%% file: main.tex
\documentclass[11pt, a4paper, oneside]{thbook}

\usepackage[square, numbers, comma, sort&compress]{natbib}

\hypersetup{urlcolor=blue, colorlinks=true}

\title{\btitle}

\usepackage{amssymb,amsmath}
\usepackage[utf8]{inputenc}
\usepackage{mathrsfs}
\usepackage[matrix,arrow,curve]{xy}
\usepackage{enumitem}

\newcommand{\projtens}{\mathbin{\widehat{\otimes}}}
\newcommand{\convol}{\ast}
\newcommand{\projmodtens}[1]{\mathbin{\widehat{\otimes}}_{#1}}
\newcommand{\isom}[1]{\mathop{\mathbin{\cong}}\limits_{#1}}

\begin{document}

% Use roman page numbering style (i, ii, iii, iv...) for the pre-content pages
\frontmatter

% Line spacing of 1.3
\setstretch{1.3}

% Define the page headers using the FancyHdr
% package and set up for one-sided printing
% Clears all page headers and footers
\fancyhead{}

% Sets the right side header to show the page number
\rhead{\thepage}

% Clears the left side page header
\lhead{}

% Finally, use the "fancy" page style to implement the FancyHdr headers
\pagestyle{fancy}

% New command to make the lines in the title page
\newcommand{\HRule}{\rule{\linewidth}{0.5mm}}

% PDF meta-data
\hypersetup{pdftitle={\btitle}}
\hypersetup{pdfsubject=\subjectname}
\hypersetup{pdfauthor=\authornames}
\hypersetup{pdfkeywords=\keywordnames}

%-------------------------------------------------------------------------------
%	TITLE PAGE
%-------------------------------------------------------------------------------

\begin{titlepage}

\begin{center}

\vfill\vfill\vfill\vfill\vfill\vfill

% Author name. Full example \href{http://www.johnsmith.com}{\authornames}.
% Remove the \href bracket to remove the link
\begin{flushleft}
\Large {\authornames}
\end{flushleft}

\vfill\vfill\vfill\vfill

% Book title
{\huge \bfseries \btitle}\\[0.4cm]

\vfill\vfill\vfill\vfill\vfill\vfill

% Date
{\large \today}\\[4cm]

\vfill
\end{center}

\end{titlepage}

%-------------------------------------------------------------------------------
%	QUOTATION PAGE
%-------------------------------------------------------------------------------

% No headers or footers for the following pages
\pagestyle{empty}

% Add some space to move the quote down the page a bit
\null\vfill

\textit{Never trust any result proved after 11 PM.}

\begin{flushright}
Professional secret
\end{flushright}

% Add some space at the bottom to position the quote just right
\vfill\vfill\vfill\vfill\vfill\vfill\null{}

\clearpage % Start a new page

%-------------------------------------------------------------------------------
%	ABSTRACT PAGE
%-------------------------------------------------------------------------------

% Add the "Abstract" page entry to the Contents
\addtotoc{Abstract}

% Add a gap in the Contents, for aesthetics
\abstract{\addtocontents{toc}{\vspace{1em}}}

In this book we study metric and topological versions of projectivity
injectivity and flatness of Banach modules over Banach algebras. These two
non-standard versions of Banach homology theories are studied in parallel under
unified approach.

Chapter 1 gives background required for our studies. In paragraph 1.1 we collect
all necessary facts on category theory, topology and measure theory. Paragraph
1.2 contains a brief introduction into Banach structures. Here we discuss Banach
spaces, Banach algebras and Banach modules. In paragraph 1.3 we give a short
introduction into relative Banach homology.

In chapter 2 we establish general properties of projective injective and flat
modules. In some cases we give complete characterizations of such modules.
Results of this chapter are extensively used in chapter 3 when dealing with
specific modules of analysis. Let us discuss contents of this chapter in more
detail. In paragraph 2.1 we derive basic properties of projective injective and
flat modules. We also study different
constructions that preserve homological triviality of Banach modules. These
results are used to characterize projectivity and flatness of cyclic modules. We
also give necessary conditions for projectivity of left ideals of Banach
algebras. As the consequence we describe projective ideals of commutative Banach
algebras that admit bounded approximate identities. Paragraph 2.2 is devoted to
Banach geometric properties of homologically trivial modules. We characterize
projective injective and flat annihilator modules and establish their strong
relation to projective injective and flat Banach spaces. Then we give several
examples confirming that homologically trivial module and its Banach algebra
have similar Banach geometric properties. Examples include the property of being
an $\mathscr{L}_1^g$-space, the Dunford-Pettis property and the l.u.st.
property. Paragraph 2.3 is quite short. Here we list conditions under which
projectivity injectivity and flatness are preserved under transition between
modules over algebra to modules over ideal. Finally, we give a necessary and
sufficient conditions of topological flatness of a Banach module and necessary
condition of injectivity of two-sided ideals.

In chapter 3 we apply general results to specific modules of analysis. In
paragraph 3.1 we investigate projectivity injectivity and flatness of ideals of
$C^*$-algebras. We describe projective left ideals of $C^*$-algebras and give a
criterion for injectivity of $AW^*$-algebras. These characterizations are
indispensable in description of homologically trivial modules over algebras of
bounded and compact operators on a Hilbert space. We perform similar research
for commutative case of algebras of bounded and vanishing functions on discrete
sets and locally compact Hausdorff spaces. In paragraph 3.2 we proceed to study
of standard modules of harmonic analysis. Due to specific Banach geometric
structure of convolution algebra and measure algebra we easily show that most of
standard modules of harmonic analysis are homologically non-trivial. The most
intriguing result of the paragraph is non-projectivity of convolution algebra of
a non-discrete group.

% Start a new page
\clearpage

%-------------------------------------------------------------------------------
%	ACKNOWLEDGEMENTS
%-------------------------------------------------------------------------------

% % Reset the line-spacing to 1.3 for body text (if it has changed)
% \setstretch{1.3}

% % Add a gap in the Contents, for aesthetics
% \acknowledgements{\addtocontents{toc}{\vspace{1em}}

% Firstly, I would like to thank ...

% % Start a new page
% \clearpage

%-------------------------------------------------------------------------------
%	LIST OF CONTENTS/FIGURES/TABLES PAGES
%-------------------------------------------------------------------------------

% The page style headers have been "empty" all this time, now use the "fancy"
% headers as defined before to bring them back
\pagestyle{fancy}

% Set the left side page header to "Contents"
\lhead{\emph{Contents}}

% Write out the Table of Contents
\tableofcontents

\mainmatter{}

% Return the page headers back to the "fancy" style
\pagestyle{fancy}

% Include the chapters of the book as separate files from the Chapters folder
% Uncomment the lines as you write the chapters

\input{./chapters/chapter1_preliminaries}
\input{./chapters/chapter2_general_theory}
\input{./chapters/chapter3_applications_to_algebras_of_analysis}

%-------------------------------------------------------------------------------
%	CONTENT - APPENDICES
%-------------------------------------------------------------------------------

% % Add a gap in the Contents, for aesthetics
% \addtocontents{toc}{\vspace{2em}}

% % Cue to tell LaTeX that the following 'chapters' are Appendices
% \appendix

% Include the appendices of the book as separate files from
% the Appendices folder. Uncomment the lines as you write the Appendices

% \input{./Appendices/AppendixA}

% \addtocontents{toc}{\vspace{2em}} % Add a gap in the Contents, for aesthetics

%-------------------------------------------------------------------------------
%	SYMBOLS
%-------------------------------------------------------------------------------

% % Start a new page
% \clearpage

% % Set the left side page header to "Symbols"
% \lhead{\emph{Symbols}}

% % Include a list of Symbols (a three column table)
% \listofnomenclature{ll}
% {
%     $a$ &
%     distance \\
%     $P$ &
%     power  \\
%     Symbol &
%     Description \\
% }

%-------------------------------------------------------------------------------
%	BIBLIOGRAPHY
%-------------------------------------------------------------------------------

\backmatter{}\label{Bibliography}

% Change the page header to say "Bibliography"
\lhead{\emph{Bibliography}}

% Use the "unsrtnat" BibTeX style for formatting the Bibliography
\bibliographystyle{unsrtnat}

% The references (bibliography) information are stored
% in the file named "Bibliography.bib"
\bibliography{main}

\end{document}

%% file: chapters/chapter1_preliminaries.tex
% chktex-file 35
% Chapter Template

% Main chapter title
% Change X to a consecutive number; for referencing 
% this chapter elsewhere, use~\ref{ChapterX}
\chapter{Preliminaries}\label{ChapterPreliminaries} 

% Change X to a consecutive number; this is for the header 
% on each page - perhaps a shortened title
\lhead{Chapter 1. \emph{Preliminaries}} 

In what follows, we present some parts in parallel fashion by listing the
respective options in order, enclosed and separate like this:
$\langle$~$\ldots$ / $\ldots$~$\rangle$. For example: a real number $x$ is
$\langle$~positive / non negative~$\rangle$ if $\langle$~$x>0$ / $x\geq
0$~$\rangle$. Sometimes one of the parts might be empty. We use symbol $:=$ for
equality by definition.

We use the following standard notation for some commonly used sets of numbers:
$\mathbb{C}$ denotes the complex numbers, $\mathbb{R}$ denotes the real numbers,
$\mathbb{Z}$ denotes the integers, $\mathbb{N}$ denotes the natural numbers,
$\mathbb{N}_n$ denotes the set of first $n$ natural numbers, $\mathbb{R}_+$
denotes the set of non negative real numbers, $\mathbb{T}$ denotes the set of
complex numbers of modulus $1$, finally, $\mathbb{D}$ denotes the set of complex
numbers with modulus less than $1$. For $z\in\mathbb{C}$ the symbol
$\overline{z}$ stands for the complex conjugate number.

For a given map $f:M\to M'$ and subset $\langle$~$N\subset M$ / $N'\subset M'$
such that $\operatorname{Im}(f)\subset N'$~$\rangle$ by $\langle$~$f|_N$ /
$f|^{N'}$~$\rangle$ we denote the  $\langle$~restriction of $f$ onto $N$ /
corestriction of $f$ onto $N'$~$\rangle$, that is $\langle$~$f|_N:N\to
M':x\mapsto f(x)$ / $f|^{N'}:M\to N':x\mapsto f(x)$~$\rangle$. The indicator
function of a subset $N$ of the set $M$ is denoted by $\chi_{N}$, so that
$\chi_N(x)=1$ for $x\in N$ and $\chi_N(x)=0$ for $x\in M\setminus N$. We also
use the shortcut $\delta_x=\chi_{ \{x \}}$ where $x\in M$. By $\mathcal{P}(M)$
we denote the set of all subsets of $M$, and $\mathcal{P}_0(M)$ stands for the
set of all finite subsets of $M$. The symbol $M^N$ stands for the set of all
functions from $N$ to $M$. By $\operatorname{Card}(M)$ we denote the cardinality
of $M$. By $\aleph_0$ we denote the cardinality of $\mathbb{N}$.

%-------------------------------------------------------------------------------
%    Broad foundations
%-------------------------------------------------------------------------------

\section{
  Broad foundations
}\label{SectionBroadFoundations} 

%-------------------------------------------------------------------------------
%    Categorical language
%-------------------------------------------------------------------------------

\subsection{
  Categorical language
}\label{SubSectionCategoricalLanguage}

Here we recall some basic facts and definitions from category theory and fix
notation we shall use. We assume that our reader is familiar with such basics of
category theory as category, functor, morphism. Otherwise see
[\cite{HelLectAndExOnFuncAn}, chapter 0] for a quick introduction or
[\cite{KashivShapCatsAndSheavs}, chapter 1] for more details.

For a given category $\mathbf{C}$ by $\operatorname{Ob}(\mathbf{C})$ we denote
the class of its objects. The symbol $\mathbf{C}^o$ stands for the opposite
category. For a given objects $X$ and $Y$ by 
$\operatorname{Hom}_{\mathbf{C}}(X, Y)$ 
we denote the set of morphisms from $X$ to $Y$. Often we shall write
$\phi:X\to Y$ instead of $\phi\in\operatorname{Hom}_{\mathbf{C}}(X,Y)$. A
morphism $\phi:X\to Y$ is called $\langle$~retraction / coretraction~$\rangle$
if it has a $\langle$~right / left~$\rangle$ inverse morphism. Morphism $\phi$
is called an isomorphism if it is retraction and coretraction. Usually we shall
express existence of isomorphism between $X$ and $Y$ as $X\isom{\mathbf{C}} Y$.
We say that two morphisms $\phi:X_1\to Y_1$ and $\psi:X_2\to Y_2$ are equivalent
in $\mathbf{C}$ if there exist isomorphisms $\alpha:X_1\to X_2$ and
$\beta:Y_1\to Y_2$ such that $\beta\phi=\psi\alpha$.

The first obvious example of the category that comes to mind is the category of
all sets and all maps between them. We denote this category by $\mathbf{Set}$.
Other examples will be given later. Two main examples of functors that any
category has are functors of morphisms. For a fixed
$X\in\operatorname{Ob}(\mathbf{C})$ we define covariant and a contravariant
functors
$$
\operatorname{Hom}_{\mathbf{C}}(X,-)
:\mathbf{C}\to\mathbf{Set}:Y\mapsto \operatorname{Hom}_{\mathbf{C}}(X,Y), 
\phi\mapsto(\psi\mapsto \phi\psi),
$$
$$
\operatorname{Hom}_{\mathbf{C}}(-,X)
:\mathbf{C}\to\mathbf{Set}:Y\mapsto \operatorname{Hom}_{\mathbf{C}}(Y,X), 
\phi\mapsto(\psi\mapsto \psi\phi).
$$
This construction has its reminiscent analogs in many categories of mathematics
with slight modification of categories between which these functors act.

We say that two covariant functors $F:\mathbf{C}\to\mathbf{D}$,
$G:\mathbf{C}\to\mathbf{D}$ are isomorphic if there exists a class of
isomorphisms $ \{\eta_X:X\in\operatorname{Ob}(\mathbf{C}) \}$ in $\mathbf{D}$
(called natural isomorphisms), such that $G(f)\eta_X=\eta_Y F(f)$ for all
$f:X\to Y$. In this case we simply write $F\cong G$. A $\langle$~covariant /
contravariant~$\rangle$ functor $F:\mathbf{C}\to\mathbf{D}$ is called
representable by object $X$ if
$\langle$~$F\cong\operatorname{Hom}_{\mathbf{C}}(X,-)$ /
$F\cong\operatorname{Hom}_{\mathbf{C}}(-,X)$~$\rangle$. If functor is
representable, then its representing object is unique up to isomorphism in
$\mathbf{C}$.

Constructions of categorical product and coproduct shall play an important role
in this book. We say that $X$ is the the $\langle$~product / coproduct~$\rangle$
of the family of objects $ \{X_\lambda:\lambda\in\Lambda \}$ if the functor
$\langle$~$
\prod_{\lambda\in\Lambda}\operatorname{Hom}_{\mathbf{C}}(-,X_{\lambda})
:\mathbf{C}\to\mathbf{Set}$
/
$
\prod_{\lambda\in\Lambda}\operatorname{Hom}_{\mathbf{C}}(X_{\lambda},-)
:\mathbf{C}\to\mathbf{Set}$~$\rangle$
is representable by object $X$. As the consequence we get that a
$\langle$~product / coproduct~$\rangle$, if it exists, is unique up to an
isomorphism. Later we shall give examples of the $\langle$~products /
coproducts~$\rangle$ in different categories of functional analysis.

%-------------------------------------------------------------------------------
%    Topology
%-------------------------------------------------------------------------------

\subsection{
  Topology
}\label{SubSectionTopology}

Let $(S,\tau)$ be a topological space. Elements of $\tau$ are called open sets,
and their complements are called closed sets. Let $E$ be an arbitrary subset of
$S$. By $\operatorname{cl}_S(E)$ we denote the closure of $E$ in $S$, that is
the smallest closed set that contains $E$. Similarly, by
$\operatorname{int}_S(E)$ we denote the interior of $E$, that is the largest
open set that contained in $E$. We say that $E$ is a neighborhood of point 
$s\in S$ if $s\in \operatorname{int}_S(E)$. We say that a set $E$ is dense 
in a set $F$ if $F\subset\operatorname{cl}_S(E)$. Note that $E$ may be regarded
as topological space, if we endow it with subspace topology which equals $
\{U\cap E:U\in\tau \}$. Topological space is called Hausdorff if any two
distinct points have disjoint open neighborhoods. In this book we shall work
with Hausdorff spaces only.

A map $f:X\to Y$ between topological spaces is called continuous if preimage
under $f$ of any open set is open. By $\mathbf{Top}$ we denote the category of
topological spaces with continuous maps in the role of morphisms. Isomorphisms
in $\mathbf{Top}$ are called homeomorphisms. The category $\mathbf{Top}$ admits
products. For a given family of topological spaces $
\{S_\lambda:\lambda\in\Lambda \}$ their product is the Tychonoff product
$\prod_{\lambda\in\Lambda}S_\lambda$, that is the Cartesian product of the
family $ \{S_\lambda:\lambda\in\Lambda \}$ with the coarsest topology making all
natural projections $p_\lambda:\prod_{\lambda\in\Lambda}S_\lambda\to S_\lambda$
continuous.

In this book we shall mainly deal with four types of topological spaces: compact
spaces, paracompact spaces, locally compact spaces and extremely disconnected
spaces. Before giving their definitions we need to recall the notion of cover.
Let $\mathcal{E}$ be a family of subsets of topological space $S$. We say that
$\mathcal{E}$ is a cover if its union equals $S$. We say that cover is open if
all its elements are open sets. A cover is called locally finite if any point of
$S$ has a neighborhood that intersects only finitely many elements of the cover.
We say that cover $\mathcal{E}_1$ is inscribed into cover $\mathcal{E}_2$ if any
element of $\mathcal{E}_1$ is a subset of some element of $\mathcal{E}_2$. A
cover $\mathcal{E}_1$ is called a subcover of $\mathcal{E}_2$ if
$\mathcal{E}_1\subset\mathcal{E}_2$. Finally, a topological space is called 
\begin{enumerate}[label = (\roman*)]
  \item compact if any its open cover admits a finite open subcover; 

  \item paracompact if any its open cover is inscribed into some locally finite
  open cover;
  
  \item pseudocompact if any locally finite family of open sets is finite;

  \item locally compact if any its point has a compact neighborhood;

  \item extremely disconnected spaces if the closure of any its open set is
  open;

  \item Stonean if it is an extremely disconnected Hausdorff compact space.
\end{enumerate}

The property of being $\langle$~compact / paracompact / locally
compact~$\rangle$ space is preserved by $\langle$~closed / closed / open and
closed~$\rangle$ subspaces. A topological space is pseudocompact iff any
continuous function on this space is bounded. Any non compact locally compact
Hausdorff space $S$ can be regarded as dense subspace of some compact Hausdorff
space. There is the smallest and the largest such compactification. The smallest
one is called the Alexandroff compactification $\alpha S$.  By definition
$\alpha S:=S\cup  \{S \}$. A subset of $\alpha S$ is called open if it is an
open subset of $S$ or has the form $ \{S \}\cup S\setminus K$ for some compact
set $K\subset S$. The largest compactification $\beta S$ is called the
Stone-Cech compactification. It may be represented as the image of the embedding
$j:S\to\prod_{f\in C}[0,1]:s\mapsto \prod_{f\in C}f(s)$, where $C$ is a set of
all continuous maps from $S$ to $[0,1]$. Stone-Cech compactification is highly
non constructive. Even $\beta\mathbb{N}$ has no explicit description, though it
is known that $\beta\mathbb{N}$ is an extremely disconnected Hausdorff compact.

Occasionally we shall apply the Urysohn's lemma to locally compact Hausdorff
spaces. It states that for any compact subset $K$ of open set $V$ in a locally
compact Hausdorff space $S$ there exists a continuous function $f:S\to [0,1]$
such that $f|_K=1$ and $f|_{S\setminus V}=0$. 

For more details on topological spaces see comprehensive
treatise~\cite{EngelGenTop}. 

%-------------------------------------------------------------------------------
%    Filters, nets and limits
%-------------------------------------------------------------------------------

\subsection{
  Filters, nets and limits
}\label{SubSectionFiltersNetsAndLimits} 

We will use two generalizations of the notion of the sequence and the limit of
the sequence.

A family $\mathfrak{F}$ of subsets of the set $M$ is called a filter on a set
$M$ if $\mathfrak{F}$ doesn't contain the empty set, $\mathfrak{F}$ is closed
under finite intersections and $\mathfrak{F}$ contains all supersets of its
elements. In general filters are too large to be described explicitly. To
overcome this difficulty we shall use filterbases. A non empty family
$\mathfrak{B}$ of subset of a set $M$ is called a filterbase on a set $M$ if
$\mathfrak{B}$ doesn't contain empty set and the intersection of any two
elements of $\mathfrak{B}$ contains some element of $\mathfrak{B}$. Given a
filterbase $\mathfrak{B}$ we can construct a filter by adding to $\mathfrak{B}$
all supersets of elements of $\mathfrak{B}$.

We say that filter $\mathfrak{F}_1$ dominates filter $\mathfrak{F}_2$ if
$\mathfrak{F}_2\subset\mathfrak{F}_1$. Therefore the set of all filters on a
given set is partially ordered set. Filters that are maximal with respect to
this order are called ultrafilters. An easy application of Zorn's lemma gives
that any filter is dominated by some ultrafilter.

Let $\mathfrak{F}$ be a filter on  a set $M$, and $\phi:M\to S$ be a map from
$M$ to the Hausdorff topological space $S$. We say that $x$ is a limit of $\phi$
along $\mathfrak{F}$ and write $x=\lim_{\mathfrak{F}} \phi(m)$ if for every open
neighborhood $U$ of $x$ holds $\phi^{-1}(U)\in\mathfrak{F}$. Directly from the
definition it follows that if $\phi$ has a limit along $\mathfrak{F}$ then it
has the same limit along any filter that dominates $\mathfrak{F}$. 

Limit along filter preserve order structure of $\mathbb{R}$. More precisely: if
two functions $\phi:M\to\mathbb{R}$ and $\psi:M\to\mathbb{R}$ have limits along
filter $\mathfrak{F}$ and $\phi\leq\psi$, then 
$$
\lim_{\mathfrak{F}}\phi(m)\leq\lim_{\mathfrak{F}}\psi(m).
$$

Limit along filter respects continuous functions. Rigorously this formulates as
follows. Assume for each $\lambda\in\Lambda$ a function $\phi_\lambda:M\to
S_\lambda$ has a limit along filter $\mathfrak{F}$, then for any continuous
function $g:\prod_{\lambda\in\Lambda}S_\lambda\to Y$ holds
$$
\lim_{\mathfrak{F}}g\left(
  \prod_{\lambda\in\Lambda}\phi_\lambda(m)
\right)
=g\left(\prod_{\lambda\in\Lambda}\lim_{\mathfrak{F}}\phi_\lambda(m)\right).
$$
In particular limit along filter is linear and multiplicative. Just like
ordinary sequences.

The most important feature of filters and the reason of our interest is the
following: if $\mathfrak{U}$ is an ultrafilter on the set $M$ and $\phi:M\to K$
is a function with values in the compact Hausdorff space $K$, then
$\lim_{\mathfrak{U}}\phi(m)$ exists. In particular, we always can speak of
limits along ultrafilters of bounded scalar valued functions.

Another approach to the generalization of the notion of the limit is a limit of
the net. A directed set is a partially ordered set $(N,\leq)$ in which any two
elements have upper bound. Every directed set gives rise to the so called
section filter, whose filterbase consist of so called sections $
\{\nu':\nu\leq\nu' \}$ for some $\nu\in N$. Any function $x:N\to X$ from the
directed set $(N,\leq)$ into the topological space $X$ is called a net. Usually
it is denoted as ${(x_\nu)}_{\nu\in N}$ to allude to sequences. 
A limit of the net $x:N\to X$ is a limit of the function $x$ along section
filter of the directed set $N$. It is denoted $\lim_\nu x_\nu$. We shall exploit
both notions of the limit.

More on this matters can be found in [\cite{BourbElemMathGenTopLivIII}, section
7].

%-------------------------------------------------------------------------------
%    Measure theory basics
%-------------------------------------------------------------------------------

\subsection{
  Measure theory
}\label{SubSectionMeasureTheory}

A family $\Sigma$ of subsets of the set $\Omega$ is called a $\sigma$-algebra if
it contains an empty set, contains complements of all its elements and closed
under countable unions. If $\Sigma$ is a $\sigma$-algebra of subsets of $\Omega$
we call $(\Omega,\Sigma)$ a measurable subspace. Elements of $\Sigma$ are called
measurable sets. 

A function $\mu:\Sigma\to[0,+\infty]$ such that:  
\begin{enumerate}[label = (\roman*)]
  \item $\mu(\varnothing)=0$; 

  \item $\mu\left(\bigcup\limits_{n\in\mathbb{N}} E_n\right)
  =\sum\limits_{n\in\mathbb{N}}\mu(E_n)$ for any family of disjoint sets
  ${(E_n)}_{n\in\mathbb{N}}$ in $\Sigma$; 
\end{enumerate}

is called a measure. The triple $(\Omega,\Sigma,\mu)$ is called a measure space.
If $\mu$ attains only finite values we may drop the first condition. The second
condition is essential and called the $\sigma$-additivity. The simplest example
of measure space is an  arbitrary set $\Lambda$ with $\sigma$-algebra of all
subsets and so called counting measure
$\mu_c:\mathcal{P}(\Lambda)\to[0,+\infty]$. By definition $\mu_c(E)$ equals
$\operatorname{Card}(E)$ if $E$ is finite and $+\infty$ otherwise. If $E$ is a
measurable set, by $\Sigma|_E$ we denote the $\sigma$-algebra $ \{F\cap
E:F\in\Sigma \}$ and by $\mu|_E$ we denote the restriction of $\mu$ to
$\Sigma|_E$. A set $E$ in $\Omega$ is called negligible if there exists a
measurable set $F$ of measure $0$ that contains $E$. Similarly, a set $E$ in
$\Omega$ is conegligible if $\Omega\setminus E$ is negligible. Let $P$ be some
property that depends on points of $\Omega$. We say that $P$ holds almost
everywhere if the set where $P$ is violated is negligible. A measure space is
called $\sigma$-finite if there exists a countable family of measurable sets of
finite measure whose union is the whole space. The class of $\sigma$-finite
spaces is enough for most applications but we shall encounter a more generic
measure spaces.

A measurable space $(\Omega,\Sigma,\mu)$ is called strictly localizable if there
exists a  family of disjoint measurable subsets 
$ \{E_\lambda:\lambda\in\Lambda \}$ of finite measure such that: 
\begin{enumerate}[label = (\roman*)]
  \item $\bigcup_{\lambda\in\Lambda}E_\lambda=\Omega$;

  \item $E$ is measurable iff $E\cap E_\lambda$ is measurable for all
  $\lambda\in\Lambda$;

  \item for any measurable $E$ holds 
  $\mu(E)=\sum_{\lambda\in\Lambda}\mu(E\cap E_\lambda)$. 
\end{enumerate}
  
The class of strictly localizable measure spaces is huge. It includes all
$\sigma$-finite measure spaces, their arbitrary unions, Haar measures of
locally compact groups, counting measures and much more. In what follows we
shall consider only strictly localizable measure spaces.

We shall exploit a more detailed classification of measure spaces. We say that a
measurable set $E$ is an atom if $\mu(E)>0$ and for any measurable subset $F$ of
$E$ either $F$ or $E\setminus F$ is negligible. Directly from the definition it
follows that that all atoms of strictly localizable measure spaces are of finite
measure. In general an atom may not be a mere singleton.

We say that a measure space is non atomic if there is no atoms for its measure.
A measure space is called purely atomic if every measurable set of positive
measure contains an atom. A straightforward application of Zorn's lemma gives
that a purely atomic measure space can be represented as disjoint union of some
family of atoms. This family is countable if measure space is $\sigma$-finite.
These facts allow us to say that the structure of purely atomic measure space is
well understood. The structure of strictly localizable non atomic measure spaces
is given by Maharam's theorem [\cite{FremMeasTh}, 332B]. 

For completeness we shall say a few words on constructions with measures. The
product measure of two measure spaces $(\Omega_1,\Sigma_1,\mu_1)$ and
$(\Omega_2,\Sigma_2,\mu_2)$ we denote by $\mu_1\times \mu_2$. The definition of
product measure for localizable measure spaces is rather involved
[\cite{FremMeasTh}, definition 251F] and we don't give it here.  For our
purposes it is enough to know that the product of two strictly localizable
measure spaces is again strictly localizable [\cite{FremMeasTh}, proposition
251N]. By direct sum of measure spaces $ \{(\Omega_\lambda, \Sigma_\lambda,
\mu_\lambda):\lambda\in\Lambda \}$ we denote the disjoint union of sets 
$\{\Omega_\lambda:\lambda\in\Lambda \}$ with $\sigma$-algebra defined as 
$\Sigma= \{
  E\subset \Omega: E\cap E_\lambda\in\Sigma_\lambda
  \mbox{ for all }\lambda\in\Lambda
 \}$ 
and measure given by the formula 
$\mu(E)=\sum_{\lambda\in\Lambda}\mu_\lambda(E\cap E_\lambda)$. It is clear now
that strictly localizable measure space are exactly direct sums of finite
measure spaces.

Assume $(\Omega,\Sigma,\mu)$ is a $\sigma$-finite measure space, then there
exists a purely atomic measure $\mu_1:\Sigma\to[0,+\infty]$ and a non atomic
measure $\mu_2:\Sigma\to[0,+\infty]$ such that $\mu=\mu_1+\mu_2$. Even more
there exist measurable sets $\Omega_a^{\mu}$ and
$\Omega_{na}^{\mu}=\Omega\setminus \Omega_a^{\mu}$ such that
$\mu_1(\Omega_{na}^{\mu})=\mu_2(\Omega_a^{\mu})=0$. The sets $\Omega_a^{\mu}$
and $\Omega_{na}^{\mu}$ are called respectively the atomic and the non atomic
parts of the measure space $(\Omega,\Sigma,\mu)$.

By measurable function we always mean a complex or real valued function on
measurable space, with the property that preimage of every open set is
measurable. We say that two measurable functions are equivalent if the set where
they are different is negligible. If $f:\Omega\to\mathbb{R}$ is an  integrable
function on $(\Omega,\Sigma,\mu)$, then we may define a new measure 
$$
f\mu:\Sigma\to[0,+\infty]:E\mapsto\int_{E}f(\omega)d\mu(\omega).
$$

The notion of measure can be extended by changing the range of values that the
measure can attain. Any $\sigma$-additive function $\mu:\Sigma\to\mathbb{C}$ on
a measurable space $(\Omega,\Sigma)$ is called a complex measure. Any complex
measure $\mu$ can be represented as $\mu=\mu_1-\mu_2+i(\mu_3-\mu_4)$, where
$\mu_1,\mu_2,\mu_3,\mu_4$ --- are finite measures. As the consequence every
complex measure is finite and therefore we have a well defined total variation
measure:
$$
|\mu|:\Sigma\to\mathbb{R}_+
:E\mapsto\sup\left \{\sum_{n\in\mathbb{N}}|\mu(E_n)|
: \{E_n:n\in\mathbb{N} \}\subset\Sigma -\mbox{partition of }E\right \}
$$

Let $\mu$ and $\nu$ be two measures on a measurable space $(\Omega,\Sigma)$. We
say that $\mu$ and $\nu$ are mutually singular and write $\mu\perp\nu$ if there
exists a measurable set $E$ such that $\mu(E)=\nu(\Omega\setminus E)=0$. The
opposite property is absolute continuity. We say that $\nu$ is absolutely
continuous with respect to $\mu$ and write $\nu\ll\mu$ if $\nu(E)=0$ for every
measurable set $E$ with $\mu(E)=0$. In general, two measures may neither be
absolutely continuous nor singular with respect to each other. We have a
Lebesgue decomposition theorem for this case. For a given two $\sigma$-finite
measures $\mu$ and $\nu$ on a measurable space $(\Omega,\Sigma)$ there exists a
measurable function $\rho_{\nu,\mu}:\Omega\to\mathbb{C}$, a $\sigma$-finite
measure $\nu_s:\Sigma\to[0,+\infty]$ and two measurable sets
$\Omega_s^{\nu,\mu}$, $\Omega_c^{\nu,\mu}=\Omega\setminus\Omega_s^{\nu,\mu}$
such that $\nu=\rho_{\nu,\mu}\mu+\nu_s$ and
$\mu(\Omega_s^{\nu,\mu})=\nu_s(\Omega_c^{\nu,\mu})=0$, i.e. $\mu\perp\nu_s$.

Finally we shall say a few words on measures defined on topological spaces.
Given a topological space $S$ we may consider the minimal $\sigma$-algebra
containing all open subsets of $S$. It is called the Borel $\sigma$-algebra of
$S$ and denoted by $Bor(S)$. Measures and complex measures defined on Borel
$\sigma$-algebras are supported with adjective Borel.

The support of a complex Borel measure $\mu$ is a set of all points
$s\in S$ for which every open neighborhood of $s$ has positive measure. We
denote the set of such points by $\operatorname{supp}(\mu)$. The support is
always closed. A Borel measure $\mu$ is called
\begin{enumerate}[label = (\roman*)]
  \item strictly positive if $\operatorname{supp}(\mu)=S$;

  \item a mesure with full support 
  if $\mu(S\setminus\operatorname{supp}(\mu))=0$;

  \item locally finite if point has an open neighborhood of finite measure; 

  \item inner regular with respect to 
  class $\mathcal{C}$ if for any Borel set $E$ holds 
  $$
  \mu(E)=\sup \{\mu(K): K\subset E, K\in\mathcal{C} \};
  $$

  \item outer regular with respect to 
  class $\mathcal{C}$ if for any Borel set $E$ holds 
  $$
  \mu(E)=\inf \{\mu(K): K\subset E, K\in\mathcal{C} \};
  $$

  \item residual if $\mu(E)=0$ for any Borel nowhere dense set $E$;

  \item normal if its residual and has full support.
\end{enumerate}

Usually we consider regularity of measures with respect to classes of open and
compact sets. Now recall a few facts on topological measures. All compact sets 
have finite measure if measure is locally finite. Any finite inner compact 
regular measure is outer open regular. Any finite inner compact regular and
inner open regular measure is normal. 

We say that a complex Borel measure $\mu:Bor(S)\to\mathbb{C}$ defined on a 
locally compact Hausdorff space $S$ is regular if it is inner compact regular.  

Most of the results and definitions in this section can be found in the first,
second and the fourth volumes of~\cite{FremMeasTh}.

%-------------------------------------------------------------------------------
%    Banach spaces, algebras and modules
%-------------------------------------------------------------------------------

\section{
  Banach structures
}\label{SectionBanachStructures}

%-------------------------------------------------------------------------------
%    Banach spaces
%-------------------------------------------------------------------------------

\subsection{
  Banach spaces
}\label{SubSectionBanachSpaces}

We assume that our reader is familiar with fundamentals of functional analysis
and its constructions, otherwise consult~\cite{HelLectAndExOnFuncAn}
or~\cite{ConwACoursInFuncAn}. In this book we will highly rely on results about
geometry of Banach spaces.
See~\cite{CarothShortCourseBanSp},~\cite{KalAlbTopicsBanSpTh}
or~\cite{FabHabBanSpTh} for a quick introduction. All Banach spaces are
considered over the complex field, unless otherwise stated. 

By $\langle$~$B_E$ / $B_E^\circ$~$\rangle$ we denote the $\langle$~closed /
open~$\rangle$ unit ball of Banach space $E$ with center at zero. For a given 
set $S\subset E$ by $\operatorname{span}(S)$ we denote its linear span. 

By $E^{cc}$ we denote a Banach space with the same set of vectors as in $E$, the
same addition but with new multiplication by conjugate scalars:  $\alpha
\overline{x}:=\overline{\overline{\alpha}x}$ for $\alpha\in\mathbb{C}$ 
and $x\in E$. Note: elements of $E^{cc}$ we denote by $\overline{x}$. Clearly,
${(E^{cc})}^{cc}=E$.

If $F$ is a closed subspace of the Banach space $E$, then $E/F$ stands for 
linear quotient Banach space. Its norm defined by equality 
$$
\Vert x+F\Vert=\inf \{\Vert x+y\Vert: y\in F \}.
$$
where $x+F\in E/F$.

Now fix two Banach spaces $E$ and $F$. A map $T:E\to F$ is called conjugate
linear if the respective map $T:E^{cc}\to F$ is linear. A linear operator
$T:E\to F$ is called:
\begin{enumerate}[label = (\roman*)]
  \item bounded if its norm $\Vert T\Vert:=\sup \{\Vert T(x)\Vert:x\in B_E \}$ 
  is finite.

  \item contractive if its norm is at most $1$;

  \item compact if $T(B_E)$ is relatively compact in $F$;

  \item nuclear if it can be represented as an absolutely convergent series of
  rank one operators.
\end{enumerate}

Any nuclear operator is compact. Any compact operator is
bounded, any bounded operator is continuous. By $\langle$~$\mathcal{B}(E,F)$ /
$\mathcal{K}(E,F)$ / $\mathcal{N}(E,F)$~$\rangle$ we denote the Banach space of
$\langle$~bounded / compact / nuclear~$\rangle$ linear operators from $E$ to
$F$. If $F=E$ we use the shortcut $\langle$~$\mathcal{B}(E)$ / $\mathcal{K}(E)$
/ $\mathcal{N}(E)$~$\rangle$ for this space. The norms in $\mathcal{B}(E,F)$ and
$\mathcal{K}(E,F)$ are just the usual operator norm. The norm of a nuclear
operator $T$ is defined by equality
$$
\Vert T\Vert
:=\inf\left \{
  \sum_{n=1}^\infty\Vert S_n\Vert
  :T=\sum_{n=1}^\infty S_n,\quad 
  {(S_n)}_{n\in\mathbb{N}} - \mbox{ rank one operators}
\right \}.
$$

By $\mathbf{Ban}$ we shall denote the category of Banach spaces with bounded
linear operators in the role of morphisms, while $\mathbf{Ban}_1$ stands for the
category of Banach spaces with contractive operators in the role of morphisms.
As the consequence, $\operatorname{Hom}_{\mathbf{Ban}}(E,F)$ is just another
name for $\mathcal{B}(E,F)$.

Two Banach spaces $E$ and $F$ are $\langle$~isometrically isomorphic /
topologically isomorphic~$\rangle$ as Banach spaces if there exists a bounded
linear operator $T:E\to F$ which is both $\langle$~isometric and surjective /
topologically injective and topologically surjective~$\rangle$. The fact that
$E$ and $F$ are $\langle$~isometrically isomorphic / topologically
isomorphic~$\rangle$ Banach spaces means that
$\langle$~$E\isom{\mathbf{Ban}_1}F$ / $E\isom{\mathbf{Ban}}F$~$\rangle$. The
Banach-Mazur distance between $E$ and $F$ is defined by the formula 
$$
d_{BM}(E,F):=\inf \{
  \Vert T\Vert\Vert T^{-1}\Vert
  :T \in \mathcal{B}(E,F) \mbox{ --- a topological isomorphism}
 \}.
$$ 
If $E$ and $F$ are not topologically isomorphic the Banach-Mazur distance
between them is infinite.

By $E^*$ we denote the space of bounded linear functionals on $E$, that is
$E^*=\mathcal{B}(E, \mathbb{C})$. It is called a dual space of $E$. Similarly
one can define the second and higher duals of $E$. An important corollary of the
Hahn-Banach theorem says that a bounded linear operator 
$$
\iota_E:E\to E^{**}:x\mapsto (f\mapsto f(x))
$$ 
is isometric. We call this operator the natural embedding of $E$ into its 
second dual $E^{**}$.

For a given bounded linear operator $T:E\to F$ its adjoint is bounded linear 
operator 
$$
T^*:F^*\to E^*:f\mapsto (x\mapsto f(T(x))).
$$
Again from Hahn Banach 
theorem it follows that $\Vert T^*\Vert=\Vert T\Vert$. One can easily check 
that $T^{**}\iota_E=\iota_F T$.  

For a given subspace $F$ of $\langle$~$E$ / $E^*$~$\rangle$ we define 
its $\langle$~right / left~$\rangle$ annihilator 
$\langle$~$F^\perp$ / ${}^\perp F$~$\rangle$ is defined as 
$\langle$~$ \{ f\in E^*: f(x)=0 \mbox{ for all } x\in F \}$ / 
$ \{ x\in E : f(x)=0 \mbox{ for all } f\in F \}$~$\rangle$. 
Clearly, $E^{\perp} = \{ 0\}$, $0^{\perp}=E^*$ 
and ${}^{\perp}E^*= \{ 0\}$, ${}^{\perp}0= E$. Annihilators are closely related
with quotient spaces. One can show that 
operators 
$$
i: F^*\to E^*/F^{\perp}: f\mapsto f|_F+F^\perp,
\quad\quad 
q:{(E/F)}^*\to F^{\perp}: f\mapsto (x\mapsto f(x + F))
$$ 
are isometric isomorphisms.

A few words on classification of bounded linear operators. A bounded linear
operator $T:E\to F$ is called:
\begin{enumerate}[label = (\roman*)]
  \item topologically injective if it performs homeomorphism on its image;

  \item topologically surjective if it is an open map;

  \item coisometric if it maps open unit ball onto open unit ball;

  \item strictly coisometric if it maps closed unit ball onto closed unit ball; 

  \item $c$-topologically injective, if $\Vert x\Vert\leq c\Vert  T(x)\Vert$ for
  all $x\in E$;

  \item $c$-topologically surjective, if $cT(B_E^\circ)\supset B_F^\circ$;

  \item strictly $c$-topologically surjective, if $cT(B_E)\supset B_F$.
\end{enumerate}

Note that $T$ is topologically $\langle$~injective / surjective~$\rangle$ iff it
is $c$-topologically $\langle$~injective / surjective~$\rangle$ for some $c>0$.
Obviously $\langle$~coisometric / strictly coisometric~$\rangle$ operators are
exactly contractive $\langle$~$1$-topologically surjective / strictly
$1$-topologically surjective~$\rangle$ operators. These classes of operators 
behave nicely with respect to taking adjoints:

\begin{enumerate}[label = (\roman*)]
  \item if $ T$ is $c$-topologically surjective or strictly $c$-topologically 
  surjective, then $ T^*$ is $c$-topologically injective;

  \item if $ T$ $c$-topologically injective, then $ T^*$ is strictly
  $c$-topologically surjective;

  \item if $ T^*$ $c$-topologically surjective or strictly $c$-topologically 
  surjective, then $ T$ is $c$-topologically injective;

  \item if $T^*$ $c$-topologically injective and $E$ is complete, then $T$ 
  is $c$-topologically surjective.
\end{enumerate}

Let us discuss relation between Banach spaces and their subspaces. A bounded 
linear operator $T:E\to F$ is called a $\langle$~$c$-retraction /
$c$-coretraction~$\rangle$ if there exist a bounded linear operator $S:F\to E$
such that $\langle$~$T S=1_F$ / $S T=1_E$~$\rangle$ 
and $\Vert T\Vert\Vert S\Vert\leq c$. In this case we say 
that $\langle$~$F$ / $E$~$\rangle$ has a
$c$-complemented copy in $E$ via topologically $\langle$~injective /
surjective~$\rangle$ operator $T$. If the natural embedding of a subspace $F$
into the ambient space $E$ is a $c$-coretraction, then we say that $F$ is
$c$-complemented in $E$. Complemented subspaces can be described in terms of so
called projections. A bounded linear operator $P:E\to E$ is called a projection 
of $E$ onto $F$ if $P^2=P$, $\operatorname{Im}(P)=F$. One can show that 
$F$ is $c$-complemented in $E$ iff there is a projection $P$ from $E$ onto $F$ 
with $\Vert P\Vert\leq c$. We say that $F$ is weakly $c$-complemented in $E$
if $F^*$ is $c$-complemented in $E^*$ via the map $i^*$, which is the adjoint of
the natural embedding. If $F$ is $c$-complemented in $E$ then it is 
weakly $c$-complemented in $E$.  The term ``contractively complemented'' will 
be a synonym for $1$-complemented. Sometimes we will omit a constant $c$ and 
simply say that one Banach space $\langle$~has a complemented copy / 
is complemented~$\rangle$ inside the other. 

All finite dimensional subspaces are complemented, but not necessarily
contractively complemented. An example of contractively complemented subspace is
the following: consider arbitrary Banach space $E$, then $E^*$ is contractively
complemented in $E^{***}$ via Dixmier projection $P=\iota_{E^*}{(\iota_E)}^*$. A
canonical example of uncomplemented subspace is $c_0(\mathbb{N})$ in
$\mathbb{\ell_\infty}(\mathbb{N})$ [\cite{KalAlbTopicsBanSpTh}, theorem 2.5.5].
Nevertheless $c_0(\mathbb{N})$ is weakly complemented in
$\ell_\infty(\mathbb{N})$ by Dixmier projection.

Let $E$, $F$ and $G$ be three Banach spaces, then a bilinear operator
$\Phi:E\times F\to G$ is called bounded if its norm 
$$
\Vert \Phi\Vert:=\sup \{\Vert \Phi(x,y)\Vert:x\in B_E, y\in B_F \}
$$ 
is finite. 
The Banach space of all bounded bilinear operators on $E\times F$ with 
values in $G$ is denoted by $\mathcal{B}(E\times F,G)$.

Now consider the algebraic tensor product $E\otimes F$ of Banach spaces $E$ and
$F$. This linear space can be endowed with different norms, but the most
important is the projective norm. For $u\in E\otimes F$ we define its projective
norm as
$$
\Vert u\Vert
:=\inf\left \{
  \sum_{i=1}^n \Vert x_i\Vert\Vert y_i\Vert
  :u=\sum_{i=1}^n x_i\otimes y_i, 
  {(x_i)}_{i\in\mathbb{N}_n}\subset E, 
  {(y_i)}_{i\in\mathbb{N}_n}\subset F
\right \}
$$
It is indeed a norm, but not complete in general. The symbol $E\projtens F$
stands for the completion of $E\otimes F$ under projective norm. We call the
resulting completion the projective tensor product of Banach spaces $E$ and $F$.
Let $T:E_1\to E_2$ and $S:F_1\to F_2$ be two bounded linear operators between
Banach spaces, then there exists a unique bounded linear operator 
$$
T\projtens S:E_1\projtens F_1\to E_2\projtens F_2
$$ 
such that  $(T\projtens S)(x\projtens y)=T(x)\projtens S(y)$ 
for all $x\in E_1$ and $y\in F_1$. 
Even more $\Vert T\projtens S\Vert=\Vert T\Vert\Vert S\Vert$. 
The main feature of projective
tensor product which makes it so important is the following universal property:
for any Banach spaces $E$, $F$ and $G$ there is a natural isometric isomorphism:
$$
\mathcal{B}(E\projtens F,G)\isom{\mathbf{Ban}_1}\mathcal{B}(E\times F,G)
$$
In other words, projective tensor product linearizes bounded bilinear operators.
Also we have the following two (natural in $E$, $F$ and $G$) isometric
isomorphisms:
$$
\mathcal{B}(E\projtens F,G)
\isom{\mathbf{Ban}_1}\mathcal{B}(E,\mathcal{B}(F,G))
\isom{\mathbf{Ban}_1}\mathcal{B}(F,\mathcal{B}(E,G))
$$
The last isomorphism is called the law of adjoint associativity. There are many
other tensor norms on the algebraic tensor product of Banach spaces. Their
thorough treatment can be found in~\cite{DiestMetTheoryOfTensProd}.

Now we are able to craft four very important functors:
$$
\mathcal{B}(-,E):\mathbf{Ban}\to\mathbf{Ban}
\qquad\qquad
\mathcal{B}(E,-):\mathbf{Ban}\to\mathbf{Ban}
$$
$$
-\projtens E:\mathbf{Ban}\to\mathbf{Ban}
\qquad\qquad
E\projtens -:\mathbf{Ban}\to\mathbf{Ban}.
$$
We shall often encounter them. For example, the well known adjoint functor
${}^*$ is nothing more than $\mathcal{B}(-,\mathbb{C})$. All these functors have
their obvious analogs on $\mathbf{Ban}_1$.

Now we  proceed to classical examples of Banach spaces. 

An important source of examples of Banach spaces are $L_p$-spaces, also known as
Lebesgue spaces. A detailed discussion of basic properties of $L_p$-spaces can
be found in~\cite{CarothShortCourseBanSp}.  Let $(\Omega,\Sigma,\mu)$ be a
measure space.  For $1\leq p<\infty$, as usually, the symbol $L_p(\Omega,\mu)$
stands for the Banach space of equivalence classes of functions
$f:\Omega\to\mathbb{C}$ such that $|f|^p$ is Lebesgue integrable with respect to
measure $\mu$. The norm of such function is defined as
$$
\Vert f\Vert
:={\left(
  \int\limits_{\Omega}{|f(\omega)|}^p d\mu(\omega)
\right)}^{1/p}.
$$ 
By $L_\infty(\Omega,\mu)$ we denote the Banach space of equivalence classes of
bounded measurable functions with norm defined as 
$$
\Vert f\Vert:=\inf\left \{
  \sup_{\omega\in\Omega\setminus N}|f(\omega)|
  :N\subset\Omega - \mbox{is negligible}
\right \}.
$$
For simplicity we shall speak of functions in $L_p(\Omega,\mu)$ instead of their
equivalence classes. All equalities and inequalities about functions of
$L_p$-spaces are understood up to negligible sets. It is well know that
${L_p(\Omega,\mu)}^*\isom{\mathbf{Ban}_1}L_{p^*}(\Omega,\mu)$ for 
$1\leq p<+\infty$ [\cite{FremMeasTh}, theorems 243G, 244K]. 
One more well known fact is that, $L_p$-spaces are reflexive for 
$1<p<+\infty$. Here we exploited the standard 
notation $p^*=+\infty$ if $p=1$ and $p^*=p/(p-1)$ if $1<p<+\infty$.
Clearly, $p^{**}=p$ for $1<p<+\infty$.

The most well known classes of Banach spaces are related to continuous
functions. Let $S$ be a locally compact Hausdorff space. We say that a function
$f:S\to\mathbb{C}$ vanishes at infinity  if for any $\epsilon>0$ there exists a
compact $K\subset S$ such that $|f(s)|\leq\epsilon$ for all $s\in S\setminus K$.
The linear space of continuous functions on $S$ vanishing at infinity is denoted
by $C_0(S)$. When endowed with $\sup$-norm $C_0(S)$ becomes a Banach space. Any
set $\Lambda$ with discrete topology may be regarded as a locally compact space
and following the traditional notation we shall write $c_0(\Lambda)$ instead of
$C_0(\Lambda)$. If $K$ is a compact Hausdorff space then all functions on $K$
vanish at infinity. We use the notation $C(K)$ for $C_0(K)$ to indicate that $K$
is compact. Some Banach spaces in fact are $C(K)$-spaces in disguise. For
example, if we are given a measure space $(\Omega,\Sigma,\mu)$, then
$B(\Omega,\Sigma)$ --- the space of bounded measurable functions with
$\sup$-norm or $L_\infty(\Omega,\mu)$ are $C(K)$ spaces for some compact
Hausdorff space $K$ [\cite{KalAlbTopicsBanSpTh}, remark 4.2.8]. If $S$ is a
locally compact Hausdorff space, then $M(S)$ stands for the Banach space of
complex finite Borel regular measures on $S$. The norm of measure $\mu\in M(S)$
is defined by equality $\Vert\mu\Vert=|\mu|(S)$, where $|\mu|$ is a total
variation measure of measure $\mu$. By Riesz-Markov-Kakutani theorem
[\cite{ConwACoursInFuncAn}, section C.18] we have
${C_0(S)}^*\isom{\mathbf{Ban}_1}M(S)$. In fact $M(S)$ is an $L_1$-space, see
discussion after [\cite{DalLauSecondDualOfMeasAlg}, proposition 2.14]. 

We shall also mention one important specific case of $L_p$-spaces. For a given
index set $\Lambda$ and a counting measure
$\mu_c:\mathcal{P}(\Lambda)\to[0,+\infty]$ the respective $L_p$-space is denoted
by $\ell_p(\Lambda)$. For this type of measure spaces we have one more important
isomorphism ${c_0(\Lambda)}^*\isom{\mathbf{Ban}_1}\ell_1(\Lambda)$. For
convenience we define $c_0(\varnothing)=\ell_p(\varnothing)= \{0 \}$ for 
$1\leq p\leq+\infty$. This example motivates the following construction.

Let $ \{E_\lambda:\lambda\in\Lambda \}$ be an arbitrary family of Banach spaces.
For each $x\in \prod_{\lambda\in\Lambda} E_\lambda$ we define 
$\Vert x\Vert_p
=\Vert{
  (\Vert x_\lambda\Vert)}_{\lambda\in\Lambda}
\Vert_{\ell_p(\Lambda)}$
for $1\leq p\leq +\infty$ and 
$\Vert x\Vert_0
=\Vert
  {(\Vert x_\lambda\Vert)}_{\lambda\in\Lambda}
\Vert_{c_0(\Lambda)}$. 
Then the Banach space
$\left \{
  x\in \prod_{\lambda\in\Lambda} E_\lambda
  :\Vert x\Vert_p<+\infty
\right \}$ 
with the norm $\Vert\cdot\Vert_p$ is denoted by $\bigoplus_p
\{E_\lambda:\lambda\in\Lambda \}$. We call these objects $\bigoplus_p$-sums of
Banach spaces $ \{E_\lambda:\lambda\in\Lambda \}$. It is almost tautological
that the Banach space $\ell_p(\Lambda)$ is the $\bigoplus_p$-sum of the family $
\{\mathbb{C}:\lambda\in\Lambda \}$. A nice property of $\bigoplus_p$-sums is
their interrelation with duality:
$$
{\left(\bigoplus\nolimits_p \{E_\lambda:\lambda\in\Lambda \}\right)}^*
\isom{\mathbf{Ban}_1}
\bigoplus\nolimits_{p^*} \{E_\lambda^*:\lambda\in\Lambda \}
$$
for all $1\leq p<+\infty$ and 
$$
{\left(\bigoplus\nolimits_0 \{E_\lambda:\lambda\in\Lambda \}\right)}^*
\isom{\mathbf{Ban}_1}
\bigoplus\nolimits_1 \{E_\lambda^*:\lambda\in\Lambda \}
$$
If $ \{T_\lambda\in\mathcal{B}(E_\lambda, F_\lambda):\lambda\in\Lambda \}$ is a
family of bounded linear operators, then for all $1\leq p\leq+\infty$ and $p=0$
we have a well defined linear operator
$$
T:\bigoplus\nolimits_p \{E_\lambda:\lambda\in\Lambda \}
  \to 
\bigoplus\nolimits_p \{ F_\lambda:\lambda\in\Lambda \}
:x\mapsto \bigoplus\nolimits_p \{ T_\lambda(x_\lambda):\lambda\in\Lambda \}
$$
which we shall denote by $\bigoplus_p \{T_\lambda:\lambda\in\Lambda \}$. Its
norm equals $\sup_{\lambda\in\Lambda}\Vert T_\lambda\Vert$.

Among different $\bigoplus_p$-sums the $\langle$~$\bigoplus_1$-sums /
$\bigoplus_\infty$-sums~$\rangle$ play a special role in Banach space theory.
The reason is that any family of Banach spaces admit $\langle$~product /
coproduct~$\rangle$ in $\mathbf{Ban}_1$ which in fact is their
$\langle$~$\bigoplus_1$-sum / $\bigoplus_\infty$-sum~$\rangle$. The same
statement holds for $\mathbf{Ban}$ if we restrict ourselves to finite families
of objects [\cite{HelLectAndExOnFuncAn}, chapter 2, section 5].Obviously, any 
summand in a $\bigoplus_p$-sum is a complemented subspace the sum. Even more, a
subspce $F$ of a Banach space $E$ is complemented if there exists a closed
subspace $G$ in $E$ such that $E\isom{\mathbf{Ban}}F\bigoplus_1 G$.

We proceed to advanced topics of Banach space theory. Below we shall discuss
several geometric properties of Banach spaces such as the property of being an
$\mathscr{L}_p^g$-space, weak sequential completeness, the Dunford-Pettis
property, the l.u.st.\ property and the approximation property. In what follows,
imitating Banach space geometers, we shall say that a Banach space $E$ contains
$\langle$~an isometric copy / a copy~$\rangle$ of Banach space $F$ if $F$ is
$\langle$~isometrically isomorphic / topologically isomorphic~$\rangle$ to some
closed subspace of $E$. 

Let $1\leq p\leq +\infty$. We say that $E$ is an $\mathscr{L}_{p,C}^g$-space if
for any $\epsilon>0$ and any finite dimensional subspace $F$ of $E$ there exists
a finite dimensional $\ell_p$-space $G$ and two bounded linear operators 
$S:F\to G$, $T:G\to E$ such that $TS|^F=1_F$ and 
$\Vert T\Vert\Vert S\Vert\leq C+\epsilon$. 
If $E$ is an $\mathscr{L}_{p,C}^g$-space for some $C\geq 1$ we
simply say, that $E$ is an $\mathscr{L}_p^g$-space. This definition
[\cite{DefFloTensNorOpId}, definition 23.1] is an improvement of the definition
of $\mathscr{L}_p$-spaces given by Lindenstrauss and Pelczynski in their
pioneering work~\cite{LinPelAbsSumOpInLpSpAndApp}. Clearly, any finite
dimensional Banach space is an $\mathscr{L}_p^g$-space for 
all $1\leq p\leq +\infty$. Any $L_p$-space is an $\mathscr{L}_{p,1}^g$-space
[\cite{DefFloTensNorOpId}, exercise 4.7], but the converse is not true. Any
$c$-complemented subspace of $\mathscr{L}_{p,C}^g$-space is an
$\mathscr{L}_{p,cC}^g$-space [\cite{DefFloTensNorOpId}, corollary 23.2.1(2)]. A
Banach space is an $\mathscr{L}_{p,C}^g$-space iff its dual is an
$\mathscr{L}_{p^*,C}^g$-space [\cite{DefFloTensNorOpId}, corollary 23.2.1(1)].
All $C(K)$-spaces are $\mathscr{L}_{\infty, 1}^g$-spaces
[\cite{DefFloTensNorOpId}, lemma 4.4]. Note that, for a given locally compact
Hausdorff space $S$ the Banach space $C_0(S)$ is complemented in $C(\alpha S)$.
Therefore $C_0(S)$-spaces are $\mathscr{L}_\infty^g$-spaces too. We will mainly
concern in $\mathscr{L}_1^g$- and $\mathscr{L}_\infty^g$-spaces.

We say that a Banach space $E$ is weakly sequentially complete if for any
sequence ${(x_n)}_{n\in\mathbb{N}}\subset E$ such that
${(f(x_n))}_{n\in\mathbb{N}}\subset\mathbb{C}$ is a Cauchy sequence for 
any $f\in E^*$ there exists a vector $x\in E$ 
such that $\lim_n f(x_n)=f(x)$ for all $f\in E^*$. That is any weakly 
Cauchy sequence converges in the weak topology. A typical example of weakly 
sequentially complete Banach space is any $L_1$-space
[\cite{WojBanSpForAnalysts}, corollary III.C.14]. This property is preserved by
closed subspaces. A typical example of Banach space that is not weakly
sequentially complete is $c_0(\mathbb{N})$, just consider the sequence
${(\sum_{k=1}^n \delta_k)}_{n\in\mathbb{N}}$.

Now we proceed to the discussion of the Dunford-Pettis property. A bounded
linear operator $T:E\to F$ is called weakly compact if it maps the unit ball of
$E$ into a relatively weakly compact subset of $F$. A bounded linear operator is
called completely continuous if the image of any weakly compact subset of $E$ is
norm compact in $F$. A Banach space $E$ is said to have the Dunford-Pettis
property if any weakly compact operator from $E$ to any Banach space $F$ is
completely continuous. There is a simple internal characterization
[\cite{KalAlbTopicsBanSpTh}, theorem 5.4.4]: a Banach space $E$ has the
Dunford-Pettis property if $\lim_n f_n(x_n)=0$ for all sequences
${(x_n)}_{n\in\mathbb{N}}\subset E$ and ${(f_n)}_{n\in\mathbb{N}}\subset E^*$, 
that both weakly converge to $0$. Now it is easy to deduce, 
that if a Banach space $E^*$ has the Dunford-Pettis property, then so does $E$.
In his seminal work~\cite{GrothApllFaiblCompSpCK} Grothendieck showed that all
$L_1$-spaces and $C(K)$-spaces have this property. The Dunford-Pettis property
passes to complemented subspaces [\cite{FabHabBanSpTh}, proposition 13.44]. This
property behaves badly with reflexive spaces: since the unit ball of a reflexive
space is weakly compact [\cite{MeggIntroBanSpTh}, theorem 2.8.2], then reflexive
Banach space with the Dunford-Pettis property has norm compact unit ball and
therefore this space is finite dimensional. 

To introduce the next Banach geometric property we need definitions of Banach
lattice and unconditional Schauder basis. 

A real Riesz space $E$ is a vector space over $\mathbb{R}$ with the structure of
partially ordered set such that $x\leq y$ implies $x+z\leq y+z$ for every
$x,y,z\in E$ and $ax\geq 0$ for every $x\geq 0$, $a\in\mathbb{R}_+$. A partially
ordered set is a lattice if any two elements ${x,y}$ have the least upper bound
$x\vee y$ and the greatest lower bound $x\wedge y$. A real vector lattice is
real Riesz space which is lattice as partially ordered set. If $E$ is a real
vector lattice, then for every $x\in E$ we define its absolute value by equality
$|x|:=x\vee(-x)$. A complex vector lattice $E$ is a vector space over
$\mathbb{C}$ such that there exists a real vector subspace
$\operatorname{Re}(E)$ which is real vector lattice and
\begin{enumerate}[label = (\roman*)]
  \item for any $x\in E$ there are unique
  $\operatorname{Re}(x),\operatorname{Im}(x)\in \operatorname{Re}(E)$ such that
  $x=\operatorname{Re}(x)+i\operatorname{Im}(x)$;

  \item for any $x\in E$ there exist an absolute value $|x|:=\sup
  \{\operatorname{Re}(e^{i\theta}x):\theta\in\mathbb{R} \}$.
\end{enumerate}

A Banach lattice is a Banach space with the structure of the complex vector
lattice such that $\Vert x\Vert\leq \Vert y\Vert$ whenever $|x|\leq |y|$. A
classical example of Banach lattice $E$ is an $L_p$-space or a $C(K)$-space. In
both cases $\operatorname{Re}(E)$ consist of real valued functions in $E$. If 
$ \{E_\lambda:\lambda\in\Lambda \}$ is a family of Banach lattices then for any
$1\leq p\leq +\infty$ or $p=0$ their $\bigoplus_p$-sum is a Banach lattice with
lattice operation defined as $x\leq y$ if $x_\lambda\leq y_\lambda$ for all
$\lambda\in\Lambda$, where 
$x,y\in\bigoplus_p \{ E_\lambda:\lambda\in\Lambda \}$. 
The dual space $E^*$ of a Banach lattice $E$ is again a Banach lattice with
lattice operation defined by $f\leq g$ if $f(x)\leq g(x)$ for all $x\geq 0$,
where $f,g\in  E^*$. A very nice account of Banach lattices can be found in
[\cite{LaceyIsomThOfClassicBanSp}, section 1].

The property of being a Banach lattice puts some restrictions on the geometry of
the space~\cite{SherOrderInOpAlg},~\cite{KadOrderPropOfBoundSAOps}. To explain
the Banach geometric reason of this phenomena we need the definition of an
unconditional Schauder basis. Let $E$ be a Banach space. A collection of
functionals ${(f_\lambda)}_{\lambda\in\Lambda}$ in $E^*$ is called 
a biorthogonal system for vectors ${(x_\lambda)}_{\lambda\in\Lambda}$ 
from $E$ if $f_\lambda(x_{\lambda'})=1$ whenever $\lambda=\lambda'$ 
and $0$ otherwise. A collection ${(x_\lambda)}_{\lambda\in\Lambda}$ 
in $E$ is called an unconditional Schauder basis if there exists 
a biorthogonal system ${(f_\lambda)}_{\lambda\in\Lambda}$ in $E^*$ 
for it such that the series
$\sum_{\lambda\in\Lambda} f_\lambda(x)x_\lambda$ unconditionally converges to
$x$ for any $x\in E$. All $\ell_p$-spaces with $1\leq p<+\infty$ have an
unconditional Schauder basis, for example, it is
${(\delta_\lambda)}_{\lambda\in\Lambda}$. A typical example of space without
unconditional basis is $C([0,1])$. Even more this Banach space can not even be a
subspace of the space with unconditional basis [\cite{KalAlbTopicsBanSpTh},
proposition 3.5.4].  Any unconditional Schauder basis
${(x_\lambda)}_{\lambda\in\Lambda}$ in $E$ satisfy the following property
[\cite{DiestAbsSumOps}, proposition 1.6]: there exists a constant $\kappa\geq 1$
such that
$$
\left\Vert \sum_{\lambda\in\Lambda}t_\lambda f_\lambda(x)x_\lambda\right\Vert
\leq
\kappa\left\Vert \sum_{\lambda\in\Lambda}f_\lambda(x)x_\lambda\right\Vert
$$
for all $x\in E$ and $t\in\ell_\infty(\Lambda)$. The least such constant
$\kappa$ among all unconditional Schauder bases of $E$ is denoted by
$\kappa(E)$. Similar constant could be defined for Banach spaces without
unconditional Schauder bases. The local unconditional constant $\kappa_u(E)$ of
Banach space $E$ is defined to be the infimum of all scalars $c$ with the
following property: given any finite dimensional subspace $F$ of $E$ there
exists a Banach space $G$ with unconditional Schauder basis and two bounded
linear operators $S:F\to G$, $T:G\to E$ such that $TS|^{F}=1_F$ and $\Vert
T\Vert\Vert S\Vert\kappa(G)\leq c$. We say that a Banach space $E$ has the local
unconditional structure property (the l.u.st.\ property for short) if
$\kappa_u(E)$ is finite. Clearly any Banach space $E$ with unconditional
Schauder basis has the l.u.st.\ property with $\kappa_u(E)=\kappa(E)$. In
particular, all finite dimensional Banach spaces have the l.u.st.\ property.
Though a general Banach lattice $E$ may not have an unconditional Schauder basis
it is still has the l.u.st.\ property with $\kappa_u(E)=1$
[\cite{DiestAbsSumOps}, theorem 17.1]. Directly from the definition it follows
that the l.u.st.\ property is preserved by complemented subspaces. More
precisely: if $F$ is a $c$-complemented subspace of $E$, 
then $\kappa_u(F)\leq c\kappa_u(E)$. 
Therefore all complemented subspaces of Banach lattices have the
l.u.st.\ property. This sufficient condition is not far from criterion
[\cite{DiestAbsSumOps}, theorem 17.5]: a Banach space $E$ has the l.u.st.
property iff $E^{**}$ is topologically isomorphic to a complemented subspace of
some Banach lattice. As the corollary of this criterion we get that $E$ has the
l.u.st.\ property iff so does $E^*$ [\cite{DiestAbsSumOps}, corollary 17.6].

The last property we shall discuss is a well known approximation property
introduced by Grothendieck in~\cite{GrothProdTenTopNucl}. We say that a Banach
space $E$ has the approximation property if for any compact set $K\subset E$ and
any $\epsilon>0$ there exists a finite rank operator $T:E\to E$ such 
that $\Vert T(x)-x\Vert<\epsilon$ for all $x\in K$. If $T$ can be 
chosen with $\Vert T\Vert\leq c$, then $E$ is said to have the $c$-bounded
approximation property. The metric approximation property is another name for
$1$-bounded approximation property. We say that $E$ has the bounded
approximation property if $E$ has the $c$-bounded approximation property for
some $c\geq 1$. None of these properties are preserved by subspaces or quotient
spaces, but the approximation property and the bounded approximation properties
are inherited by complemented subspaces [\cite{DefFloTensNorOpId}, exercise
5.5]. All $L_p$-spaces and $C(K)$-spaces have the metric approximation property
[\cite{DefFloTensNorOpId}, section 5.2(3)], but their subspaces may fail the
approximation property [\cite{DefFloTensNorOpId}, section 5.2(1)]. Any Banach
space with unconditional Schauder basis has the approximation property
[\cite{RyanIntroTensNormsBanSp}, example 4.4]. If $E^*$ has the approximation
property, then so does $E$ [\cite{DefFloTensNorOpId}, corollary 5.7.2]. The
reason why approximation property is so important is rather simple --- it has a
lot of equivalent reformulations that involve many nice properties of Banach
spaces. For example, the following properties of Banach space $E$ are equivalent
[\cite{DefFloTensNorOpId}, sections 5.3, 5.6]:
\begin{enumerate}[label = (\roman*)]
  \item $E$ has the approximation property;

  \item the natural mapping $Gr:E^*\projtens E\to\mathcal{N}(E)$ is an isometric
  isomorphism;

  \item for any Banach space $F$ every compact operator $T:F\to E$ can be
  approximated in the operator norm by finite rank operators.
\end{enumerate}

There is much more to list, but we confine ourselves with these three
properties.

%-------------------------------------------------------------------------------
%    Banach algebras and their modules
%-------------------------------------------------------------------------------

\subsection{
  Banach algebras and their modules
}\label{SubSectionBanachAlgebrasAndTheirModules}

A thorough treatment of Banach algebras and Banach modules can be found
in~\cite{HelBanLocConvAlg} or~\cite{HelHomolBanTopAlg}
or~\cite{DalBanAlgAutCont}. We shall describe only the bare minimum required for
us.

A Banach algebra $A$ is an associative algebra over $\mathbb{C}$ which is a
Banach space and the multiplication bilinear operator 
$\cdot:A\times A\to A:(a,b)\mapsto ab$ is of the norm at most $1$. A typical
example of commutative Banach algebra is the algebra of continuous functions on
a compact Hausdorff space with pointwise multiplication. A typical non
commutative example is the algebra of bounded linear operators on the Hilbert
space with composition in the role of multiplication. Both examples belong to a
very important class of $C^*$-algebras to be discussed below. By $\langle$~left
/ right / two-sided~$\rangle$ ideal $I$ of a Banach algebra $A$ we always mean a
closed subalgebra of $A$ such that $\langle$~$ax$ / $xa$ / $ax$ and
$xa$~$\rangle$ belong to $I$ for all $a\in A$ and $x\in I$.

We say that an element $p$ of a Banach algebra $A$ is a $\langle$~left /
right~$\rangle$ identity of $A$ if $\langle$~$pa=a$ / $ap=a$~$\rangle$ for all
$a\in A$. The element which is both left and right identity is called the
identity and denoted by $e_A$. In general we do not assume that Banach algebras
are unital, i.e.\ has an identity. Even if a Banach algebra $A$ is unital we do
not require its identity to be of norm $1$. We use notation
$A_+=A\bigoplus_1\mathbb{C}$ for the standard unitization of Banach algebras.
The multiplication in $A_+$ is defined as $(a\oplus_1 z)(b\oplus_1
w)=(ab+wa+zb)\oplus_1 zw$, for $a,b\in A$ and $z,w\in\mathbb{C}$. Clearly
$(0,1)$ is the identity of $A_+$. By $A_\times$ we denote the conditional
unitization of $A$, i.e. $A_\times=A$ if $A$ has identity of norm one and
$A_\times=A_+$ otherwise. Even in the absence of identity in case of Banach
algebras there are good substitutes for it which are called approximate
identities. We say that a net ${(e_\nu)}_{\nu\in N}$ in $A$ is a 
$\langle$~left / right / two-sided~$\rangle$ approximate identity if 
$\langle$~$\lim_\nu e_\nu a=a$ / $\lim_\nu ae_\nu=a$ / 
$\lim_\nu e_\nu a=\lim_\nu ae_\nu=a$~$\rangle$ for
all $a\in A$. In all these three definitions convergence of nets is understood
in the norm topology. If we will consider weak topology, we shall get
definitions of left, right and two-sided weak approximate identities. We say
that an approximate identity ${(e_\nu)}_{\nu\in N}$ is $c$-bounded if
$\sup_\nu\Vert e_\nu\Vert\leq c$. An approximate identity ${(e_\nu)}_{\nu\in N}$
is called $\langle$~bounded / contractive~$\rangle$ if it is
$\langle$~$1$-bounded / $c$-bounded for some $c\geq 1$~$\rangle$. Occasionally
we will use the following simple fact: if $A$ is a Banach algebra with
$\langle$~left / right~$\rangle$ identity $p$ and $\langle$~right /
left~$\rangle$ approximate identity ${(e_\nu)}_{\nu\in N}$, then $A$ is unital
with identity $p$ of norm $\lim_\nu\Vert e_\nu\Vert$. 

If $A$ is a unital Banach algebra we define the spectrum
$\operatorname{sp}_A(a)$ of element $a$ in $A$ as the set of all complex numbers
$z$ such that $a-ze_A$ is not invertible in $A$. For Banach algebras the
spectrum of any element is a non empty compact subset of $\mathbb{C}$
[\cite{HelBanLocConvAlg}, corollary 2.1.16].

A character on a Banach algebra $A$ is a non zero linear homomorphism
$\varkappa:A\to\mathbb{C}$. All characters are continuous and are contained in
the unit ball of $A^*$  [\cite{HelBanLocConvAlg}, theorem 1.2.6]. Therefore we
may consider the set of characters with the induced weak$^*$ topology. The
resulting topological space is Hausdorff and locally compact. It is called the
spectrum of Banach algebra $A$ and denoted by $\operatorname{Spec}(A)$. If $A$
is unital then its spectrum is compact [\cite{HelBanLocConvAlg}, theorem
1.2.50]. Now for a given Banach algebra $A$ with non empty spectrum we can
construct a contractive homomorphism $\Gamma_A:A\to
C_0(\operatorname{Spec}(A)):a\mapsto(\varkappa\mapsto \varkappa(a))$ called the
Gelfand transform of $A$ [\cite{HelBanLocConvAlg}, theorem 4.2.11]. The kernel
of this homomorphism is called the Jacobson's radical and denoted by
$\operatorname{Rad}(A)$. For Banach algebra $A$ with empty spectrum we define
$\operatorname{Rad}(A)=A$. If $\operatorname{Rad}(A)= \{0 \}$, then $A$ is
called semisimple. By Shilov's idempotent theorem [\cite{KaniBanAlg}, section
3.5] any semisimple Banach algebra with compact spectrum is unital.

Most of standard constructions for Banach spaces have their counterparts for
Banach algebras. For example $\bigoplus_p$-sum of Banach algebras endowed with
componentwise multiplication is a Banach algebra, the quotient of a given Banach
algebra by its two sided ideal is a Banach algebra. Even the projective tensor
product of two Banach algebras is a Banach algebra with multiplication defined
on elementary tensors the same way as in pure algebra.

We shall proceed to the discussion of the most important class of Banach
algebras. Let $A$ be an associative algebra over $\mathbb{C}$, then a conjugate
linear operator ${}^*:A\to A$ is called an involution if ${(ab)}^*=b^*a^*$ and
$a^{**}=a$ for all $a,b\in A$. Algebras with involution are called
${}^*$-algebras. Homomorphisms between ${}^*$-algebras that preserve involution
are called ${}^*$-homomorphisms. A Banach algebra with isometric involution is
called a ${}^*$-Banach algebra. An example of such algebra is the Banach algebra
of bounded linear operators on Hilbert space with operation of taking the
Hilbert adjoint operator in the role of involution. In fact there is much more
to this algebra than one could expect. We say that a ${}^*$-Banach algebra $A$
is a $C^*$-algebra if it satisfies $\Vert a^*a\Vert=\Vert a\Vert^2$ for all
$a\in A$. One of the biggest advantages of $C^*$-algebras is their celebrated
representation theorems by Gelfand and Naimark. The first representation theorem
[\cite{HelBanLocConvAlg}, theorem 4.7.13] states that any commutative
$C^*$-algebra $A$ is isometrically isomorphic as ${}^*$-algebra to
$C_0(\operatorname{Spec}(A))$. The second theorem [\cite{HelBanLocConvAlg},
theorem 4.7.57] gives a description of generic $C^*$-algebras as closed
${}^*$-Banach subalgebras of $\mathcal{B}(H)$ for some Hilbert space $H$. Such
representation is not unique, but a norm (if it exists) that turn a
${}^*$-algebra into a $C^*$-algebra is always unique. If a ${}^*$-subalgebra of
$\mathcal{B}(H)$ is weak${}^*$ closed it is called a von Neumann algebra. If a
$C^*$-algebra is isomorphic as ${}^*$-algebra to a von Neumann algebra it is
called a $W^*$-algebra. By well known Sakai's theorem [\cite{BlackadarOpAlg},
theorem III.2.4.2] each $C^*$-algebra which is dual as Banach space is a
$W^*$-algebra, but beware a $W^*$-algebra may be represented as non weak${}^*$
closed ${}^*$-subalgebra in $\mathcal{B}(H)$ for some Hilbert space $H$. 

A lot of standard constructions pass to $C^*$-algebras from Banach algebras, but
not all. For example a $\bigoplus_\infty$-sum of $C^*$-algebras is again a
$C^*$-algebra. A quotient of $C^*$-algebra by closed two-sided ideal is a
$C^*$-algebra too. Meanwhile the projective tensor product of $C^*$-algebras is
rarely a $C^*$-algebra, though there a lot of norms that may turn their
algebraic tensor product into a $C^*$-algebra. In this book we shall exploit one
specific and highly important for $C^*$-algebras construction of matrix
algebras. For a given $C^*$-algebra $A$ by $M_n(A)$ we denote the linear space
of $n\times n$ matrices with entries in $A$. In fact $M_n(A)$ is ${}^*$-algebra
with involution and multiplication defined by equalities 
$$
{(ab)}_{i,j}=\sum_{k=1}^n a_{i,k}b_{k,j}
\qquad\qquad
{(a^*)}_{i,j}=(a_{j,i}^*)
$$ 
for all $a,b\in M_n(A)$ and $i,j\in\mathbb{N}_n$. There is a unique norm on
$M_n(A)$ that makes it a $C^*$-algebra [\cite{MurphyCStarAlgsAndOpTh}, theorem
3.4.2]. Obviously, $M_n(\mathbb{C})$ is isometrically isomorphic as
${}^*$-algebra to $\mathcal{B}(\ell_2(\mathbb{N}_n))$. From
[\cite{MurphyCStarAlgsAndOpTh}, remark 3.4.1] it follows that the natural
embeddings 
$i_{k,l}:A\to M_n(A):a\mapsto{(a\delta_{i,k}\delta_{j,l})}_{i,j\in\mathbb{N}_n}$
and projections
$\pi_{k,l}:M_n(A)\to A:a\mapsto a_{k,l}$ are continuous. Therefore for a given
bounded linear operator $\phi:A\to B$ between $C^*$-algebras $A$ and $B$ the
linear operator 
$$
M_n(\phi):M_n(A)\to M_n(B):a\mapsto {(\phi(a_{i,j}))}_{i,j\in\mathbb{N}_n}
$$ 
is also bounded. Even more if $\phi$ is an $A$-morphism or ${}^*$-homomorphism,
then so does $M_n(\phi)$. Finally we shall mention two isometric isomorphisms
that will be of use:
$$
M_n\left(\bigoplus\nolimits_\infty \{A_\lambda:\lambda\in\Lambda \}\right)
\isom{\mathbf{Ban}_1}
\bigoplus\nolimits_\infty \{M_n\left(A_\lambda\right):\lambda\in\Lambda \},
$$
$$
M_n(C(K))\isom{\mathbf{Ban}_1}C(K,M_n(\mathbb{C}))
$$

Now a few facts on approximate identities and identities of $C^*$-algebras and
their ideals. Any two-sided closed ideal of $C^*$-algebra has a two-sided
contractive positive approximate identity [\cite{HelBanLocConvAlg}, theorem
4.7.79], and any left ideal has a right contractive positive approximate
identity. In some cases even an approximate identity is not enough, so for this
situation there is a procedure to endow a $C^*$-algebra with identity and
preserve $C^*$-algebraic structure [\cite{HelBanLocConvAlg}, proposition 4.7.6].
This type of unitization we shall denote as $A_\#$. Till the end of this
paragraph we assume that $A$ is a unital $C^*$-algebra. An element $a\in A$ is
called a projection (or an orthogonal projection) if $a=a^*=a^2$, self-adjoint
if $a=a^*$, positive if $a=b^*b$ for some $b\in A$, unitary if $a^*a=aa^*=e_A$.
The set $A_{pos}$ of all positive elements of $A$ is a closed cone in $A$. If an
element $a\in A$ is $\langle$~self-adjoint / positive~$\rangle$, then
$\langle$~$\operatorname{sp}_A(a)\subset[-\Vert a\Vert, \Vert a\Vert]$ /
$\operatorname{sp}_A(a)\subset[0,\Vert a\Vert]$~$\rangle$. For a given
self-adjoint element $a\in A$, there always exists the isometric
${}^*$-homomorphism $\operatorname{Cont}_a:C(\operatorname{sp}_A(a))\to A$ such
that $\operatorname{Cont}_a(f)=a$, where
$f:\operatorname{sp}_A(a)\to\mathbb{C}:t\mapsto t$. It is called the continuous
functional calculus [\cite{HelBanLocConvAlg}, theorem 4.7.24]. Loosely speaking
it allows to take continuous functions of self-adjoint elements of
$C^*$-algebras, so following standard convention we shall write $f(a)$ instead
of $\operatorname{Cont}_a(f)$. Another related result called the spectral
mapping theorem allows to compute the spectrum of elements given by continuous
functional calculus: $\operatorname{sp}_A(f(a))=f(\operatorname{sp}_A(a))$.

We proceed to the discussion of more general objects --- Banach modules. Let $A$
be a Banach algebra, we say that $X$ is a $\langle$~left / right~$\rangle$
Banach $A$-module if $X$ is a Banach space endowed with bilinear operator
$\langle$~$\cdot:A\times X\to X$ / $\cdot: X\times A\to X$~$\rangle$ of norm at
most $1$ (called a module action), such that 
$\langle$~$a\cdot(b\cdot x)=ab\cdot x$ / 
$(x\cdot a)\cdot b=x\cdot ab$~$\rangle$ for all $a,b\in A$ and $x\in X$.
Any Banach space $E$ can be turned into a $\langle$~left / right~$\rangle$
Banach $A$-module be defining $\langle$~$a\cdot x=0$ / $x\cdot a=0$~$\rangle$
for all $a\in A$ and $x\in E$. Any Banach algebra $A$ can be regarded as a left
and right Banach $A$-module --- the module action coincides with algebra
multiplication. Of course, there are more meaningful examples too.  Usually we
shall discuss only left Banach modules since for their right sided counterparts
all definitions and results are similar. We call a left Banach module $X$ over
unital Banach algebra $A$ unital if $e_A\cdot x=x$ for all $x\in X$. For a given
left Banach $A$-module $X$ and $S\subset A$, $M\subset X$ we define their
products $S\cdot M= \{a\cdot x:a\in S, x\in M \}$, 
$SM=\operatorname{span} (S\cdot M)$ and algebraic annihilators 
$$
S^{\perp M}= \{a\in S:a\cdot M= \{0 \} \},
\qquad\qquad
{}^{S\perp}M= \{x\in M: S\cdot x= \{0 \} \}.
$$ 
The essential and annihilator
parts of $X$ are defined as $X_{ess}=\operatorname{cl}_X(A X)$,
$X_{ann}={}^{A\perp}X$. The module $X$ is called $\langle$~faithful /
annihilator / essential~$\rangle$ if $\langle$~${}^{A\perp}X= \{0 \}$ /
$X=X_{ann}$ / $X=X_{ess}$~$\rangle$. Clearly, 
${(X^*)}^{\perp S}={\operatorname{cl}_X(S X)}^\perp$. Setting $S=A$ we get that 
$X$ is an essential $A$-module iff $X^*$ is a faithful $A$-module.

Let $X$ and $Y$ be $\langle$~left / right~$\rangle$ Banach $A$-modules. We say
that a linear operator $\phi:X\to Y$ is an $A$-module map of $\langle$~left /
right~$\rangle$ modules if $\langle$~$\phi(a\cdot x)=a\cdot \phi(x)$ /
$\phi(x\cdot a)=\phi(x)\cdot a$~$\rangle$ for all $a\in A$ and $x\in X$. A
bounded $A$-module map is called an $A$-morphism. The set of $A$-morphisms
between $\langle$~left / right~$\rangle$ $A$-modules $X$ and $Y$ we denote as
$\langle$~${}_A\mathcal{B}(X,Y)$ / $\mathcal{B}_A(X,Y)$~$\rangle$. Note that if
$X$ and $Y$ are $\langle$~left / right~$\rangle$ annihilator $A$-modules, then
$\langle$~${}_A\mathcal{B}(X,Y)=\mathcal{B}(X,Y)$ /
$\mathcal{B}_A(X,Y)=\mathcal{B}(X,Y)$~$\rangle$.

By $\langle A-\mathbf{mod}$ / $\mathbf{mod}-A\rangle$ we shall denote the
category of $\langle$~left / right~$\rangle$ $A$-modules with continuous
$A$-module maps in the role of morphisms. By $\langle$~$A-\mathbf{mod}_1$ /
$\mathbf{mod}_1-A$~$\rangle$ we denote its subcategory of
$\langle$~$A-\mathbf{mod}$ / $\mathbf{mod}-A$~$\rangle$ with the same objects
and contractive morphisms only. Therefore 
$\langle$~$\operatorname{Hom}_{A-\mathbf{mod}}(X,Y)={}_A\mathcal{B}(X,Y)$ /
$\operatorname{Hom}_{\mathbf{mod}-A}(X,Y)=\mathcal{B}_A(X,Y)$~$\rangle$. 

As in any category we can speak of retraction and     coretractions in the
category of Banach modules. But for this particular case we have several
refinements for the standard definitions. An $A$-morphism $\xi:X\to Y$ is called
a $\langle$~$c$-retraction / $c$-coretraction~$\rangle$ if there exist an
$A$-morphism $\eta:Y\to X$ such that $\langle$~$\xi\eta=1_Y$ /
$\eta\xi=1_X$~$\rangle$ and $\Vert\xi\Vert\Vert\eta\Vert\leq c$. From the
definition it follows that composition of $\langle$~$c_1$- and $c_2$-retraction
/ $c_1$- and $c_2$-coretraction~$\rangle$ gives a $\langle$~$c_1c_2$-retraction
/ $c_1c_2$-coretraction~$\rangle$. Clearly, the adjoint of
$\langle$~$c$-retraction / $c$-coretraction~$\rangle$ is a
$\langle$~$c$-coretraction / $c$-retraction~$\rangle$. Finally, an $A$-morphism
$\xi:X\to Y$ is called a $c$-isomorphism if there exists an $A$-morphism
$\eta:Y\to X$ such that $\xi\eta=1_Y$, $\eta\xi=1_X$ and
$\Vert\xi\Vert\Vert\eta\Vert\leq c$. In this case we say that $A$-modules $X$
and $Y$ are $c$-isomorphic.

Now we mention several constructions over Banach modules that we will encounter
in this book.  Any left Banach $A$-module can be regarded as unital Banach
module over $A_+$, and we put by definition $(a\oplus_1 z)\cdot x=a\cdot x+zx$
for all $a\in A$, $x\in X$ and $z\in\mathbb{C}$. Most constructions used for
Banach spaces transfer to Banach modules.  We say that a linear subspace $ Y$ of
a left Banach $A$-module $X$ is a left $A$-submodule of $X$ 
if $A\cdot Y\subset Y$. For example, any left ideal $I$ of a 
Banach algebra $A$ is a left $A$-submodule of $A$. If $Y$ is a closed left
$A$-submodule of the left Banach $A$-module $X$, then the Banach space $X/Y$ can
be endowed with the structure of the left Banach $A$-module, just put by
definition $a\cdot(x+Y)=a\cdot x+Y$ for all $a\in A$ and $x+Y\in X/Y$. This
object is called the quotient $A$-module. Quotient modules of the form $A/I$,
where $I$ is a left ideal of $A$, are called cyclic modules. For motivation for
this term see [\cite{HelBanLocConvAlg}, proposition 6.2.2]. Clearly, $X/X_{ess}$
is an annihilator $A$-module. If $X$ is a left Banach $A$-module and $E$ is a
Banach space, then $\langle$~$\mathcal{B}(X,E)$ / $\mathcal{B}(E,X)$~$\rangle$
is a $\langle$~right / left~$\rangle$ Banach $A$-module with module action
defined by $\langle$~$(T\cdot a)(x)=T(a\cdot x)$ for all $a\in A$, $x\in X$ and
$T\in\mathcal{B}(X, E)$ / $(a\cdot T)(x)=a\cdot T(x)$ for all $a\in A$, $x\in E$
and $T\in\mathcal{B}(E, X)$~$\rangle$. In particular, $X^*$ is a right Banach
$A$-module. If $ \{X_\lambda:\lambda\in\Lambda \}$ is a family of left Banach
$A$-modules and $1\leq p\leq +\infty$ or $p=0$, then their $\bigoplus_p$-sum is
a left Banach $A$-module with module action defined 
by $a\cdot x=\bigoplus_p \{a\cdot x_\lambda:\lambda\in\Lambda \}$, 
where $a\in A$, $x\in\bigoplus_p \{X_\lambda:\lambda\in\Lambda \}$. 
Again, as in Banach space theory, any family
of $A$-modules admits the $\langle$~product / coproduct~$\rangle$ in
$A-\mathbf{mod}_1$ which in fact is their $\langle$~$\bigoplus_1$-sum /
$\bigoplus_\infty$-sum~$\rangle$. The category $A-\mathbf{mod}$ admits
$\langle$~products / coproducts~$\rangle$ only for finite families of objects.
Similar statements are valid for $\mathbf{mod}-A$ and $\mathbf{mod}_1-A$.

Projective tensor product of Banach spaces also has its module version, it is
called the projective module tensor product. Assume $X$ is a right and $Y$ is a
left Banach $A$-module. Their projective module tensor product
$X\projmodtens{A}Y$ is defined as quotient space $X\projtens Y / N$ where
$N
=\operatorname{cl}_{X\projtens Y}(\operatorname{span} \{
  x\cdot a\projtens y-x\projtens a\cdot y
  :x\in X,y\in Y,a\in A
 \})$. 
Let $\phi\in\mathcal{B}_A(X_1,X_2)$ and $\psi\in{}_A\mathcal{B}(Y_1,Y_2)$ 
for right Banach $A$-modules $X_1$, $X_2$ and left 
Banach $A$-modules $Y_1$, $Y_2$, then there exists a unique bounded 
linear operator 
$$
\phi\projmodtens{A} \psi:X_1\projmodtens{A} Y_1\to X_2\projmodtens{A} Y_2
$$ 
such that
$(\phi\projmodtens{A} \psi)(x\projmodtens{A} y)=\phi(x)\projmodtens{A} \psi(y)$
for all $x\in X_1$ and $y\in Y_1$. Even more 
$\Vert \phi\projmodtens{A} \psi\Vert\leq\Vert \phi\Vert\Vert \psi\Vert$. 
The projective module tensor product has its own universal property: for any
right Banach $A$-module $X$, any left Banach $A$-module $Y$ and any Banach space
$E$ there exists an isometric isomorphism:
$$
\mathcal{B}(X\projmodtens{A}Y,E)
\isom{\mathbf{Ban}_1}
\mathcal{B}_{bal}(X\times Y, E)
$$
where $\mathcal{B}_{bal}(X\times Y, E)$ stands for the Banach space of bilinear
operators $\Phi:X\times Y\to E$ satisfying $\Phi(x\cdot a,y)=\Phi(x,a\cdot y)$
for all $x\in X$, $y\in Y$ and $a\in A$. Such bilinear operators are called
balanced. Furthermore we have two (natural in $X$, $Y$ and $E$) isometric
isomorphisms:
$$
\mathcal{B}(X\projmodtens{A}Y,E)
\isom{\mathbf{Ban}_1}
{}_A\mathcal{B}(Y,\mathcal{B}(X,E))
\isom{\mathbf{Ban}_1}
\mathcal{B}_A(X,\mathcal{B}(Y,E))
$$
Analogously to Banach space theory we may define the following functors:
$$
\mathcal{B}(-,E):A-\mathbf{mod}\to \mathbf{mod}-A
\qquad\qquad
\mathcal{B}(E,-):\mathbf{mod}-A\to \mathbf{mod}-A
$$
$$
-\projmodtens{A} Y:\mathbf{mod}-A\to\mathbf{Ban}
\qquad\qquad
X\projmodtens{A} -:A-\mathbf{mod}\to\mathbf{Ban}
$$
where $E$ is a Banach space, $X$ is a right $A$-module and $Y$ is a left
$A$-module. All these functors have their counterparts for categories
$A-\mathbf{mod}_1$, $\mathbf{mod}_1-A$. 

In some cases it is possible to explicitly compute the projective module tensor
product. For example [\cite{HelBanLocConvAlg}, proposition 6.3.24] if $I$ is a
left closed ideal of $A_+$ with left $\langle$~contractive / bounded~$\rangle$
approximate identity, and $X$ is a left Banach module then the linear operator 
$$
i_{I,X}
:I\projmodtens{A}X \to \operatorname{cl}_X(IX)
:a\projmodtens{A} x\mapsto a\cdot x
$$
is $\langle$~a topological isomorphism / an isometric isomorphism~$\rangle$ of
Banach spaces. If $I$ is a two-sided ideal, then $i_{I,X}$ is a morphism of left
$A$-modules. We call reduced all left Banach modules of the form
$A\projmodtens{A}X$. 

Most of what have been said here can be generalized to Banach bimodules, but in
this book we shall not exploit them much. In those rare case when we shall
encounter bimodules, the respective definitions and results are easily
recoverable from their one sided counterparts.

%-------------------------------------------------------------------------------
%    Banach homology
%-------------------------------------------------------------------------------

\section{
  Banach homology
}\label{SectionBanachHomology}

%-------------------------------------------------------------------------------
%    Relative homology
%-------------------------------------------------------------------------------

\subsection{
  Relative homology
}\label{SubSectionRelativeHomology}

Further we briefly discuss ABCs of relative homology introduced and intensively
studied by Helemskii. In this text we shall use a little bit more involved
definition given by White in~\cite{WhiteInjmoduAlg}. It is equivalent to the one
given by Helemskii. Fix an arbitrary Banach algebra $A$. We say that a
morphism $\xi:X\to Y$ of left $A$-modules $X$ and $Y$ is a $c$-relatively 
admissible epimorphism if it admits a right inverse bounded linear operator of 
the norm at most $c$. A left $A$-module $P$ is called $C$-relatively projective 
if for any $c$-relatively admissible epimorphism $\xi:X\to Y$ and 
for any  $A$-morphism $\phi:P\to Y$ there exists 
an $A$-morphism $\psi:P\to X$ such 
that $\Vert \psi\Vert\leq cC\Vert \phi\Vert$ and the diagram
$$
\xymatrix{
& {X} \ar[d]^{\xi}\\  % chktex 3
{P} \ar@{-->}[ur]^{\psi} \ar[r]^{\phi} &{Y}}  % chktex 3
$$
is commutative. Such $A$-morphism $\psi$ is called a lifting of $\phi$ and it is
not unique in general. Similarly,  we say that a morphism $\xi:Y\to X$ of right
$A$-modules $X$ and $Y$ is a $c$-relatively admissible monomorphism if it admits
a left inverse bounded linear operator of the norm at most $c$. A right 
$A$-module $J$ is called $C$-relatively injective if for any $c$-relatively 
admissible monomorphism $\xi:Y\to X$ and for 
any $A$-morphism $\phi:Y\to J$ there exists an $A$-morphism $\psi:X\to J$ such 
that $\Vert\psi\Vert\leq cC\Vert\phi\Vert$ and the diagram
$$
\xymatrix{
& {X} \ar@{-->}[dl]_{\psi} \\  % chktex 3
{J} &{Y} \ar[l]_{\phi} \ar[u]_{\xi}}  % chktex 3
$$
is commutative. Such $A$-morphism $\psi$ is called an extension of $\phi$ and it
is not unique in general.

Another closely related homological property is flatness. A
left $A$-module $F$ is called $C$-relatively flat if for each
$c$-relatively admissible monomorphism $\xi:X\to Y$ the operator 
$\xi\projmodtens{A} 1_F:X\projmodtens{A} F\to Y\projmodtens{A} F$ is
$cC$-topologically injective. A left $A$-module $F$ is $C$-relatively flat iff
$F^*$ is $C$-relatively injective. 

We say that an $A$-morphism is relatively admissible if it is $c$-relatively 
admissible for some $c\geq 1$. Similarly, we say that an $A$-module $X$ is 
relatively $\langle$~projective / injective / flat~$\rangle$ if it is  
$C$-relatively $\langle$~projective / injective / flat~$\rangle$ for 
some $C\geq 1$.

The reason for considering relatively admissible morphisms in these definitions
is the intention of separation of Banach geometric and algebraic motives that 
may prevent an $A$-module to be relatively projective, injective or flat. 
A straightforward check shows that any $c$-retract of 
$C$-relatively $\langle$~projective / injective / flat~$\rangle$ $A$-module is 
again $cC$-relatively $\langle$~projective / injective / flat~$\rangle$. 
Obviously, any $c$-relatively admissible $\langle$~epimorphism
/ monomorphism~$\rangle$ $\langle$~onto / from~$\rangle$ a $C$-relatively
$\langle$~projective / injective~$\rangle$ $A$-module is a 
$\langle$~$cC$-retraction / $cC$-coretraction~$\rangle$.

A special class of relatively $\langle$~projective / injective~$\rangle$
$A$-modules is the class of so-called 
relatively $\langle$~free / cofree~$\rangle$ modules. 
These are modules of the form $\langle$~$A_+\projtens E$ /
$\mathcal{B}(A_+,E)$~$\rangle$ for some Banach space $E$. Their main feature is
the following: for any $A$-module $X$ there exists a relatively $\langle$~free /
cofree~$\rangle$ $A$-module $F$, which in fact is $\langle$~$A_+\projtens X$ /
$\mathcal{B}(A_+,X)$~$\rangle$ and a $1$-relatively admissible 
$\langle$~epimorphism $\xi:F\to X$ / monomorphism $\xi:X\to F$~$\rangle$. 
If $X$ is $C$-relatively
$\langle$~projective / injective~$\rangle$ we immediately get that $\xi$ is a
$\langle$~$C$-retraction / $C$-coretraction~$\rangle$. Therefore an $A$-module 
is $C$-relatively $\langle$~projective / injective~$\rangle$ iff it is 
a $C$-retract of relatively $\langle$~free / cofree~$\rangle$ $A$-module. 

It is worth to emphasize one more time that major nuance of relative Banach
homology is deliberate balance between algebra and topology in choice of
admissible morphisms. This choice allowed one to build homological theory with
some interesting phenomena with no analogs in pure algebra. We demonstrate one
example related to Banach algebras. Consider morphism of $A$-bimodules
$\Pi_A:A\projtens A\to A:a\projtens b\mapsto ab$. We say that a Banach algebra
$A$ is 
\begin{enumerate}[label = (\roman*)]
  \item $c$-relatively biprojective if $\Pi_A$ is a $c$-retraction of
  $A$-bimodules;

  \item $c$-relatively biflat if $\Pi_A^*$ is a $c$-coretraction of 
  $A$-bimodules;

  \item $c$-relatively contractible if $\Pi_{A_+}$ is a $c$-retraction of
  $A$-bimodules;

  \item $c$-relatively amenable if $\Pi_{A_+}^*$ is a $c$-coretraction of
  $A$-bimodules.
\end{enumerate}

We say that $A$ is relatively $\langle$~biprojective / biflat / contractive /
amenable~$\rangle$ if it is $c$-relatively $\langle$~biprojective / biflat /
contractive / amenable~$\rangle$ for some $c\geq 1$. The infimum of the
constants $c$ is called the $\langle$~biprojectivity / biflatness /
contractivity / amenability~$\rangle$ constant.  With slight modifications of
[\cite{HelBanLocConvAlg}, proposition 7.1.72] one can show that $A$ is
$c$-relatively $\langle$~contractible / amenable~$\rangle$ iff there exists
$\langle$~an element $d\in A\projtens A$ / a net 
${(d_\nu)}_{\nu\in N}\subset A\projtens A$~$\rangle$ with 
norm not greater than $c$ such that for all $a\in A$ holds 
$\langle$~$a\cdot d-d\cdot a=0$ and $a\Pi_A(d)=a$ / 
$\lim_\nu(a\cdot d_\nu-d_\nu\cdot a)=0$ 
and $\lim_\nu a\Pi_A(d_\nu)=a$~$\rangle$. Note that 
$\langle$~such element $d$ / such net ${(d_\nu)}_{\nu\in N}$~$\rangle$ is called
$\langle$~a diagonal / an approximate diagonal~$\rangle$. From homological point
of view, the main advantage of relatively $\langle$~biprojective / biflat /
contractible / amenable~$\rangle$ Banach algebras is that $\langle$~any reduced
/ any reduced / any / any~$\rangle$ left and right Banach $A$-module is
relatively $\langle$~projective / flat / projective / flat~$\rangle$
[\cite{HelBanLocConvAlg}, theorem 7.1.60]. As for flatness such phenomena is
typical for relative Banach homology, but not for the purely algebraic one.

%% file: chapters/chapter2_general_theory.tex
% chktex-file 35
% Chapter Template

% Main chapter title
% Change X to a consecutive number; for referencing 
% this chapter elsewhere, use~\ref{ChapterX}
\chapter{General theory}\label{ChapterGeneralTheory} 

% Change X to a consecutive number; this is for the header 
% on each page - perhaps a shortened title
\lhead{Chapter 2. \emph{General theory}} 

%-------------------------------------------------------------------------------
%    Projectivity, injectivity and flatness
%-------------------------------------------------------------------------------

\section{
    Projectivity, injectivity and flatness
}\label{SectionProjectivityInjectivityAndFlatness}

%-------------------------------------------------------------------------------
%    Metric and topological projectivity
%-------------------------------------------------------------------------------

\subsection{
    Metric and topological projectivity
}\label{SubSectionMetricAndTopologicalProjectivity}

In what follows $A$ denotes a not necessary unital Banach algebra. We
immediately start from the most important definitions in this book.

\begin{definition}[\cite{HelMetrFrQMod}, 
    definition 1.4;~\cite{WhiteInjmoduAlg}, 
    definition 2.4]\label{MetCTopProjMod} 
An $A$-module $P$ is 
called $\langle$~metrically / $C$-topologically~$\rangle$ projective if for any 
strictly $\langle$~coisometric / $c$-topologically surjective~$\rangle$
$A$-morphism $\xi:X\to Y$ and for any $A$-morphism $\phi:P\to Y$ 
there exists an $A$-morphism $\psi:P\to X$ such that $\xi\psi=\phi$  and
$\langle$~$\Vert\psi\Vert\leq\Vert\phi\Vert$ / 
$\Vert \psi\Vert\leq c C\Vert\phi\Vert$~$\rangle$. We say that an $A$-module 
$P$ is topologically projective if it is $C$-topologically projective for 
some $C\geq 1$.
\end{definition}

The task of constructing an $A$-morphism $\psi$ for a given $A$-morphisms $\phi$ 
and $\xi$ in the definition~\ref{MetCTopProjMod} is called a lifting problem 
and $\psi$ is called a lifting of $\phi$ along $\xi$.

A short but more involved equivalent definition 
of $\langle$~metric / $C$-topological~$\rangle$ projectivity is the following: 
an $A$-module $P$ is called $\langle$~metrically / $C$-topologically~$\rangle$ 
projective, if the functor
$\langle$~$\operatorname{Hom}_{A-\mathbf{mod}_1}(P,-)
:A-\mathbf{mod}_1\to\mathbf{Ban}_1$
/
$\operatorname{Hom}_{A-\mathbf{mod}}(P,-)
:A-\mathbf{mod}\to\mathbf{Ban}$~$\rangle$
maps strictly $\langle$~coisometric / $c$-topologically surjective~$\rangle$
$A$-morphisms into strictly $\langle$~coisometric / 
$c C$-topologically surjective~$\rangle$ operators. 

In category $\langle$~$A-\mathbf{mod}_1$ / $A-\mathbf{mod}$~$\rangle$ there
is a special class of $\langle$~metrically / $1$-topologically~$\rangle$
projecive modules of the form $A_+\projtens \ell_1(\Lambda)$ for 
some set $\Lambda$. They are called free modules. These modules play a crucial 
role in our studies of projectivity.

\begin{proposition}[\cite{WhiteInjmoduAlg}, 
    lemma 2.6]\label{MetCTopFreeMod} Let $\Lambda$ be an arbitrary 
set, then the left $A$-modules $A_+ \projtens \ell_1(\Lambda)$ 
and $A_\times \projtens \ell_1(\Lambda)$ 
are $\langle$~metrically / $1$-topologically~$\rangle$ projective. 
\end{proposition}
\begin{proof} By $A_\bullet$ we denote either $A_+$ or $A_\times$. 
Consider arbitrary $A$-morphism 
$\phi:A_\bullet\projtens\ell_1(\Lambda)\to X$ and a 
strictly $\langle$~coisometric / $c$-topologically surjective~$\rangle$ 
$A$-morphism $\xi:X\to Y$. Fix arbitrary $\lambda\in\Lambda$ and 
consider $y_\lambda=\phi(e_{A_\bullet}\projtens \delta_\lambda)$. 
Clearly $\Vert y_\lambda\Vert\leq \Vert\phi\Vert$. Since $\xi$ is 
strictly $\langle$~coisometric / $c$-topologically surjective~$\rangle$, 
then there exists $x_\lambda\in X$ such that $\xi(x_\lambda)=y_\lambda$
and $\Vert x_\lambda\Vert\leq K\Vert y_\lambda\Vert$ for $\langle$~$K=1$ / 
$K=c$~$\rangle$. Now we can define a bounded linear operator
$\psi:A_\bullet\projtens \ell_1(\Lambda)\to X$ such that 
$\psi(a\projtens \delta_\lambda)=a\cdot x_\lambda$ for $\lambda\in\Lambda$. 
It is routine to check that $\psi$ is an $A$-morphism with $\xi\psi=\phi$ 
and $\Vert\psi\Vert\leq K\Vert\phi\Vert$. Thus for a given $\phi$ and $\xi$
we have constructed an $A$-morphism $\psi$ such that $\xi\psi=\phi$ and
$\langle$~$\Vert\psi\Vert\leq \Vert\phi\Vert$ / $\Vert\psi\Vert\leq
c\Vert\phi\Vert$~$\rangle$. Hence the $A$-module $A_\bullet\projtens
\ell_1(\Lambda)$ is $\langle$~metrically / $1$-topologically~$\rangle$
projective.
\end{proof}

\begin{proposition}\label{UnitalAlgIsMetTopProj} The left $A$-module $A_\times$
is $\langle$~metrically / $1$-topologically~$\rangle$ projective.
\end{proposition} 
\begin{proof} Consier set $\Lambda=\mathbb{N}_1$. 
By proposition~\ref{MetCTopFreeMod} the 
$A$-module $A_\times\projtens\ell_1(\Lambda)$ is metrically 
and $1$-topologically projective. Now it remains to note that 
$A_\times\projtens\ell_1(\Lambda)
\isom{A-\mathbf{mod}_1}
A_\times\projtens \mathbb{C}
\isom{A-\mathbf{mod}_1}
A_\times$.
\end{proof}

\begin{proposition}[\cite{WhiteInjmoduAlg}, 
    lemma 2.7]\label{RetrMetCTopProjIsMetCTopProj} Any 
$\langle$~$1$-retract / $c$-retract~$\rangle$ of
$\langle$~metrically / $C$-topologically~$\rangle$ projective module is
$\langle$~metrically / $c C$-topologically~$\rangle$ projective.
\end{proposition}
\begin{proof} Suppose that $P$ is a $c$-retract of 
$\langle$~metrically / $C$-topologically~$\rangle$ projective $A$-module $P'$.
Then there exist $A$-morphisms $\eta:P\to P'$ and $\zeta: P'\to P$ such that
$\zeta\eta=1_{P}$ and $\Vert\zeta\Vert\Vert\eta\Vert\leq c$ 
for $\langle$~$c=1$ / $c\geq 1$~$\rangle$. Consider arbitrary
strictly $\langle$~coisometric / $c'$-topologically surjective~$\rangle$ 
$A$-morphism $\xi:X\to Y$ and an arbitrary $A$-morphism $\phi:P\to Y$. 
Consider $A$-morphism $\phi'=\phi\zeta$. Since $P'$ is 
$\langle$~metrically / $C$-topologically~$\rangle$ projective, then there 
exists an $A$-morphism $\psi':P'\to X$ such that $\phi'=\xi\psi'$ 
and $\Vert\psi'\Vert \leq K\Vert\phi'\Vert$ 
for $\langle$~$K=1$ / $K=c' C$~$\rangle$. Now it is routine to check that 
for the $A$-morphism $\psi=\psi'\eta$ 
holds $\xi\psi=\phi$ and $\Vert\psi\Vert \leq cK\Vert\phi\Vert$. 
Thus for a given $\phi$ and $\xi$ we have
constructed an $A$-morphism $\psi$ such that $\xi\psi=\phi$  and
$\langle$~$\Vert\psi\Vert\leq\Vert\phi\Vert$ / 
$\Vert \psi\Vert\leq c C\Vert\phi\Vert$~$\rangle$. Hence $P$ 
is $\langle$~metrically / $c C$-topologically~$\rangle$ projective.
\end{proof}

It is easy to show that, for any $A$-module $X$ there exists 
a strictly $\langle$~coisometric / $1$-topologically surjective~$\rangle$ 
$A$-morphism
$$
\pi_X^+:A_+\projtens \ell_1(B_X)\to X:a\projtens \delta_x\mapsto a\cdot x.
$$

\begin{proposition}[\cite{WhiteInjmoduAlg}, 
    proposition 2.10]\label{MetCTopProjModViaCanonicMorph} An $A$-module $P$ is
$\langle$~metrically / $C$-topologically~$\rangle$ projective iff $\pi_P^+$ is 
a $\langle$~$1$-retraction / $C$-retraction~$\rangle$ in $A-\mathbf{mod}$.
\end{proposition}
\begin{proof}
Suppose $P$ is 
$\langle$~metrically / $C$-topologically~$\rangle$ projective, then consider 
strictly $\langle$~coisometric / $1$-topologically surjective~$\rangle$
$A$-morphism $\pi_P^+$. Consider lifting problem with $\phi=1_P$ 
and $\xi=\pi_P^+$, then from $\langle$~metric / $C$-topological~$\rangle$ 
projectivity of $P$ we get an $A$-morphism $\sigma^+$ such 
that $\pi_P^+\sigma^+=1_P$
and $\Vert\sigma^+\Vert\leq K$ for $\langle$~$K=1$ / $K=C$~$\rangle$.
Since $\Vert\pi_P^+\Vert\Vert \sigma^+\Vert\leq K$ 
we conclude that $\pi_P^+$ is a
$\langle$~$1$-retraction / $C$-retraction~$\rangle$ in $A-\mathbf{mod}$.

Conversely, assume that $\pi_P^+$ is 
a $\langle$~$1$-retraction / $C$-retraction~$\rangle$. In other words $P$ is 
a $\langle$~$1$-retract / $C$-retract~$\rangle$ of $A_+\projtens \ell_1(B_P)$. 
By proposition~\ref{MetCTopFreeMod} the $A$-module $A_+\projtens \ell_1(B_P)$ 
is $\langle$~metrically / $1$-topologically~$\rangle$ projective. So from 
proposition~\ref{RetrMetCTopProjIsMetCTopProj} 
its $\langle$~$1$-retract / $C$-retract~$\rangle$ $P$ is 
$\langle$~metrically / $C$-topologically~$\rangle$ projective.
\end{proof}

\begin{proposition}[\cite{WhiteInjmoduAlg}, 
    proposition 2.9]\label{MetProjIsTopProjAndTopProjIsRelProj} Every metrically
projective module is $1$-topologically projective and every $C$-topologically 
projective module is $C$-relatively projective.
\end{proposition}
\begin{proof} By proposition~\ref{MetCTopProjModViaCanonicMorph} every 
metrically projecive $A$-module $P$ is a $1$-retract 
of $A_+\projtens \ell_1(B_P)$. Hence by the same proposition $P$ 
is $1$-topologically projecive. Again by 
proposition~\ref{MetCTopProjModViaCanonicMorph} every $C$-topologically 
projecive $A$-module $P$ is a $C$-retract of $A$-module $A_+\projtens \ell_1(B_P)$. 
In other words $P$ is a $C$-retract of the module $A\projtens E$ for some Banach
space $E$. Therefore $P$ is $C$-relatively projecive.
\end{proof}

Clearly, every $C$-topologically projecive module is $C'$-topologically 
projecive for $C'\geq C$. But we can state even more.

\begin{proposition}\label{MetProjIsOneTopProj} An $A$-module is metrically
projective iff it is $1$-topologically projective.
\end{proposition}
\begin{proof} The result immediately follows 
from propositions~\ref{MetProjIsTopProjAndTopProjIsRelProj} 
and~\ref{MetCTopProjModViaCanonicMorph}.
\end{proof}

Let us proceed to examples. Note that the category of Banach spaces may be
regarded as the category of left Banach modules over zero algebra. As the
results we get the definition of $\langle$~metrically / topologically~$\rangle$
projective Banach space. All the results mentioned above hold for this type of
projectivity. Both types of projective objects are described by now.
In~\cite{KotheTopProjBanSp} K{\"o}the proved that all topologically projective
Banach spaces are topologically isomorphic to $\ell_1(\Lambda)$ for some index
set $\Lambda$. Using result of Grothendieck from~\cite{GrothMetrProjFlatBanSp}
Helemskii showed that metrically projective Banach spaces are isometrically
isomorphic to $\ell_1(\Lambda)$ for some index set $\Lambda$
[\cite{HelMetrFrQMod}, proposition 3.2]. Thus the zoo of projective Banach
spaces is wide but conformed.

\begin{proposition}\label{NonDegenMetTopProjCharac}  Let $P$ be an essential
$A$-module. Then $P$ is $\langle$~metrically / $C$-topologically~$\rangle$
projective iff the map 
$\pi_P:A\projtens\ell_1(B_P):a\projtens\delta_x\mapsto a\cdot x$ 
is a $\langle$~$1$-retraction / $C$-retraction~$\rangle$ in $A-\mathbf{mod}$.
\end{proposition} 
\begin{proof}
If $P$ is $\langle$~metrically / $C$-topologically~$\rangle$ projective, then by
proposition~\ref{MetCTopProjModViaCanonicMorph} the morphism 
$\pi_P^+$ has a right inverse morphism $\sigma^+$ of norm 
$\langle$~at most $1$ / at most $C$~$\rangle$. Then 
$$
\sigma^+(P)
=\sigma^+(\operatorname{cl}_{A_+\projtens\ell_1(B_P)}(AP))
\subset \operatorname{cl}_{A_+\projtens\ell_1(B_P)}(A\cdot\sigma(P))=
$$
$$
\operatorname{cl}_{A_+\projtens\ell_1(B_P)}(A\cdot(A_+\projtens\ell_1(B_P)))
=A\projtens\ell_1(B_P).
$$ 
So we have well a defined corestriction $\sigma:P\to A\projtens\ell_1(B_P)$
which is also an $A$-morphism with norm $\langle$~at most $1$ / at most
$C$~$\rangle$. Clearly, $\pi_P\sigma=1_P$, so $\pi_P$ is a
$\langle$~$1$-retraction / $C$-retraction~$\rangle$ in $A-\mathbf{mod}$.

Conversely, assume $\pi_P$ has a right inverse morphism $\sigma$ of norm
$\langle$~at most $1$ / at most $C$~$\rangle$. Then its coextension $\sigma^+$
also is a right inverse morphism to $\pi_P^+$ with the same norm. Again, by
proposition~\ref{MetCTopProjModViaCanonicMorph} the module $P$ is
$\langle$~metrically / $C$-topologically~$\rangle$ projective. 
\end{proof}

It is worth to mention that $\langle$~an arbitrary / only finite~$\rangle$
family of objects in $\langle$~$A-\mathbf{mod}_1$ / $A-\mathbf{mod}$~$\rangle$
have the categorical coproduct which is in fact their $\bigoplus_1$-sum. This is
reason why we make additional assumption in the second paragraph of the next
proposition.

\begin{proposition}\label{MetTopProjModCoprod} Let
${(P_\lambda)}_{\lambda\in\Lambda}$ be a family of $A$-modules. Then 
\begin{enumerate}[label = (\roman*)]
    \item $\bigoplus_1 \{P_\lambda:\lambda\in\Lambda \}$ is metrically
    projective iff for all $\lambda\in\Lambda$ the $A$-module $P_\lambda$ is
    metrically projective;

    \item $\bigoplus_1 \{P_\lambda:\lambda\in\Lambda \}$ is $C$-topologically
    projective iff for all $\lambda\in\Lambda$ the $A$-module $P_\lambda$ is
    $C$-topologically projective.
\end{enumerate}
\end{proposition}
\begin{proof} Denote $P:=\bigoplus_1 \{P_\lambda:\lambda\in\Lambda \}$.

$(i)$ The proof is literally the same as in paragraph $(ii)$.

$(ii)$ Assume that $P$ is $C$-topologically projective. Note that, for any
$\lambda\in\Lambda$ the $A$-module $P_\lambda$ is a $1$-retract of $P$ via
natural projection $p_\lambda:P\to P_\lambda$. By
proposition~\ref{RetrMetCTopProjIsMetCTopProj} the $A$-module $P_\lambda$ is
$C$-topologically projective.

Conversely, let each $A$-module $P_\lambda$ be $C$-topologically projective. By
proposition~\ref{MetCTopProjModViaCanonicMorph} we have a family 
of $C$-retractions $\pi_\lambda:A_+\projtens\ell_1(S_\lambda)\to P_\lambda$. 
It follows that
$\bigoplus_1 \{\pi_\lambda:\lambda\in\Lambda \}$ is a $C$-retraction in
$A-\mathbf{mod}$. As the result $P$ is a $C$-retract of 
$$
\bigoplus\nolimits_1\left \{A_+\projtens
\ell_1(S_\lambda):\lambda\in\Lambda\right \} \isom{A-\mathbf{mod}_1}
\bigoplus\nolimits_1\left \{\bigoplus\nolimits_1 \{A_+:s\in S_\lambda
\}:\lambda\in\Lambda\right \}
$$
$$
\isom{A-\mathbf{mod}_1} \bigoplus\nolimits_1 \{A_+:s\in S \}
\isom{A-\mathbf{mod}_1} A_+ \projtens \ell_1(S)
$$
in $A-\mathbf{mod}$ where $S=\bigsqcup_{\lambda\in\Lambda}S_\lambda$. Clearly,
the latter module is $1$-topologically projective, so by
proposition~\ref{RetrMetCTopProjIsMetCTopProj} the $A$-module $P$ is 
$C$-topologically projective.
\end{proof}

\begin{corollary}\label{MetTopProjTensProdWithl1} Let $P$ be an $A$-module and
$\Lambda$ be an arbitrary set. Then $P\projtens \ell_1(\Lambda)$ is
$\langle$~metrically / $C$-topologically~$\rangle$ projective iff $P$ is
$\langle$~metrically / $C$-topologically~$\rangle$ projective.
\end{corollary}
\begin{proof} 
Note that 
$P\projtens \ell_1(\Lambda)
\isom{A-\mathbf{mod}_1}
\bigoplus_1 \{P:\lambda\in\Lambda \}$. It remains to set $P_\lambda=P$ for all
$\lambda\in\Lambda$ and apply proposition~\ref{MetTopProjModCoprod}.
\end{proof}

The property of being metrically, topologically or relatively projective module 
puts some restrictions on the Banach geometric structure of the module.

\begin{proposition}[\cite{RamsHomPropSemgroupAlg}, 
    proposition 2.1.1]\label{MetTopRelProjModCompIdealPartCompl} Let $P$ 
be a $\langle$~metrically / $C$-topologically / $C$-relatively~$\rangle$ 
projective $A$-moudle, and let $I$ 
be a $\langle$~$1$-complemented / $c$-complemented / $c$-complemented~$\rangle$
right ideal of $A$. Then $\operatorname{cl}_P(I P)$ 
is $\langle$~$1$-complemented / $cC$-complemented / $cC$-complemented~$\rangle$ 
in $P$.
\end{proposition}
\begin{proof} Since $A$ is $1$-complemented in $A_+$, then $I$ 
is $\langle$~$1$-complemented / $c$-complemented / $c$-complemented~$\rangle$
in $A_+$. Hence there exists a bounded linear operator $r:A_+\to I$ of 
$\langle$~norm $1$ / norm $c$ / norm $c$~$\rangle$ such that $r|_I=1_I$. 
Consider operator $R=r\projtens 1_P$, then $R|_{I\projtens P}=1_{I\projtens P}$ 
and $\Vert R\Vert\leq\Vert r\Vert$. Since $P$ 
is $\langle$~metrically / $C$-topologically / $C$-relatively~$\rangle$ 
projective, by proposition~\ref{MetCTopProjModViaCanonicMorph} the 
$A$-morphism $\pi_P^+$ has a right inverse morphism $\sigma^+$ of norm
$\langle$~at most $1$ / at most $C$ / at most $C$~$\rangle$. Now consider 
bounded linear operator $p=\pi_P^+ R \sigma^+$. Clearly,
$\operatorname{Im}(p)=\pi_P^+(R(\sigma^+(P)))
\subset\pi_P^+(R(A_+\projtens P))
\subset \pi_P^+(I\projtens P)
\subset \operatorname{cl}_P(I P)$.
Since $I$ is a right ideal of $A$ and $\sigma^+$ is an $A$-morphism, then 
$$
\sigma^+(\operatorname{cl}_P(I P))
\subset \operatorname{cl}_{A_+\projtens P}(\sigma^+(I P))
\subset \operatorname{cl}_{A_+\projtens P}(I\sigma^+(P))
\subset \operatorname{cl}_{A_+\projtens P}(I (A_+\projtens P))
\subset I \projtens P.
$$
In particular, for all $x\in \operatorname{cl}_P(I P)$, we have 
$p(x)
=\pi_P^+(R(\sigma^+(x)))
=\pi_P^+(1_{I\projtens P}(\sigma^+(x)))
=\pi_P^+(\sigma^+(x))
=x$. 
Thus $p$ is projection onto $\operatorname{cl}_P(I P)$ with norm at 
most $\Vert \sigma^+\Vert\Vert r\Vert$. Hence $\operatorname{cl}_P(I P)$ 
is $\langle$~$1$-complemented / $cC$-complemented / 
$cC$-complemented~$\rangle$ in $P$.
\end{proof}

\begin{corollary}\label{MetTopRelProjModPartCompl} Let $P$ 
be a $\langle$~metrically / $C$-topologically / $C$-relatively~$\rangle$ 
projective $A$-moudle. Then $P_{ess}$ 
is $\langle$~$1$-complemented / $C$-complemented / $C$-complemented~$\rangle$ 
in $P$.
\end{corollary}
\begin{proof} The result directly follows from 
proposition~\ref{MetTopRelProjModCompIdealPartCompl}.
\end{proof}

%-------------------------------------------------------------------------------
%    Metric and topological projectivity of ideals and cyclic modules
%-------------------------------------------------------------------------------

\subsection{
    Metric and topological projectivity of ideals and cyclic modules
}\label{SubSectionMetricAndTopologicalProjectivityOfIdealsAndCyclicModules}

As we shall see idempotents play a significant role in the study of metric and
topological projectivity, so we shall recall one of the corollaries of Shilov's
idempotent theorem [\cite{KaniBanAlg}, section 3.5]: every semisimple
commutative Banach algebra with compact spectrum admits an identity, but not
necessarily of norm 1. 

\begin{proposition}\label{UnIdeallIsMetTopProj} Let $I$ be a left ideal of a
Banach algebra $A$. Then

\begin{enumerate}[label = (\roman*)]
    \item if $I=Ap$ for some $\langle$~norm one idempotent /
    idempotent~$\rangle$ $p\in I$, then $I$ is $\langle$~metrically / 
    $\Vert p\Vert$-topologically~$\rangle$ projective $A$-module;

    \item if $I$ is commutative, semisimple and $\operatorname{Spec}(I)$ is
    compact then $I$ is topologically projective $A$-module.
\end{enumerate}
\end{proposition}
\begin{proof} 
$(i)$ For $A$-module maps $\pi:A_\times\to I:x\mapsto xp$ and 
$\sigma:I\to A_\times:x\mapsto x$ we clearly have $\pi\sigma=1_I$. 
Therefore $I$ is a $\langle$~$1$-retract / $\Vert p\Vert$-retract~$\rangle$ 
of $A_\times$. Now the result follows from 
propositions~\ref{RetrMetCTopProjIsMetCTopProj} and~\ref{UnitalAlgIsMetTopProj}.

$(ii)$ By Shilov's idempotent theorem the ideal $I$ is unital. In general the
norm of this unit is not less than $1$. By paragraph $(i)$ the ideal $I$ is
topologically projective.
\end{proof}

The assumption of semisimplicity in~\ref{UnIdeallIsMetTopProj} is not necessary.
From [\cite{DalesIntroBanAlgOpHarmAnal}, exercise 2.3.7] we know, that there
exists a commutative non semisimple unital Banach algebra $A$. By
proposition~\ref{UnIdeallIsMetTopProj} it is topologically projective as
$A$-module. To prove the main result of this section we need two preparatory
lemmas.

\begin{lemma}\label{ImgOfAMorphFromBiIdToA} Let $I$ be a two-sided ideal of a
Banach algebra $A$ which is essential as left $I$-module and let $\phi:I\to A$
be an $A$-morphism. Then $\operatorname{Im}(\phi)\subset I$.
\end{lemma}
\begin{proof} Since $I$ is a right ideal, then $\phi(ab)=a\phi(b)\in I$ for all
$a,b\in I$. So $\phi(I\cdot I)\subset I$. Since $I$ is an essential left
$I$-module then $I=\operatorname{cl}_A(\operatorname{span}(I\cdot I))$ and
$\operatorname{Im}(\phi)
\subset\operatorname{cl}_A(\operatorname{span}\phi(I\cdot I))
=\operatorname{cl}_A(\operatorname{span}I)
=I$.
\end{proof}

\begin{lemma}\label{GoodIdealMetTopProjIsUnital} Let $I$ be a left ideal of a
Banach algebra $A$, which is $\langle$~metrically / $C$-topologically~$\rangle$
projective as $A$-module. Then the following holds:

\begin{enumerate}[label = (\roman*)]
    \item Assume $I$ has a left $\langle$~contractive / $c$-bounded~$\rangle$
    approximate identity and for each  morphism $\phi:I\to A$ of left
    $A$-modules there exists a morphism $\psi:I\to I$ of right $I$-modules such
    that $\phi(x)y=x\psi(y)$ for all $x,y\in I$. Then $I$ has the identity of
    norm $\langle$~at most $1$ / at most $c$~$\rangle$;

    \item Assume $I$ has a right $\langle$~contractive / $c$-bounded~$\rangle$
    approximate identity and for $\langle$~$k=1$ / some $k\geq 1$~$\rangle$ and
    each morphism $\phi:I\to A$ of left $A$-modules there exists a morphism
    $\psi:I\to I$ of right $I$-modules such that $\Vert\psi\Vert\leq
    k\Vert\phi\Vert$ and $\phi(x)y=x\psi(y)$ for all $x,y\in I$. Then $I$ has a
    right identity of norm $\langle$~at most $1$ / at most $ckC$~$\rangle$.
\end{enumerate}
\end{lemma} 
\begin{proof} If either $(i)$ or $(ii)$ holds then $I$ has a one-sided bounded
approximate identity. So $I$ is an essential left $I$-module, and a fortiori an
essential $A$-module. By proposition~\ref{NonDegenMetTopProjCharac} we have a
right inverse $A$-morphism $\sigma:I\to A\projtens \ell_1(B_I)$ of $\pi_I$ with
norm  $\langle$~at most $1$ / at most $C$~$\rangle$. For each $d\in B_I$
consider $A$-morphisms $p_d:A\projtens \ell_1(B_I)\to A:a\projtens
\delta_x\mapsto \delta_x(d)a$ and $\sigma_d=p_d\sigma$. Then
$\sigma(x)=\sum_{d\in B_I}\sigma_d(x)\projtens \delta_d$ for all $x\in I$. From
identification 
$A\projtens\ell_1(B_I)\isom{\mathbf{Ban}_1}\bigoplus_1 \{A:d\in B_I \}$ 
we have $\Vert\sigma(x)\Vert=\sum_{d\in B_I} \Vert\sigma_d(x)\Vert$ for
all $x\in I$. Since $\sigma$ is a right inverse morphism of $\pi_I$, then
$x=\pi_I(\sigma(x))=\sum_{d\in B_I}\sigma_d(x)d$ for all $x\in I$. 

From assumption, for each $d\in B_I$ there exists a morphism of right
$I$-modules $\tau_d:I\to I$ such that $\sigma_d(x)d=x\tau_d(d)$ for all 
$x\in I$. 

Assume $(i)$ holds. From assumption, for each $d\in B_I$ there exists a morphism
of right $I$-modules $\tau_d:I\to I$ such that $\sigma_d(x)d=x\tau_d(d)$ for all
$x\in I$.  Let ${(e_\nu)}_{\nu\in N}$ be a left $\langle$~contractive /
bounded~$\rangle$ approximate identity of $I$ bounded in norm by constant
$\langle$~$D=1$ / $D=c$~$\rangle$. Since $\tau_d(d)\in I$ for all $d\in B_I$,
then for all $S\in\mathcal{P}_0(B_I)$ holds
$$
\sum_{d\in S}\Vert \tau_d(d)\Vert
=\sum_{d\in S}\lim_{\nu}\Vert e_\nu \tau_d(d) \Vert
=\lim_{\nu}\sum_{d\in S}\Vert e_\nu \tau_d(d)\Vert
=\lim_{\nu}\sum_{d\in S}\Vert \sigma_d(e_\nu)d \Vert
$$
$$
\leq\liminf_{\nu}\sum_{d\in S}\Vert\sigma_d(e_\nu)\Vert\Vert d\Vert 
\leq\liminf_{\nu}\sum_{d\in S}\Vert\sigma_d(e_\nu)\Vert
\leq\liminf_{\nu}\sum_{d\in B_I}\Vert\sigma_d(e_\nu)\Vert
$$
$$
=\liminf_{\nu}\Vert\sigma(e_\nu)\Vert
\leq\Vert\sigma\Vert\liminf_{\nu}\Vert e_\nu\Vert
\leq D\Vert\sigma\Vert 
$$
Since $S\in \mathcal{P}_0(B_I)$ is arbitrary we have well defined element
$p=\sum_{d\in B_I}\tau_d(d)$ with norm $\langle$ at most $1$ / at most
$cC$~$\rangle$. For all $x\in I$ we have $x=\sum_{d\in
B_I}\sigma_d(x)d=\sum_{d\in B_I}x\tau_d(d)=xp$, i.e. $p$ is a right identity for
$I$. But $I$ admits a left $\langle$~contractive / $c$-bounded~$\rangle$
approximate identity, so $p$ is the identity of $I$ with $\Vert
p\Vert=\lim_\nu\Vert e_\nu\Vert$. Therefore the norm of $p$ is $\langle$~at most
$1$ / at most $c$~$\rangle$.

Assume $(ii)$ holds. From assumption, for each $d\in B_I$ there exists a
morphism of right $I$-modules $\tau_d:I\to I$ such that
$\sigma_d(x)d=x\tau_d(d)$ for all $x\in I$ and 
$\Vert\tau_d\Vert\leq k\Vert\sigma_d\Vert$. 
Let ${(e_\nu)}_{\nu\in N}$ be a right $\langle$~contractive / bounded~$\rangle$ 
approximate identity of $I$ bounded in norm by constant 
$\langle$~$D=1$ / $D=c$~$\rangle$. For all $x\in I$ we have
$$
\Vert\sigma_d(x)\Vert
=\Vert\sigma_d(\lim_\nu x e_\nu)\Vert
=\lim_\nu\Vert x\sigma_d(e_\nu)\Vert
\leq\Vert x\Vert\liminf_\nu\Vert\sigma_d(e_\nu)\Vert
$$
so $\Vert\sigma_d\Vert\leq \liminf_\nu\Vert\sigma_d(e_\nu)\Vert$. 
Then for all $S\in\mathcal{P}_0(B_I)$ holds
$$
\sum_{d\in S}\Vert \tau_d(d)\Vert
\leq \sum_{d\in S}\Vert \tau_d\Vert\Vert d\Vert
\leq k\sum_{d\in S}\Vert \sigma_d\Vert
\leq k\sum_{d\in S}\liminf_\nu \Vert \sigma_d(e_\nu)\Vert
\leq k\liminf_{\nu}\sum_{d\in S}\Vert \sigma_d(e_\nu) \Vert
$$
$$
\leq k\liminf_{\nu}\sum_{d\in B_I}\Vert \sigma_d(e_\nu) \Vert
=k\liminf_{\nu}\Vert\sigma(e_\nu)\Vert
\leq k\Vert\sigma\Vert\liminf_{\nu}\Vert e_\nu\Vert
\leq kD\Vert\sigma\Vert
$$
Since $S\in \mathcal{P}_0(B_I)$ is arbitrary we have well defined element
$p=\sum_{d\in B_I}\tau_d(d)$ with norm $\langle$ at most $1$ / at most
$ckC$~$\rangle$. For all $x\in I$ 
we have $x=\sum_{d\in B_I}\sigma_d(x)d=\sum_{d\in B_I}x\tau_d(d)=xp$, 
i.e. $p$ is a right identity for $I$.
\end{proof}

\begin{theorem}\label{GoodCommIdealMetTopProjIsUnital} Let $I$ be an ideal of a
commutative Banach algebra $A$ and $I$ has a $\langle$~contractive /
$c$-bounded~$\rangle$ approximate identity. Then $I$ is $\langle$~metrically /
$c$-topologically~$\rangle$ projective as $A$-module iff $I$ has the identity of
norm $\langle$~at most $1$ / at most $c$~$\rangle$.
\end{theorem} 
\begin{proof} Assume $I$ is $\langle$~metrically / $c$-topologically~$\rangle$
projective as $A$-module. Since $A$ is commutative, then for any $A$-morphism
$\phi:I\to A$ and $x,y\in I$ we have $\phi(x)y=x\phi(y)$. Since $I$ has a 
bounded approximate identity and $I$ is commutative we can apply
lemma~\ref{ImgOfAMorphFromBiIdToA} to get that $\phi(y)\in I$. Now by paragraph
$(i)$ of lemma~\ref{GoodIdealMetTopProjIsUnital} we get that $I$ has the
identity of norm $\langle$~at most $1$ / at most $c$~$\rangle$.

The converse immediately follows from proposition~\ref{UnIdeallIsMetTopProj}.
\end{proof}

There is no analogous criterion of this theorem in relative theory. The most
general result of this kind gives only a necessary condition: any ideal in a
commutative Banach algebra $A$ which is relatively projective as $A$-module has
a paracompact spectrum. This result is due to Helemskii
[\cite{HelHomolBanTopAlg}, theorem IV.3.6]. 

Note that existence of bounded approximate identity is not necessary for an
ideal of a commutative Banach algebra to be even topologically projective.
Indeed, consider Banach algebra  $A_0(\mathbb{D})$ --- the ideal of the disk
algebra consisting of functions vanishing at zero. By combination of 
propositions 4.3.5 and 4.3.13 paragraph $(iii)$ from~\cite{DalBanAlgAutCont} 
we get that $A_0(\mathbb{D})$ has no bounded approximate identities. 
On the other hand, from [\cite{HelBanLocConvAlg}, example IV.2.2] we know that
$A_0(\mathbb{D})\isom{A_0(\mathbb{D})-\mathbf{mod}} {A_0(\mathbb{D})}_+$, so
$A_0(\mathbb{D})$ is topologically projective by
proposition~\ref{UnitalAlgIsMetTopProj}.

Next proposition is an obvious adaptation of purely algebraic argument on
projective cyclic modules. It is almost identical to [\cite{WhiteInjmoduAlg},
proposition 2.11].

\begin{proposition}\label{MetTopProjCycModCharac} Let $I$ be a left ideal in
$A_\times $. Assume the natural projection $\pi:A_\times\to A_\times/I$ is
strictly $\langle$~coisometric / $c$-topologically surjective~$\rangle$. 
Then the following holds:

\begin{enumerate}[label = (\roman*)]
    \item If $A_\times /I$ is $\langle$~metrically / $C$-topologically~$\rangle$
    projective as $A$-module, then there exists an idempotent $p\in I$ such that
    $I=Ap$ and $\Vert e_{A_\times}-p\Vert$ is $\langle$~at most $1$ / at most
    $cC$~$\rangle$;

    \item If there exists an idempotent $p\in I$ such that $I=A_\times  p$ and
    $\Vert e_{A_\times }-p\Vert$ is $\langle$~at most $1$ / at most
    $C$~$\rangle$, then $A_\times/I$ is $\langle$~metrically /
    $C$-topologically~$\rangle$ projective.
\end{enumerate}
\end{proposition}
\begin{proof} $(i)$ Since the natural quotient map $\pi$ is strictly 
$\langle$~coisometric / $c$-topologically surjective~$\rangle$ and $A_\times /I$
is $\langle$~metrically / $C$-topologically~$\rangle$ projective, then $\pi$ has
a right inverse $A$-morphism $\sigma$ with norm $\langle$~at most $1$ / at most
$cC$~$\rangle$. We set $e_{A_\times }-p=(\sigma\pi)(e_{A_\times })$, then
$(\sigma\pi)(a)=a(e_{A_\times }-p)$. By construction, $\pi\sigma=1_{A_\times }$,
so  
$$
e_{A_\times }-p
=(\sigma\pi)(e_{A_\times })
=(\sigma\pi)(\sigma\pi)(e_{A_\times })
=(\sigma\pi)(e_{A_\times }-p)
=(e_{A_\times }-p)(\sigma\pi)(e_{A_\times })
={(e_{A_\times }-p)}^2
$$
This equality shows that $p^2=p$. Therefore $A_\times
p=\operatorname{Ker}(\sigma\pi)$ because $(\sigma\pi)(a)=a-ap$. Since $\sigma$
is injective this is equivalent to $A_\times p=\operatorname{Ker}(\pi)$ which
equals to $I$. Finally, note that 
$\Vert e_{A_\times }-p\Vert
=\Vert(\sigma\pi)(e_{A_\times})\Vert
\leq\Vert\sigma\Vert\Vert\pi\Vert\Vert e_{A_\times }\Vert
=\Vert\sigma\Vert$.

$(ii)$ Since $p^2=p$, then we have a well defined left ideal $I=A_\times p$ and
an $A$-module map $\sigma:A_\times /I\to A_\times:a+I\mapsto a-ap$. It is easy
to check that $\pi\sigma=1_{A_\times/I}$ 
and $\Vert\sigma\Vert\leq\Vert e_{A_\times }-p\Vert$. 
This means that $\pi:A_\times \to A_\times /I$ is a
$\langle$~$1$-retraction / $C$-retraction~$\rangle$. From
propositions~\ref{UnitalAlgIsMetTopProj} 
and~\ref{RetrMetCTopProjIsMetCTopProj} it follows that $A_\times /I$ is
$\langle$~metrically / $C$-topologically~$\rangle$ projective.
\end{proof} 

In contrast with topological theory, there is no description of relatively
projective cyclic modules. There are partial answers under additional
assumptions. For example, if an ideal $I$ is complemented as Banach space in
$A_\times$, then almost the same criterion as in previous proposition holds in
relative theory [\cite{HelBanLocConvAlg}, proposition 7.1.29]. There are other
characterizations of relatively projective cyclic modules under more mild
assumptions on Banach geometry. For example, Selivanov proved that if $I$ is a
two-sided ideal and either $A/I$ has the approximation property or all
irreducible $A$-modules have the approximation property, then $A/I$ is
relatively projective iff $A_\times\isom{A-\mathbf{mod}}I\bigoplus_1 I'$ for
some left ideal $I'$ of $A$. For details see [\cite{HelHomolBanTopAlg}, chapter
IV, \S 4].

%-------------------------------------------------------------------------------
%    Metric and topological injectivity
%-------------------------------------------------------------------------------

\subsection{
    Metric and topological injectivity
}\label{SubSectionMetricAndTopologicalInjectivity}

Unless otherwise stated we shall consider injectivity of right modules.

\begin{definition}[\cite{HelMetrFrQMod}, 
    definition 4.3;~\cite{WhiteInjmoduAlg}, 
    definition 3.4]\label{MetCTopInjMod} 
An $A$-module $J$ is 
called $\langle$~metrically / $C$-topologically~$\rangle$ injective if for any 
$\langle$~isometric / $c$-topologically injective~$\rangle$
$A$-morphism $\xi:Y\to X$ and for any $A$-morphism $\phi:Y\to J$ 
there exists an $A$-morphism $\psi:X\to J$ such that $\psi\xi=\phi$  and
$\langle$~$\Vert\psi\Vert\leq\Vert\phi\Vert$ / 
$\Vert \psi\Vert\leq c C\Vert\phi\Vert$~$\rangle$. We say that an $A$-module 
$J$ is topologically injective if it is $C$-topologically injective for 
some $C\geq 1$.
\end{definition}

The task of constructing an $A$-morphism $\psi$ for a given $A$-morphisms $\phi$ 
and $\xi$ in the definition~\ref{MetCTopInjMod} is called an extension problem 
and $\psi$ is called an extension of $\phi$ along $\xi$.

A short but more involved equivalent definition 
of $\langle$~metric / $C$-topological~$\rangle$ injectivity is the following: 
an $A$-module $J$ is called $\langle$~metrically / $C$-topologically~$\rangle$ 
injective, if the functor
$\langle$~$\operatorname{Hom}_{A-\mathbf{mod}_1}(-,J)
:A-\mathbf{mod}_1\to\mathbf{Ban}_1$
/
$\operatorname{Hom}_{A-\mathbf{mod}}(-,J)
:A-\mathbf{mod}\to\mathbf{Ban}$~$\rangle$
maps $\langle$~isometric / $c$-topologically injective~$\rangle$
$A$-morphisms into strictly $\langle$~coisometric / 
$c C$-topologically surjective~$\rangle$ operators. 

In category $\langle$~$A-\mathbf{mod}_1$ / $A-\mathbf{mod}$~$\rangle$ there
is a special class of $\langle$~metrically / $1$-topologically~$\rangle$
injective modules of the form $\mathcal{B}(A_+, \ell_\infty(\Lambda))$ for 
some set $\Lambda$. They are called cofree modules. These modules play a crucial 
role in our studies of injectivity.

\begin{proposition}[\cite{WhiteInjmoduAlg}, 
    lemma 3.6]\label{MetCTopCofreeMod} Let $\Lambda$ be an arbitrary 
set, then the left $A$-modules $\mathcal{B}(A_+, \ell_\infty(\Lambda))$ 
and $\mathcal{B}(A_\times, \ell_\infty(\Lambda))$
are $\langle$~metrically / $1$-topologically~$\rangle$ injective. 
\end{proposition}
\begin{proof} By $A_\bullet$ we denote either $A_+$ or $A_\times$.
Consider arbitrary $A$-morphism 
$\phi:Y\to \mathcal{B}(A_\bullet, \ell_\infty(\Lambda))$ and 
$\langle$~an isometric / $c$-topologically injective~$\rangle$ 
$A$-morphism $\xi:Y\to X$. Fix arbitrary $\lambda\in\Lambda$ and define 
a bounded linear functional 
$h_\lambda:Y\to\mathbb{C}:y\mapsto \phi(y)(e_{A_\times})(\lambda)$. 
Clearly, $\Vert h_\lambda\Vert\leq\Vert \phi\Vert$.
Denote $X_0=\operatorname{Im}(\xi)$ and consider 
$A$-morphism $\eta=\xi|^{X_0}$. Since $\xi$ is 
$\langle$~an isometric / $c$-topologically injective~$\rangle$, then $X_0$ is
closed and $\eta$ has a left inverse bounded linear operator $\zeta:X_0\to Y$, 
such that $\zeta$ has norm at most $\langle$~$K=1$ / $K=c$~$\rangle$. 
Now consider a bounded linear functional  $f_\lambda=h_\lambda\zeta\in X_0^*$. 
By Hahn-Banach theorem we can extend $f_\lambda$ to some bounded linear functional
$g_\lambda:X\to\mathbb{C}$ with norm  
$\Vert g_\lambda\Vert
=\Vert f_\lambda\Vert
\leq\Vert h_\lambda\Vert\Vert\zeta\Vert
\leq K\Vert \phi\Vert$. 
Consider $A$-morphism 
$\psi
:X\to \mathcal{B}(A_\bullet, \ell_\infty(\Lambda))
:x\mapsto (a\mapsto (\lambda\mapsto g_\lambda(x\cdot a)))$. 
It is easy to check that $\psi\xi=\phi$ and 
$\Vert\psi\Vert\leq K\Vert\phi\Vert$. Thus for a given $\phi$ 
and $\xi$ we have constructed an $A$-morphism $\psi$ such that $\psi\xi=\phi$ 
and $\langle$~$\Vert\psi\Vert\leq\Vert\phi\Vert$ /
$\Vert\psi\Vert\leq c\Vert\phi\Vert$~$\rangle$. Hence the 
$A$-module $\mathcal{B}(A_\bullet, \ell_\infty(\Lambda))$ is 
$\langle$~metrically / $1$-topologically~$\rangle$ injective.
\end{proof}

\begin{proposition}\label{DualOfUnitalAlgIsMetTopInj} The right $A$-module
$A_\times^*$ is metrically and $1$-topologically injective.
\end{proposition}
\begin{proof} Consier set $\Lambda=\mathbb{N}_1$. 
By proposition~\ref{MetCTopCofreeMod} the 
$A$-module $\mathcal{B}(A_\times, \ell_\infty(\Lambda))$ is metrically 
and $1$-topologically injective. Now it remains to note that 
$\mathcal{B}(A_\times, \ell_\infty(\Lambda))
\isom{\mathbf{mod}_1-A}\\
\mathcal{B}(A_\times, \mathbb{C})
\isom{\mathbf{mod}_1-A}
A_\times^*$.
\end{proof}

\begin{proposition}[\cite{WhiteInjmoduAlg}, 
    lemma 3.7]\label{RetrMetCTopInjIsMetCTopInj} Any 
$\langle$~$1$-retract / $c$-retract~$\rangle$ of
$\langle$~metrically / $C$-topologically~$\rangle$ injective module is
$\langle$~metrically / $c C$-topologically~$\rangle$ injective.
\end{proposition}
\begin{proof} Suppose that $J$ is a $c$-retract of 
$\langle$~metrically / $C$-topologically~$\rangle$ injective $A$-module $J'$.
Then there exist $A$-morphisms $\eta:J\to J'$ and $\zeta: J'\to J$ such that
$\zeta\eta=1_{J}$ and $\Vert\zeta\Vert\Vert\eta\Vert\leq c$ 
for $\langle$~$c=1$ / $c\geq 1$~$\rangle$. Consider 
arbitrary $\langle$~isometric / $c'$-topologically injective~$\rangle$ 
$A$-morphism $\xi:Y\to X$ and an arbitrary $A$-morphism $\phi:Y\to J$. 
Consider $A$-morphism $\phi'=\eta\phi$. Since $J'$ is 
$\langle$~metrically / $C$-topologically~$\rangle$ injective, then there 
exists an $A$-morphism $\psi':X\to J'$ such that $\phi'=\psi'\xi$ 
and $\Vert\psi'\Vert \leq K\Vert\phi'\Vert$ 
for $\langle$~$K=1$ / $K=c' C$~$\rangle$. Now it is routine to check that 
for the $A$-morphism $\psi=\zeta\psi'$ 
holds $\psi\xi=\phi$ and $\Vert\psi\Vert \leq cK\Vert\phi\Vert$. 
Thus for a given $\phi$ and $\xi$ we have
constructed an $A$-morphism $\psi$ such that $\psi\xi=\phi$  and
$\langle$~$\Vert\psi\Vert\leq\Vert\phi\Vert$ / 
$\Vert \psi\Vert\leq c C\Vert\phi\Vert$~$\rangle$. Hence $J$ 
is $\langle$~metrically / $c C$-topologically~$\rangle$ injective.
\end{proof}

It is easy to show that, for any $A$-module $X$ there exists 
$\langle$~an isometric / a $1$-topologically injective~$\rangle$ 
$A$-morphism
$$
\rho_X^+
:X\to\mathcal{B}(A_+, \ell_\infty(B_{X^*}))
:x\mapsto(a\mapsto(f\mapsto f(x\cdot a)))
$$

\begin{proposition}[\cite{WhiteInjmoduAlg}, 
    proposition 3.10]\label{MetCTopInjModViaCanonicMorph} An $A$-module $J$ is
$\langle$~metrically / $C$-topologically~$\rangle$ injective iff $\rho_J^+$ is 
a $\langle$~$1$-retraction / $C$-retraction~$\rangle$ in $A-\mathbf{mod}$.
\end{proposition}
\begin{proof}
Suppose $J$ is 
$\langle$~metrically / $C$-topologically~$\rangle$ injective, then consider 
$\langle$~isometric / $1$-topologically injective~$\rangle$
$A$-morphism $\rho_J^+$. Consider extension problem with $\phi=1_J$ 
and $\xi=\rho_J^+$, then from $\langle$~metric / $C$-topological~$\rangle$ 
injectivity of $J$ we get an $A$-morphism $\tau^+$ such 
that $\tau^+\rho_J^+=1_J$
and $\Vert\tau^+\Vert\leq K$ for $\langle$~$K=1$ / $K=C$~$\rangle$. 
Since $\Vert\rho_J^+\Vert\Vert \tau^+\Vert\leq K$ 
we conclude that $\rho_J^+$ is 
a $\langle$~$1$-retraction / $C$-retraction~$\rangle$ in $A-\mathbf{mod}$.

Conversely assume that $\rho_J^+$ is 
a $\langle$~$1$-retraction / $C$-retraction~$\rangle$. In other words $J$ is 
a $\langle$~$1$-retract / $C$-retract~$\rangle$ of 
$\mathcal{B}(A_+, \ell_\infty(B_{J^*}))$. By proposition~\ref{MetCTopCofreeMod} 
the $A$-module $\mathcal{B}(A_+, \ell_\infty(B_{J^*}))$ 
is $\langle$~metrically / $1$-topologically~$\rangle$ injective. So from 
proposition~\ref{RetrMetCTopInjIsMetCTopInj} 
its $\langle$~$1$-retract / $C$-retract~$\rangle$ $P$ is 
$\langle$~metrically / $C$-topologically~$\rangle$ injective.
\end{proof}

\begin{proposition}[\cite{WhiteInjmoduAlg}, 
    proposition 3.9]\label{MetInjIsTopInjAndTopInjIsRelInj} Every metrically
injective module is $1$-topologically injective and every $C$-topologically 
injective module is $C$-relatively injective.
\end{proposition}
\begin{proof} By proposition~\ref{MetCTopInjModViaCanonicMorph} every 
metrically injecive $A$-module $J$ is a $1$-retract 
of $\mathcal{B}(A_+, \ell_\infty(B_{J^*}))$. Hence by the same proposition $J$ 
is $1$-topologically injecive. Again by 
proposition~\ref{MetCTopInjModViaCanonicMorph} every $C$-topologically 
injecive $A$-module $J$ is a $C$-retract 
of $A$-module $\mathcal{B}(A_+, \ell_\infty(B_{J^*}))$. 
In other words $J$ is a $C$-retract of the module $\mathcal{B}(A_+, E)$ for 
some Banach space $E$. Therefore $J$ is $C$-relatively injecive.
\end{proof}

Clearly, every $C$-topologically injecive module is $C'$-topologically 
injecive for $C'\geq C$. But we can state even more.

\begin{proposition}\label{MetInjIsOneTopInj} An $A$-module is metrically
injective iff it is $1$-topologically injective.
\end{proposition}
\begin{proof} The result immediately follows 
from propositions~\ref{MetInjIsTopInjAndTopInjIsRelInj} 
and~\ref{MetCTopInjModViaCanonicMorph}.
\end{proof}

Let us proceed to examples. If we regard the category of Banach spaces as the
category of right Banach modules over zero algebra, we may speak of
$\langle$~metrically / topologically~$\rangle$ injective Banach spaces. All
results mentioned above hold for this type of injectivity. An equivalent
definition says that a Banach space is $\langle$~metrically /
topologically~$\rangle$ injective if it is $\langle$~contractively complemented
/ complemented~$\rangle$ in any ambient Banach space. The typical examples of
metrically injective Banach spaces are $L_\infty$-spaces. Only metrically
injective Banach spaces are completely understood --- these spaces are
isometrically isomorphic to $C(K)$-space for some extremely disconnected compact
Hausdorff space $K$ [\cite{LaceyIsomThOfClassicBanSp}, theorem 3.11.6]. Usually
such topological spaces are referred to as Stonean spaces.  For the contemporary
results on topologically injective Banach spaces see
[\cite{JohnLinHandbookGeomBanSp}, chapter 40].

\begin{proposition}\label{NonDegenMetTopInjCharac}  Let $J$ be a faithful
$A$-module. Then $J$ is $\langle$~metrically / $C$-topologically~$\rangle$
injective iff the map
$\rho_J:J\to\mathcal{B}(A,\ell_\infty(B_{J^*})):x\mapsto(a\mapsto(f\mapsto
f(x\cdot a)))$ is a $\langle$~$1$-coretraction / $C$-coretraction~$\rangle$ in
$\mathbf{mod}-A$.
\end{proposition} 
\begin{proof}
If $J$ is $\langle$~metrically / $C$-topologically~$\rangle$ injective, then by
proposition~\ref{MetCTopInjModViaCanonicMorph} the $A$-morphism $\rho_J^+$ has
right inverse morphism $\tau^+$, with norm $\langle$~at most $1$ / at most
$C$~$\rangle$. Assume we are given an operator $T\in
\mathcal{B}(A_+,\ell_\infty(B_{J^*}))$, such that $T|_A=0$. Fix $a\in A$, then
$T\cdot a=0$, and so $\tau^+(T)\cdot a=\tau^+(T\cdot a)=0$. Since $J$ is
faithful and $a\in A$ is arbitrary, then $\tau^+(T)=0$. Define $p:A_+\to A$ be
the natural projection from $A_+$ onto $A$, then define $A$-morphisms
$j=\mathcal{B}(p,\ell_\infty(B_{J^*}))$ and $\tau =\tau^+ j$. For any $a\in A$
and $T\in\mathcal{B}(A,\ell_\infty(B_{J^*}))$ 
we have $\tau (T\cdot a)-\tau (T)\cdot a=\tau^+(j(T\cdot a)-j(T)\cdot a)=0$, 
because $j(T\cdot a)-j(T)\cdot a|_A=0$. Therefore $\tau $ is an $A$-morphism. 
Note that $\Vert\tau \Vert\leq\Vert\tau^+\Vert\Vert j\Vert\leq\Vert\tau^+\Vert$.
Therefore $\tau$ has norm $\langle$~at most $1$ / at most $C$~$\rangle$. 
Obviously, for all $x\in J$ we have $\rho_J^+(x)-j(\rho_J(x))|_A=0$, 
so $\tau^+(\rho_J^+(x)-j(\rho_J(x)))=0$. 
As a consequence $\tau (\rho_J(x))=\tau^+(j(\rho_J(x)))=\tau^+(\rho_J^+(x))=x$ 
for all $x\in J$. Since $\tau \rho_J=1_J$, then $\rho_J$ 
is a  $\langle$~$1$-coretraction / $C$-coretraction~$\rangle$ 
in $\mathbf{mod}-A$.

Conversely, assume $\rho_J$ is a $\langle$~$1$-coretraction /
$C$-coretraction~$\rangle$, that is has a right inverse morphism $\tau $ with
norm $\langle$~at most $1$ / at most $C$~$\rangle$. Define $i:A\to A_+$ to be
the natural embedding of $A$ into $A_+$ and define $A$-morphism
$q=\mathcal{B}(i,\ell_\infty(B_{J^*}))$. Obviously, $\rho_J=q\rho_J^+$. Consider
$A$-morphism $\tau^+=\tau q$. 
Note that $\Vert\tau^+\Vert\leq\Vert\tau \Vert\Vert q\Vert\leq \Vert\tau \Vert$.
Therefore $\tau^+$ has norm $\langle$~at most $1$ / at most $C$~$\rangle$. 
Clearly $\tau^+\rho_J^+=\tau q\rho_J^+=\tau \rho_J=1_J$. So $\rho_J^+$ 
is a $\langle$~$1$-coretraction / $C$-coretraction~$\rangle$ and by 
proposition~\ref{MetCTopInjModViaCanonicMorph} the $A$-module $J$ is
$\langle$~metrically / $C$-topologically~$\rangle$ injective.
\end{proof}

It is worth to mention that $\langle$~arbitrary / only finite~$\rangle$ family
of objects in $\langle$~$\mathbf{mod}_1-A$ / $\mathbf{mod}-A$~$\rangle$ have the
categorical product which in fact is their $\bigoplus_\infty$-sum. This is the
reason why we make additional assumption in the second paragraph of the next
proposition.

\begin{proposition}\label{MetTopInjModProd} Let
${(J_\lambda)}_{\lambda\in\Lambda}$ be a family of $A$-modules. Then 

\begin{enumerate}[label = (\roman*)]
    \item $\bigoplus_\infty \{J_\lambda:\lambda\in\Lambda \}$ is metrically
    injective iff for all $\lambda\in\Lambda$ the $A$-module $J_\lambda$ is
    metrically injective;

    \item $\bigoplus_\infty \{J_\lambda:\lambda\in\Lambda \}$ is
    $C$-topologically injective iff for all $\lambda\in\Lambda$ the $A$-module
    $J_\lambda$ is a $C$-topologically injective.
\end{enumerate}
\end{proposition}
\begin{proof} Denote $J:=\bigoplus_\infty \{J_\lambda:\lambda\in\Lambda \}$.

$(i)$ The proof is literally the same as in paragraph $(ii)$.

$(ii)$ Assume that $J$ is $C$-topologically injective. Note that, for any
$\lambda\in\Lambda$ the $A$-module $J_\lambda$ is a $1$-retract of $J$ via
natural projection $p_\lambda:J\to J_\lambda$. By
proposition~\ref{RetrMetCTopInjIsMetCTopInj} the $A$-module $J_\lambda$ is
$C$-topologically injective.

Conversely, let each $A$-module $J_\lambda$ be $C$-topologically injective. By
proposition~\ref{MetCTopInjModViaCanonicMorph} we have a family of
$C$-coretractions
$\rho_\lambda:J_\lambda\to\mathcal{B}(A_+,\ell_\infty(S_\lambda))$. It follows
that $\bigoplus_\infty \{\rho_\lambda:\lambda\in\Lambda \}$ is a
$C$-coretraction in $A-\mathbf{mod}$. As the result $J$ is a $C$-retract of 
$$
\bigoplus\nolimits_\infty \{
    \mathcal{B}(A_+,\ell_\infty(S_\lambda)):\lambda\in\Lambda
 \}
\isom{\mathbf{mod}_1-A}
\bigoplus\nolimits_\infty\left \{
    \bigoplus\nolimits_\infty \{ A_+^*:s\in S_\lambda \}:\lambda\in\Lambda
\right \}
\isom{\mathbf{mod}_1-A}
$$
$$
\bigoplus\nolimits_\infty \{A_+^*:s\in S \}
\isom{\mathbf{mod}_1-A}
\mathcal{B}(A_+,\ell_\infty(S))
$$
in $\mathbf{mod}-A$, where $S=\bigsqcup_{\lambda\in\Lambda}S_\lambda$. Clearly,
the latter module is $1$-topologically injective, so by
proposition~\ref{RetrMetCTopInjIsMetCTopInj} the $A$-module $J$ is $C$-topologically
injective.
\end{proof}

\begin{corollary}\label{MetTopInjlInftySum} Let $J$ be an $A$-module and
$\Lambda$ be an arbitrary set. Then $\bigoplus_\infty \{J:\lambda\in\Lambda \}$
is $\langle$~metrically / $C$-topologically~$\rangle$ injective iff $J$ is
$\langle$~metrically / $C$-topologically~$\rangle$ injective.
\end{corollary}
\begin{proof} The result immediately follows from
proposition~\ref{MetTopInjModProd} if one set $J_\lambda=J$ 
for all $\lambda\in\Lambda$.
\end{proof}

\begin{proposition}\label{MapsFroml1toMetTopInj} Let $J$ be an $A$-module and
$\Lambda$ be an arbitrary set. Then $\mathcal{B}(\ell_1(\Lambda),J)$ is
$\langle$~metrically / $C$-topologically~$\rangle$ injective iff $J$ is
$\langle$~metrically / $C$-topologically~$\rangle$ injective.
\end{proposition}
\begin{proof} 
Assume $\mathcal{B}(\ell_1(\Lambda), J)$ is $\langle$~metrically /
$C$-topologically~$\rangle$ injective. Take any $\lambda\in\Lambda$ and consider
contractive $A$-morphisms
$i_\lambda:J\to\mathcal{B}(\ell_1(\Lambda),J):x\mapsto(f\mapsto f(\lambda)x)$
and $p_\lambda:\mathcal{B}(\ell_1(\Lambda),J)\to J:T\mapsto T(\delta_\lambda)$.
Clearly, $p_\lambda i_\lambda=1_J$, so by 
proposition~\ref{RetrMetCTopInjIsMetCTopInj} the $A$-module $J$ 
is $\langle$~metrically / $C$-topologically~$\rangle$ injective as $1$-retract 
of $\langle$~metrically / $1$-topologically~$\rangle$ injective 
$A$-module $\mathcal{B}(\ell_1(\Lambda),J)$.

Conversely, since $J$ is $\langle$~metrically / $C$-topologically~$\rangle$
injective, by proposition~\ref{MetCTopInjModViaCanonicMorph} 
the $A$-morphism $\rho_J^+$ is a
$\langle$~$1$-coretraction / $C$-coretraction~$\rangle$. Apply the functor
$\mathcal{B}(\ell_1(\Lambda),-)$ to this coretraction to get another
$\langle$~$1$-coretraction / $C$-coretraction~$\rangle$ denoted by
$\mathcal{B}(\ell_1(\Lambda),\rho_J^+)$. Note that 
$$
\mathcal{B}(\ell_1(\Lambda),\ell_\infty(B_{J^*}))
\isom{\mathbf{Ban}_1}
{(\ell_1(\Lambda)\projtens \ell_1(B_{J^*}))}^*
\isom{\mathbf{Ban}_1}
{\ell_1(\Lambda\times B_{J^*})}^*
\isom{\mathbf{Ban}_1}
\ell_\infty(\Lambda\times B_{J^*}),
$$ 
so we have isometric isomorphisms of Banach modules
$$
\mathcal{B}(\ell_1(\Lambda),\mathcal{B}(A_+,\ell_\infty(B_{J^*})))
\isom{\mathbf{mod}_1-A}
\mathcal{B}(A_+,\mathcal{B}(\ell_1(\Lambda),\ell_\infty(B_{J^*}))
\isom{\mathbf{mod}_1-A}
\mathcal{B}(A_+,\ell_\infty(\Lambda\times B_{J^*})).
$$ 
Therefore $\mathcal{B}(\ell_1(\Lambda),J)$ is a $\langle$~$1$-retract /
$C$-retract~$\rangle$ of $\langle$~metrically / $1$-topologically~$\rangle$
injective $A$-module $\mathcal{B}(A_+,\ell_\infty(\Lambda\times B_{J^*}))$. By
proposition~\ref{RetrMetCTopInjIsMetCTopInj} the $A$-module
$\mathcal{B}(\ell_1(\Lambda), J)$ is $\langle$~metrically /
$C$-topologically~$\rangle$ injective.
\end{proof}

The property of being metrically, topologically or relatively injective module 
puts some restrictions on the Banach geometric structure of the module.

\begin{proposition}[\cite{RamsHomPropSemgroupAlg}, 
    proposition 2.2.1]\label{MetTopRelInjModCompIdealAnnihCompl} Let $J$ 
be a $\langle$~metrically / $C$-topologically / $C$-relatively~$\rangle$ 
injective $A$-moudle, and let $I$ 
be a $\langle$~$1$-complemented / $c$-complemented / $c$-complemented~$\rangle$
right ideal of $A$. Then $J^{\perp I}$ 
is $\langle$~$2$-complemented / $1+cC$-complemented / 
$1+cC$-complemented~$\rangle$ in $J$.
\end{proposition}
\begin{proof} Since $A$ is $1$-complemented in $A_+$, then $I$ 
is $\langle$~$1$-complemented / $c$-complemented / $c$-complemented~$\rangle$
in $A_+$. Hence there exists a bounded linear operator $r:A_+\to I$ of 
$\langle$~norm $1$ / norm $c$ / norm $c$~$\rangle$ such that $r|_I=1_I$. 
Since $J$ is $\langle$~metrically / $C$-topologically / $C$-relatively~$\rangle$ 
injective, by proposition~\ref{MetCTopInjModViaCanonicMorph} the 
morphism $\rho_J^+$ has a left inverse $A$-morphism $\tau^+$ of norm 
$\langle$~at most $1$ / at most $C$ / at most $C$~$\rangle$. Now consider 
a bounded linear operator $p:A\to A:x\mapsto x-\tau^+(\rho_J^+(x)r)$. Clearly,
$\Vert p\Vert\leq 1+\Vert \tau^+\Vert\Vert r\Vert$. 
Since $\operatorname{Im}(r)\subset I$, then for all $x\in J^{\perp I}$ we 
have $\rho_J^+(x)r=0$. Hence $p(x)=x$ for all $x\in J^{\perp I}$. Since $I$ is 
a right ideal, then one can check that for all $x\in J$ and $a\in I$ holds 
$(\rho_J^+(x)r)\cdot a=\rho_J^+(x\cdot a)$. As a consequence 
$$
p(x)\cdot a
=x\cdot a-\tau^+(\rho_J^+(x)r)\cdot a
=x\cdot a-\tau^+((\rho_J^+(x)r)\cdot a)
=x\cdot a-\tau^+(\rho_J^+(x\cdot a))
=0.
$$
In other words $\operatorname{Im}(p)\subset J^{\perp I}$. Thus $p$ is 
projection onto $J^{\perp I}$ with norm at 
most $1+\Vert \tau^+\Vert\Vert r\Vert$. Hence $J^{\perp I}$ 
is $\langle$~$2$-complemented / $1+cC$-complemented / 
$1+cC$-complemented~$\rangle$ complemented in $J$.
\end{proof}

\begin{corollary}\label{MetTopRelInjModAnnihCompl} Let $J$ 
be a $\langle$~metrically / $C$-topologically / $C$-relatively~$\rangle$ 
injective $A$-moudle. Then $J_{ann}$ 
is $\langle$~$2$-complemented / $1+C$-complemented / 
$1+C$-complemented~$\rangle$ in $J$.
\end{corollary}
\begin{proof} The result directly follows from 
proposition~\ref{MetTopRelInjModCompIdealAnnihCompl}.
\end{proof}

%-------------------------------------------------------------------------------
%    Metric and topological flatness
%-------------------------------------------------------------------------------

\subsection{
    Metric and topological flatness
}\label{SubSectionMetricAndTopologicalFlatness}

The definition of metrically flat modules was given by Helemeskii 
in~\cite{HelMetrFlatNorMod} under the name of extremely flat modules. The 
notion of topological flatness was first studied by 
White in~\cite{WhiteInjmoduAlg}. We use a somewhat different, but equivalent  
terminology. Unfortunately, his definition was erroneous, meanwhile the results 
were correct. So we take responsibility to fix the definition.

\begin{definition}[\cite{HelMetrFlatNorMod}, 
    definition I;~\cite{WhiteInjmoduAlg}, definition 4.4]\label{MetCTopFlatMod} 
A left $A$-module $F$ is 
called $\langle$~metrically / $C$-topologically~$\rangle$ 
flat if for each $\langle$~isometric / $c$-topologically injective~$\rangle$ 
$A$-morphism $\xi:X\to Y$ of right $A$-modules the 
operator $\xi\projmodtens{A} 1_F:X\projmodtens{A} F\to Y\projmodtens{A} F$ 
is $\langle$~isometric / $c C$-topologically injecive~$\rangle$. We say that an 
$A$-module $F$ is topologically flat if it is $C$-topologically flat for 
some $C\geq 1$.
\end{definition}

A short but more involved definition is the following: an $A$-module $F$ is
called $\langle$~metrically / $C$-topologically~$\rangle$ flat, if the functor
$\langle$~$-\projmodtens{A}F:A-\mathbf{mod}_1\to\mathbf{Ban}_1$ /
$-\projmodtens{A}F:A-\mathbf{mod}\to\mathbf{Ban}$~$\rangle$ maps
$\langle$~isometric / $c$-topologically injective~$\rangle$ $A$-morphisms into
$\langle$~isometric / $c C$-topologically injective~$\rangle$ operators.

The key result in the study of flatness is the following.

\begin{proposition}[\cite{WhiteInjmoduAlg}, lemma 4.10]\label{MetCTopFlatCharac}
An $A$-module $F$ is
$\langle$~metrically / $C$-topologically~$\rangle$ flat iff $F^*$ is
$\langle$~metrically / $C$-topologically~$\rangle$ injective.
\end{proposition}
\begin{proof} Consider any $\langle$~isometric / $c$-topologically
injective~$\rangle$ morphism of right $A$-modules, call it $\xi:X\to Y$. The
operator $\xi\projmodtens{A} 1_F$ is $\langle$~isometric / $c C$-topologically
injective~$\rangle$ iff the adjoint operator ${(\xi\projmodtens{A} 1_F)}^*$ is
strictly $\langle$~coisometric / $c C$-topologically surjective~$\rangle$. 
Since operators ${(\xi\projmodtens{A} 1_F)}^*$ and $\mathcal{B}_A(\xi,F^*)$ 
are equivalent in $\mathbf{Ban}_1$ via universal property of projective module 
tensor product, then we get that $\xi\projmodtens{A} 1_F$ 
is $\langle$~isometric / $c C$-topologically injective~$\rangle$ 
iff $\mathcal{B}_A(\xi,F^*)$ is 
strictly $\langle$~coisometric / $c C$-topologically surjective~$\rangle$. 
Since $\xi$ is arbitrary we conclude that $F$ 
is $\langle$~metrically / $C$-topologically~$\rangle$ flat iff
$F^*$ is $\langle$~metrically / $C$-topologically~$\rangle$ injective.
\end{proof}

This characterization allows one to prove many properties of flat modules by 
considering respective duals.

\begin{proposition}\label{RetrMetCTopFlatIsMetCTopFlat} 
Any $\langle$~$1$-retract / $c$-retract~$\rangle$ of a
$\langle$~metrically / $C$-topologically~$\rangle$ flat module is 
$\langle$~metrically / $cC$-topologically~$\rangle$ flat.
\end{proposition}
\begin{proof} The result follows from propositions~\ref{MetCTopFlatCharac} 
and~\ref{RetrMetCTopInjIsMetCTopInj}
\end{proof}

\begin{proposition}\label{MetFlatIsTopFlatAndTopFlatIsRelFlat} Every metrically
flat module is $1$-topologically flat and every $C$-topologically flat module is
$C$-relatively flat.
\end{proposition}
\begin{proof} The result follows from propositions~\ref{MetCTopFlatCharac} 
and~\ref{MetInjIsTopInjAndTopInjIsRelInj}.
\end{proof}

\begin{proposition}\label{MetFlatIsOneTopFlat} An $A$-module is metrically flat
iff it is $1$-topologically flat.
\end{proposition}
\begin{proof} The result follows from propositions~\ref{MetCTopFlatCharac}
and~\ref{MetInjIsOneTopInj}.
\end{proof}

Loosely speaking flatness is ``projectivity with respect to second duals''. 
We can give this statement a precise meaning.

\begin{proposition}\label{MetTopFlatSecondDualCharac} Let $F$ be a left Banach
$A$-module. Then the following are equivalent:

\begin{enumerate}[label = (\roman*)]
    \item $F$ is $\langle$~metrically / $C$-topologically~$\rangle$ flat;

    \item for any strictly $\langle$~coisometric / $c$-topologically
    surjective~$\rangle$ $A$-morphism $\xi:X\to Y$ and for any $A$-morphism
    $\phi:F\to Y$ there exists an $A$-morphism $\psi:P\to X^{**}$ such that
    $\xi^{**}\psi=\iota_Y\phi$ and $\langle$~$\Vert\psi\Vert=\Vert\phi\Vert$ /
    $\Vert\psi\Vert\leq cC\Vert\phi\Vert$~$\rangle$;

    \item there exists an $A$-morphism $\sigma:F\to {(A_+\projtens
    \ell_1(B_F))}^{**}$ with norm $\langle$~at most $1$ / at most $C$~$\rangle$
    such that ${(\pi_F^+)}^{**}\sigma=\iota_F$.
\end{enumerate}
\end{proposition}
\begin{proof} $(i)\implies (ii)$ Again, consider arbitrary strictly 
$\langle$~coisometric / $c$-topologically surjective~$\rangle$ $A$-morphism
$\xi:X\to Y$ and arbitrary $A$-morphism $\phi:F\to Y$. By
[\cite{HelLectAndExOnFuncAn}, exercise 4.4.6] we know that $\xi^*$ is
$\langle$~isometric / $c$-topologically injective~$\rangle$. Since $F$ is
$\langle$~metrically / $C$-topologically~$\rangle$ flat then
$(\xi^*\projmodtens{A}1_F)$ is $\langle$~isometric / $cC$-topologically
injective~$\rangle$ too. Therefore ${(\xi^*\projmodtens{A}1_F)}^*$ is
$\langle$~strictly coisometric / strictly $cC$-topologically
surjective~$\rangle$ by [\cite{HelLectAndExOnFuncAn}, exercise 4.4.7]. Note that
${(\xi^*\projmodtens{A}1_F)}^*$ and $\mathcal{B}_A(F,\xi^{**})$ are equivalent
in $\mathbf{Ban}_1$ thanks to the law of adjoint associativity. So
$\mathcal{B}_A(F,\xi^{**})$ is strictly $\langle$~coisometric / 
$cC$-topologically surjective~$\rangle$ too. The latter implies that for the
operator $\iota_Y\phi$ we can find an $A$-morphism $\psi:Y\to X^{**}$ such that
$\xi^{**}\psi=\iota_Y\phi$ and
$\langle$~$\Vert\psi\Vert=\Vert\iota_Y\phi\Vert=\Vert\phi\Vert$ /
$\Vert\psi\Vert\leq cC\Vert\iota_Y\phi\Vert=cC\Vert\phi\Vert$~$\rangle$

$(ii)\implies (iii)$ Set $\xi=\pi_F^+$ and $\phi=1_F$. Since $\xi$ is
strictly $\langle$~coisometric/ $1$-topologically surjective~$\rangle$,
then from assumption we get an $A$-morphism $\sigma:F\to
{(A_+\projtens\ell_1(B_F))}^{**}$ such that 
${(\pi_F^+)}^{**}\sigma
=\iota_F 1_F
=\iota_F$ 
and $\langle$~$\Vert\sigma\Vert\leq \Vert\phi\Vert=1$ /
$\Vert\sigma\Vert\leq 1\cdot C\Vert\phi\Vert=C$~$\rangle$.

$(iii)\implies (i)$ Let $\sigma$ be a right inverse $A$-morphism for
${(\pi_F^+)}^{**}$ with norm $\langle$~at most $1$ / at most $C$~$\rangle$.
Consider $A$-morphism $\tau=\sigma^*\iota_{{(A_+\projtens\ell_1(B_F))}^*}$.
Clearly, its norm is $\langle$~at most $1$ / at most $C$~$\rangle$. For any
$f\in F^*$ and $x\in F$ we have
$$
({\tau(\pi_F^+)}^*)(f)(x)
=\sigma^*(\iota_{{(A_+\projtens\ell_1(B_F))}^*}({(\pi_F^+)}^*(f)))(x)
=\iota_{{(A_+\projtens\ell_1(B_F))}^*}({(\pi_F^+)}^*(f))(\sigma(x))
$$
$$
=\sigma(x)({(\pi_F^+)}^*(f))
={(\pi_F^+)}^{**}(\sigma(x))(f)
=\iota_F(x)(f)
=f(x)
$$
So ${\tau(\pi_F^+)}^*=1_{F^*}$, which means $F^*$ is a with norm
$\langle$~$1$-retract / $C$-retract~$\rangle$ of
${(A_+\projtens\ell_1(B_F))}^*$. The latter module is $\langle$~metrically /
$1$-topologically~$\rangle$ injective, because 
$$
{(A_+\projtens\ell_1(B_F))}^*
\isom{\mathbf{mod}_1-A}
\mathcal{B}(A_+,\ell_\infty(B_F)).
$$ 
By proposition~\ref{RetrMetCTopInjIsMetCTopInj} the $A$-module $F^*$ is
$\langle$~metrically / $C$-topologically~$\rangle$ injective. By 
proposition~\ref{MetCTopFlatCharac} this is
equivalent to $\langle$~metric / $C$-topological~$\rangle$ flatness of $F$.
\end{proof}

Let us proceed to examples. Consider the category of Banach spaces as 
the category of left Banach modules over zero algebra, then we get the 
definition of $\langle$~metrically / topologically~$\rangle$ flat Banach space. 
From Grothendieck's paper~\cite{GrothMetrProjFlatBanSp} it follows that any 
metrically flat Banach space is isometrically isomorphic to $L_1(\Omega,\mu)$ 
for some measure space $(\Omega,\Sigma,\mu)$. For topologically flat Banach 
spaces, in contrast with topologically injective ones, we also have a criterion
[\cite{DefFloTensNorOpId}, corollary 23.5(1)]: a Banach space is topologically
flat iff it is an $\mathscr{L}_1^g$-space.

\begin{proposition}\label{MetTopFlatModCoProd} Let
${(F_\lambda)}_{\lambda\in\Lambda}$ be family of $A$-modules. Then 

\begin{enumerate}[label = (\roman*)]
    \item $\bigoplus_1 \{F_\lambda:\lambda\in\Lambda \}$ is metrically flat iff
    for all $\lambda\in\Lambda$ the $A$-module $F_\lambda$ is metrically flat;

    \item $\bigoplus_1 \{F_\lambda:\lambda\in\Lambda \}$ is $C$-topologically
    flat iff for all $\lambda\in\Lambda$ the $A$-module $F_\lambda$ is
    $C$-topologically flat.
\end{enumerate}
\end{proposition}
\begin{proof} By proposition~\ref{MetCTopFlatCharac} an $A$-module $F$ 
is $\langle$~metrically / $C$-topologically~$\rangle$ flat iff $F^*$ 
is $\langle$~metrically / $C$-topologically~$\rangle$ injective. 
It remains to apply
proposition~\ref{MetTopInjModProd} with $J_\lambda=F_\lambda^*$ for all
$\lambda\in\Lambda$ and recall that 
${\left(\bigoplus_1 \{
    F_\lambda:\lambda\in\Lambda \}
\right)}^*
\isom{\mathbf{mod}_1-A}
\bigoplus_\infty \{ F_\lambda^*:\lambda\in\Lambda \}$.
\end{proof}

The following propositions demonstrate a close relationsip between flatness and
projectivity.

\begin{proposition}\label{DualMetTopProjIsMetrInj} Let $P$ be a
$\langle$~metrically / $C$-topologically~$\rangle$ projective $A$-module, and
$\Lambda$ be an arbitrary set. Then $\mathcal{B}(P,\ell_\infty(\Lambda))$ is
$\langle$~metrically / $C$-topologically~$\rangle$ injective $A$-module. In
particular, $P^*$ is $\langle$~metrically / $C$-topologically~$\rangle$
injective $A$-module.
\end{proposition}
\begin{proof} From proposition~\ref{MetCTopProjModViaCanonicMorph} we know 
that $\pi_P^+$ is a $\langle$~$1$-retraction / $C$-retraction~$\rangle$. 
Then the $A$-morphism
$\rho^+=\mathcal{B}(\pi_P^+,\ell_\infty(\Lambda))$ is a
$\langle$~$1$-coretraction / $C$-coretraction~$\rangle$. Note that, 
$$
\mathcal{B}(A_+\projtens\ell_1(B_P),\ell_\infty(\Lambda))
\isom{\mathbf{mod}_1-A}
\mathcal{B}(A_+,\mathcal{B}(\ell_1(B_P),\ell_\infty(\Lambda)))
\isom{\mathbf{mod}_1-A}
\mathcal{B}(A_+,\ell_\infty(B_P\times\Lambda)).
$$ 
Thus we showed that $\rho^+$ is a $\langle$~$1$-coretraction /
$C$-coretraction~$\rangle$ from $\mathcal{B}(P,\ell_\infty(\Lambda))$ into
$\langle$~metrically / $1$-topologically~$\rangle$ injective $A$-module. By
proposition~\ref{RetrMetCTopInjIsMetCTopInj} the $A$-module
$\mathcal{B}(P,\ell_\infty(\Lambda))$ is $\langle$~metrically /
$C$-topologically~$\rangle$ injective. To prove the last claim, just set
$\Lambda=\mathbb{N}_1$.
\end{proof}

\begin{proposition}\label{MetTopProjIsMetTopFlat} Every $\langle$~metrically /
$C$-topologically~$\rangle$ projective module is $\langle$~metrically /
$C$-topologically~$\rangle$ flat.
\end{proposition}
\begin{proof} The result follows from propositions~\ref{MetCTopFlatCharac}
and~\ref{DualMetTopProjIsMetrInj}.
\end{proof}

The property of being metrically, topologically or relatively flat module 
puts some restrictions on the Banach geometric structure of the module.

\begin{proposition}[\cite{RamsHomPropSemgroupAlg}, 
    corollary 2.2.2]\label{MetTopRelFlatModCompIdealPartCompl} Let $F$ 
be a $\langle$~metrically / $C$-topologically / $C$-relatively~$\rangle$ 
flat $A$-moudle, and let $I$ 
be a $\langle$~$1$-complemented / $c$-complemented / $c$-complemented~$\rangle$
right ideal of $A$. Then $\operatorname{cl}_F(I F)$ 
is weakly $\langle$~$2$-complemented / $1+cC$-complemented / 
$1+cC$-complemented~$\rangle$  in $F$.
\end{proposition}
\begin{proof} From 
propositions~\ref{MetCTopFlatCharac},~\ref{MetTopRelInjModCompIdealAnnihCompl} 
it follows that ${(F^*)}^{\perp I}$ is complemented in $F^*$. It remains to 
recall that ${(F^*)}^{\perp I}=\operatorname{cl}_F(I F)$. 
\end{proof}

\begin{corollary}\label{MetTopRelFlatModPartCompl} Let $F$ 
be a $\langle$~metrically / $C$-topologically / $C$-relatively~$\rangle$ 
flat $A$-moudle. Then $F_{ess}$ 
is weakly $\langle$~$2$-complemented / $1+C$-complemented / 
$1+C$-complemented~$\rangle$ in $F$.
\end{corollary}
\begin{proof} The result directly follows from 
proposition~\ref{MetTopRelFlatModCompIdealPartCompl}.
\end{proof}

%-------------------------------------------------------------------------------
%    Metric and topological flatness of ideals and cyclic modules
%-------------------------------------------------------------------------------

\subsection{
    Metric and topological flatness of ideals and cyclic modules
}\label{SubSectionMetricAndTopologicalFlatnessOfIdealsAndCyclicModules}

In this section we study conditions under which ideals and cyclic modules are
metrically or topologically flat. The proofs are somewhat similar to those used
in the study of relative flatness of ideals and cyclic modules.

\begin{proposition}\label{MetTopFlatIdealsInUnitalAlg} Let $I$ be a left ideal
of $A_\times $ and $I$ has a right $\langle$~contractive / $c$-bounded~$\rangle$
approximate identity. Then $I$ is $\langle$~metrically /
$c$-topologically~$\rangle$ flat.
\end{proposition}
\begin{proof} Let $\mathfrak{F}$ be the section filter on $N$ and let
$\mathfrak{U}$ be an ultrafilter dominating $\mathfrak{F}$. 
For a fixed $f\in I^*$ 
and $a\in A_\times $ we have 
$|f(a e_\nu)|
\leq\Vert f\Vert\Vert a\Vert\Vert e_\nu\Vert
\leq c\Vert f\Vert\Vert a\Vert$ 
i.e. ${(f(ae_\nu))}_{\nu\in N}$ is a bounded net of complex numbers. 
Therefore we have a well defined limit $\lim_{\mathfrak{U}}f(ae_\nu)$ along 
ultrafilter $\mathfrak{U}$. Now it is routine to check that 
$\sigma:A_\times ^*\to I^*:f\mapsto (a\mapsto \lim_{\mathfrak{U}}f(ae_\nu))$ 
is an $A$-morphism with norm $\langle$~at most $1$ / at most $c$~$\rangle$. 
Let $\rho:I\to A_\times$ be the natural embedding, then for 
all $f\in A_\times^*$ and $a\in I$ holds
$$
\rho^*(\sigma(f))(a)
=\sigma(f)(\rho(a))
=\sigma(f)(a)
=\lim_{\mathfrak{U}}f(a e_\nu)
=\lim_{\nu}f(a e_\nu)
=f(\lim_{\nu}a e_\nu)
=f(a)
$$
i.e. $\sigma:I^*\to A_\times^*$ is a $\langle$~$1$-coretraction /
$c$-coretraction~$\rangle$. The right $A$-module $A_\times ^*$ is
$\langle$~metrically / $1$-topologically~$\rangle$ injective by
proposition~\ref{DualOfUnitalAlgIsMetTopInj}, hence its $\langle$~$1$-retract /
$c$-retract~$\rangle$ $I^*$ is $\langle$~metrically /
$c$-topologically~$\rangle$ injective. Now from 
proposition~\ref{MetCTopFlatCharac} we conclude that the $A$-module $I$ 
is $\langle$~metrically / $c$-topologically~$\rangle$ flat.
\end{proof}

Note that the same sufficient condition holds for relative flatness for ideals
of $A_\times$ [\cite{HelBanLocConvAlg}, proposition 7.1.45]. Now we are able to
give an example of a metrically flat module which is not even topologically
projective. Clearly $\ell_\infty(\mathbb{N})$-module $c_0(\mathbb{N})$ is not
unital as ideal but admits a contractive approximate identity. By
theorem~\ref{GoodCommIdealMetTopProjIsUnital} it is not topologically
projective, but it is metrically flat by
proposition~\ref{MetTopFlatIdealsInUnitalAlg}.

The ``metric'' part of the following proposition is a slight modification of
[\cite{WhiteInjmoduAlg}, proposition 4.11]. The case of topological flatness was
solved by Helemskii in [\cite{HelHomolBanTopAlg}, theorem VI.1.20].

\begin{proposition}\label{MetTopFlatCycModCharac} Let $I$ be a left proper ideal
of $A_\times $. Then the following are equivalent:

\begin{enumerate}[label = (\roman*)]
    \item $A_\times /I$ is $\langle$~metrically / $C$-topologically~$\rangle$
    flat $A$-module;

    \item $I$ has a right bounded approximate identity ${(e_\nu)}_{\nu\in N}$
    with $\sup_{\nu\in N}\Vert e_{A_\times }-e_\nu\Vert$ $\langle$~at most $1$ /
    at most $C$~$\rangle$
\end{enumerate}
\end{proposition}
\begin{proof} $(i)\implies (ii)$  Since $A_\times /I$ is $\langle$~metrically
/ $C$-topologically~$\rangle$ flat, then by proposition~\ref{MetCTopFlatCharac} 
the right $A$-module ${(A_\times /I)}^*$ is $\langle$~metrically /
$C$-topologically~$\rangle$ injective. Let $\pi:A_\times \to A_\times /I$ be the
natural quotient map, then $\pi^*:{(A_\times /I)}^*\to A_\times ^*$ is an
isometry. Since ${(A_\times /I)}^*$ is $\langle$~metrically /
$C$-topologically~$\rangle$ injective, then $\pi^*$ is a coretraction, i.e.\
there exists a $\langle$~strictly coisometric / topologically
surjective~$\rangle$ $A$-morphism $\tau:A_\times ^*\to {(A_\times /I)}^*$ of
norm $\langle$~at most $1$ / at most $C$~$\rangle$ such that
$\tau\pi^*=1_{{(A_\times /I)}^*}$. Consider $p\in A^{**}$ such that
$\iota_{A_\times }(e_{A_\times })-p
=\tau^*(\pi^{**}(\iota_{A_\times }(e_{A_\times })))$. 
Fix $f\in I^\perp$. Since $I^\perp=\pi^*({(A_\times /I)}^*)$, 
then there exists $g\in {(A_\times /I)}^*$ such that $f=\pi^*(g)$. Thus,
$$
(\iota_{A_\times }(e_{A_\times })-p)(f)
=\tau^*(\pi^{**}(\iota_{A_\times }(e_{A_\times })))(\pi^*(g))
=\pi^{**}(\iota_{A_\times }(e_{A_\times }))(\tau(\pi^*(g)))
$$
$$
=\pi^{**}(\iota_{A_\times }(e_{A_\times }))(g)
=\iota_{A_\times }(e_{A_\times })(\pi^*(g))
=\iota_{A_\times }(e_{A_\times })(f).
$$
Therefore $p(f)=0$ for all $f\in I^\perp$, i.e. $p\in I^{\perp\perp}$. Recall
that $I^{\perp\perp}$ is the weak${}^*$ closure of $I$ in $A^{**}$, so we can
choose a net ${(e_\nu'')}_{\nu\in N''}\subset I$ such that
${(\iota_I(e_\nu''))}_{\nu\in N''}$ converges to $p$ in weak${}^*$ topology.
Clearly ${(\iota_{A_\times }(e_{A_\times }-e_\nu''))}_{\nu\in N''}$ converges to
$\iota_{A_\times }(e_{A_\times })-p$ in the same topology. By
[\cite{PosAndApproxIdinBanAlg}, lemma 1.1] there exists a net in the convex hull
$\operatorname{conv}{
    (\iota_{A_\times }(e_{A_\times }-e_\nu''))
}_{\nu\in N''}
=\iota_{A_\times }(e_{A_\times })
-\operatorname{conv}{(\iota_{A_\times}(e_\nu''))}_{\nu\in N''}$ 
that weak${}^*$ converges to $\iota_{A_\times }(e_{A_\times })-p$ with 
norm bound $\Vert \iota_{A_\times }(e_{A_\times})-p\Vert$. 
Denote this net as 
${(\iota_{A_\times }(e_{A_\times })-\iota_{A_\times }(e_\nu'))}_{\nu\in N'}$, 
then ${(\iota_{A_\times }(e_\nu'))}_{\nu\in N'}$ weak${}^*$ converges to $p$. 
For any $a\in I$ and $f\in I^*$ we have
$$
\lim_{\nu}f(ae_\nu')
=\lim_{\nu}\iota_{A_\times }(e_\nu')(f\cdot a)
=p(f\cdot a)
=\iota_{A_\times }(e_{A_\times })(f\cdot a)
-\tau^*(\pi^{**}(\iota_{A_\times }(e_{A_\times })))(f\cdot a)
$$
$$
=f(a)-\iota_{A_\times }(e_{A_\times })(\pi^*(\tau(f\cdot a)))
=f(a)-\pi^*(\tau(f)\cdot a)(e_{A_\times })
=f(a)-\tau(f)(\pi(a))
=f(a)
$$
hence ${(e_\nu')}_{\nu\in N'}$ is a weak right bounded approximate identity for
$I$. By [\cite{AppIdAndFactorInBanAlg}, proposition 33.2] there is a net
${(e_\nu)}_{\nu\in N}\subset\operatorname{conv}{(e_\nu')}_{\nu\in N'}$ which is
a right bounded approximate identity for $I$. For any $\nu\in N$ we have
$e_{A_\times }-e_\nu\in\operatorname{conv}{(e_{A_\times }-e_\nu')}_{\nu\in N'}$,
so taking into account the norm bound 
on ${(\iota_{A_\times }(e_{A_\times}-e_\nu'))}_{\nu\in N'}$ we get 
$$
\sup_{\nu\in N}\Vert e_{A_\times }-e_\nu\Vert
\leq\Vert \iota_{A_\times }(e_{A_\times })-p\Vert
\leq\Vert\tau^*(\pi^{**}(\iota_{A_\times }(e_{A_\times })))\Vert
\leq\Vert\tau^*\Vert\Vert\pi^{**}\Vert\Vert\iota_{A_\times }(e_{A_\times })\Vert
=\Vert\tau\Vert
$$
Since $\tau$ has norm $\langle$~at most $1$ / at most $C$~$\rangle$ we get the
desired bound. By construction, ${(e_\nu)}_{\nu\in N}$ is a right bounded
approximate identity for $I$.

$(ii)\implies (i)$ Denote $D=\sup_{\nu\in N}\Vert e_{A_\times }-e_\nu\Vert$.
Let $\mathfrak{F}$ be the section filter on $N$ and let $\mathfrak{U}$ be an
ultrafilter dominating $\mathfrak{F}$. For a fixed $f\in A_\times ^*$ 
and $a\in A_\times $ we have 
$|f(a-a e_\nu)|=|f(a(e_{A_\times }-e_\nu))|
\leq\Vert f\Vert\Vert a\Vert\Vert e_{A_\times }-e_\nu\Vert
\leq D\Vert f\Vert\Vert a\Vert$
i.e. ${(f(a-ae_\nu))}_{\nu\in N}$ is a bounded net of complex numbers. Therefore
we have a well defined limit $\lim_{\mathfrak{U}}f(a-ae_\nu)$ along ultrafilter
$\mathfrak{U}$. Since ${(e_\nu)}_{\nu\in N}$ is a right approximate identity for
$I$ and $\mathfrak{U}$ contains section filter then for all $a\in I$ we have
$\lim_{\mathfrak{U}}f(a-ae_\nu)=\lim_{\nu}f(a-ae_\nu)=0$. Therefore for each
$f\in A_\times ^*$ we have well a defined map $\tau(f):A_\times /I\to
\mathbb{C}:a+I\mapsto \lim_{\mathfrak{U}} f(a-ae_\nu)$. Clearly, this is a
linear functional and from inequalities above we see its norm bounded 
by $D\Vert f\Vert$. Now it is routine to check 
that $\tau:A_\times ^*\to {(A_\times /I)}^*:f\mapsto \tau(f)$ is 
an $A$-morphism with norm $\langle$~at most $1$ / at most $C$~$\rangle$. 
For all $g\in{(A_\times /I)}^*$ and $a+I\in A_\times /I$ holds
$$
\tau(\pi^*(g))(a+I)
=\lim_{\mathfrak{U}}\pi^*(g)(a-ae_\nu)
=\lim_{\mathfrak{U}} g(\pi(a-ae_\nu))
=\lim_{\mathfrak{U}} g(a+I)
=g(a+I)
$$
i.e. $\tau:A_\times ^*\to {(A_\times /I)}^*$ is a retraction. The right
$A$-module $A_\times ^*$ is $\langle$~metrically / $1$-topologically~$\rangle$
injective by proposition~\ref{DualOfUnitalAlgIsMetTopInj}, hence its
$\langle$~$1$-retract / $C$-retract~$\rangle$ ${(A_\times /I)}^*$ is
$\langle$~metrically / $C$-topologically~$\rangle$ injective. 
Proposition~\ref{MetCTopFlatCharac} gives that the
module $A_\times /I$ is $\langle$~metrically / $C$-topologically~$\rangle$ flat.
\end{proof}

It is worth to mention that every operator algebra $A$ (not necessary self
adjoint) with contractive approximate identity has a contractive approximate
identity ${(e_\nu)}_{\nu\in N}$ such 
that $\sup_{\nu\in N}\Vert e_{A_\#}-e_\nu\Vert\leq 1$ and 
even $\sup_{\nu\in N}\Vert e_{A_\#}-2e_\nu\Vert\leq 1$. Here $A_\#$ is 
a unitization of $A$ as operator algebra. For details
see~\cite{PosAndApproxIdinBanAlg},~\cite{BleContrAppIdInOpAlg}.

Again we shall compare our result on metric and topological flatness of cyclic
modules with their relative counterpart. Helemeskii and Sheinberg showed
[\cite{HelHomolBanTopAlg}, theorem VII.1.20] that a cyclic module is relatively
flat if $I$ admits a right bounded approximate identity. In case when $I^\perp$
is complemented in $A_\times^*$ the converse is also true. In topological theory
we don't need this assumption, so we have a criterion. Metric flatness of cyclic
modules is a much stronger property due to specific restriction on the norm of
approximate identity. As we shall see in the next section, it is so restrictive
that it doesn't allow to construct any non zero annihilator metrically
projective, injective or flat module over a non zero Banach algebra.

%-------------------------------------------------------------------------------
%    The impact of Banach geometry
%-------------------------------------------------------------------------------

\section{
    The impact of Banach geometry
}\label{SectionTheImpactOfBanachGeometry}

%-------------------------------------------------------------------------------
%    Homologically trivial annihilator modules
%-------------------------------------------------------------------------------

\subsection{
    Homologically trivial annihilator modules
}\label{SubSectionHomoligicallyTrivialAnnihilatorModules}

In this section we concentrate on the study of metrically and topologically
projective, injective and flat annihilator modules. Unless otherwise stated, all
Banach spaces in this section are regarded as annihilator modules. Note the
obvious fact that we shall often use in this section: any bounded linear
operator between annihilator $A$-modules is an $A$-morphism.

\begin{proposition}\label{AnnihCModIsRetAnnihMod} Let $X$ be a non zero
annihilator $A$-module. Then $\mathbb{C}$ is a $1$-retract of $X$ in
$A-\mathbf{mod}_1$.
\end{proposition}
\begin{proof} Take any $x_0\in X$ with $\Vert x_0\Vert=1$ and using Hahn-Banach
theorem choose $f_0\in X^*$ such that $\Vert f_0\Vert=f_0(x_0)=1$. Consider
contractive linear operators $\pi:X\to \mathbb{C}:x\mapsto f_0(x)$,
$\sigma:\mathbb{C}\to X:z\mapsto zx_0$. It is easy to check that $\pi$ and
$\sigma$ are contractive $A$-morphisms and what is more
$\pi\sigma=1_\mathbb{C}$. In other words $\mathbb{C}$ is a $1$-retract of $X$ in
$A-\mathbf{mod}_1$.
\end{proof}

Now it is time to recall that any Banach algebra $A$ can always be regarded as
proper maximal ideal of $A_+$, and what is more
$\mathbb{C}\isom{A-\mathbf{mod}_1} A_+/A$. If we regard $\mathbb{C}$ as a right
annihilator $A$-module we also have
$\mathbb{C}\isom{\mathbf{mod}_1-A}{(A_+/A)}^*$. 

\begin{proposition}\label{MetTopProjModCCharac} An annihilator $A$-module
$\mathbb{C}$ is $\langle$~metrically / $C$-topologically~$\rangle$ projective
iff $\langle$~$A= \{0 \}$ / $A$ has a right identity of norm at most
$C-1$~$\rangle$.
\end{proposition}
\begin{proof} 
It is enough to study $\langle$~metric / $C$-topological~$\rangle$ projectivity
of $A_+/A$. Since the natural quotient map $\pi:A_+\to A_+/A$ is a strict
coisometry, then by proposition~\ref{MetTopProjCycModCharac} $\langle$~metric /
$C$-topological~$\rangle$ projectivity of $A_+/A$ is equivalent to existence of
idempotent $p\in A$ such that $A=A_+p$ and $e_{A_+}-p$ has norm $\langle$~at
most $1$ / at most $C$~$\rangle$. It remains to note that this norm bound holds
iff $\langle$~$p=0$ and therefore $A=A_+p= \{0 \}$ / $p$ has norm at most
$C-1$~$\rangle$.
\end{proof}

\begin{proposition}\label{MetTopProjOfAnnihModCharac} Let $P$ be a non zero
annihilator $A$-module. Then the following are equivalent:

\begin{enumerate}[label = (\roman*)]
    \item $P$ is $\langle$~metrically / $C$-topologically~$\rangle$ projective
    $A$-module;

    \item $\langle$~$A= \{0 \}$ / $A$ has a right identity of norm at most
    $C-1$~$\rangle$ and $P$ is a $\langle$~metrically /
    $C$-topologically~$\rangle$ projective Banach space. As the consequence
    $\langle$~$P\isom{\mathbf{Ban}_1}\ell_1(\Lambda)$ /
    $P\isom{\mathbf{Ban}}\ell_1(\Lambda)$~$\rangle$ for some set $\Lambda$.
\end{enumerate}
\end{proposition}
\begin{proof} $(i)\implies (ii)$ By 
propositions~\ref{RetrMetCTopProjIsMetCTopProj} 
and~\ref{AnnihCModIsRetAnnihMod} the
$A$-module $\mathbb{C}$ is $\langle$~metrically / $C$-topologically~$\rangle$
projective as $1$-retract of $\langle$~metrically / $C$-topologically~$\rangle$
projective module $P$. Proposition~\ref{MetTopProjModCCharac} gives that
$\langle$~$A= \{0 \}$ / $A$ has right identity of norm at most $C-1$~$\rangle$.
By corollary~\ref{MetTopProjTensProdWithl1} the annihilator $A$-module
$\mathbb{C}\projtens\ell_1(B_P)\isom{A-\mathbf{mod}_1}\ell_1(B_P)$ is
$\langle$~metrically / $C$-topologically~$\rangle$ projective. Consider strict
coisometry $\pi:\ell_1(B_P)\to P$ well defined by equality $\pi(\delta_x)=x$.
Since $P$ and $\ell_1(B_P)$ are annihilator modules, then $\pi$ is also an
$A$-module map. Since $P$ is $\langle$~metrically / $C$-topologically~$\rangle$
projective, then the $A$-morphism $\pi$ has a right inverse morphism $\sigma$ of
norm  $\langle$~at most $1$ / at most $C-1$~$\rangle$. Therefore $P$ is a
$\langle$~$1$-retract / $C$-retract~$\rangle$ of $\langle$~metrically /
$1$-topologically~$\rangle$ projective Banach space $\ell_1(B_P)$. By
proposition~\ref{RetrMetCTopProjIsMetCTopProj} the Banach space $P$ is
$\langle$~metrically / $C$-topologically~$\rangle$ projective. Now from
$\langle$~[\cite{HelMetrFrQMod}, proposition 3.2] / results
of~\cite{KotheTopProjBanSp}~$\rangle$ we get that $P$ is isomorphic to
$\ell_1(\Lambda)$ in $\langle$~$\mathbf{Ban}_1$ / $\mathbf{Ban}$~$\rangle$ for
some set $\Lambda$. 

$(ii)\implies (i)$ By proposition~\ref{MetTopProjModCCharac} the annihilator
$A$-module $\mathbb{C}$ is $\langle$~metrically / $C$-topologically~$\rangle$
projective. Therefore by corollary~\ref{MetTopProjTensProdWithl1} the
annihilator $A$-module
$\mathbb{C}\projtens\ell_1(\Lambda)\isom{A-\mathbf{mod}_1}\ell_1(\Lambda)$ is
$\langle$~metrically / $C$-topologically~$\rangle$ projective too.
\end{proof}

\begin{proposition}\label{MetTopInjModCCharac} A right annihilator $A$-module
$\mathbb{C}$ is $\langle$~metrically / $C$-topologically~$\rangle$ injective iff
$\langle$~$A= \{0 \}$ / $A$ has right $(C-1)$-bounded approximate
identity~$\rangle$.
\end{proposition}
\begin{proof} Because of proposition~\ref{MetCTopFlatCharac} it is enough 
to study $\langle$~metric / $C$-topological~$\rangle$ flatness of $A_+/A$. By
proposition~\ref{MetTopFlatCycModCharac} this is equivalent to existence of
right bounded approximate identity ${(e_\nu)}_{\nu\in N}$ in $A$ with
$\langle$~$\sup_{\nu\in N}\Vert e_{A_+}-e_\nu\Vert\leq 1$ / 
$\sup_{\nu\in N}\Vert e_{A_+}-e_\nu\Vert\leq C$~$\rangle$. 
It remains to note that the latter inequality holds 
iff $\langle$~iff $e_\nu=0$ and therefore $A= \{0 \}$. /
${(e_\nu)}_{n\in N}$ is a right $C-1$-bounded approximate identity~$\rangle$.
\end{proof}

\begin{proposition}\label{MetTopInjOfAnnihModCharac} Let $J$ be a non zero right
annihilator $A$-module. Then the following are equivalent:

\begin{enumerate}[label = (\roman*)]
    \item $J$ is $\langle$~metrically / $C$-topologically~$\rangle$ injective
    $A$-module;

    \item $\langle$~$A= \{0 \}$ / $A$ has a right $(C-1)$-bounded approximate
    identity~$\rangle$ and $J$ is a $\langle$~metrically /
    $C$-topologically~$\rangle$ injective Banach space. $\langle$~As the
    consequence $J\isom{\mathbf{Ban}_1}C(K)$ for some Stonean space $K$
    /~$\rangle$.
\end{enumerate}
\end{proposition}
\begin{proof} $(i)\implies (ii)$ By 
propositions~\ref{RetrMetCTopInjIsMetCTopInj} and~\ref{AnnihCModIsRetAnnihMod} 
the $A$-module $\mathbb{C}$ is
$\langle$~metrically / $C$-topologically~$\rangle$ injective as $1$-retract of
$\langle$~metrically / $C$-topologically~$\rangle$ injective module $J$.
Proposition~\ref{MetTopInjModCCharac} gives that $\langle$~$A= \{0 \}$ / $A$ has
a right $(C-1)$-bounded approximate identity~$\rangle$. By
proposition~\ref{MapsFroml1toMetTopInj} the annihilator $A$-module
$\mathcal{B}(\ell_1(B_{J^*}),\mathbb{C})
\isom{\mathbf{mod}_1-A}
\ell_\infty(B_{J^*})$ is $\langle$~metrically / $C$-topologically~$\rangle$ 
injective. Consider isometry $\rho:J\to\ell_\infty(B_{J^*})$ 
well defined by $\rho(x)(f)=f(x)$. Since $J$ and $\ell_\infty(B_{J^*})$ are 
annihilator modules, then $\rho$ is also an $A$-module map. 
Since $J$ is $\langle$~metrically / $C$-topologically~$\rangle$ injective, 
then the $A$-morphism $\rho$ has a left inverse morphism $\tau$ with 
norm $\langle$~at most $1$ / at most $C$~$\rangle$.
Therefore $J$ is $\langle$~$1$-retract / $C$-retract~$\rangle$ of
$\langle$~metrically / $1$-topologically~$\rangle$ injective Banach space
$\ell_\infty(B_{J^*})$. By 
proposition $\langle$~\ref{RetrMetCTopInjIsMetCTopInj} the Banach space $J$ is
$\langle$~metrically / $C$-topologically~$\rangle$ injective. $\langle$~From
[\cite{LaceyIsomThOfClassicBanSp}, theorem 3.11.6] the Banach space $J$ is
isometrically isomorphic to $C(K)$ for some Stonean space $K$. /~$\rangle$ 

$(ii)\implies (i)$ By proposition~\ref{MetTopInjModCCharac} the annihilator
$A$-module $\mathbb{C}$ is $\langle$~metrically / $C$-topologically~$\rangle$
injective. Even more, by proposition~\ref{MapsFroml1toMetTopInj} 
the annihilator $A$-module
$\mathcal{B}(\ell_1(B_{J^*}),\mathbb{C})
\isom{\mathbf{mod}_1-A}
\ell_\infty(B_{J^*})$ is $\langle$~metrically / $C$-topologically~$\rangle$ 
injective too. Since $J$ is a $\langle$~metrically / $C$-topologically~$\rangle$
injective Banach space and there an isometric 
embedding $\rho:J\to \ell_\infty(B_{J^*})$, then as
Banach space $J$ is a $\langle$~$1$-retract / $C$-retract~$\rangle$ of
$\ell_\infty(B_{J^*})$. Recall, that $J$ and $\ell_\infty(B_{J^*})$ are
annihilator modules, so in fact we have a retraction in
$\langle$~$\mathbf{mod}_1-A$ / $\mathbf{mod}-A$~$\rangle$. By 
proposition~\ref{RetrMetCTopInjIsMetCTopInj} the $A$-module $J$ 
is $\langle$~metrically / $C$-topologically~$\rangle$ injective.
\end{proof}

\begin{proposition}\label{MetTopFlatAnnihModCharac} Let $F$ be a non zero
annihilator $A$-module. Then the following are equivalent:

\begin{enumerate}[label = (\roman*)]
    \item $F$ is $\langle$~metrically / $C$-topologically~$\rangle$ flat
    $A$-module;

    \item $\langle$~$A= \{0 \}$ / $A$ has a right $(C-1)$-bounded approximate
    identity~$\rangle$ and $F$ is a $\langle$~metrically /
    $C$-topologically~$\rangle$ flat Banach space, that is
    $\langle$~$F\isom{\mathbf{Ban}_1}L_1(\Omega,\mu)$ for some measure space
    $(\Omega, \Sigma, \mu)$ / $F$ is an $\mathscr{L}_{1,C}^g$-space~$\rangle$.
\end{enumerate}
\end{proposition}
\begin{proof} By $\langle$~[\cite{GrothMetrProjFlatBanSp}, theorem 1] /
[\cite{DefFloTensNorOpId}, corollary 23.5(1)]~$\rangle$ the Banach space $F$ is
$\langle$~metrically / $C$-topologically~$\rangle$ flat iff
$\langle$~$F\isom{\mathbf{Ban}_1}L_1(\Omega,\mu)$ for some measure space
$(\Omega, \Sigma, \mu)$ / $F$ is an $\mathscr{L}_{1,C}^g$-space~$\rangle$. Now
the equivalence follows from propositions~\ref{MetTopInjOfAnnihModCharac}
and~\ref{MetCTopFlatCharac}.
\end{proof}

We obliged to compare these results with similar ones in relative theory. From
$\langle$~[\cite{RamsHomPropSemgroupAlg}, proposition 2.1.7] /
[\cite{RamsHomPropSemgroupAlg}, proposition 2.1.10]~$\rangle$ we know that an
annihilator $A$-module over Banach algebra $A$ is relatively
$\langle$~projective / flat~$\rangle$ iff $A$ has $\langle$~a right identity / a
right bounded approximate identity~$\rangle$.   
In metric and topological theory, in comparison with relative one, homological
triviality of annihilator modules puts restrictions not only on the algebra
itself but on the geometry of the module too. These geometric restrictions
forbid existence of certain homologically excellent algebras. One of the most
important properties of relatively $\langle$~contractible / amenable~$\rangle$
Banach algebra is $\langle$~projectivity / flatness~$\rangle$ of all (and in
particular of all annihilator) left Banach modules over it. In a sharp contrast
in metric and topological theories such algebras can't exist.

\begin{proposition} There is no Banach algebra $A$ such that all $A$-modules are
$\langle$~metrically / topologically~$\rangle$ flat. A fortiori, there is no
such Banach algebras that all $A$-modules are $\langle$~metrically /
topologically~$\rangle$ projective.
\end{proposition}
\begin{proof} Consider any infinite dimensional $\mathscr{L}_\infty^g$-space $X$
(say $\ell_\infty(\mathbb{N})$) as an annihilator $A$-module. From remark right
after [\cite{DefFloTensNorOpId}, corollary 23.3] we know that $X$ is not an
$\mathscr{L}_1^g$-space. Therefore by proposition~\ref{MetTopFlatAnnihModCharac}
the $A$-module $X$ is not topologically flat. By
proposition~\ref{MetFlatIsTopFlatAndTopFlatIsRelFlat} it is not metrically flat.
Now from proposition~\ref{MetTopProjIsMetTopFlat} we see that $X$ is neither
metrically nor topologically projective.
\end{proof}

%-------------------------------------------------------------------------------
%    Homologically trivial modules over Banach algebras with specific geometry
%-------------------------------------------------------------------------------

\subsection{
    Homologically trivial modules over Banach algebras with specific geometry
}\label{
    SubSectionHomologicallyTrivialModulesOverBanachAlgebrasWithSpecificGeometry
}

The purpose of this section is to convince our reader that homologically trivial
modules over certain Banach algebras have similar geometric structure of those
algebras. For the case of metric theory the following proposition was proved by
Graven in~\cite{GravInjProjBanMod}.

\begin{proposition}\label{TopProjInjFlatModOverL1Charac} Let $A$ be a Banach
algebra which is as Banach space isometrically isomorphic to $L_1(\Theta,\nu)$
for some measure space $(\Theta,\Sigma,\nu)$. Then

\begin{enumerate}[label = (\roman*)]
    \item if $P$ is a $\langle$~metrically / topologically~$\rangle$ projective
    $A$-module, then $P$ is $\langle$~an $L_1$-space / complemented in some
    $L_1$-space~$\rangle$.

    \item if $J$ is a $\langle$~metrically / topologically~$\rangle$ injective
    $A$-module, then  $J$ is a $\langle$~$C(K)$-space for some Stonean space $K$
    / topologically injective Banach space~$\rangle$.

    \item if $F$ is a $\langle$~metrically / topologically~$\rangle$ flat
    $A$-module, then $F$ is an $\langle$~$L_1$-space /
    $\mathscr{L}_1^g$-space~$\rangle$.
\end{enumerate}
\end{proposition}
\begin{proof} 

Denote by $(\Theta',\Sigma',\nu')$ the measure space $(\Theta,\Sigma,\nu)$ with
singleton atom adjoined, then $A_+\isom{\mathbf{Ban}_1} L_1(\Theta',\nu')$.

$(i)$ Since $P$ is a $\langle$~metrically / topologically~$\rangle$ projective
$A$-module, then by proposition~\ref{MetCTopProjModViaCanonicMorph} it is a
retract of $A_+\projtens \ell_1(B_P)$ in $\langle$~$A-\mathbf{mod}_1$ /
$A-\mathbf{mod}$~$\rangle$. Let $\mu_c$ be the counting measure on $B_P$, then
by Grothendieck's theorem [\cite{HelLectAndExOnFuncAn}, theorem 2.7.5]
$$
A_+\projtens\ell_1(B_P)
\isom{\mathbf{Ban}_1}L_1(\Theta',\nu')\projtens L_1(B_P,\mu_c)
\isom{\mathbf{Ban}_1}L_1(\Theta'\times B_P,\nu'\times \mu_c)
$$
Therefore $P$ is $\langle$~$1$-complemented / complemented~$\rangle$ in some
$L_1$-space. It remains to recall that any $1$-complemented subspace of
$L_1$-space is again an $L_1$-space [\cite{LaceyIsomThOfClassicBanSp}, theorem
6.17.3].

$(ii)$ Since $J$ is $\langle$~metrically / topologically~$\rangle$ injective
$A$-module, then by proposition~\ref{MetCTopInjModViaCanonicMorph} it is a
retract of $\mathcal{B}(A_+,\ell_\infty(B_{J^*}))$ in
$\langle$~$\mathbf{mod}_1-A$ / $\mathbf{mod}-A$~$\rangle$. Let $\mu_c$ be the
counting measure on $B_{J^*}$, then by Grothendieck's theorem
[\cite{HelLectAndExOnFuncAn}, theorem 2.7.5]
$$
\mathcal{B}(A_+,\ell_\infty(B_{J^*}))
\isom{\mathbf{Ban}_1}
{(A_+\projtens \ell_1(B_{J^*}))}^*
\isom{\mathbf{Ban}_1}
{(L_1(\Theta',\nu')\projtens L_1(B_P,\mu_c))}^*
$$
$$
\isom{\mathbf{Ban}_1}{L_1(\Theta'\times B_P,\nu'\times \mu_c)}^*
\isom{\mathbf{Ban}_1}L_\infty(\Theta'\times B_P,\nu'\times \mu_c)
$$
Therefore $J$ is $\langle$~$1$-complemented / complemented~$\rangle$ in some
$L_\infty$-space. Since $L_\infty$-space is $\langle$~metrically /
topologically~$\rangle$ injective Banach space, then so does $J$. It remains to
recall that every metrically injective Banach space is a $C(K)$-space for some
Stonean space $K$ [\cite{LaceyIsomThOfClassicBanSp}, theorem 3.11.6].

$(iii)$  By $\langle$~[\cite{GrothMetrProjFlatBanSp}, theorem 1] / remark after
[\cite{DefFloTensNorOpId}, corollary 23.5(1)]~$\rangle$ the Banach space $F^*$
is $\langle$~metrically / topologically~$\rangle$ injective iff $F$ is an
$\langle$~$L_1$-space / $\mathscr{L}_1^g$-space~$\rangle$. Now the implication
follows from paragraph $(ii)$ and proposition~\ref{MetCTopFlatCharac}.
\end{proof}

\begin{proposition}\label{TopProjInjFlatModOverMthscrL1SpCharac} Let $A$ be a
Banach algebra which is topologically isomorphic as Banach space to some
$\mathscr{L}_1^g$-space. Then any topologically $\langle$~projective / injective
/ flat~$\rangle$ $A$-module is an $\langle$~$\mathscr{L}_1^g$-space /
$\mathscr{L}_\infty^g$-space / $\mathscr{L}_1^g$-space~$\rangle$.
\end{proposition}
\begin{proof} If $A$ is an $\mathscr{L}_1^g$-space, then so does $A_+$. 

Let $P$ be a topologically projective $A$-module. Then by
proposition~\ref{MetCTopProjModViaCanonicMorph} it is a retract 
of $A_+\projtens \ell_1(B_P)$ in $A-\mathbf{mod}$ and a fortiori 
in $\mathbf{Ban}$. Since
$\ell_1(B_P)$ is an $\mathscr{L}_1^g$-space, then so does
$A_+\projtens\ell_1(B_P)$ as projective tensor product of
$\mathscr{L}_1^g$-spaces [\cite{DefFloTensNorOpId}, exercise 23.17(c)].
Therefore $P$ is an $\mathscr{L}_1^g$-space as complemented subspace of
$\mathscr{L}_1^g$-space [\cite{DefFloTensNorOpId}, corollary 23.2.1(2)].

Let $J$ be a topologically injective $A$-module, then by
proposition~\ref{MetCTopInjModViaCanonicMorph} it is a retract of
$\mathcal{B}(A_+,\ell_\infty(B_{J^*}))
\isom{\mathbf{mod}_1-A}
{(A_+\projtens\ell_1(B_{J^*}))}^*$
in $\mathbf{mod}-A$ and a fortiori in $\mathbf{Ban}$. As we showed in the
previous paragraph $A_+\projtens\ell_1(B_{J^*})$ is an $\mathscr{L}_1^g$-space,
therefore its dual $\mathcal{B}(A_+,\ell_\infty(B_{J^*}))$ is an
$\mathscr{L}_\infty^g$-space [\cite{DefFloTensNorOpId}, corollary 23.2.1(1)]. It
remains to recall that any complemented subspace of $\mathscr{L}_\infty^g$-space
is again an $\mathscr{L}_\infty^g$-space [\cite{DefFloTensNorOpId}, corollary
23.2.1(2)].

Finally, let $F$ be a topologically flat $A$-module, then $F^*$ is topologically
injective $A$-module by proposition~\ref{MetCTopFlatCharac}. From previous
paragraph it follows that $F^*$ is an $\mathscr{L}_\infty^g$-space. By
[\cite{DefFloTensNorOpId}, corollary 23.5(1)] we get that $F$ is an
$\mathscr{L}_1^g$-space.
\end{proof}

We proceed to the discussion of the Dunford-Pettis property for homologically
trivial modules.   

\begin{proposition}\label{C0SumOfL1SpHaveDPP} Let $(\Omega, \Sigma, \mu)$ be a
measure spaces and $\Lambda$ be an arbitrary set. Then the Banach space
$\bigoplus_0 \{L_1(\Omega,\mu):\lambda\in\Lambda \}$ and all its duals has the
Dunford-Pettis property. In particular, $\bigoplus_1
\{L_\infty(\Omega,\mu):\lambda\in\Lambda \}$ and $\bigoplus_\infty
\{L_1(\Omega,\mu):\lambda\in\Lambda \}$ have this property.
\end{proposition}
\begin{proof} Consider one point compactification $\alpha\Lambda:=\Lambda\cup
\{\Lambda \}$ of the set $\Lambda$ with discrete topology. From
[\cite{BourgOnTheDPP}, corollary 7] we know 
that $C(\alpha\Lambda, L_1(\Omega, \mu))$ and all its duals have 
the Dunford-Pettis property. Since $c_0(\Lambda)$ is complemented 
in $C(\alpha\Lambda)$ via projection 
$P:C(\alpha\Lambda)\to C(\alpha\Lambda):x\mapsto x(\lambda)-x( \{\Lambda \})$, 
then $E:=c_0(\Lambda, L_1(\Omega,\mu))$ is complemented 
in $C(\alpha\Lambda, L_1(\Omega, \mu))$. The same statement holds for 
any $n$-th dual of $E$, because we can take $n$-th
adjoint of $P$ in the role of projection. Now it remains to note that
$E=\bigoplus_0 \{L_1(\Omega,\mu):\lambda\in\Lambda \}$ and that the
Dunford-Pettis property is inherited by complemented subspaces
[\cite{FabHabBanSpTh}, proposition 13.44]. 

As the consequence of previous paragraph the Banach spaces
$E^*
\isom{\mathbf{Ban}_1}
\bigoplus_1 \{
    L_\infty(\Omega,\mu):\lambda\in\Lambda
 \}$ 
and $E^{**}
\isom{\mathbf{Ban}_1}
    \bigoplus_\infty \{{L_1(\Omega,\mu)}^{**}:\lambda\in\Lambda
 \}$ 
have the Dunford-Pettis property.
From [\cite{DefFloTensNorOpId}, proposition B10] we know that any $L_1$-space is
contractively complemented in its second dual. By $Q$ we denote the respective
projection. Therefore the Banach space $\bigoplus_\infty
\{L_1(\Omega,\mu):\lambda\in\Lambda \}$ is contractively complemented in
$E^{**}$ via projection $\bigoplus_\infty  \{Q:\lambda\in\Lambda \}$. Since
$E^{**}$ has the Dunford-Pettis property, then by [\cite{FabHabBanSpTh},
proposition 13.44] so does its complemented 
subspace $\bigoplus_\infty \{L_1(\Omega,\mu):\lambda\in\Lambda \}$.
\end{proof}

\begin{proposition}\label{ProdOfL1SpHaveDPP} Let $
\{(\Omega_\lambda,\Sigma_\lambda,\mu_\lambda):\lambda\in\Lambda \}$ be a family
of measure spaces. 
Then 
$\bigoplus_0 \{
    L_1(\Omega_\lambda, \mu_\lambda):\lambda\in\Lambda
 \}$, 
 $\bigoplus_1 \{
     L_\infty(\Omega_\lambda, \mu_\lambda):\lambda\in\Lambda
 \}$ 
and $\bigoplus_\infty \{
    L_1(\Omega_\lambda,\mu_\lambda):\lambda\in\Lambda
 \}$ have the Dunford-Pettis property.
\end{proposition}
\begin{proof} Let $(\Omega, \Sigma, \mu)$ be a direct sum of $
\{(\Omega_\lambda,\Sigma_\lambda,\mu_\lambda):\lambda\in\Lambda \}$. By
construction each Banach space $L_1(\Omega_\lambda,\mu_\lambda)$ for
$\lambda\in\Lambda$ is $1$-complemented in $L_1(\Omega, \mu)$. Therefore, their
$\bigoplus_0$-, $\bigoplus_1$- and $\bigoplus_\infty$-sums are contractively
complemented in $\bigoplus_0 \{L_1(\Omega, \mu):\lambda\in\Lambda \}$,
$\bigoplus_1 \{L_\infty(\Omega, \mu):\lambda\in\Lambda \}$ 
and $\bigoplus_\infty \{L_1(\Omega,\mu):\lambda\in\Lambda \}$ respectively. 
It remains to combine proposition~\ref{C0SumOfL1SpHaveDPP} and the fact 
that the Dunford-Pettis property is preserved by complemented 
subspaces [\cite{FabHabBanSpTh}, proposition 13.44].
\end{proof}

\begin{proposition}\label{ProdOfDualsOfMthscrLInftySpHaveDPP} Let $E$ be an
$\mathscr{L}_\infty^g$-space and $\Lambda$ be an arbitrary set. Then
$\bigoplus_\infty \{E^*:\lambda\in\Lambda \}$ has the Dunford-Pettis property.
\end{proposition}
\begin{proof} Since $E$ is an $\mathscr{L}_\infty^g$-space, then $E^*$ is a
$\mathscr{L}_{1}^g$-space [\cite{DefFloTensNorOpId}, corollary 23.2.1(1)]. Then
from [\cite{DefFloTensNorOpId}, corollary 23.2.1(3)] it follows that $E^{***}$
is a retract of $L_1$-space. Recall that $E^*$ is complemented in $E^{***}$ via
Dixmier projection, so $E^*$ is complemented in some $L_1$-space too. Thus we
have a bounded linear projection $P:L_1(\Omega,\mu)\to L_1(\Omega,\mu)$ with
image topologically isomorphic to $E$. In this 
case $\bigoplus_\infty \{E^*:\lambda\in\Lambda \}$ is complemented 
in $\bigoplus_\infty \{ L_1(\Omega,\mu):\lambda\in\Lambda \}$ via 
projection $\bigoplus_\infty \{P:\lambda\in\Lambda \}$. The 
space $\bigoplus_\infty \{ L_1(\Omega,\mu):\lambda\in\Lambda \}$ has 
the Dunford-Pettis property by proposition~\ref{ProdOfL1SpHaveDPP}. By 
proposition 13.44 in~\cite{FabHabBanSpTh} so does 
$\bigoplus_\infty \{ E^*:\lambda\in\Lambda \}$ as complemented subspace 
of $\bigoplus_\infty \{ L_1(\Omega,\mu):\lambda\in\Lambda \}$.
\end{proof}

\begin{proposition}\label{MthscrL1LInftyHaveDPP} Any
$\langle$~$\mathscr{L}_1^g$-space / $\mathscr{L}_\infty^g$-space~$\rangle$
admits the Dunford-Pettis property.
\end{proposition}
\begin{proof} Assume $E$ is an $\langle$~$\mathscr{L}_1^g$-space /
$\mathscr{L}_\infty^g$-space~$\rangle$. Then $E^{**}$ is complemented in some
$\langle$~$L_1$-space / $L_\infty$-space~$\rangle$ [\cite{DefFloTensNorOpId},
corollary 23.2.1(3)]. Since any $\langle$~$L_1$-space /
$L_\infty$-space~$\rangle$ admits the Dunford-Pettis
property~\cite{GrothApllFaiblCompSpCK}, then so does $E^{**}$ as complemented
subspace [\cite{FabHabBanSpTh}, proposition 13.44]. It remains to recall that a
Banach space have the Dunford-Pettis property whenever so does its dual.
\end{proof}

\begin{theorem}\label{TopProjInjFlatModOverMthscrL1OrLInftySpHaveDPP} Let $A$ be
a Banach algebra which is an $\mathscr{L}_1^g$-space or
$\mathscr{L}_\infty^g$-space as Banach space. Then any topologically projective,
injective or flat $A$-module has the Dunford-Pettis property.
\end{theorem}
\begin{proof} If $A$ is an $\mathscr{L}_1^g$-space, we just need to combine
propositions~\ref{MthscrL1LInftyHaveDPP}
and~\ref{TopProjInjFlatModOverMthscrL1SpCharac}. 

Assume $A$ is an $\mathscr{L}_\infty^g$-space, then so does $A_+$. Let $J$ be a
topologically injective $A$-module, then by
proposition~\ref{MetCTopInjModViaCanonicMorph} it is a retract of 
$$
\mathcal{B}(A_+,\ell_\infty(B_{J^*}))
\isom{\mathbf{mod}_1-A}
{(A_+\projtens\ell_1(B_{J^*}))}^*
\isom{\mathbf{mod}_1-A}
{\left(\bigoplus\nolimits_1 \{ A_+:\lambda\in B_{J^*} \}\right)}^*
$$
$$
\isom{\mathbf{mod}_1-A}
\bigoplus\nolimits_\infty \{ A_+^*:\lambda\in B_{J^*} \}
$$ 
in $\mathbf{mod}-A$ and a fortiori in $\mathbf{Ban}$. By
proposition~\ref{ProdOfDualsOfMthscrLInftySpHaveDPP} this space has the
Dunford-Pettis property. As $J$ is its retract, then $J$ also has this property
[\cite{FabHabBanSpTh}, proposition 13.44]. 

If $F$ is a topologically flat $A$-module, then $F^*$ is a topologically
injective $A$-module by proposition~\ref{MetCTopFlatCharac}. By previous
paragraph $F^*$ has the Dunford-Pettis property and so does $F$.

If $P$ is a topologically projective $A$-module, it is also topologically flat
by proposition~\ref{MetTopProjIsMetTopFlat}. From previous paragraph it follows
that $P$ has the Dunford-Pettis property.
\end{proof}

\begin{corollary}\label{NoInfDimRefMetTopProjInjFlatModOverMthscrL1OrLInfty} Let
$A$ be a Banach algebra which $\mathscr{L}_1^g$-space or
$\mathscr{L}_\infty^g$-space as Banach space. Then there is no topologically
projective, injective or flat infinite dimensional reflexive $A$-modules. A
fortiori there is no metrically projective, injective or flat infinite
dimensional reflexive $A$-modules.
\end{corollary}
\begin{proof} From theorem~\ref{TopProjInjFlatModOverMthscrL1OrLInftySpHaveDPP}
we know that any topologically injective $A$-module has the Dunford-Pettis
property. On the other hand there is no infinite dimensional reflexive Banach
spaces with the Dunford-Pettis property. Thus we get the desired result
regarding topological injectivity. Since dual of reflexive module is reflexive,
from proposition~\ref{MetCTopFlatCharac} we get the result for topological
flatness. It remains to recall that by proposition~\ref{MetTopProjIsMetTopFlat}
every topologically projective module is topologically flat. To prove the last
claim note that metric $\langle$~projectivity / injectivity / flatness~$\rangle$
implies topological $\langle$~projectivity / injectivity / flatness~$\rangle$ by
proposition $\langle$~\ref{MetProjIsTopProjAndTopProjIsRelProj}
/~\ref{MetInjIsTopInjAndTopInjIsRelInj}
/~\ref{MetFlatIsTopFlatAndTopFlatIsRelFlat}~$\rangle$.
\end{proof}

Note that in relative theory there are examples of infinite dimensional
relatively projective injective and flat reflexive modules over Banach algebras
that are $\mathscr{L}_1^g$- or $\mathscr{L}_\infty^g$-spaces. Here are two
examples. The first one is about convolution algebra $L_1(G)$ on a locally
compact group $G$ with Haar measure. It is an $\mathscr{L}_1^g$-space. In
[\cite{DalPolHomolPropGrAlg}, \S6] and~\cite{RachInjModAndAmenGr} it was proved
that for $1<p<+\infty$ the $L_1(G)$-module $L_p(G)$ is relatively
$\langle$~projective / injective / flat~$\rangle$ iff $G$ is $\langle$~compact /
amenable / amenable~$\rangle$. Note that any compact group is amenable
[\cite{PierAmenLCA}, proposition 3.12.1], so in case $G$ is compact $L_p(G)$ is
relatively projective injective and flat for all $1<p<+\infty$.  The second
example is about $\mathscr{L}_\infty^g$-spaces $c_0(\Lambda)$ and
$\ell_\infty(\Lambda)$ for an infinite set $\Lambda$. In
proposition~\ref{c0AndlInftyModsRelTh} we shall show that $\ell_p(\Lambda)$ for
$1<p<\infty$ is relatively projective injective and flat as $c_0(\Lambda)$- or
$\ell_\infty(\Lambda)$-module. 

We finalize this lengthy section by short remark on the l.u.st.\ property of
topologically projective, injective and flat modules. 

\begin{proposition} Let $A$ be a Banach algebra which as Banach space has the
l.u.st.\ property. Then any topologically projective, injective or flat
$A$-module has the l.u.st.\ property.
\end{proposition}
\begin{proof} 

If $J$ is topologically injective $A$-module, then by
proposition~\ref{MetCTopInjModViaCanonicMorph} it is a retract of
$\mathcal{B}(A_+,\ell_\infty(B_{J^*}))
\isom{\mathbf{mod}_1-A}
\bigoplus_\infty \{ A_+^*:\lambda\in B_{J^*} \}$ 
in $\mathbf{mod}-A$ and a fortiori in $\mathbf{Ban}$. If $A$ has 
the l.u.st.\ property, then $A^{**}$ is complemented
in some Banach lattice $E$ [\cite{DiestAbsSumOps}, theorem 17.5]. As the
consequence $A_+^{***}$ is complemented in the Banach lattice
$F:={\left(E\bigoplus_1\mathbb{C}\right)}^*$ via some bounded projection 
$P:F\to F$. Thus $\bigoplus_\infty \{A_+^{***}:\lambda\in B_{J^*} \}$ is 
complemented in the Banach lattice $\bigoplus_\infty \{F:\lambda\in B_{J^*} \}$ 
via projection $\bigoplus_\infty \{ P:\lambda\in B_{J^*} \}$. Any Banach 
lattice has the l.u.st.\ property [\cite{DiestAbsSumOps}, theorem 17.1]. 
This property is inherited by complemented subspaces, 
so $\bigoplus_\infty \{A_+^{***}:\lambda\in B_{J^*} \}$ has 
the l.u.st.\ property too. Note that $A_+^*$ is contractively
complemented in $A_+^{***}$ by Dixmier projection $Q$, therefore
$\bigoplus_\infty \{A_+^*:\lambda\in B_{J^*} \}$ is complemented in
$\bigoplus_\infty \{A_+^{***}:\lambda\in B_{J^*} \}$ via contractive projection
$\bigoplus_\infty \{Q:\lambda\in B_{J^*} \}$. Since the latter space has the
l.u.st.\ property, then so does its 
retract $\bigoplus_\infty \{A_+^*:\lambda\in B_{J^*} \}$. Finally $J$ is a 
retract of the $\bigoplus_\infty \{A_+^*:\lambda\in B_{J^*} \}$, therefore 
also has this property.

If $F$ is topologically flat $A$-module, then $F^*$ is topologically injective
by proposition~\ref{MetCTopFlatCharac}. By previous paragraph $F^*$ has the
l.u.st.\ property. Corollary 17.6 from~\cite{DiestAbsSumOps} gives that $F$ has
this property too.

If $P$ is topologically projective $A$-module, it is topologically flat by
proposition~\ref{MetTopProjIsMetTopFlat}. So $P$ has the l.u.st.\ property by
previous paragraph.
\end{proof}

%-------------------------------------------------------------------------------
%    Further properties of projective injective and flat modules
%-------------------------------------------------------------------------------

\section{
    Further properties of projective injective and flat modules
}\label{SectionFurtherPropertiesOfProjectiveInjectiveAndFlatModules}

%----------------------------------------------------------------------------------------
%    Homological triviality of modules under change of algebra
%----------------------------------------------------------------------------------------

\subsection{
    Homological triviality of modules under change of algebra
}\label{SubSectionHomologicalTrivialityOfModulesUnderChangeOfAlgebra}

The following three propositions are metric and topological versions of
propositions 2.3.2, 2.3.3 and 2.3.4 in~\cite{RamsHomPropSemgroupAlg}.

\begin{proposition}\label{MorphCoincide} Let $X$ and $Y$ be $\langle$~left /
right~$\rangle$ $A$-modules. Assume one of the following holds

\begin{enumerate}[label = (\roman*)]
    \item $I$ is a $\langle$~left / right~$\rangle$ ideal of $A$ and $X$ is an
    essential $I$-module;

    \item $I$ is a $\langle$~right / left~$\rangle$ ideal of $A$ and $Y$ is a
    faithful $I$-module. 
\end{enumerate}

Then $\langle$~${}_A\mathcal{B}(X,Y)={}_I\mathcal{B}(X,Y)$ /
$\mathcal{B}_A(X,Y)=\mathcal{B}_I(X,Y)$~$\rangle$.
\end{proposition}
\begin{proof} We shall prove both statements only for left modules, since their
right counterparts are proved with minimal modifications. 
Fix $\phi\in{}_I\mathcal{B}(X,Y)$.

$(i)$ Take $x\in I\cdot X$, then $x=a'\cdot x'$ for some $a'\in I$, $x'\in X$.
For any $a\in A$ we have $\phi(a\cdot x)=\phi(aa'\cdot
x')=aa'\cdot\phi(x')=a\cdot\phi(a'\cdot x')=a\cdot\phi(x)$. Therefore
$\phi(a\cdot x)=a\cdot\phi(x)$ for all $a\in A$ 
and $x\in \operatorname{cl}_X(IX)=X$. Hence $\phi\in {}_A\mathcal{B}(X,Y)$.

$(ii)$ For any $a\in I$ and $a'\in A$, $x\in X$ we have $a\cdot(\phi(a'\cdot
x)-a'\cdot\phi(x))=\phi(aa'\cdot x)-aa'\cdot\phi(x)=0$. Since $Y$ is faithful
$I$-module we have $\phi(a'\cdot x)=a'\cdot \phi(x)$ for all $x\in X$, 
$a'\in A$. Hence $\phi\in{}_A\mathcal{B}(X,Y)$.

In both cases we proved that $\phi\in{}_A\mathcal{B}(X,Y)$ for any
$\phi\in{}_I\mathcal{B}(X,Y)$, 
therefore ${}_I\mathcal{B}(X,Y)\subset {}_A\mathcal{B}(X,Y)$. 
The reverse inclusion is obvious.
\end{proof}

\begin{proposition}\label{MetTopProjUnderChangeOfAlg} Let $I$ be a closed
subalgebra of $A$ and $P$ be an $A$-module which is essential as $I$-module.
Then

\begin{enumerate}[label = (\roman*)]
    \item if $I$ is a left ideal of $A$ and $P$ is $\langle$~metrically /
    $C$-topologically~$\rangle$  projective $I$-module, then $P$ is
    $\langle$~metrically / $C$-topologically~$\rangle$ projective $A$-module;

    \item if $I$ is a $\langle$~$1$-complemented / $c$-complemented~$\rangle$
    right ideal of $A$ and $P$ is $\langle$~metrically /
    $C$-topologically~$\rangle$ projective $A$-module, then $P$ is
    $\langle$~metrically / $cC$-topologically~$\rangle$ projective $I$-module.
\end{enumerate}
\end{proposition}
\begin{proof} By $\widetilde{\pi}_P: I\projtens \ell_1(B_P)\to P$ and
$\pi_P:A\projtens \ell_1(B_P)\to P$ we will denote the standard epimorphisms.

$(i)$ By proposition~\ref{NonDegenMetTopProjCharac} the morphism
$\widetilde{\pi}_P$ has a right inverse morphism in $\langle$~$I-\mathbf{mod}_1$
/ $I-\mathbf{mod}$~$\rangle$, say $\widetilde{\sigma}$ of norm $\langle$~at most
$1$ / at most $C$~$\rangle$. Let $i:I\to A$ be the natural embedding, then
consider $\langle$~contractive / bounded~$\rangle$ $I$-morphism
$\sigma=(i\projtens 1_{\ell_1(B_P)})\widetilde{\sigma}$. By paragraph $(i)$ of
proposition~\ref{MorphCoincide} we have that $\sigma$ is an $A$-morphism.
Clearly, $\sigma$ has norm $\langle$~at most $1$ / at most $C$~$\rangle$. For
$\pi_P:A\projtens \ell_1(B_P)\to P$ we obviously 
have $\pi_P(i\projtens 1_{\ell_1(B_P)})=\widetilde{\pi}_P$, 
hence 
$\pi_P\sigma
=\pi_P(i\projtens 1_{\ell_1(B_P)})\widetilde{\sigma}
=\widetilde{\pi}_P\widetilde{\sigma}
=1_P$.
Thus $\pi_P$ is a $\langle$~$1$-retraction / $C$-retraction~$\rangle$ in
$\langle$~$A-\mathbf{mod}_1$ / $A-\mathbf{mod}$~$\rangle$. So by
proposition~\ref{NonDegenMetTopProjCharac} the $A$-module $P$ is
$\langle$~metrically / $C$-topologically~$\rangle$ projective.

$(ii)$ Since $P$ is an essential $I$-module it is a fortiori an essential
$A$-module. By proposition~\ref{NonDegenMetTopProjCharac} the morphism $\pi_P$
has a right inverse morphism $\sigma$ in $\langle$~$A-\mathbf{mod}_1$ /
$A-\mathbf{mod}$~$\rangle$ with norm $\langle$~at most $1$ / at most
$C$~$\rangle$. Obviously $\sigma$ is a right inverse for $\pi_P$ in
$\langle$~$I-\mathbf{mod}_1$ / $I-\mathbf{mod}$~$\rangle$ too. By $i:I\to A$ we
denote the natural embedding, and by $r:A\to I$ the $\langle$~contractive /
bounded~$\rangle$ left inverse. By assumption $\Vert r\Vert\leq c$. Consider
$\langle$~contractive / bounded~$\rangle$ linear operator
$\widetilde{\sigma}=(r\projtens 1_{\ell_1(B_P)})\sigma$. Clearly, its norm is
$\langle$~at most $1$ / at most $cC$~$\rangle$. Since $I$ is a right ideal of
$A$ and $P$ is an essential $I$-module then
$\sigma(P)
=\sigma(\operatorname{cl}_P(IP))
=\operatorname{cl}_{A\projtens \ell_1(B_P)}(I\cdot (A\projtens \ell_1(B_P)))
=I\projtens \ell_1(B_P)$, so
$\sigma=(ir\projtens 1_{\ell_1(B_P)})\sigma$. Even more, 
since $\sigma(P)\subset I\projtens\ell_1(B_P)$ and $r|_I=1_I$, 
then $\sigma$ is an $I$-morphism.
Clearly, $\pi_P(i\projtens 1_{\ell_1(B_P)})=\widetilde{\pi}_P$, so
$$
\widetilde{\pi}_P\widetilde{\sigma}
=\pi_P(i\projtens 1_{\ell_1(B_P)})(r \projtens 1_{\ell_1(B_P)})\sigma
=\pi_P(ir\projtens 1_{\ell_1(B_P)})\sigma
=\pi_P\sigma
=1_P
$$ 
Thus $\widetilde{\pi}_P$ is 
a $\langle$~$1$-retraction / $cC$-retraction~$\rangle$ 
in $\langle$~$I-\mathbf{mod}_1$ / $I-\mathbf{mod}$~$\rangle$, so by
proposition~\ref{NonDegenMetTopProjCharac} the $I$-module $P$ is
$\langle$~metrically / $cC$-topologically~$\rangle$ projective.
\end{proof}

\begin{proposition}\label{MetTopInjUnderChangeOfAlg} Let $I$ be a closed
subalgebra of $A$ and $J$ be a right $A$-module which is faithful as $I$-module.
Then

\begin{enumerate}[label = (\roman*)]
    \item if $I$ is a left ideal of $A$ and $J$ is $\langle$~metrically /
    $C$-topologically~$\rangle$  injective $I$-module, then $J$ is
    $\langle$~metrically / $C$-topologically~$\rangle$ injective $A$-module;

    \item if $I$ is a weakly $\langle$~$1$-complemented /
    $c$-complemented~$\rangle$ right ideal of $A$ and $J$ is
    $\langle$~metrically / $C$-topologically~$\rangle$ injective $A$-module,
    then $J$ is $\langle$~metrically / $cC$-topologically~$\rangle$ injective
    $I$-module.
\end{enumerate}
\end{proposition}
\begin{proof} By $\widetilde{\rho}_J:J\to\mathcal{B}(I,\ell_\infty(B_{J^*}))$
and $\rho_J:J\to\mathcal{B}(A,\ell_\infty(B_{J^*}))$ we will denote the standard
monomorphisms.

$(i)$ By proposition~\ref{NonDegenMetTopInjCharac} the morphism
$\widetilde{\rho}_J: J\to\mathcal{B}(I,\ell_\infty(B_{J^*}))$ has a left inverse
morphism in $\langle$~$\mathbf{mod}_1-I$ / $\mathbf{mod}-I$~$\rangle$, say
$\widetilde{\tau}$ of norm $\langle$~at most $1$ / at most $C$~$\rangle$. Let
$i:I\to A$ be the natural embedding, and define $I$-morphism
$q=\mathcal{B}(i,\ell_\infty(B_{J^*}))$. Obviously
$\widetilde{\rho}_J=q\rho_J$. Consider $I$-morphism $\tau =\widetilde{\tau} q$.
By paragraph $(ii)$ of proposition~\ref{MorphCoincide} it is also an
$A$-morphism. Note that 
$\Vert\tau \Vert
\leq\Vert\widetilde{\tau}\Vert\Vert q\Vert
\leq\Vert\widetilde{\tau}\Vert$, 
so $\tau$ has norm $\langle$~at most $1$ / at most $C$~$\rangle$. 
Clearly, 
$\tau \rho_J=\widetilde{\tau} q\rho_J=\widetilde{\tau}\widetilde{\rho}_J=1_J$. 
Thus $\rho_J$ is a $\langle$~$1$-coretraction / $C$-coretraction~$\rangle$, 
so by proposition~\ref{NonDegenMetTopInjCharac} the $A$-module $J$ is
$\langle$~metrically / $C$-topologically~$\rangle$ injective.

$(ii)$ If $J$ is $\langle$~metrically / $C$-topologically~$\rangle$ injective as
$A$-module, then by proposition~\ref{NonDegenMetTopInjCharac} the $A$-morphism
$\rho_J$ has a left inverse in $\langle$~$\mathbf{mod}_1-A$ /
$\mathbf{mod}-A$~$\rangle$, say $\tau $ of norm $\langle$~at most $1$ / at most
$C$~$\rangle$. Assume we are given an operator $T\in
\mathcal{B}(A,\ell_\infty(B_{J^*}))$, such that $T|_I=0$. Fix $a\in I$, then
$T\cdot a=0$, and so $\tau (T)\cdot a=\tau (T\cdot a)=0$. Since $J$ is faithful
$I$-module and $a\in I$ is arbitrary, then $\tau (T)=0$. Since $I$ is weakly 
$\langle$~$1$-complemented /$c$-complemented~$\rangle$ then $i^*$ has left 
inverse $r:I^*\to A^*$ with norm $\langle$~at most $1$ / at most $c$~$\rangle$.
For a given $f\in B_{J^*}$ we define a bounded linear operator 
$g_f:\mathcal{B}(I,\ell_\infty(B_{J^*}))\to I^*:T\mapsto(x\mapsto T(x)(f))$. 
Now we can define two more bounded linear operators
$$
j
:\mathcal{B}(I,\ell_\infty(B_{J^*}))\to \mathcal{B}(A,\ell_\infty(B_{J^*}))
:T\mapsto (a\mapsto (f\mapsto r(g_{f}(T))(a)))
$$ 
and $\widetilde{\tau}=\tau  j$. Fix $a\in I$ 
and $T\in\mathcal{B}(I,\ell_\infty(B_{J^*}))$. Since $r$ is a left inverse 
of $i^*$ we have $r(h)(a)=h(a)$ for all $h\in I^*$. Now it is routine to check 
that $(j(T\cdot a)-j(T)\cdot a)|_I=0$. As we have shown earlier this implies
that $\widetilde{\tau}(T\cdot a)-\widetilde{\tau}(T)\cdot a
=\tau (j(T\cdot a)-j(T)\cdot a)=0$. Since $a\in I$ 
and $T\in\mathcal{B}(I,\ell_\infty(B_{J^*}))$ are arbitrary the 
map $\widetilde{\tau}$ is an $I$-morphism.
Note that 
$\Vert\widetilde{\tau}\Vert\leq\Vert\tau \Vert\Vert j\Vert\leq \Vert
\tau\Vert\Vert r\Vert$, so $\widetilde{\tau}$ has 
norm $\langle$~at most $1$ / at most $cC$~$\rangle$.
In the same way one can show that, for all $x\in J$ 
holds $\rho_J(x)-j(\widetilde{\rho}_J(x))|_I=0$,
so $\tau (\rho_J(x)-j(\widetilde{\rho}_J(x)))=0$. As a consequence,
$\widetilde{\tau}(\widetilde{\rho}_J(x))
=\tau (j(\widetilde{\rho}_J(x)))=\tau(\rho_J(x))=x$ 
for all $x\in J$. Since $\widetilde{\tau}\widetilde{\rho}_J=1_J$,
then $\widetilde{\rho}_J$ is a  $\langle$~$1$-coretraction /
$cC$-coretraction~$\rangle$ in $\langle$~$\mathbf{mod}_1-I$ /
$\mathbf{mod}-I$~$\rangle$, so by proposition~\ref{NonDegenMetTopInjCharac} the
$I$-module $J$ is $\langle$~metrically / $cC$-topologically~$\rangle$ injective.
\end{proof}

\begin{proposition}\label{MetTopFlatUnderChangeOfAlg} Let $I$ be a closed
subalgebra of $A$ and $F$ be an $A$-module which is essential as $I$-module.
Then

\begin{enumerate}[label = (\roman*)]
    \item if $I$ is a left ideal of $A$ and $F$ is $\langle$~metrically /
    $C$-topologically~$\rangle$  flat $I$-module, then $F$ is
    $\langle$~metrically / $C$-topologically~$\rangle$ flat $A$-module;

    \item if $I$ is a weakly $\langle$~$1$-complemented / 
    $c$-complemented~$\rangle$ right ideal of $A$ and $F$ 
    is $\langle$~metrically / $C$-topologically~$\rangle$ flat $A$-module, 
    then $F$ is $\langle$~metrically / $cC$-topologically~$\rangle$ 
    flat $I$-module.
\end{enumerate}
\end{proposition}
\begin{proof} Note that the dual of essential module is faithful. Now the result
follows from propositions~\ref{MetCTopFlatCharac} 
and~\ref{MetTopInjUnderChangeOfAlg}.
\end{proof}    

\begin{proposition}\label{MetTopProjInjFlatUnderSumOfAlg} Let
${(A_\lambda)}_{\lambda\in\Lambda}$ be a family of Banach algebras and for each
$\lambda\in\Lambda$ let $X_\lambda$ be $\langle$~an essential / a faithful / an
essential~$\rangle$ $A_\lambda$-module. Denote $A=\bigoplus_p
\{A_\lambda:\lambda\in\Lambda \}$ for $1\leq p\leq +\infty$ or $p=0$. Let $X$
denote $\langle$~$\bigoplus_1 \{X_\lambda:\lambda\in\Lambda \}$ /
$\bigoplus_\infty \{X_\lambda:\lambda\in\Lambda \}$ / 
$\bigoplus_1 \{X_\lambda:\lambda\in\Lambda \}$~$\rangle$. Then

\begin{enumerate}[label = (\roman*)]
    \item $X$ is metrically $\langle$~projective / injective / flat~$\rangle$
    $A$-module iff for all $\lambda\in\Lambda$ the $A_\lambda$-module
    $X_\lambda$ is metrically $\langle$~projective / injective / flat~$\rangle$;

    \item $X$ is $C$-topologically $\langle$~projective / injective /
    flat~$\rangle$ $A$-module iff for all $\lambda\in\Lambda$ the
    $A_\lambda$-module $X_\lambda$ is $C$-topologically $\langle$~projective /
    injective / flat~$\rangle$.
\end{enumerate}
\end{proposition}
\begin{proof} Note that for each $\lambda\in\Lambda$ the natural embedding
$i_\lambda:A_\lambda\to A$ allows to regard $A_\lambda$ as complemented in
$\mathbf{Ban}_1$ two sided ideal of $A$.

$(i)$ The proof is literally the same as in paragraph $(ii)$.

$(ii)$ Assume $X_\lambda$ is $C$-topologically $\langle$~projective / injective
/ flat~$\rangle$ $A_\lambda$-module for all $\lambda\in\Lambda$, then by
paragraph $(i)$ of proposition $\langle$~\ref{MetTopProjUnderChangeOfAlg}
/~\ref{MetTopInjUnderChangeOfAlg} /~\ref{MetTopFlatUnderChangeOfAlg}~$\rangle$
it is $C$-topologically $\langle$~projective / injective / flat~$\rangle$ as
$A$-module. It remains to apply the proposition
$\langle$~\ref{MetTopProjModCoprod} /~\ref{MetTopInjModProd}
/~\ref{MetTopFlatModCoProd}~$\rangle$. 

Conversely, assume that $X$ is $C$-topologically $\langle$~projective /
injective / flat~$\rangle$ as $A$-module. Fix arbitrary $\lambda\in\Lambda$.
Clearly, we may regard $X_\lambda$ as $A$-module and even more $X_\lambda$ is a
$1$-retract of $X$ in $\langle$~$A-\mathbf{mod}_1$ / $\mathbf{mod}_1-A$ /
$A-\mathbf{mod}_1$~$\rangle$. By proposition
$\langle$~\ref{RetrMetCTopProjIsMetCTopProj} /~\ref{RetrMetCTopInjIsMetCTopInj}
/~\ref{RetrMetCTopFlatIsMetCTopFlat}~$\rangle$ we get that $X_\lambda$ is
$C$-topologically $\langle$~projective / injective / flat~$\rangle$ as
$A$-module. It remains to apply paragraph $(ii)$ of proposition
$\langle$~\ref{MetTopProjUnderChangeOfAlg} /~\ref{MetTopInjUnderChangeOfAlg}
/~\ref{MetTopFlatUnderChangeOfAlg}~$\rangle$.
\end{proof} 

%----------------------------------------------------------------------------------------
%    Further properties of flat modules
%----------------------------------------------------------------------------------------

\subsection{
    Further properties of flat modules
}\label{SubSectionFurtherPropertiesOfFlatModules}

Based on results obtained above, we collect more interesting facts on metric and
topological injectivity and flatness of Banach modules.

\begin{proposition}\label{DualBanModDecomp} Let $B$ be a unital Banach algebra,
$A$ be its subalgebra with two-sided bounded approximate identity
${(e_\nu)}_{\nu\in N}$ and $X$ be a left $B$-module. 
Denote $c_1=\sup_{\nu\in N}\Vert e_\nu\Vert$, 
$c_2=\sup_{\nu\in N}\Vert e_B-e_\nu\Vert$ and $X_{ess}=\operatorname{cl}_X(AX)$.
Then 

\begin{enumerate}[label = (\roman*)]
    \item $X^*$ is $c_2(c_1+1)$-isomorphic as a right $A$-module to
    $X_{ess}^*\bigoplus_\infty {(X/X_{ess})}^*$;

    \item $\langle$~$X_{ess}^*$ / ${(X/X_{ess})}^*$~$\rangle$ is a $\langle$
    $c_1$-retract / $c_2$-retract~$\rangle$ of $A$-module $X^*$;

    \item if $X$ is an $\mathscr{L}_{1,C}^g$-space, then $\langle$~$X_{ess}$ /
    $X/X_{ess}$~$\rangle$ is an $\langle$~$\mathscr{L}_{1,c_1C}^g$-space /
    $\mathscr{L}_{1,c_2C}^g$-space~$\rangle$.
\end{enumerate}

\end{proposition}
\begin{proof} $(i)$ Define the natural embedding $\rho:X_{ess}\to X:x\mapsto x$
and the quotient map  $\pi:X\to X/X_{ess}:x\mapsto x+X_{ess}$. Let
$\mathfrak{F}$ be the section filter on $N$ and let $\mathfrak{U}$ be an
ultrafilter dominating $\mathfrak{F}$. For a fixed $f\in X ^*$ and $x\in X $ we
have 
$|f(x-e_\nu\cdot x)|
\leq\Vert f\Vert\Vert e_B - e_\nu\Vert\Vert x\Vert
\leq c_2\Vert f\Vert\Vert x\Vert$ i.e. ${(f(x-e_\nu\cdot x))}_{\nu\in N}$ is a
bounded net of complex numbers. Therefore we have a well defined limit
$\lim_{\mathfrak{U}}f(x-e_\nu\cdot x)$ along ultrafilter $\mathfrak{U}$. Since
${(e_\nu)}_{\nu\in N}$ is a two-sided approximate identity for $A$ and
$\mathfrak{U}$ contains section filter then for all $x\in X_{ess}$ we have
$\lim_{\mathfrak{U}}f(x-e_\nu\cdot x)=\lim_{\nu}f(x-e_\nu\cdot x)=0$. Therefore
for each $f\in X ^*$ we have a well defined map 
$\tau(f)
:X /X_{ess}\to \mathbb{C}
:x+X_{ess}\mapsto \lim_{\mathfrak{U}} f(x-e_\nu\cdot x)$. 
Clearly this is a linear functional, and from inequalities above we see 
its norm is bounded by $c_2\Vert f\Vert$. Now it is routine to check 
that $\tau:X^*\to {(X/ X_{ess})}^*:f\mapsto \tau(f)$ is an $A$-morphism 
with norm not greater than $c_2$. Similarly, one can show that 
$\sigma
:X_{ess}^*\to X^*
:h\mapsto(x\mapsto \lim_{\mathfrak{U}}h(e_\nu\cdot x))$ is an $A$-morphism 
with norm not greater than $c_1$. For any $f\in X^*$, $g\in {(X/X_{ess})}^*$, 
$h\in X_{ess}^*$ and $x\in X$, $y\in X_{ess}$ we have
$$
\sigma(h)(y)
=\lim_{\mathfrak{U}}h(e_\nu\cdot y)
=\lim_{\nu}h(e_\nu\cdot y)
=h(y),
\qquad
(\rho^*\sigma)(h)(y)
=\sigma(h)(\rho(y))
\sigma(h)(y)
=h(y),
$$
$$
(\tau\pi^*)(g)(x+X_{ess})
=\lim_{\mathfrak{U}}\pi^*(g)(x-e_\nu\cdot x)
=\lim_{\mathfrak{U}}g(x+X_{ess})
=g(x+X_{ess}),
$$
$$
(\tau\sigma)(h)(x+X_{ess})
=\lim_{\mathfrak{U}}\sigma(h)(x-e_\nu\cdot x)
=\lim_{\mathfrak{U}}(\sigma(h)(x)-h(e_\nu\cdot x))
=\sigma(h)(x)-\lim_{\mathfrak{U}}h(e_\nu\cdot x)=0,
$$
$$
(\pi^*\tau + \sigma\rho^*)(f)(x)
=\tau(f)(x+X_{ess})+\lim_{\mathfrak{U}}\rho^*(f)(e_\nu\cdot x)
=\lim_{\mathfrak{U}}f(x - e_\nu\cdot x)+\lim_{\mathfrak{U}}f(e_\nu\cdot x)
=f(x).
$$
Therefore $\tau \pi^*=1_{{(X/X_{ess})}^*}$, $\rho^*\sigma=1_{X_{ess}^*}$ and
$\pi^*\tau+\sigma\rho^*=1_{X^*}$. Now it is easy to check that the linear maps
$$
\xi
:X^*\to X_{ess}^*\bigoplus{}_\infty {(X/X_{ess})}^*
:f\mapsto \rho^*(f)\bigoplus{}_\infty \tau(f)
$$
$$
\eta
:X_{ess}^*\bigoplus{}_\infty {(X/X_{ess})}^*\to X^*
:h\bigoplus{}_\infty g\mapsto \pi^*(h)+\sigma(g)
$$
are isomorphism of right $A$-modules with $\Vert\xi \Vert\leq c_2$ 
and $\Vert \eta\Vert\leq c_1+1$. Hence $X^*$ is $c_2(c_1+1)$-isomorphic 
in $\mathbf{mod}-A$ to $X_{ess}^*\bigoplus_\infty {(X/X_{ess})}^*$.

$(ii)$ Both results immediately follow from equalities
$\rho^*\sigma=1_{X_{ess}^*}$, $\tau \pi^*=1_{{(X/X_{ess})}^*}$ and estimates
$\Vert \rho^*\Vert\Vert \sigma\Vert\leq c_1$, 
$\Vert\tau\Vert\Vert \pi^*\Vert\leq c_2$.

$(iii)$ Now consider case when $X$ is an $\mathscr{L}_{1,C}^g$-space. Then $X^*$
is an $\mathscr{L}_{\infty,C}^g$-space [\cite{DefFloTensNorOpId}, corollary
23.2.1(1)]. As $\langle$~$X_{ess}^*$ / ${(X/X_{ess})}^*$~$\rangle$ is
$\langle$~$c_1$-complemented / $c_2$-complemented~$\rangle$ in $X^*$ it is an
$\langle$~$\mathscr{L}_{\infty,c_1C}^g$-space /
$\mathscr{L}_{\infty,c_2C}^g$-space~$\rangle$ by [\cite{DefFloTensNorOpId},
corollary 23.2.1(1)]. Again we apply [\cite{DefFloTensNorOpId}, corollary
23.2.1(1)] to conclude that $\langle$~$X_{ess}$  / $X/X_{ess}$~$\rangle$ is an
$\langle$~$\mathscr{L}_{1,c_1C}^g$-space /
$\mathscr{L}_{1,c_2C}^g$-space~$\rangle$.
\end{proof}

The following proposition is an analog 
of [\cite{RamsHomPropSemgroupAlg}, proposition 2.1.11].

\begin{proposition}\label{TopFlatModCharac} Let $A$ be a Banach algebra with
two-sided $c$-bounded approximate identity, and $F$ be a left $A$-module. Then

\begin{enumerate}[label = (\roman*)]
    \item if $F$ is $C$-topologically flat $A$-module, then $F_{ess}$ is
    $(1+c)C$-topologically flat $A$-module and $F/F_{ess}$ is an
    $\mathscr{L}_{1,(1+c)C}^g$-space;

    \item if $F_{ess}$ is $C_1$-topologically flat $A$-module and $F/F_{ess}$ is
    an $\mathscr{L}_{1,C_1}^g$-space, then $F$ is ${(1+c)}^2\max(C_1,
    C_2)$-topologically flat $A$-module.

    \item $F$ is topologically flat $A$-module iff $F_{ess}$  is topologically
    flat $A$-module and $F/F_{ess}$ is an $\mathscr{L}_1^g$-space.
\end{enumerate}
\end{proposition}
\begin{proof} We regard $A$ as closed subalgebra of unital Banach algebra
$B:=A_+$. Then $F$ is unital left $B$-module. Using notation of
proposition~\ref{DualBanModDecomp} we may say that $c_1=c$ and $c_2=1+c$, so the
right $A$-modules $F_{ess}^*$ and ${(F/F_{ess})}^*$ are $(1+c)$-retracts of
$F^*$.

$i)$ By proposition~\ref{MetCTopFlatCharac} the right $A$-module $F^*$ is
$C$-topologically injective. Therefore from
propositions~\ref{RetrMetCTopInjIsMetCTopInj},~\ref{MetCTopFlatCharac} the modules
$F_{ess}$ and $F/F_{ess}$ are $(1+c)C$-topologically flat. It remains to note
that $F/F_{ess}$ is an annihilator $A$-module, so by
proposition~\ref{MetTopFlatAnnihModCharac} it is an
$\mathscr{L}_{1,(1+c)C}^g$-space.

$(ii)$ Again, by proposition~\ref{MetCTopFlatCharac} the right $A$-modules
$F_{ess}^*$ and ${(F/F_{ess})}^*$ are $C_1$- and $C_2$-topologically injective
respectively. So from proposition~\ref{MetTopInjModProd} their product is
$\max(C_1,C_2)$-topologically injective. By proposition~\ref{DualBanModDecomp}
this product is ${(1+c)}^2$-isomorphic to $F^*$ in $\mathbf{mod}-A$. Therefore
$F^*$ is ${(1+c)}^2\max(C_1, C_2)$-topologically injective $A$-module. Now the
result follows from proposition~\ref{MetCTopFlatCharac}.

$(iii)$ The result immediately follows from paragraphs $(i)$ and $(ii)$.
\end{proof}

\begin{proposition}\label{MetTopEssL1FlatModAoverAmenBanAlg} Let $A$ be a
$\langle$~$1$-relatively / $c$-relatively~$\rangle$ amenable Banach algebra and
$F$ be an essential Banach $A$-module which is an $\langle$~$L_1$-space /
$\mathscr{L}_{1,C}^g$-space~$\rangle$. Then $F$ is a $\langle$~metrically /
$c^2C$-topologically~$\rangle$ flat $A$-module.
\end{proposition}
\begin{proof} We may assume that $A$ is $c$-relatively amenable for
$\langle$~$c=1$ / $c\geq 1$~$\rangle$. Let ${(d_\nu)}_{\nu\in N}$ be an
approximate diagonal for $A$ with norm bound at most $c$. Recall, that
${(\Pi_A(d_\nu))}_{\nu\in N}$ is a two-sided $\langle$~contractive /
bounded~$\rangle$ approximate identity for $A$. Since $F$ is essential left
$A$-module, then $\lim_{\nu}\Pi_A(d_\nu)\cdot x=x$ for all $x\in F$
[\cite{HelHomolBanTopAlg}, proposition 0.3.15]. As the consequence
$c\pi_F(B_{A\projtens\ell_1(B_F)})$ is dense in $B_F$. Then for all $f\in F^*$
we have
$$
\Vert\pi_F^*(f)\Vert
=\sup \{|f(\pi_F(u))|:u\in B_{A\projtens\ell_1(B_F)} \}
=\sup \{|f(x)|:x\in \operatorname{cl}_F(\pi_F(B_{A\projtens\ell_1(B_F)})) \}
$$
$$
\geq\sup \{c^{-1}|f(x)|:x\in B_F \}=c^{-1}\Vert f\Vert.
$$
This means, that $\pi_F^*$ is $c$-topologically injective. By assumption $F$ is
an $\langle$~$L_1$-space / $\mathscr{L}_{1,C}^g$-space~$\rangle$, then by
$\langle$~[\cite{GrothMetrProjFlatBanSp}, theorem 1] / remark after
[\cite{DefFloTensNorOpId}, corollary 23.5(1)]~$\rangle$ the Banach space $F^*$
is $\langle$~metrically / $C$-topologically~$\rangle$ injective. Since operator
$\pi_F^*$ is $\langle$~isometric / $c$-topologically injective~$\rangle$, then
there exists a linear operator $R:{(A\projtens\ell_1(B_F))}^*\to F^*$ of norm
$\langle$~at most $1$ / at most $cC$~$\rangle$ such that $R\pi_F^*=1_{F^*}$.

Fix $h\in {(A\projtens\ell_1(B_F))}^*$ and $x\in F$. Consider bilinear
functional $M_{h,x}:A\times A\to\mathbb{C}:(a,b)\mapsto R(h\cdot a)(b\cdot x)$.
Clearly, $\Vert M_{h,x}\Vert\leq\Vert R\Vert\Vert h\Vert\Vert x\Vert$. By
universal property of projective tensor product we have a bounded linear
functional 
$m_{h,x}:A\projtens A\to\mathbb{C}:a\projtens b\mapsto R(h\cdot a)(b\cdot x)$. 
Note that $m_{h,x}$ is linear in $h$ and $x$. Even more, for any
$u\in A\projtens A$, $a\in A$ and $f\in F^*$ we have
$m_{\pi_F^*(f),x}(u)=f(\Pi_A(u)\cdot x)$, $m_{h\cdot a,x}(u)=m_{h,x}(a\cdot u)$,
$m_{h,a\cdot x}(u)=m_{h,x}(u\cdot a)$. It easily checked for elementary tensors.
Then it is enough to recall that their linear span is dense in $A\projtens A$.

Let $\mathfrak{F}$ be the section filter on $N$ and let $\mathfrak{U}$ be an
ultrafilter dominating $\mathfrak{F}$. For any 
$h\in {(A\projtens\ell_1(B_F))}^*$ and $x\in F$ we 
have $|m_{h,x}(d_\nu)|\leq c\Vert R\Vert\Vert h\Vert\Vert x\Vert$, 
i.e. ${(m_{h,x}(d_\nu))}_{\nu\in N}$ is a
bounded net of complex numbers. Therefore we have a well defined limit
$\lim_{\mathfrak{U}}m_{h,x}(d_\nu)$ along ultrafilter $\mathfrak{U}$. Consider
linear operator 
$\tau
:{(A\projtens\ell_1(B_F))}^*\to F^*
:h\mapsto(x\mapsto\lim_{\mathfrak{U}}m_{h,x}(d_\nu))$. From 
norm estimates for $m_{h,x}$ it follows that $\tau$ is bounded 
with $\Vert\tau\Vert\leq c\Vert R\Vert$. For all $a\in A$, $x\in F$ 
and $h\in {(A\projtens\ell_1(B_F))}^*$ we have
$$
\tau(h\cdot a)(x)-(\tau(h)\cdot a)(x)
=\tau(h\cdot a)(x)-\tau(h)(a\cdot x)
=\lim_{\mathfrak{U}}m_{h\cdot a,x}(d_\nu)
-\lim_{\mathfrak{U}}m_{h,a\cdot x}(d_\nu).
$$
$$
=\lim_{\mathfrak{U}}m_{h,x}(a\cdot d_\nu)-m_{h,x}(d_\nu\cdot a)
=m_{h,x}\left(\lim_{\mathfrak{U}}(a\cdot d_\nu-d_\nu\cdot a)\right)
$$
$$
=m_{h,x}\left(\lim_{\nu}(a\cdot d_\nu-d_\nu\cdot a)\right)
=m_{h,x}(0)
=0.
$$
Therefore $\tau$ is a morphism of right $A$-modules. Now for all $f\in F^*$ and
$x\in F$ we have
$$
(\tau(\pi_F^*)(f))(x)
=\lim_{\mathfrak{U}}m_{\pi_F^*(f),x}(d_\nu)
=\lim_{\mathfrak{U}}f(\Pi_A(d_\nu)\cdot x)
=\lim_{\nu}f(\Pi_A(d_\nu)\cdot x)
$$
$$
=f\left(\lim_{\nu}\Pi_A(d_\nu)\cdot x\right)
=f(x).
$$
So $\tau\pi_F^*=1_{F^*}$. This means that $F^*$ is a $\langle$~$1$-retract /
 $c^2 C$-retract~$\rangle$ of ${(A\projtens\ell_1(B_F))}^*$ in
 $\langle$~$\mathbf{mod}_1-A$ / $\mathbf{mod}-A$~$\rangle$. The latter
 $A$-module is $\langle$~metrically / topologically~$\rangle$ injective, because
 ${(A_+\projtens\ell_1(B_F))}^*
 \isom{\mathbf{mod}_1-A}
 \mathcal{B}(A_+,\ell_\infty(B_F))$
 and by proposition~\ref{RetrMetCTopInjIsMetCTopInj} so does its 
 retract $F^*$. By proposition~\ref{MetCTopFlatCharac} this is equivalent 
 to $\langle$~metric / $c^2 C$-topological~$\rangle$ flatness of $F$.
\end{proof}

\begin{theorem}\label{TopL1FlatModAoverAmenBanAlg} Let $A$ be a $c$-relatively
amenable Banach algebra and $F$ be a left Banach $A$-module which as Banach
space is an $\mathscr{L}_{1, C}^g$-space. Then $F$ is a
${(1+c)}^2C\max(c^2,(1+c))$-topologically flat $A$-module.
\end{theorem}
\begin{proof} Since $A$ is amenable it admits a two-sided $c$-bounded
approximate identity. By proposition~\ref{DualBanModDecomp} the annihilator
$A$-module $F/F_{ess}$ is an $\mathscr{L}_{1,1+c}^g$-space. From
proposition~\ref{MetTopEssL1FlatModAoverAmenBanAlg} we get that the essential
$A$-module $F_{ess}$ is $c^2 C$-topologically flat. Now the result follows from
proposition~\ref{TopFlatModCharac}.
\end{proof}

We must point out here that in relative Banach homology any left Banach module
over relatively amenable Banach algebra is relatively flat
[\cite{HelBanLocConvAlg}, theorem 7.1.60]. Even topological theory is so
restrictive that in some cases, as the following proposition shows, we can
obtain complete characterization of all flat modules.

\begin{proposition}\label{TopFlatModAoverAmenL1BanAlgCharac} Let $A$ be a
relatively amenable Banach algebra which as Banach space is an
$\mathscr{L}_1^g$-space. Then for a Banach $A$-module $F$ the following are
equivalent:

\begin{enumerate}[label = (\roman*)]
    \item $F$ is topologically flat $A$-module; 

    \item $F$ is an $\mathscr{L}_1^g$-space.
\end{enumerate}
\end{proposition}
\begin{proof} The equivalence follows from
proposition~\ref{TopProjInjFlatModOverMthscrL1SpCharac} and
theorem~\ref{TopL1FlatModAoverAmenBanAlg}.
\end{proof}

Finally we are able to give an example of relatively flat, but not topologically
flat ideal in a Banach algebra. Consider $A=L_1(\mathbb{T})$. It is known, that
$A$ has a translation invariant infinite dimensional closed subspace $I$
isomorphic to a Hilbert space [\cite{RosProjTransInvSbspLpG}, p.52]. By
[\cite{KaniBanAlg}, proposition 1.4.7] we have that $I$ is a two-sided ideal of
$A$, as any translation invariant subspace of $A$. By [\cite{DefFloTensNorOpId},
section 23.3] this ideal is not an $\mathscr{L}_1^g$-space. So from
proposition~\ref{TopFlatModAoverAmenL1BanAlgCharac} we get that $I$ is not
topologically flat as $A$-module. We claim it is still relatively flat. Since
$\mathbb{T}$ is a compact group, then it is amenable [\cite{PierAmenLCA},
proposition 3.12.1]. Thus $A$ is relatively amenable [\cite{HelBanLocConvAlg},
proposition VII.1.86], so all left ideals of $A$ are relatively flat
[\cite{HelBanLocConvAlg}, proposition VII.1.60(I)]. In particular, $I$ is
relatively flat.

%-------------------------------------------------------------------------------
%    Injectivity of ideals
%-------------------------------------------------------------------------------

\subsection{
    Injectivity of ideals
}\label{SubSectionInjectivityOfIdeals}

Injective ideals are rare creatures but we need to say a few words about them.
Results of this section needed for the study of metric and topological
injectivity of $C^*$-algebras.

\begin{proposition}\label{MetTopInjOfId} Let $I$ be a right ideal of a Banach
algebra $A$. Assume $I$ is $\langle$~metrically / $C$-topologically~$\rangle$
injective $A$-module. Then $I$ has a left identity of norm $\langle$~at most $1$
/ at most $C$~$\rangle$ and is a $\langle$~$1$-retract / $C$-retract~$\rangle$
of $A$ in $\mathbf{mod}-A$.
\end{proposition}
\begin{proof} Consider isometric embedding $\rho^+:I\to A_+$ of $I$ into $A_+$.
Clearly, this is an $A$-morphism. Since $I$ is $\langle$~metrically /
$C$-topologically~$\rangle$ injective, then $\rho^+$ has left inverse
$A$-morphism $\tau^+:A_+\to I$ with norm $\langle$~at most $1$ / at most
$C$~$\rangle$. Now for all $x\in I$ we have
$x
=\tau^+(\rho^+(x))
=\tau^+(e_{A_+}\rho^+(x))
=\tau(e_{A_+})\rho^+(x)=\tau^+(e_{A_+})x$.
In other words $p=\tau^+(e_{A_+})\in I$ is a left unit for $I$. Clearly, 
$\Vert p\Vert
\leq\Vert\tau^+\Vert\Vert e_{A_+}\Vert
\leq\Vert\tau^+\Vert$. Consider maps
$\rho:I\to A:x\mapsto x$ and $\tau:A\to I:x\mapsto p x$. Clearly, they are
morphisms of right $A$-modules and $\tau\rho=1_I$. Hence $I$ is a
$\langle$~$1$-retract / $C$-retract~$\rangle$ of $A$ in $\mathbf{mod}-A$.
\end{proof}

\begin{proposition}\label{ReduceInjIdToInjAlg} Let $I$ be a two-sided ideal of
a Banach algebra $A$, which is faithful as right $I$-module. Then

\begin{enumerate}[label = (\roman*)]
    \item if $I$ is $\langle$~metrically / $C$-topologically~$\rangle$ injective
    $I$-module, then $I$ is $\langle$~metrically / $C$-topologically~$\rangle$
    injective $A$-module; 

    \item if $I$ is $\langle$~metrically / $C$-topologically~$\rangle$ injective
    $A$-module, then $I$ is $\langle$~metrically / $C^2$-topologically~$\rangle$
    injective $I$-module.
\end{enumerate}
\end{proposition}
\begin{proof} $(i)$ The result immediately follows from paragraph $(i)$ of
proposition~\ref{MetTopInjUnderChangeOfAlg}.

$(ii)$ By proposition~\ref{MetTopInjOfId} the $A$-module $I$ is
$\langle$~$1$-complemented / $C$-complemented~$\rangle$ in $A$. By paragraph
$(ii)$ of proposition~\ref{MetTopInjUnderChangeOfAlg} the $I$-module $I$ is
$\langle$~metrically / $C^2$-topologically~$\rangle$ injective.
\end{proof}

%% file: chapters/chapter3_applications_to_algebras_of_analysis.tex
% chktex-file 35
% Chapter Template
% Main chapter title
% Change X to a consecutive number; for referencing this 
% chapter elsewhere, use~\ref{ChapterX}
\chapter{
    Applications to algebras of analysis
}\label{ChapterApplicationsToAlgebrasOfAnalysis} 

% Change X to a consecutive number; this is for the header on 
% each page - perhaps a shortened title
\lhead{Chapter 3. \emph{Applications to algebras of analysis}} 

Vaguely speaking there are three types of Banach modules depending on the type
of module action: modules with pointwise multiplication, modules with
composition of operators in the role of module action and modules with
convolution. We shall investigate main examples of these types. Following the
style of Dales and Polyakov from~\cite{DalPolHomolPropGrAlg} we shall
systematize all results on classical modules of analysis, but this time for
metric and topological theory. We shall consider modules over operator algebras,
sequence algebras, algebras of continuous functions and, finally, classical
modules of harmonic analysis.

%-------------------------------------------------------------------------------
%	Applications to operator algebras
%-------------------------------------------------------------------------------

\section{
    Applications to modules over \texorpdfstring{$C^*$}{C*}-algebras
}\label{SectionApplicationsToModulesOverCStarAlgebras}

%-------------------------------------------------------------------------------
%	Spatial modules
%-------------------------------------------------------------------------------

\subsection{
    Spatial modules
}\label{SubSectionSpatialModules}

We start from the simplest  examples of modules over operator algebras --- the
spatial modules. By Gelfand-Naimark's theorem (see e.g.
[\cite{HelBanLocConvAlg}, theorem 4.7.57]) for any $C^*$-algebra $A$ there
exists a Hilbert space $H$ and an isometric ${}^*$-homomorphism
$\varrho:A\to\mathcal{B}(H)$. For Hilbert spaces that admit such homomorphism we
may consider the left $A$-module $H_\varrho$ with module action defined as
$a\cdot x=\varrho(a)(x)$. Automatically we get the structure of right $A$-module
on $H^*$ which is by Riesz's theorem is isometrically isomorphic to $H^{cc}$.
This isomorphism allows one to define the right $A$-module structure on $H^{cc}$
by $\overline{x}\cdot a=\overline{\varrho(a^*)(x)}$. For a given $x_1,x_2\in H$
we define the rank one 
operator $x_1\bigcirc x_2:H\to H:x\mapsto \langle x, x_2\rangle x_1$. 

\begin{proposition}\label{SpatModOverCStarAlgProp} Let $A$ be a $C^*$-algebra
and $\varrho:A\to\mathcal{B}(H)$ be an isometric ${}^*$-homomorphism, such that
its image contains a subspace of rank one operators of the 
form $ \{x\bigcirc x_0:x\in H \}$ for some non zero $x_0\in H$. 
Then the left $A$-module $H_\varrho$ is metrically projective and flat, 
while the  right $A$-module $H_\varrho^{cc}$ is metrically injective.
\end{proposition}
\begin{proof} Without loss of generality we may assume that $\Vert x_0\Vert=1$.
Consider linear operators 
$\pi:A_+\to H_\varrho:a\oplus_1 z\mapsto \varrho(a)(x_0)+zx_0$ 
and $\sigma:H_\varrho\to A_+:x\mapsto \varrho^{-1}(x\bigcirc x_0)$. 
It is straightforward to check that $\pi$ and
$\sigma$ are contractive $A$-morphisms such that $\pi\sigma=1_{H_\varrho}$.
Therefore $H_\varrho$ is a retract of $A_+$ in $A-\mathbf{mod}_1$. From
propositions~\ref{UnitalAlgIsMetTopProj} and~\ref{RetrMetCTopProjIsMetCTopProj} 
it follows that $H_\varrho$ is metrically projective $A$-module. From
proposition~\ref{MetTopProjIsMetTopFlat} it follows that $H_\varrho$ is
metrically flat too. Since $H_\varrho^{cc}\isom{\mathbf{mod}_1-A}H_\varrho^*$,
proposition~\ref{DualMetTopProjIsMetrInj} gives that $H_\varrho^{cc}$ is
metrically injective.
\end{proof}

In what follows we shall use the following simple application of the above
result.

\begin{proposition}\label{FinDimNHModTopProjFlat} Let $H$ be a finite
dimensional Hilbert space. Then $\mathcal{N}(H)$ is topologically projective and
hence flat as $\mathcal{B}(H)$-module.
\end{proposition}
\begin{proof} From [\cite{HelBanLocConvAlg}, proposition 0.3.38] we know that
$\mathcal{N}(H)\isom{\mathbf{Ban}_1}H\projtens H^*$. Let
$\varrho=1_{\mathcal{B}(H)}$, then we can claim a little bit more:
$\mathcal{N}(H)\isom{\mathcal{B}(H)-\mathbf{mod}_1} H_\varrho\projtens H^*$.
Since $H^*$ is finite dimensional, then
$H^*\isom{\mathbf{Ban}}\ell_1(\mathbb{N}_n)$ for $n=\dim(H)$ and as the result
$\mathcal{N}(H)\isom{\mathcal{B}(H)-\mathbf{mod}}
H_\varrho\projtens\ell_1(\mathbb{N}_n)$. By
proposition~\ref{SpatModOverCStarAlgProp} the module $H_\varrho$ it
topologically projective, so from corollary~\ref{MetTopProjTensProdWithl1} we
get that $\mathcal{N}(H)$ is topologically projective as
$\mathcal{B}(H)$-module. The last claim of theorem follows from
proposition~\ref{MetTopProjIsMetTopFlat}.
\end{proof}

%-------------------------------------------------------------------------------
%	Projective ideals of C^*-algebras
%-------------------------------------------------------------------------------

\subsection{
    Projective ideals of \texorpdfstring{$C^*$}{C*}-algebras
}\label{SubSectionProjectiveIdealsOfCStarAlgebras}

The study of homologically trivial ideals of $C^*$-algebras we start from
projectivity, but before stating the main result we need a preparatory lemma.

\begin{lemma}\label{ContFuncCalcOnIdealOfCStarAlg} Let $I$ be a left ideal of a
unital $C^*$-algebra $A$. Assume $a\in I$ is a self-adjoint element and let $E$
be the real subspace of real valued functions in $C(\operatorname{sp}_A(a))$
vanishing at zero. Then there is an isometric homomorphism
$\operatorname{RCont}_a^0:E\to I$ well defined by
$\operatorname{RCont}_a^0(f)=a$, where
$f:\operatorname{sp}_A(a)\to\mathbb{C}:t\mapsto t$.
\end{lemma}
\begin{proof} By $\mathbb{R}_0[t]$ we denote the real linear subspace of $E$
consisting of polynomials vanishing at zero. Since $I$ is an ideal of $A$ and
and $p\in\mathbb{R}_0[t]$ has no constant term then $p(a)\in I$.  Hence we have
well defined $\mathbb{R}$-linear homomorphism of algebras
$\operatorname{RPol}_a^0:\mathbb{R}_0[t]\to I:p\mapsto p(a)$. By continuous
functional calculus for any polynomial $p$ 
holds $\Vert p(a)\Vert=\Vert p|_{\operatorname{sp}_A(a)}\Vert_\infty$, so
$\Vert\operatorname{RPol}_a^0(p)\Vert
=\Vert p|_{\operatorname{sp}_A(a)}\Vert_\infty$. 
Thus $\operatorname{RPol}_a^0$ is isometric. 
As $\mathbb{R}_0[t]$ is dense in $E$ and $I$ is complete, then
$\operatorname{RPol}_a^0$ has an isometric extension
$\operatorname{RCont}_a^0:E\to I$ which is also an $\mathbb{R}$-linear
homomorphism. 
\end{proof}

The following proof is inspired by ideas of D. P. Blecher and T. Kania. In
[\cite{BleKanFinGenCStarAlgHilbMod}, lemma 2.1] they proved that any
algebraically finitely generated left ideal of $C^*$-algebras is principal.  

\begin{theorem}\label{LeftIdealOfCStarAlgMetTopProjCharac} Let $I$ be a left
ideal of a $C^*$-algebra $A$. Then the following are equivalent:

\begin{enumerate}[label = (\roman*)]
    \item $I=Ap$ for some self-adjoint idempotent $p\in I$;

    \item $I$ is metrically projective $A$-module;

    \item $I$ is topologically projective $A$-module.
\end{enumerate}
\end{theorem}
\begin{proof} $(i)\implies (ii)$ Since $p$ is a self-adjoint idempotent, then
$\Vert p\Vert=1$, so by proposition~\ref{UnIdeallIsMetTopProj} paragraph $(i)$
the ideal $I$ is metrically projective as $A$-module.

$(ii)\implies (iii)$ See
proposition~\ref{MetProjIsTopProjAndTopProjIsRelProj}.

$(iii) \implies (i)$ Let ${(e_\nu)}_{\nu\in N}$ be a right contractive
approximate identity of ideal $I$ [\cite{HelBanLocConvAlg}, theorem 4.7.79].
Since $I$ admits a right approximate identity, then it is an essential left
$I$-module, and a fortiori an essential $A$-module. By
proposition~\ref{NonDegenMetTopProjCharac} we have a right inverse $A$-morphism
$\sigma:I\to A\projtens \ell_1(B_I)$ of $\pi_I$ in $A-\mathbf{mod}$. For each
$d\in B_I$ consider $A$-morphisms $p_d:A\projtens \ell_1(B_I)\to A:a\projtens
\delta_x\mapsto \delta_x(d)a$ and $\sigma_d=p_d\sigma$. Then
$\sigma(x)=\sum_{d\in B_I}\sigma_d(x)\projtens \delta_d$ for all $x\in I$. From
identification 
$A\projtens\ell_1(B_I)\isom{\mathbf{Ban}_1}\bigoplus_1 \{ A:d\in B_I \}$, 
for all $x\in I$ we 
have $\Vert\sigma(x)\Vert=\sum_{d\in B_I} \Vert\sigma_d(x)\Vert$. 
Since $\sigma$ is a right inverse morphism of $\pi_I$ we
have $x=\pi_I(\sigma(x))=\sum_{d\in B_I}\sigma_d(x)d$ for all $x\in I$. 

For all $x\in I$ we have 
$\Vert\sigma_d(x)\Vert
=\Vert\sigma_d(\lim_\nu xe_\nu)\Vert
=\lim_\nu\Vert x\sigma_d(e_\nu)\Vert 
\leq\Vert x\Vert\liminf_\nu\Vert\sigma_d(e_\nu)\Vert$, 
so $\Vert\sigma_d\Vert\leq \liminf_\nu\Vert\sigma_d(e_\nu)\Vert$. 
Then for all $S\in\mathcal{P}_0(B_I)$ holds
$$
\sum_{d\in S}\Vert \sigma_d\Vert
\leq \sum_{d\in S}\liminf_\nu\Vert \sigma_d(e_\nu)\Vert
\leq \liminf_\nu\sum_{d\in S}\Vert \sigma_d(e_\nu)\Vert
\leq \liminf_\nu\sum_{d\in B_I}\Vert \sigma_d(e_\nu) \Vert
$$
$$
=\liminf_{\nu}\Vert\sigma(e_\nu)\Vert
\leq \Vert\sigma\Vert\liminf_{\nu}\Vert e_\nu\Vert
\leq \Vert\sigma\Vert
$$
Since $S\in \mathcal{P}_0(B_I)$ is arbitrary, then the sum 
$\sum_{d\in B_I}\Vert\sigma_d\Vert$ is finite. As the consequence, the 
sum $\sum_{d\in B_I}\Vert\sigma_d\Vert^2$ is finite too. 

Now we regard $A$ as an ideal in its unitization $A_\#$, then $I$ is an ideal of
$A_\#$ too. Fix a natural number $m\in\mathbb{N}$ and a real number
$\epsilon>0$. Then there exists a set $S\in\mathcal{P}_0(B_I)$ such that
$\sum_{d\in B_I\setminus S}\Vert\sigma_d\Vert<\epsilon$. Denote its cardinality
by $N$. Consider positive element 
$b=\sum_{d\in B_I}\Vert\sigma_d\Vert^2 d^*d\in I$. Now we perform 
a ``power trick'' by considering different powers $b^{1/m}$ of positive element 
$b$, where $m\in\mathbb{N}$. By lemma~\ref{ContFuncCalcOnIdealOfCStarAlg} we 
have that $b^{1/m}\in I$, so $b^{1/m}=\sum_{d\in B_I}\sigma_d(b^{1/m})d$. 
By continuous functional calculus we have 
$\Vert b^{1/m}\Vert
=\sup_{t\in\operatorname{sp}_{A_\#}(b)} t^{1/m}\leq\Vert b\Vert^{1/m}$, 
then $\limsup_{m\to\infty}\Vert b^{1/m}\Vert\leq 1$. 
Therefore $\Vert b^{1/m}\Vert\leq 2$ for sufficiently big $m$. Denote
$\varsigma_d=\sigma_d(b^{1/m})$, $u=\sum_{d\in S}\varsigma_d d$ and
$v=\sum_{d\in B_I\setminus S}\varsigma_d d$, so 
$$
b^{2/m}={(b^{1/m})}^*b^{1/m}=u^*u+u^*v+v^*u+v^*v
$$
Clearly, $\varsigma_d^*\varsigma_d\leq \Vert \varsigma_d\Vert^2 e_{A_\#}\leq
\Vert \sigma_d\Vert^2\Vert b^{1/m}\Vert^2 e_{A_\#}\leq 4\Vert \sigma_d\Vert^2
e_{A_\#}$. For any $x,y\in A$ we have $x^*x+y^*y-(x^*y+y^*x)={(x-y)}^*(x-y)\geq
0$, therefore 
$$
d^*\varsigma_d^* \varsigma_c c+c^*\varsigma_c^* \varsigma_d d
\leq d^*\varsigma_d^*\varsigma_d d + c^*\varsigma_c^*\varsigma_c c
\leq 4\Vert \sigma_d\Vert^2 d^*d+4\Vert \sigma_c\Vert^2 c^*c
$$
for all $c,d\in B_I$. We sum up these inequalities over $c\in S$ and $d\in S$,
then 
$$
\begin{aligned}
\sum_{c\in S}\sum_{d\in S}c^*\varsigma_c^* \varsigma_d d
&=\frac{1}{2}\left(
    \sum_{c\in S}\sum_{d\in S}d^*\varsigma_d^* \varsigma_c c
    +
    \sum_{c\in S}\sum_{d\in S}c^*\varsigma_c^* \varsigma_d d
\right)\\
&\leq\frac{1}{2}\left(4 N\sum_{d\in S} \Vert \sigma_d\Vert^2 d^*d+
4 N\sum_{c\in S} \Vert \sigma_c\Vert^2 c^*c\right)\\
&=4 N\sum_{d\in S} \Vert \sigma_d\Vert^2 d^*d.
\end{aligned}
$$
Therefore
$$
u^*u
={\left(\sum_{c\in S}\varsigma_c c\right)}^*
\left(\sum_{d\in S}\varsigma_d d\right)
=\sum_{c\in S}\sum_{d\in S}c^*\varsigma_c^* \varsigma_d d
\leq 4N\sum_{d\in S} \Vert \sigma_d\Vert^2 d^*d
= 4N b
$$
Note that
$$
\Vert u\Vert
\leq \sum_{d\in S}\Vert\varsigma_d\Vert\Vert d\Vert
\leq \sum_{d\in S}2\Vert\sigma_d\Vert
\leq 2\Vert\sigma\Vert,
\qquad
\Vert v\Vert
\leq \sum_{d\in B_I\setminus S}\Vert\varsigma_d\Vert\Vert d\Vert
\leq \sum_{d\in B_I\setminus S}2\Vert\sigma_d\Vert
\leq 2\epsilon,
$$
so $\Vert u^*v+v^*u\Vert\leq 8\Vert\sigma\Vert\epsilon$ 
and $\Vert v^*v\Vert\leq 4\epsilon^2$. Since $u^*v+v^*u$ and $v^*v$ are self 
adjoint, then $u^*v+v^*u\leq 8\Vert\sigma\Vert\epsilon e_{A_\#}$ 
and $v^*v\leq 4\epsilon^2 e_{A_\#}$ Therefore for any $\epsilon>0$ and 
sufficiently big $m$ we have 
$$
b^{2/m}
=u^*u+u^*v+v^*u+v^*v
\leq 4Nb+\epsilon(8\Vert\sigma\Vert+4\epsilon)e_{A_\#}.
$$

In other words $g_m(b)\geq 0$ for continuous function
$g_m:\mathbb{R}_+\to\mathbb{R}:t\mapsto
4Nt+\epsilon(8\Vert\sigma\Vert+4\epsilon)-t^{2/m}$. Now choose $\epsilon>0$ such
that $M:=\epsilon(8\Vert\sigma\Vert+4\epsilon)<1$.  By spectral mapping theorem
[\cite{HelLectAndExOnFuncAn}, theorem 6.4.2] we get
$g_m(\operatorname{sp}_{A_\#}(b))
=\operatorname{sp}_{A_\#}(g_m(b))\subset\mathbb{R}_+$.
It is routine to check that $g_m$ has the only one extreme point
$t_{0,m}={(2Nm)}^{\frac{m}{2-m}}$ where the minimum of $g_m$ is attained. Since
$\lim_{m\to\infty} g_m(t_{0,m})=M-1<0$, $g_m(0)=M>0$ and $\lim_{t\to\infty}
g_m(t)=+\infty$, then for sufficiently big $m$ the function $g_m$ has exactly
two roots: $0<t_{1,m}<t_{0,m}$ and $t_{2,m} > t_{0,m}$. Therefore $g_m(t)\geq 0$
for $0\leq t\leq t_{1,m}$ or $t\geq t_{2,m}$. Hence for all sufficiently large
$m$ holds $\operatorname{sp}_{A_\#}(b)\subset T_{t_{1,m},t_{2,m}}$, 
where $T_{a, b}:=\{t\in\mathbb{R} : 0\leq t \leq a \vee t \geq b\}$. 
Since $\lim_{m\to\infty} t_{0,m}=0$ then $\lim_{m\to\infty} t_{1,m}=0$ too. 
Note that $g_m(1)=4N+M-1>0$, so for sufficiently big $m$ we also 
have $t_{2,m}\leq 1$. Consequently,
$\operatorname{sp}_{A_\#}(b)\subset T_{0,1}$.

Consider a continuous function $h:\mathbb{R}_+\to\mathbb{R}:t\mapsto\min(1, t)$,
uthen from lemma~\ref{ContFuncCalcOnIdealOfCStarAlg} we get an idempotent
$p=h(b)=\operatorname{RCont}_b^0(h)\in I$ such that $\Vert
p\Vert=\sup_{t\in\operatorname{sp}_{A_\#}(b)}|h(t)|\leq 1$. Therefore $p$ is a
self-adjoint idempotent. Since $h(t)t=th(t)=t$ for all $t\in
\operatorname{sp}_{A_\#}(b)$ we have $bp=pb=b$. The last equality implies
$$
0=(e_{A_\#}-p)b(e_{A_\#}-p)
=\sum_{d\in B_I}
    {(\Vert\sigma_d\Vert d(e_{A_\#}-p))}^*\Vert\sigma_d\Vert d(e_{A_\#}-p).
$$
Since the right hand side of this equality is non negative, then $d=dp$ for all
$d\in B_I$ with $\sigma_d\neq 0$. Finally, for any $x\in I$ we obtain
$xp=\sum_{d\in B_I}\sigma_d(x)dp=\sum_{d\in B_I}\sigma_d(x)d=x$, i.e. $I=Ap$,
for self-adjoint idempotent $p\in I$.
\end{proof}

It is worth to point out here that in relative theory there no such description
for relative projectivity of left ideals of $C^*$-algebras. Though for the case
of separable $C^*$-algebras (that is $C^*$-algebras that are separable as Banach
spaces) all left ideals are relatively projective. See
[\cite{LykProjOfBanAndCStarAlgsOfContFld}, section 6] for a nice overview of the
current state of the problem.

\begin{corollary}\label{BiIdealOfCStarAlgMetTopProjCharac} Let $I$ be a
two-sided ideal of a $C^*$-algebra $A$. Then the following are equivalent:

\begin{enumerate}[label = (\roman*)]
    \item $I$ is unital;

    \item $I$ is metrically projective $A$-module;

    \item $I$ is topologically projective $A$-module.
\end{enumerate}
\end{corollary}
\begin{proof} The ideal $I$ has a contractive approximate identity
[\cite{HelBanLocConvAlg}, theorem 4.7.79]. Therefore $I$ has a right identity
iff $I$ is unital. Now all equivalences follow from
theorem~\ref{LeftIdealOfCStarAlgMetTopProjCharac}. 
\end{proof}

\begin{corollary}\label{IdealofCommCStarAlgMetTopProjCharac} Let $S$ be a
locally compact Hausdorff space, and $I$ be an ideal of $C_0(S)$. Then the
following are equivalent:

\begin{enumerate}[label = (\roman*)]
    \item $\operatorname{Spec}(I)$ is compact;

    \item $I$ is metrically projective $C_0(S)$-module;

    \item $I$ is topologically projective $C_0(S)$-module.
\end{enumerate}
\end{corollary}
\begin{proof} By Gelfand-Naimark's theorem
$I\isom{\mathbf{Ban}_1}C_0(\operatorname{Spec}(I))$, therefore $I$ is
semisimple. Now by Shilov's idempotent theorem $I$ is unital iff
$\operatorname{Spec}(I)$ is compact. It remains to apply
corollary~\ref{BiIdealOfCStarAlgMetTopProjCharac}. 
\end{proof}

It is worth to mention that the class of relatively projective ideals of
$C_0(S)$ is much larger. In fact a closed ideal $I$ of $C_0(S)$ is relatively
projective iff $\operatorname{Spec}(I)$ is paracompact
[\cite{HelHomolBanTopAlg}, chapter IV,\S\S 2--3].

%-------------------------------------------------------------------------------
%	Injective ideals of C^*-algebras
%-------------------------------------------------------------------------------

\subsection{
    Injective ideals of \texorpdfstring{$C^*$}{C*}-algebras
}\label{SubSectionInjectiveIdealsOfCStarAlgebras}

Now we proceed to injectivity of two-sided ideals of $C^*$-algebras.
Unfortunately we don't have a complete characterization at hand, but some
necessary conditions and several examples. The following proposition shows that
we may restrict investigation of injective ideals to the case of $C^*$-algebras
that $\langle$~metrically / topologically~$\rangle$ injective over themselves as
right modules.

\begin{proposition}\label{MetTopInjOvrAlgIffOvrItslf} Let $I$ be a two-sided
ideal of a $C^*$-algebra $A$, then $I$ is $\langle$~metrically /
topologically~$\rangle$ injective as $A$-module iff $I$ is $\langle$~metrically
/ topologically~$\rangle$ injective as $I$-module.
\end{proposition}
\begin{proof} Note that any two-sided ideal $I$ of a $C^*$-algebra $A$ is again
a $C^*$-algebra with contractive approximate identity [\cite{HelBanLocConvAlg},
theorem 4.7.79]. Therefore $I$ is faithful as right $I$-module. Now
proposition~\ref{ReduceInjIdToInjAlg} gives the desired equivalence.
\end{proof}

We shall say a few words on so called $AW^*$-algebras, since they are key
players here. In attempts to find a purely algebraic definition of
$W^*$-algebras Kaplanski introduced this class of $C^*$-algebras
in~\cite{KaplProjInBanAlg}. A $C^*$-algebra $A$ is called an $AW^*$-algebra if
for any subset $S\subset A$ the right algebraic annihilator
$S^{\perp A}$ is of the form $pA$ for
some self-adjoint idempotent $p\in A$. This class contains all $W^*$-algebras,
but strictly larger. Note that for the case of commutative $C^*$-algebras the
property of being an $AW^*$-algebra is equivalent to  $\operatorname{Spec}(A)$
being a Stonean space [\cite{BerbBaerStarRings}, theorem 1.7.1]. The main
reference to $AW^*$-algebras and more general Baer ${}^*$-rings
is~\cite{BerbBaerStarRings}. 

The following proposition is a combination of results by M. Hamana and M.
Takesaki.

\begin{proposition}[Hamana, Takesaki]\label{MetInjCStarAlgCharac} Let $A$ be a
$C^*$-algebra, then it is metrically injective right $A$-module iff it is a
commutative $AW^*$-algebra, that is $\operatorname{Spec}(A)$ is a Stonean space.
\end{proposition}
\begin{proof} 

If $A$ is metrically injective as $A$-module, then it has norm one left identity
by proposition~\ref{MetTopInjOfId}. But $A$ also has a contractive approximate
identity  [\cite{HelBanLocConvAlg}, theorem 4.7.79], therefore $A$ is unital.
Now by result of Hamana  [\cite{HamInjEnvBanMod}, proposition 2] the
$C^*$-algebra $A$ is a commutative $AW^*$-algebra. Hamana's argument was for
left modules, but one can easily modify his proof to get the result for right
modules.

The converse proved by Takesaki in [\cite{TakHanBanThAndJordDecomOfModMap},
theorem 2]. Although only two-sided Banach modules were treated there, the
reasoning is essentially the same for right modules.
\end{proof}

It remains to consider topological injectivity. As the following proposition
shows topologically injective $C^*$-algebras are not so far from commutative
ones. This proposition exploits the l.u.st.\ property, for its definition see
section~\ref{
    SubSectionHomologicallyTrivialModulesOverBanachAlgebrasWithSpecificGeometry
}.

\begin{proposition}\label{TopInjIdHaveLUST} Let $A$ be a $C^*$-algebra which is
topologically injective as an $A$-module. Then $A$ has the l.u.st.\ property and
as the consequence it can't contain  $\mathcal{B}(\ell_2(\mathbb{N}_n))$ as
${}^*$-subalgebra for arbitrary $n\in\mathbb{N}$.
\end{proposition}
\begin{proof} By Gelfand-Naimark's theorem [\cite{HelBanLocConvAlg}, theorem
4.7.57] there exists a Hilbert space $H$ and an isometric ${}^*$-homomorphism
$\varrho:A\to\mathcal{B}(H)$. Denote $\Lambda:=B_{H_\varrho^{cc}}$. It is easy
to check that 
$$
\rho
:A
    \to
\bigoplus\nolimits_\infty \{H_\varrho^{cc}:\overline{x}\in \Lambda \}
:a
    \mapsto 
\bigoplus\nolimits_\infty \{\overline{x}\cdot a:\overline{x}\in \Lambda \}
$$
is an isometric $A$-morphism of right $A$-modules. Since $A$ is topologically
injective $A$-module, then $\rho$ has a left inverse $A$-morphism $\tau$.
Therefore $A$ is complemented in $E:=\bigoplus_\infty
\{H_\varrho^{cc}:\overline{x}\in \Lambda \}$ via projection $\rho\tau$. Note
that $H_{\varrho}^{cc}$ is a Banach lattice as any Hilbert space, then so does
$E$. As any Banach lattice $E$ has the l.u.st.\ property [\cite{DiestAbsSumOps},
theorem 17.1], then so does $A$, because the l.u.st.\ property is inherited by
complemented subspaces.

Assume $A$ contains $\mathcal{B}(\ell_2(\mathbb{N}_n))$ as ${}^*$-subalgebra for
arbitrarily large $n\in\mathbb{N}$. In fact this copy of
$\mathcal{B}(\ell_2(\mathbb{N}_n))$ is $1$-complemented in $A$
[\cite{LauLoyWillisAmnblOfBanAndCStarAlgsOfLCG}, lemma 2.1]. Therefore we have
an inequality for local unconditional constants
$\kappa_u(\mathcal{B}(\ell_2(\mathbb{N}_n)))\leq \kappa_u(A)$. By theorem 5.1
in~\cite{GorLewAbsSmOpAndLocUncondStrct} we know that $\lim_n
\kappa_u(\mathcal{B}(\ell_2(\mathbb{N}_n)))=+\infty$, so $\kappa_u(A)=+\infty$.
This contradicts the l.u.st.\ property of $A$. Hence $A$ can't contain
$\mathcal{B}(\ell_2(\mathbb{N}_n))$ as ${}^*$-subalgebra for arbitrary
$n\in\mathbb{N}$.
\end{proof}

As the proposition~\ref{TopInjIdHaveLUST} shows $C^*$-algebras that are
topologically injective over themselves can't contain
$\mathcal{B}(\ell_2(\mathbb{N}_n))$ as ${}^*$-subalgebra for arbitrary
$n\in\mathbb{N}$. Such $C^*$-algebras are called subhomogeneous, and in fact
[\cite{BlackadarOpAlg}, proposition IV.1.4.3] they can be treated as closed
${}^*$-subalgebras of $M_n(C(K))$ for some compact Hausdorff space $K$ and some
natural number $n$. For more on subhomogeneous $C^*$-algebras see
[\cite{BlackadarOpAlg}, section IV.1.4]. 

We shall give two important examples of non commutative $C^*$-algebras that are
topologically injective over themselves.

\begin{proposition}\label{FinDimBHModTopInj} Let $H$ be a finite dimensional
Hilbert space. Then $\mathcal{B}(H)$ is topologically injective as
$\mathcal{B}(H)$-module. 
\end{proposition}
\begin{proof} Note that
$\mathcal{B}(H)\isom{\mathbf{mod}_1-\mathcal{B}(H)}{\mathcal{N}(H)}^*$, and the
the result immediately follows from propositions~\ref{FinDimNHModTopProjFlat}
and~\ref{DualMetTopProjIsMetrInj}.
\end{proof}

\begin{proposition}\label{CKMatrixModTopInj} Let $K$ be a Stonean space and
$n\in\mathbb{N}$, then $M_n(C(K))$ is topologically injective
$M_n(C(K))$-module.
\end{proposition}
\begin{proof} For a fixed $s\in K$ by $\mathbb{C}_s$ we denote the right
$C(K)$-module $\mathbb{C}$ with outer action defined by $z\cdot a=a(s)z$ for all
$a\in C(K)$ and $z\in\mathbb{C}$. By $M_n(\mathbb{C}_s)$ we denote the right
Banach $M_n(C(K))$-module $M_n(\mathbb{C})$ with outer action defined by
${(x\cdot a)}_{i,j}=\sum_{k=1}^n x_{i,k}a_{k,j}(s)$ for $a\in M_n(C(K))$ and 
$x\in M_n(\mathbb{C}_s)$. The $C^*$-algebra $M_n(C(K))$ is nuclear
[\cite{BroOzaCStarAlgFinDimApprox}, corollary 2.4.4], then by
[\cite{HaaNucCStarAlgAmen}, theorem 3.1] this algebra is relatively amenable and
even $1$-relatively amenable [\cite{RundeAmenConstFour}, example 2]. Since
$M_n(\mathbb{C}_s)$ is finite dimensional, it is an $\mathscr{L}_{1, C}^g$-space
for some constant $C\geq 1$ independent of $s$. Thus, by
proposition~\ref{MetTopEssL1FlatModAoverAmenBanAlg} the $M_n(C(K))$-module
${M_n(\mathbb{C}_s)}^*$ is $C$-topologically flat. Since the latter module is
essential, by proposition~\ref{MetCTopFlatCharac} the right $M_n(C(K))$-module
${M_n(\mathbb{C}_s)}^{**}$ is $C$-topologically injective. Note that
${M_n(\mathbb{C}_s)}^{**}$ is isometrically isomorphic to $M_n(\mathbb{C}_s)$ as
right $M_n(C(K))$-module, so from proposition~\ref{MetTopInjModProd} we get that
$\bigoplus_\infty \{M_n(\mathbb{C}_s):s\in K \}$ is topologically injective as
$M_n(C(K))$-module.

Note that by proposition~\ref{MetInjCStarAlgCharac} the $C(K)$-module $C(K)$ is
metrically injective, therefore an isometric $C(K)$-morphism
$\widetilde{\rho}
:C(K)\to\bigoplus_\infty \{ \mathbb{C}_s:s\in K \}
:x\mapsto \bigoplus_\infty \{x(s):s\in K \}$ 
admits a left inverse contractive $C(K)$-morphism 
$\widetilde{\tau}:\bigoplus_\infty \{ \mathbb{C}_s:s\in K \} \to C(K)$. 
It is routine to check now that linear operators
$$
\rho
:M_n(C(K))\to\bigoplus\nolimits_\infty \{M_n(\mathbb{C}_s):s\in K \}
:x\mapsto \bigoplus\nolimits_\infty \{
    {(x_{i,j}(s))}_{i,j\in\mathbb{N}_n}:s\in K
 \}
$$
$$
\tau
:\bigoplus\nolimits_\infty \{M_n(\mathbb{C}_s):s\in K \}\to M_n(C(K))
:y\mapsto {\left(
    \widetilde{\tau}\left(\bigoplus\nolimits_\infty \{y_{s,i,j}:s\in K \}\right)
\right)}_{i,j\in\mathbb{N}_n}
$$
are bounded $M_n(C(K))$-morphisms such that $\tau \rho=1_{M_n(C(K))}$. That is
$M_n(C(K))$ is a retract of topologically injective $M_n(C(K))$-module
$\bigoplus_\infty \{M_n(\mathbb{C}_s):s\in K \}$ in $\mathbf{mod}_1-M_n(C(K))$.
Finally, from proposition~\ref{RetrMetCTopInjIsMetCTopInj} we conclude that
$M_n(C(K))$ is topologically injective $M_n(C(K))$-module.
\end{proof}

\begin{theorem}\label{TopInjAWStarAlgCharac} Let $A$ be a $C^*$-algebra. Then
the following are equivalent:

\begin{enumerate}[label = (\roman*)]
    \item $A$ is an $AW^*$-algebra which is topologically injective 
    as $A$-module;

    \item $A
    =\bigoplus_\infty \{M_{n_\lambda}(C(K_\lambda)):\lambda\in\Lambda \}$
    for some finite families of natural 
    numbers ${(n_\lambda)}_{\lambda\in\Lambda}$ and Stonean 
    spaces ${(K_\lambda)}_{\lambda\in\Lambda}$.
\end{enumerate}
\end{theorem}
\begin{proof}$(i)\implies (ii)$ From proposition 6.6
in~\cite{SmithDecompPropCStarAlg} we know that an $AW^*$-algebra is either
isomorphic as $C^*$-algebra 
to  $\bigoplus_\infty \{M_{n_\lambda}(C(K_\lambda)):\lambda\in\Lambda \}$ 
for some finite families of natural numbers ${(n_\lambda)}_{\lambda\in\Lambda}$ 
and Stonean spaces ${(K_\lambda)}_{\lambda\in\Lambda}$ or contains 
a ${}^*$-subalgebra 
$\bigoplus_\infty \{ \mathcal{B}(\ell_2(\mathbb{N}_n)):n\in\mathbb{N} \}$. The
second option is canceled out by proposition~\ref{TopInjIdHaveLUST}.

$(ii)\implies (i)$ For each $\lambda\in\Lambda$ the algebra
$M_{n_\lambda}(C(K_\lambda))$ is unital because $K_\lambda$ is compact.
Therefore $M_{n_\lambda}(C(K_\lambda))$ is faithful as
$M_{n_\lambda}(C(K_\lambda))$-module. It is also topologically injective as
$M_{n_\lambda}(C(K_\lambda))$-module by proposition~\ref{CKMatrixModTopInj}.
Now the topological injectivity of $A$ as $A$-module follows from paragraph
$(ii)$ of proposition~\ref{MetTopProjInjFlatUnderSumOfAlg} with $p=\infty$ and
$X_\lambda=A_\lambda=M_{n_\lambda}(C(K_\lambda))$ for all $\lambda\in\Lambda$. 

For all $\lambda\in\Lambda$ the algebra $C(K_\lambda)$ is an $AW^*$-algebra,
because $K_\lambda$ is a Stonean space [\cite{BerbBaerStarRings}, theorem
1.7.1]. Therefore $M_{n_\lambda}(C(K_\lambda))$ is an $AW^*$-algebra too
[\cite{BerbBaerStarRings}, corollary 9.62.1]. Finally $A$ is an $AW^*$-algebra
as $\bigoplus_\infty$-sum of such algebras [\cite{BerbBaerStarRings},
proposition 1.10.1].
\end{proof}

It is desirable to prove that any topologically injective over itself
$C^*$-algebra is an $AW^*$-algebra, but it seems to be a challenge even in the
commutative case.

%-------------------------------------------------------------------------------
%	Flat ideals of C^*-algebras
%-------------------------------------------------------------------------------

\subsection{
    Flat ideals of \texorpdfstring{$C^*$}{C*}-algebras
}\label{SubSectionFlatIdealsOfCStarAlgebras}

By considering flatness we finalize this lengthy investigations of ideals of
$C^*$-algebras.

\begin{proposition}\label{IdealofCstarAlgisMetTopFlat} Let $I$ be a left ideal
of a $C^*$-algebra $A$. Then $I$ is metrically and topologically flat as
$A$-module.
\end{proposition}
\begin{proof} From [\cite{HelBanLocConvAlg}, proposition 4.7.78] it follows that
$I$ has a right contractive identity. It remains to apply
proposition~\ref{MetTopFlatIdealsInUnitalAlg}.
\end{proof}

\begin{proposition}\label{CStarAlgIsL1IfFinDim} Let $A$ be a $C^*$-algebra, then
$A$ is an $\langle$~$L_1$-space / $\mathscr{L}_1^g$-space~$\rangle$ iff
$\langle$~$\operatorname{dim}(A)\leq 1$ / $A$ is finite dimensional~$\rangle$.
\end{proposition}
\begin{proof} Assume $A$ is an $\mathscr{L}_1^g$-space, then $A^{**}$ is
complemented in some $L_1$-space [\cite{DefFloTensNorOpId}, corollary
23.2.1(2)]. Since $A$ isometrically embeds in its second dual via $\iota_{A}$ we
may regard $A$ as closed subspace of some $L_1$-space. The latter space is
weakly sequentially complete [\cite{WojBanSpForAnalysts}, corollary III.C.14].
The property of being weakly sequentially complete is preserved by closed
subspaces, therefore $A$ is weakly sequentially complete too. By proposition 2
in~\cite{SakWeakCompOpOnOpAlg} every weakly sequentially complete $C^*$-algebra
is finite dimensional, hence $A$ is finite dimensional. Conversely, if $A$ is
finite dimensional it is an $\mathscr{L}_1^g$-space as any finite dimensional
Banach space.

Assume $A$ is an $L_1$-space and, a fortiori, an $\mathscr{L}_1^g$-space. As was
noted above $A$ is a finite dimensional, so
$A\isom{\mathbf{Ban}_1}\ell_1(\mathbb{N}_n)$ for $n=\operatorname{dim}(A)$. On
the other hand, $A$ is a finite dimensional $C^*$-algebra, so it is
isometrically isomorphic to $\bigoplus_\infty \{
\mathcal{B}(\ell_2(\mathbb{N}_{n_k})):k\in\mathbb{N}_m \}$ for some natural
numbers ${(n_k)}_{k\in\mathbb{N}_m}$ 
[\cite{DavCSatrAlgByExmpl}, theorem III.1.1].
Assume $\operatorname{dim}(A)>1$, then $A$ contains an isometric copy of
$\ell_\infty(\mathbb{N}_2)$. Therefore we have an isometric embedding of
$\ell_\infty(\mathbb{N}_2)$ into $\ell_1(\mathbb{N}_n)$. This is impossible by
theorem 1 from~\cite{LyubIsomEmdbFinDimLp}. 
Therefore $\operatorname{dim}(A)\leq 1$. 
\end{proof}

\begin{proposition}\label{CStarAlgIsTopFlatOverItsIdeal} Let $I$ be a proper
two-sided ideal of a  $C^*$-algebra $A$. Then the following are equivalent:

\begin{enumerate}[label = (\roman*)]
    \item $A$ is $\langle$~metrically / topologically~$\rangle$ flat $I$-module;

    \item $\langle$~$\operatorname{dim}(A)=1$, $I= \{0 \}$ / $A/I$ is finite
    dimensional~$\rangle$.
\end{enumerate}
\end{proposition}
\begin{proof} We may regard $I$ as an  ideal of unitazation $A_\#$ of $A$. Since
$I$ is a two-sided ideal, then it has a contractive approximate identity
${(e_\nu)}_{\nu\in N}$ such that $0\leq e_\nu\leq e_{A_\#}$
[\cite{HelBanLocConvAlg}, proposition 4.7.79]. As a 
corollary $\sup_{\nu\in N}\Vert e_{A_\#}-e_\nu\Vert\leq 1$. Since $I$ has an 
approximate identity we also have $A_{ess}:=\operatorname{cl}_A(IA)=I$. 
Since $I$ is a two sided ideal then $A/I$ is a $C^*$-algebra 
[\cite{HelBanLocConvAlg}, theorem 4.7.81].

Assume, $A$ is a metrically flat $I$-module. Since 
$\sup_{\nu\in N}\Vert e_{A_\#}-e_\nu\Vert\leq 1$, then paragraph $(ii)$ of
proposition~\ref{DualBanModDecomp} tells us that ${(A/A_{ess})}^*={(A/I)}^*$ is 
a retract of $A^*$ in $\mathbf{mod}_1-I$. Now from
propositions~\ref{MetCTopFlatCharac} and~\ref{RetrMetCTopInjIsMetCTopInj} it
follows that $A/I$ is metrically flat $I$-module. Since this is an annihilator
module, then from proposition~\ref{MetTopFlatAnnihModCharac} it follows that 
$I= \{0 \}$ and $A/I$ is an $L_1$-space. Now from
proposition~\ref{CStarAlgIsL1IfFinDim} we get that 
$\operatorname{dim}(A/I)\leq 1$. Since $A$ contains a proper 
ideal $I= \{0 \}$, then $\operatorname{dim}(A)=1$. Conversely, 
if $I= \{0 \}$ and $\operatorname{dim}(A)=1$, then we have an 
annihilator $I$-module $A$ which is isometrically isomorphic 
to $\ell_1(\mathbb{N}_1)$. By proposition~\ref{MetTopFlatAnnihModCharac} 
it is metrically flat. 

By proposition~\ref{TopFlatModCharac} the $I$-module $A$ is topologically flat
iff $A_{ess}=I$ and $A/A_{ess}=A/I$ are topologically flat $I$-modules. By
proposition~\ref{IdealofCstarAlgisMetTopFlat} the ideal $I$ is topologically
flat $I$-module. By proposition~\ref{MetTopFlatAnnihModCharac} the annihilator
$I$-module $A/I$ is topologically flat iff it is an $\mathscr{L}_1^g$-space. By
proposition~\ref{CStarAlgIsL1IfFinDim} this is equivalent to $A/I$ being finite
dimensional.
\end{proof}

%-------------------------------------------------------------------------------
%	\mathcal{K}(H)- and \mathcal{B}(H)-modules
%-------------------------------------------------------------------------------

\subsection{
    \texorpdfstring{$\mathcal{K}(H)$}{K (H)}- and
    \texorpdfstring{$\mathcal{B}(H)$}{B (H)}-modules
}\label{SubSectionKHAndBHModules}

In this section we apply general results on ideals obtained above to classical
modules over $C^*$-algebras.  For a given Hilbert space $H$ we consider
$\mathcal{B}(H)$, $\mathcal{K}(H)$ and $\mathcal{N}(H)$ as left and right Banach
modules over $\mathcal{B}(H)$ and $\mathcal{K}(H)$. For all modules the module
action is just the composition of operators. The Schatten-von Neumann
isomorphisms $\mathcal{N}(H)\isom{\mathbf{Ban}_1}{\mathcal{K}(H)}^*$,
$\mathcal{B}(H)\isom{\mathbf{Ban}_1}{\mathcal{N}(H)}^*$ (see
[\cite{TakThOpAlgVol1}, theorems II.1.6, II.1.8]) will be of use here. They are
in fact isomorphisms of left and right $\mathcal{B}(H)$-modules and a fortiori
of $\mathcal{K}(H)$-modules.

\begin{proposition}\label{KHAndBHModBH} Let $H$ be a Hilbert space. Then

\begin{enumerate}[label = (\roman*)]
    \item $\mathcal{B}(H)$ is metrically and topologically projective and 
    flat as $\mathcal{B}(H)$-module;

    \item $\mathcal{B}(H)$ is metrically or topologically projective or flat as
    $\mathcal{K}(H)$-module iff $H$ is finite dimensional;

    \item $\mathcal{B}(H)$ is topologically injective as $\mathcal{B}(H)$- or
    $\mathcal{K}(H)$-module iff $H$ is finite dimensional;

    \item $\mathcal{B}(H)$ is metrically injective as $\mathcal{B}(H)$- or
    $\mathcal{K}(H)$-module iff $\dim(H)\leq 1$.
\end{enumerate}
\end{proposition}
\begin{proof} $(i)$ Since $\mathcal{B}(H)$ is a unital algebra it is metrically
and topologically projective as $\mathcal{B}(H)$-module by
proposition~\ref{UnitalAlgIsMetTopProj}. Both results regarding flatness follow
from proposition~\ref{MetTopProjIsMetTopFlat}.

$(ii)$ For infinite dimensional $H$ the Banach space
$\mathcal{B}(H)/\mathcal{K}(H)$ is of infinite dimension, so by
proposition~\ref{CStarAlgIsTopFlatOverItsIdeal} the module $\mathcal{B}(H)$
neither topologically nor metrically flat as $\mathcal{K}(H)$-module. Both
claims regarding projectivity follow from
proposition~\ref{MetTopProjIsMetTopFlat}. If $H$ is finite dimensional, then
$\mathcal{K}(H)=\mathcal{B}(H)$, so the result follows from paragraph $(i)$.

$(iii)$ If $H$ is infinite dimensional, then $\mathcal{B}(H)$ contains
$\mathcal{B}(\ell_2(\mathbb{N}_n))$ as ${}^*$-subalgebra for all
$n\in\mathbb{N}$. By proposition~\ref{TopInjIdHaveLUST} we get that
$\mathcal{B}(H)$ is not topologically injective as $\mathcal{B}(H)$-module. The
rest follows from paragraph $(i)$ of
proposition~\ref{MetTopInjUnderChangeOfAlg}. If $H$ is finite dimensional, then
$\mathcal{K}(H)=\mathcal{B}(H)$, so the result follows from
proposition~\ref{FinDimBHModTopInj}.

$(iv)$ If $\dim(H)>1$, then $C^*$-algebra $\mathcal{B}(H)$ is not commutative. 
By proposition~\ref{MetInjCStarAlgCharac} we get that it is not metrically
injective as $\mathcal{B}(H)$-module. Now from paragraph $(i)$
of~\ref{MetTopInjUnderChangeOfAlg} we get that $\mathcal{B}(H)$ is not
metrically injective as $\mathcal{K}(H)$-module. If $\dim(H)\leq 1$ both claims
obviously follow from~\ref{MetInjCStarAlgCharac}.
\end{proof}

\begin{proposition}\label{KHAndBHModKH} Let $H$ be a Hilbert space. Then 

\begin{enumerate}[label = (\roman*)]
    \item $\mathcal{K}(H)$ is metrically and topologically flat 
    as $\mathcal{B}(H)$- or $\mathcal{K}(H)$-module;

    \item $\mathcal{K}(H)$ is metrically or topologically projective as
    $\mathcal{B}(H)$- or $\mathcal{K}(H)$-module iff $H$ is finite dimensional;

    \item $\mathcal{K}(H)$ is topologically injective as $\mathcal{B}(H)$- 
    or $\mathcal{K}(H)$-module iff $H$ is finite dimensional;

    \item $\mathcal{K}(H)$ is metrically injective as $\mathcal{B}(H)$- 
    or $\mathcal{K}(H)$-module iff $\dim(H)\leq 1$.
\end{enumerate}
\end{proposition}
\begin{proof} Let $A$ be either $\mathcal{B}(H)$ or $\mathcal{K}(H)$. Note that
$\mathcal{K}(H)$ is a two-sided ideal of $A$. 

$(i)$ Recall that $\mathcal{K}(H)$ has a contractive approximate identity
consisting of orthogonal projections onto all finite-dimensional subspaces of
$H$. Since $\mathcal{K}(H)$ is a two-sided ideal of $A$, then the result follows
from proposition~\ref{IdealofCstarAlgisMetTopFlat}.

$(ii)$, $(iii)$, $(iv)$ If $H$ is infinite dimensional, then $\mathcal{K}(H)$ is
not unital as Banach algebra. From
corollary~\ref{BiIdealOfCStarAlgMetTopProjCharac} and
proposition~\ref{MetTopInjOfId} the $A$-module $\mathcal{K}(H)$ is neither
metrically nor topologically projective or injective. If $H$ is finite
dimensional, then $\mathcal{K}(H)=\mathcal{B}(H)$, so both results follow from
paragraphs $(i)$, $(iii)$ and $(iv)$ of proposition~\ref{KHAndBHModBH}.
\end{proof}

\begin{proposition}\label{KHAndBHModNH} Let $H$ be a Hilbert space. Then

\begin{enumerate}[label = (\roman*)]
    \item $\mathcal{N}(H)$ is metrically and topologically injective as
    $\mathcal{B}(H)$- or $\mathcal{K}(H)$-module;

    \item $\mathcal{N}(H)$ is topologically projective or flat 
    as $\mathcal{B}(H)$- or $\mathcal{K}(H)$-module 
    iff $H$ is finite dimensional;

    \item $\mathcal{N}(H)$ is metrically projective or flat 
    as $\mathcal{B}(H)$- or $\mathcal{K}(H)$-module iff $\dim(H)\leq 1$.
\end{enumerate}
\end{proposition}
\begin{proof} Let $A$ be either $\mathcal{B}(H)$ or $\mathcal{K}(H)$.

$(i)$ Note that $\mathcal{N}(H)\isom{\mathbf{mod}_1-A}{\mathcal{K}(H)}^*$, 
so the result follows from proposition~\ref{MetCTopFlatCharac} and 
paragraph $(i)$ of proposition~\ref{KHAndBHModKH}.

$(ii)$ Assume $H$ is infinite dimensional. Note that
$\mathcal{B}(H)\isom{\mathbf{mod}_1-A}{\mathcal{N}(H)}^*$, so from
proposition~\ref{DualMetTopProjIsMetrInj} and paragraph $(iii)$ of
proposition~\ref{KHAndBHModBH} we get that $\mathcal{N}(H)$ is not topologically
projective as $A$-module. Both results regarding flatness follow from
proposition~\ref{MetTopProjIsMetTopFlat}. If $H$ is finite dimensional we use
proposition~\ref{FinDimNHModTopProjFlat}.

$(iii)$ Assume $\dim(H)>1$, then by paragraph $(iv)$ of
proposition~\ref{KHAndBHModBH} the $A$-module $\mathcal{B}(H)$ is not metrically
injective. Since $\mathcal{B}(H)\isom{\mathbf{mod}_1-A}{\mathcal{N}(H)}^*$, then
from proposition~\ref{MetCTopFlatCharac} we get that $\mathcal{N}(H)$ is not
metrically flat as $A$-module. By proposition~\ref{MetTopProjIsMetTopFlat}, it
is not metrically projective as $A$-module. If $\dim(H)\leq 1$, then
$\mathcal{N}(H)=\mathcal{K}(H)=\mathcal{B}(H)$, so both results follow from
paragraph $(i)$ of proposition~\ref{KHAndBHModBH}.
\end{proof}

\begin{proposition}\label{KHAndBHModsRelTh} Let $H$ be a Hilbert space. Then
\begin{enumerate}[label = (\roman*)]
    \item as $\mathcal{K}(H)$-modules $\mathcal{N}(H)$ is relatively projective
    injective and flat, $\mathcal{K}(H)$ is relatively projective and flat, but
    relatively injective only for finite dimensional $H$, $\mathcal{B}(H)$ is
    relatively injective and flat, but relatively projective only for finite
    dimensional $H$;

    \item as $\mathcal{B}(H)$-modules $\mathcal{N}(H)$ is relatively projective
    injective and flat, $\mathcal{K}(H)$ is relatively projective and flat,
    $\mathcal{B}(H)$ is relatively projective, injective and flat.
\end{enumerate}
\end{proposition}
\begin{proof} $(i)$ Note that $H$ is relatively projective as
$\mathcal{K}(H)$-module [\cite{HelBanLocConvAlg}, theorem 7.1.27], so from
proposition 7.1.13 in~\cite{HelBanLocConvAlg} we get that
$\mathcal{N}(H)\isom{\mathcal{K}(H)-\mathbf{mod}_1}H\projtens H^*$ is also
relatively projective as $\mathcal{K}(H)$-module. By theorem IV.2.16
in~\cite{HelHomolBanTopAlg} the $\mathcal{K}(H)$-module $\mathcal{K}(H)$ is
relatively projective. A fortiori $\mathcal{N}(H)$ and $\mathcal{K}(H)$ are
relatively flat $\mathcal{K}(H)$-modules 
[\cite{HelBanLocConvAlg}, proposition 7.1.40], 
so $\mathcal{N}(H)\isom{\mathbf{mod}_1-\mathcal{K}(H)}{\mathcal{K}(H)}^*$
and  $\mathcal{B}(H)\isom{\mathbf{mod}_1-\mathcal{K}(H)}{\mathcal{N}(H)}^*$ are
relatively injective $\mathcal{K}(H)$-modules. From
[\cite{RamsHomPropSemgroupAlg}, proposition 2.2.8  (i)] we know that a Banach
algebra relatively injective over itself as a right module, necessarily has a
left identity. Therefore $\mathcal{K}(H)$ is not relatively injective
$\mathcal{K}(H)$-module for infinite dimensional $H$. If $H$ is finite
dimensional, then $\mathcal{K}(H)$-module $\mathcal{K}(H)$ is relatively
injective because $\mathcal{K}(H)=\mathcal{B}(H)$ and $\mathcal{B}(H)$ is
relatively injective $\mathcal{K}(H)$-module as was shown above. By corollary
5.5.64 from~\cite{DalBanAlgAutCont} the algebra $\mathcal{K}(H)$ is relatively
amenable, so all its left modules are relatively flat [\cite{HelBanLocConvAlg},
theorem 7.1.60]. In particular $\mathcal{B}(H)$ is relatively flat
$\mathcal{K}(H)$-module. From [\cite{HelHomolBanTopAlg}, exercise V.2.20] we
know that $\mathcal{B}(H)$ is not relatively projective as
$\mathcal{K}(H)$-module when $H$ is infinite dimensional. If $H$ is finite
dimensional then $\mathcal{B}(H)$ is relatively projective
$\mathcal{K}(H)$-module because $\mathcal{B}(H)=\mathcal{K}(H)$ and
$\mathcal{K}(H)$ is relatively projective $\mathcal{K}(H)$-module as was shown
above.

$(ii)$ From proposition~\ref{KHAndBHModBH} paragraph $(i)$ and
proposition~\ref{MetProjIsTopProjAndTopProjIsRelProj} it follows that
$\mathcal{B}(H)$ is relatively projective $\mathcal{B}(H)$-module. From
[\cite{RamsHomPropSemgroupAlg}, propositions 2.3.3, 2.3.4] we know that
$\langle$~an essential relatively projective / a faithful relatively
injective~$\rangle$ module over ideal of Banach algebra is $\langle$~relatively
projective / relative injective~$\rangle$ over algebra itself. Since
$\mathcal{K}(H)$ and $\mathcal{N}(H)$ are essential and faithful
$\mathcal{K}(H)$-modules, then from results of previous paragraph
$\mathcal{N}(H)$ is relatively projective and injective, while $\mathcal{K}(H)$
is relatively projective as $\mathcal{B}(H)$-modules. Now, by
[\cite{HelBanLocConvAlg}, proposition 7.1.40] all aforementioned modules are
relatively flat $\mathcal{B}(H)$-modules. In particular
$\mathcal{B}(H)\isom{\mathbf{mod}_1-\mathcal{B}(H)}{\mathcal{N}(H)}^*$ is
relatively injective $\mathcal{B}(H)$-module.
\end{proof}

Results of this section are summarized in the following three tables. Each cell
contains a condition under which the respective module has the respective
property and propositions where this is proved. We use ``?'' symbol to indicate
open problems. These tables confirm that the property of being homologically
trivial in metric and topological theory is too restrictive. It is easier to
mention cases where metric and topological properties coincide with relative
ones: flatness of $\mathcal{K}(H)$ as $\mathcal{B}(H)$- or
$\mathcal{K}(H)$-module, injectivity of $\mathcal{N}(H)$ as $\mathcal{B}(H)$- or
$\mathcal{K}(H)$-module, projectivity and flatness of $\mathcal{B}(H)$-module
$\mathcal{B}(H)$. In the remaining cases $H$ needs to be at least finite
dimensional in order to these properties be equivalent in metric, topological
and relative theory.

\begin{scriptsize}
    \begin{longtable}{|c|c|c|c|c|c|c|} 
    \multicolumn{7}{c}{
        \mbox{
            Homologically trivial $\mathcal{K}(H)$- 
            and $\mathcal{B}(H)$-modules in metric theory
        }
    }                                                                                                                                                                                                                                                                                                                                                                                                                                                    \\
    \hline &
    \multicolumn{3}{c|}{
        $\mathcal{K}(H)$-modules
    } & 
    \multicolumn{3}{c|}{
        $\mathcal{B}(H)$-modules
    }                                                                                                                                                                                                                       \\
    \hline & 
        \mbox{Projectivity} & 
        \mbox{Injectivity} & 
        \mbox{Flatness} & 
        \mbox{Projectivity} & 
        \mbox{Injectivity} & 
        \mbox{Flatness} \\ 
    \hline
        $\mathcal{N}(H)$ & 
        \begin{tabular}{@{}c@{}}
            $\dim(H)\leq 1$ \\
            \mbox{\ref{KHAndBHModNH}}
        \end{tabular} & 
        \begin{tabular}{@{}c@{}}
            $H$\mbox{ is any } \\
            \mbox{\ref{KHAndBHModNH}}
        \end{tabular} & 
        \begin{tabular}{@{}c@{}}
            $\dim(H)\leq 1$ \\
            \mbox{\ref{KHAndBHModNH}}
        \end{tabular} & 
        \begin{tabular}{@{}c@{}}
            $\dim(H)\leq 1$ \\
            \mbox{\ref{KHAndBHModNH}}
        \end{tabular} & 
        \begin{tabular}{@{}c@{}}
            $H$\mbox{ is any } \\
            \mbox{\ref{KHAndBHModNH}}
        \end{tabular} & 
        \begin{tabular}{@{}c@{}}
            $\dim(H)\leq 1$ \\
            \mbox{\ref{KHAndBHModNH}}
        \end{tabular} \\
    \hline
        $\mathcal{B}(H)$ & 
        \begin{tabular}{@{}c@{}}
            $\dim(H)<\aleph_0$ \\
            \mbox{\ref{KHAndBHModBH}}
        \end{tabular} & 
        \begin{tabular}{@{}c@{}}
            $\dim(H)\leq 1$ \\
            \mbox{\ref{KHAndBHModBH}}
        \end{tabular} & 
        \begin{tabular}{@{}c@{}}
            $\dim(H)<\aleph_0$ \\
            \mbox{\ref{KHAndBHModBH}}
        \end{tabular} & 
        \begin{tabular}{@{}c@{}}
            $H$\mbox{ is any } \\
            \mbox{\ref{KHAndBHModBH}}
        \end{tabular} & 
        \begin{tabular}{@{}c@{}}
            $\dim(H)\leq 1$ \\
            \mbox{\ref{KHAndBHModBH}}
        \end{tabular} & 
        \begin{tabular}{@{}c@{}}
            $H$\mbox{ is any } \\
            \mbox{\ref{KHAndBHModBH}}
        \end{tabular}          \\ 
    \hline
        $\mathcal{K}(H)$ & 
        \begin{tabular}{@{}c@{}}
            $\dim(H)<\aleph_0$ \\
            \mbox{\ref{KHAndBHModKH}}
        \end{tabular} & 
        \begin{tabular}{@{}c@{}}
            $\dim(H)\leq 1$ \\
            \mbox{\ref{KHAndBHModKH}}
        \end{tabular} & 
        \begin{tabular}{@{}c@{}}
            $H$\mbox{ is any } \\
            \mbox{\ref{KHAndBHModKH}}
        \end{tabular} & 
        \begin{tabular}{@{}c@{}}
            $\dim(H)<\aleph_0$ \\
            \mbox{\ref{KHAndBHModKH}}
        \end{tabular} & 
        \begin{tabular}{@{}c@{}}
            $\dim(H)\leq 1$ \\
            \mbox{\ref{KHAndBHModKH}}
        \end{tabular} & 
        \begin{tabular}{@{}c@{}}
            $H$\mbox{ is any } \\
            \mbox{\ref{KHAndBHModKH}}
        \end{tabular} \\ 
    \hline
    \multicolumn{7}{c}{
        \mbox{
            Homologically trivial $\mathcal{K}(H)$- 
            and $\mathcal{B}(H)$-modules in topological theory
        }
    } \\
    \hline & 
    \multicolumn{3}{c|}{
        $\mathcal{K}(H)$-modules
    } & 
    \multicolumn{3}{c|}{
        $\mathcal{B}(H)$-modules
    } \\
    \hline & 
        \mbox{Projectivity} & 
        \mbox{Injectivity} & 
        \mbox{Flatness} & 
        \mbox{Projectivity} & 
        \mbox{Injectivity} & 
        \mbox{Flatness} \\ 
    \hline
        $\mathcal{N}(H)$ & 
        \begin{tabular}{@{}c@{}}
            $\dim(H)<\aleph_0$ \\
            \mbox{\ref{KHAndBHModNH}}
        \end{tabular} & 
        \begin{tabular}{@{}c@{}}
            $H$\mbox{ is any } \\
            \mbox{\ref{KHAndBHModNH}}
        \end{tabular} & 
        \begin{tabular}{@{}c@{}}
            $\dim(H)<\aleph_0$ \\
            \mbox{\ref{KHAndBHModNH}}
        \end{tabular} & 
        \begin{tabular}{@{}c@{}}
            $\dim(H)<\aleph_0$ \\
            \mbox{\ref{KHAndBHModNH}}
        \end{tabular} & 
        \begin{tabular}{@{}c@{}}
            $H$\mbox{ is any } \\
            \mbox{\ref{KHAndBHModNH}}
        \end{tabular} & 
        \begin{tabular}{@{}c@{}}
            $\dim(H)<\aleph_0$ \\
            \mbox{\ref{KHAndBHModNH}}
        \end{tabular} \\
    \hline
        $\mathcal{B}(H)$ & 
        \begin{tabular}{@{}c@{}}
            $\dim(H)<\aleph_0$ \\
            \mbox{\ref{KHAndBHModBH}}
        \end{tabular} & 
        \begin{tabular}{@{}c@{}}
            $\dim(H)<\aleph_0$ \\
            \mbox{\ref{KHAndBHModBH}}
        \end{tabular} & 
        \begin{tabular}{@{}c@{}}
            $\dim(H)<\aleph_0$ \\
            \mbox{\ref{KHAndBHModBH}}
        \end{tabular} & 
        \begin{tabular}{@{}c@{}}
            $H$\mbox{ is any } \\
            \mbox{\ref{KHAndBHModBH}}
        \end{tabular} & 
        \begin{tabular}{@{}c@{}}
            $\dim(H)<\aleph_0$ \\
            \mbox{\ref{KHAndBHModBH}}
        \end{tabular} & 
        \begin{tabular}{@{}c@{}}
            $H$\mbox{ is any } \\
            \mbox{\ref{KHAndBHModBH}}
        \end{tabular} \\ 
    \hline
        $\mathcal{K}(H)$ &
        \begin{tabular}{@{}c@{}}
            $\dim(H)<\aleph_0$ \\
            \mbox{\ref{KHAndBHModKH}}
        \end{tabular} & 
        \begin{tabular}{@{}c@{}}
            $\dim(H)<\aleph_0$ \\
            \mbox{\ref{KHAndBHModKH}}
        \end{tabular} & 
        \begin{tabular}{@{}c@{}}
            $H$\mbox{ is any } \\
            \mbox{\ref{KHAndBHModKH}}
        \end{tabular} & 
        \begin{tabular}{@{}c@{}}
            $\dim(H)<\aleph_0$ \\
            \mbox{\ref{KHAndBHModKH}}
        \end{tabular} & 
        \begin{tabular}{@{}c@{}}
            $\dim(H)<\aleph_0$ \\
            \mbox{\ref{KHAndBHModKH}}
        \end{tabular} & 
        \begin{tabular}{@{}c@{}}
            $H$\mbox{ is any } \\
            \mbox{\ref{KHAndBHModKH}}
        \end{tabular} \\ 
    \hline
    \multicolumn{7}{c}{
        \mbox{
            Homologically trivial $\mathcal{K}(H)$- 
            and $\mathcal{B}(H)$-modules in relative theory
        }
    } \\
    \hline & 
    \multicolumn{3}{c|}{
        $\mathcal{K}(H)$-modules
    } & 
    \multicolumn{3}{c|}{
        $\mathcal{B}(H)$-modules
    } \\
    \hline & 
        \mbox{Projectivity} & 
        \mbox{Injectivity} & 
        \mbox{Flatness} & 
        \mbox{Projectivity} & 
        \mbox{Injectivity} & 
        \mbox{Flatness} \\ 
    \hline
        $\mathcal{N}(H)$ & 
        \begin{tabular}{@{}c@{}}
            $H$\mbox{ is any } \\
            \mbox{\ref{KHAndBHModsRelTh}}, (i)
        \end{tabular} & 
        \begin{tabular}{@{}c@{}}
            $H$\mbox{ is any } \\
            \mbox{\ref{KHAndBHModsRelTh}}, (i)
        \end{tabular} & 
        \begin{tabular}{@{}c@{}}
            $H$\mbox{ is any } \\
            \mbox{\ref{KHAndBHModsRelTh}}, (i)
        \end{tabular} & 
        \begin{tabular}{@{}c@{}}
            $H$\mbox{ is any } \\
            \mbox{\ref{KHAndBHModsRelTh}}, (ii)
        \end{tabular} & 
        \begin{tabular}{@{}c@{}}
            $H$\mbox{ is any } \\
            \mbox{\ref{KHAndBHModsRelTh}}, (ii)
        \end{tabular} & 
        \begin{tabular}{@{}c@{}}
            $H$\mbox{ is any } \\
            \mbox{\ref{KHAndBHModsRelTh}}, (ii)
        \end{tabular} \\
    \hline
        $\mathcal{B}(H)$ & 
        \begin{tabular}{@{}c@{}}
            $\dim(H)<\aleph_0$ \\
            \mbox{\ref{KHAndBHModsRelTh}}, (i)
        \end{tabular} & 
        \begin{tabular}{@{}c@{}}
            $H$\mbox{ is any } \\
            \mbox{\ref{KHAndBHModsRelTh}}, (i)
        \end{tabular} & 
        \begin{tabular}{@{}c@{}}
            $H$\mbox{ is any } \\
            \mbox{\ref{KHAndBHModsRelTh}}, (i)
        \end{tabular} & 
        \begin{tabular}{@{}c@{}}
            $H$\mbox{ is any } \\
            \mbox{\ref{KHAndBHModsRelTh}}, (ii)
        \end{tabular} & 
        \begin{tabular}{@{}c@{}}
            $H$\mbox{ is any } \\
            \mbox{\ref{KHAndBHModsRelTh}}, (ii)
        \end{tabular} & 
        \begin{tabular}{@{}c@{}}
            $H$\mbox{ is any } \\
            \mbox{\ref{KHAndBHModsRelTh}}, (ii)
        \end{tabular} \\
    \hline
        $\mathcal{K}(H)$ & 
        \begin{tabular}{@{}c@{}}
            $H$\mbox{ is any } \\
            \mbox{\ref{KHAndBHModsRelTh}}, (i)
        \end{tabular} & 
        \begin{tabular}{@{}c@{}}
            $\dim(H)<\aleph_0$ \\
            \mbox{\ref{KHAndBHModsRelTh}}, (i)
        \end{tabular} & 
        \begin{tabular}{@{}c@{}}
            $H$\mbox{ is any } \\
            \mbox{\ref{KHAndBHModsRelTh}}, (i)
        \end{tabular} & 
        \begin{tabular}{@{}c@{}}
            $H$\mbox{ is any } \\
            \mbox{\ref{KHAndBHModsRelTh}}, (ii)
        \end{tabular} & 
        \begin{tabular}{@{}c@{}} 
            {?}
        \end{tabular} & 
        \begin{tabular}{@{}c@{}}
            $H$\mbox{ is any } \\
            \mbox{\ref{KHAndBHModsRelTh}}, (ii)
        \end{tabular} \\
    \hline
    \end{longtable}
\end{scriptsize}

%-------------------------------------------------------------------------------
%	c_0(\Lambda)- and l_infty(\Lambda)-modules
%-------------------------------------------------------------------------------

\subsection{
    \texorpdfstring{$c_0(\Lambda)$}{c0 (Lambda)}- and
    \texorpdfstring{$\ell_\infty(\Lambda)$}{lInfty (Lambda)}-modules
}\label{SubSectionc0AndlInftyModules}

We continue our study of modules over $C^*$-algebras and move to commutative
examples. For a given index set $\Lambda$ we consider spaces $c_0(\Lambda)$ and
$\ell_p(\Lambda)$ for $1\leq p\leq+\infty$ as left and right modules over
algebras $c_0(\Lambda)$ and $\ell_\infty(\Lambda)$. For all these modules the
module action is just the pointwise multiplication. It is well known that
${c_0(\Lambda)}^*\isom{\mathbf{Ban}_1}\ell_1(\Lambda)$ and
${\ell_p(\Lambda)}^*\isom{\mathbf{Ban}_1}\ell_{p^*}(\Lambda)$ 
for $1\leq p<+\infty$. In fact these isomorphisms are isomorphisms of
$\ell_\infty(\Lambda)$- and $c_0(\Lambda)$-modules. 

For a given $\lambda\in\Lambda$ we define $\mathbb{C}_\lambda$ as left or right
$\ell_\infty(\Lambda)$- or $c_0(\Lambda)$-module $\mathbb{C}$ with module action
defined by
$$
a\cdot_\lambda z=a(\lambda)z,\qquad z\cdot_\lambda a=a(\lambda) z.
$$
for $a\in \ell_\infty(\Lambda)$ and $z\in\mathbb{C}_s$. 

\begin{proposition}\label{OneDimlInftyc0ModMetTopProjIngFlat} Let $\Lambda$ be a
set and $\lambda\in\Lambda$. Then $\mathbb{C}_\lambda$ is metrically and
topologically projective injective and flat $\ell_\infty(\Lambda)$- or
$c_0(\Lambda)$-module.
\end{proposition}
\begin{proof} Let $A$ be either $\ell_\infty(\Lambda)$ or $c_0(\Lambda)$. One
can easily check that the 
maps $\pi:A_+\to\mathbb{C}_\lambda:a\oplus_1 z\mapsto a(\lambda)+z$ 
and $\sigma:\mathbb{C}_\lambda\to A_+:z\mapsto z\delta_\lambda\oplus_1 0$ 
are contractive $A$-morphisms of left $A$-modules.
Since $\pi\sigma=1_{\mathbb{C}_\lambda}$, then $\mathbb{C}_\lambda$ is retract
of $A_+$ in $A-\mathbf{mod}_1$. From proposition~\ref{UnitalAlgIsMetTopProj}
and~\ref{RetrMetCTopProjIsMetCTopProj} it follows that $\mathbb{C}_\lambda$ is
metrically and topologically projective left $A$-module and a fortiori
metrically and topologically flat by proposition~\ref{MetTopProjIsMetTopFlat}.
By proposition~\ref{DualMetTopProjIsMetrInj} we have that $\mathbb{C}_\lambda^*$
is metrically and topologically injective as $A$-module. Now metric and
topological injectivity of $\mathbb{C}_\lambda$ follow from isomorphism
$\mathbb{C}_\lambda\isom{\mathbf{mod}_1-A}\mathbb{C}_\lambda^*$.
\end{proof}

\begin{proposition}\label{c0AndlInftyModlIfty} Let $\Lambda$ be a set. Then

\begin{enumerate}[label = (\roman*)]
    \item $\ell_\infty(\Lambda)$ is metrically and topologically projective 
    and flat as $\ell_\infty(\Lambda)$-module;

    \item $\ell_\infty(\Lambda)$ is metrically or topologically projective 
    or flat as $c_0(\Lambda)$-module iff $\Lambda$ is finite;

    \item $\ell_\infty(\Lambda)$ is metrically and topologically injective as
    $\ell_\infty(\Lambda)$- and $c_0(\Lambda)$-module.
\end{enumerate}
\end{proposition}
\begin{proof} $(i)$ Since $\ell_\infty(\Lambda)$ is a unital algebra, then it is
metrically and topologically projective as $\ell_\infty(\Lambda)$-module by
proposition~\ref{UnitalAlgIsMetTopProj}. Results on flatness follow from
proposition~\ref{MetTopProjIsMetTopFlat}.

$(ii)$ For infinite $\Lambda$ the Banach space
$\ell_\infty(\Lambda)/c_0(\Lambda)$ is of infinite dimension, so by
proposition~\ref{CStarAlgIsTopFlatOverItsIdeal} the module
$\ell_\infty(\Lambda)$ neither topologically nor metrically flat as
$c_0(\Lambda)$-module. Both claims regarding projectivity follow from
proposition~\ref{MetTopProjIsMetTopFlat}. If $\Lambda$ is finite, then
$c_0(\Lambda)=\ell_\infty(\Lambda)$, so the result follows from paragraph $(i)$.

$(iii)$ Let $A$ be either $\ell_\infty(\Lambda)$ or $c_0(\Lambda)$. Note that
$\ell_\infty(\Lambda)
\isom{A-\mathbf{mod}_1}
\bigoplus_\infty \{\mathbb{C}_\lambda:\lambda\in\Lambda \}$, then from
propositions~\ref{OneDimlInftyc0ModMetTopProjIngFlat} and~\ref{MetTopInjModProd}
it follows that $\ell_\infty(\Lambda)$ is metrically injective as $A$-module.
Topological injectivity follows from
proposition~\ref{MetInjIsTopInjAndTopInjIsRelInj}.
\end{proof}

\begin{proposition}\label{c0AndlInftyModc0} Let $\Lambda$ be a set. Then 

\begin{enumerate}[label = (\roman*)]
    \item $c_0(\Lambda)$ is metrically and topologically flat as
    $\ell_\infty(\Lambda)$- or $c_0(\Lambda)$-module;

    \item $c_0(\Lambda)$ is metrically or topologically projective as
    $\ell_\infty(\Lambda)$- or $c_0(\Lambda)$-module iff $\Lambda$ is finite;

    \item $c_0(\Lambda)$ is metrically or topologically injective as
    $\ell_\infty(\Lambda)$- or $c_0(\Lambda)$-module iff $\Lambda$ is finite.
\end{enumerate}
\end{proposition}
\begin{proof} Let $A$ be either $\ell_\infty(\Lambda)$ or $c_0(\Lambda)$. Note
that $c_0(\Lambda)$ is a two-sided ideal of $A$. 

$(i)$ Recall that $c_0(\Lambda)$ has a contractive approximate identity of the
form ${(\sum_{\lambda\in S}\delta_\lambda)}_{S\in\mathcal{P}_0(\Lambda)}$. Since
$c_0(\Lambda)$ is a two-sided ideal of $A$, then the result follows from
proposition~\ref{IdealofCstarAlgisMetTopFlat}.

$(ii)$, $(iii)$ If $\Lambda$ is infinite, then $c_0(\Lambda)$ is not unital as
Banach algebra. From corollary~\ref{BiIdealOfCStarAlgMetTopProjCharac} and
proposition~\ref{MetTopInjOfId} the $A$-module $c_0(\Lambda)$ is neither
metrically nor topologically projective or injective. If $\Lambda$ is finite,
then $c_0(\Lambda)=\ell_\infty(\Lambda)$, so both results follow from paragraphs
$(i)$ and $(iii)$ of proposition~\ref{c0AndlInftyModlIfty}.
\end{proof}

\begin{proposition}\label{c0AndlInftyModl1} Let $\Lambda$ be a set. Then

\begin{enumerate}[label = (\roman*)]
    \item $\ell_1(\Lambda)$ is metrically and topologically injective as
    $\ell_\infty(\Lambda)$- or $c_0(\Lambda)$-module;

    \item $\ell_1(\Lambda)$ is metrically and topologically projective and 
    flat as $\ell_\infty(\Lambda)$- or $c_0(\Lambda)$-module;
\end{enumerate}
\end{proposition}
\begin{proof} Let $A$ be either $\ell_\infty(\Lambda)$ or $c_0(\Lambda)$.

$(i)$ Note that $\ell_1(\Lambda)\isom{\mathbf{mod}_1-A}{c_0(\Lambda)}^*$, so the
result follows from proposition~\ref{MetCTopFlatCharac} and paragraph $(i)$ of
proposition~\ref{c0AndlInftyModc0}.

$(ii)$ Note that 
$\ell_1(\Lambda)
\isom{A-\mathbf{mod}_1}
\bigoplus_1 \{\mathbb{C}_\lambda:\lambda\in\Lambda \}$, then from
propositions~\ref{OneDimlInftyc0ModMetTopProjIngFlat}
and~\ref{MetTopProjModCoprod} it follows that $\ell_1(\Lambda)$ is metrically
projective as $A$-module. Topological projectivity follows from
proposition~\ref{MetProjIsTopProjAndTopProjIsRelProj}. Metric and topological
flatness follow from proposition~\ref{MetTopProjIsMetTopFlat}.
\end{proof}

\begin{proposition}\label{c0AndlInftyModlp} Let $\Lambda$ be a set and
$1<p<+\infty$. Then $\ell_p(\Lambda)$ is metrically or topologically projective,
injective or flat as $\ell_\infty(\Lambda)$- or $c_0(\Lambda)$-module iff
$\Lambda$ is finite.
\end{proposition}
\begin{proof} Let $A$ be either $\ell_\infty(\Lambda)$ or $c_0(\Lambda)$, then
$A$ is an $\mathscr{L}_\infty^g$-space. Since $\ell_p(\Lambda)$ is reflexive for
$1<p<+\infty$, then from corollary
~\ref{NoInfDimRefMetTopProjInjFlatModOverMthscrL1OrLInfty} it follows that
$\ell_p(\Lambda)$ is necessarily finite dimensional if it is metrically or
topologically projective injective or flat. This is equivalent to $\Lambda$
being finite. If $\Lambda$ is finite then
$\ell_p(\Lambda)\isom{A-\mathbf{mod}}\ell_1(\Lambda)$ and
$\ell_p(\Lambda)\isom{\mathbf{mod}-A}\ell_1(\Lambda)$, so topological
projectivity injectivity and flatness follow from
proposition~\ref{c0AndlInftyModl1}.
\end{proof}

\begin{proposition}\label{c0AndlInftyModsRelTh} Let $\Lambda$ be a set. Then

\begin{enumerate}[label = (\roman*)]
    \item as $c_0(\Lambda)$-modules $\ell_p(\Lambda)$ for $1\leq p<+\infty$ and
    $\mathbb{C}_\lambda$ for $\lambda\in\Lambda$ are relatively projective,
    injective and flat, $c_0(\Lambda)$ relatively projective and flat, but
    relatively injective only for finite $\Lambda$, $\ell_\infty(\Lambda)$ is
    relatively injective and flat, but relatively projective only for finite
    $\Lambda$;

    \item as $\ell_\infty(\Lambda)$-modules $\ell_p(\Lambda)$ for 
    $1\leq p\leq+\infty$ and $\mathbb{C}_\lambda$ for $\lambda\in\Lambda$ 
    are relatively projective, injective and flat, $c_0(\Lambda)$ is 
    relatively projective, injective and flat.
\end{enumerate}
\end{proposition}
\begin{proof} $(i)$ The algebra $c_0(\Lambda)$ is relatively biprojective
[\cite{HelHomolBanTopAlg}, theorem IV.5.26] and admits a contractive approximate
identity, so by [\cite{HelBanLocConvAlg}, theorem 7.1.60] all essential
$c_0(\Lambda)$-modules are projective. Thus $c_0(\Lambda)$ and $\ell_p(\Lambda)$
for $1\leq p<+\infty$ are relatively projective $c_0(\Lambda)$-modules. A
fortiori they are relatively flat as $c_0(\Lambda)$-modules
[\cite{HelBanLocConvAlg}, proposition 7.1.40]. By the same proposition
$\ell_1(\Lambda)\isom{\mathbf{mod}_1-c_0(\Lambda)}{c_0(\Lambda)}^*$ and
$\ell_{p^*}(\Lambda)\isom{\mathbf{mod}_1-c_0(\Lambda)}{\ell_p(\Lambda)}^*$ for
$1\leq p<+\infty$ are relatively injective $c_0(\Lambda)$-modules. From
[\cite{RamsHomPropSemgroupAlg}, proposition 2.2.8 (i)] we know that a Banach
algebra relatively injective over itself as right module, necessarily has a left
identity. Therefore $c_0(\Lambda)$ is not relatively injective
$c_0(\Lambda)$-module for infinite $\Lambda$. If $\Lambda$ is finite, then
$c_0(\Lambda)$-module $c_0(\Lambda)$ is relatively injective because
$c_0(\Lambda)=\ell_\infty(\Lambda)$ and $\ell_\infty(\Lambda)$ is relatively
injective $c_0(\Lambda)$-module as was shown above. From
[\cite{HelHomolBanTopAlg}, corollary V.2.16(II)] we know that
$\ell_\infty(\Lambda)$ is not relatively projective as $c_0(\Lambda)$-module
provided $\Lambda$ is infinite. If $\Lambda$ is finite then
$\ell_\infty(\Lambda)$ is relatively projective $c_0(\Lambda)$-module because
$\ell_\infty(\Lambda)=c_0(\Lambda)$ and $c_0(\Lambda)$ is relatively projective
$c_0(\Lambda)$-module as was shown above.
Propositions~\ref{OneDimlInftyc0ModMetTopProjIngFlat},
~\ref{MetProjIsTopProjAndTopProjIsRelProj},
~\ref{MetInjIsTopInjAndTopInjIsRelInj}
and~\ref{MetFlatIsTopFlatAndTopFlatIsRelFlat} give the result for modules
$\mathbb{C}_\lambda$, where $\lambda\in\Lambda$.

$(ii)$ From proposition~\ref{c0AndlInftyModlIfty} paragraph $(i)$ and
proposition~\ref{MetProjIsTopProjAndTopProjIsRelProj} it follows that
$\ell_\infty(\Lambda)$ is relatively projective $\ell_\infty(\Lambda)$-module. 
In [\cite{NemANoteOnRelInjC0ModC0}, theorem 4.4] it was shown 
that $\ell_\infty(\Lambda)$-module $c_0(\Lambda)$ is relatively injective. 
From [\cite{RamsHomPropSemgroupAlg}, propositions 2.3.3, 2.3.4] we know that
$\langle$~an essential relatively projective / a  faithful relatively
injective~$\rangle$ module over an ideal of a Banach algebra is 
$\langle$~relatively projective / relatively injective~$\rangle$ over algebra 
itself. Since $c_0(\Lambda)$ and $\ell_p(\Lambda)$ for $1\leq p<+\infty$ are
essential and faithful $c_0(\Lambda)$-modules then from results of previous
paragraph $\ell_p(\Lambda)$ for $1\leq p<+\infty$ are relatively projective and
injective as $\ell_\infty(\Lambda)$-modules. Also we get that $c_0(\Lambda)$ is
relatively projective $\ell_\infty(\Lambda)$-module. Therefore all these
$\ell_\infty(\Lambda)$-modules are relatively flat [\cite{HelBanLocConvAlg},
proposition 7.1.40]. As the consequence $\ell_\infty(\Lambda)
\isom{\mathbf{mod}_1-\ell_\infty(\Lambda)} {\ell_1(\Lambda)}^*$ is relatively
injective $\ell_\infty(\Lambda)$-module.
Propositions~\ref{OneDimlInftyc0ModMetTopProjIngFlat},
~\ref{MetProjIsTopProjAndTopProjIsRelProj},
~\ref{MetInjIsTopInjAndTopInjIsRelInj}
and~\ref{MetFlatIsTopFlatAndTopFlatIsRelFlat} give the result for modules
$\mathbb{C}_\lambda$, where $\lambda\in\Lambda$.
\end{proof}

Results of this section are summarized in the following three tables. Each cell
contains a condition under which the respective module has the respective
property and propositions where this is proved. For the case 
of $\ell_\infty(\Lambda)$- and $c_0(\Lambda)$-modules $\ell_p(\Lambda)$ 
for $1<p<+\infty$ we don't have a criterion of homological triviality in metric
theory, just a necessary condition. We indicate this fact via symbol ${}^*$.
From these table one can easily see that for modules over commutative
$C^*$-algebras, there is much more in common between relative, metric and
topological  theory. For example $\ell_1(\Lambda)$ is projective injective and
flat as $\ell_\infty(\Lambda)$- or $c_0(\Lambda)$-module in all three theories.

\begin{scriptsize}
    \begin{longtable}{|c|c|c|c|c|c|c|} 
    \multicolumn{7}{c}{
        \mbox{
            Homologically trivial $c_0(\Lambda)$- 
            and $\ell_\infty(\Lambda)$-modules in metric theory
        }
    } \\ 
    \hline & 
    \multicolumn{3}{c|}{
        $c_0(\Lambda)$-modules
    } & 
    \multicolumn{3}{c|}{
        $\ell_\infty(\Lambda)$-modules
    } \\
    \hline & 
        \mbox{Projectivity} & 
        \mbox{Injectivity} & 
        \mbox{Flatness} & 
        \mbox{Projectivity} & 
        \mbox{Injectivity} & 
        \mbox{Flatness} \\ 
    \hline
        $\ell_1(\Lambda)$ & 
        \begin{tabular}{@{}c@{}}
            $\Lambda$\mbox{ is any } \\
            \mbox{\ref{c0AndlInftyModl1}}
        \end{tabular} & 
        \begin{tabular}{@{}c@{}}
            $\Lambda$\mbox{ is any } \\
            \mbox{\ref{c0AndlInftyModl1}}
        \end{tabular} & 
        \begin{tabular}{@{}c@{}}
            $\Lambda$\mbox{ is any } \\
            \mbox{\ref{c0AndlInftyModl1}}
        \end{tabular} & 
        \begin{tabular}{@{}c@{}}
            $\Lambda$\mbox{ is any } \\
            \mbox{\ref{c0AndlInftyModl1}}
        \end{tabular} & 
        \begin{tabular}{@{}c@{}}
            $\Lambda$\mbox{ is any }  \\
            \mbox{\ref{c0AndlInftyModl1}}
        \end{tabular} & 
        \begin{tabular}{@{}c@{}}
            $\Lambda$\mbox{ is any } \\
            \mbox{\ref{c0AndlInftyModl1}}
        \end{tabular} \\
    \hline 
        $\ell_p(\Lambda)$ & 
        \begin{tabular}{@{}c@{}}
            $\operatorname{Card}(\Lambda)<\aleph_0$ \\ 
            \mbox{\ref{c0AndlInftyModlp}}${}^*$
        \end{tabular} & 
        \begin{tabular}{@{}c@{}}
            $\operatorname{Card}(\Lambda)<\aleph_0$ \\ 
            \mbox{\ref{c0AndlInftyModlp}}${}^*$
        \end{tabular} & 
        \begin{tabular}{@{}c@{}}
            $\operatorname{Card}(\Lambda)<\aleph_0$ \\ 
            \mbox{\ref{c0AndlInftyModlp}}${}^*$
        \end{tabular} & 
        \begin{tabular}{@{}c@{}}
            $\operatorname{Card}(\Lambda)<\aleph_0$ \\ 
            \mbox{\ref{c0AndlInftyModlp}}${}^*$
        \end{tabular} & 
        \begin{tabular}{@{}c@{}}
            $\operatorname{Card}(\Lambda)<\aleph_0$ \\ 
            \mbox{\ref{c0AndlInftyModlp}}${}^*$
        \end{tabular} & 
        \begin{tabular}{@{}c@{}}
            $\operatorname{Card}(\Lambda)<\aleph_0$ \\ 
            \mbox{\ref{c0AndlInftyModlp}}${}^*$
        \end{tabular} \\
    \hline
        $\ell_\infty(\Lambda)$ & 
        \begin{tabular}{@{}c@{}}
            $\operatorname{Card}(\Lambda)<\aleph_0$ \\
            \mbox{\ref{c0AndlInftyModlIfty}}
        \end{tabular} & 
        \begin{tabular}{@{}c@{}}
            $\Lambda$\mbox{ is any } \\
            \mbox{\ref{c0AndlInftyModlIfty}}
        \end{tabular} & 
        \begin{tabular}{@{}c@{}}
            $\operatorname{Card}(\Lambda)<\aleph_0$ \\
            \mbox{\ref{c0AndlInftyModlIfty}}
        \end{tabular} & 
        \begin{tabular}{@{}c@{}}
            $\Lambda$\mbox{ is any } \\
            \mbox{\ref{c0AndlInftyModlIfty}}
        \end{tabular} & 
        \begin{tabular}{@{}c@{}}
            $\Lambda$\mbox{ is any } \\
            \mbox{\ref{c0AndlInftyModlIfty}}
        \end{tabular} & 
        \begin{tabular}{@{}c@{}}
            $\Lambda$\mbox{ is any } \\
            \mbox{\ref{c0AndlInftyModlIfty}}
        \end{tabular} \\ 
    \hline
        $c_0(\Lambda)$ & 
        \begin{tabular}{@{}c@{}}
            $\operatorname{Card}(\Lambda)<\aleph_0$ \\
            \mbox{\ref{c0AndlInftyModc0}}
        \end{tabular} & 
        \begin{tabular}{@{}c@{}}
            $\operatorname{Card}(\Lambda)< \aleph_0$ \\
            \mbox{\ref{c0AndlInftyModc0}}
        \end{tabular} & 
        \begin{tabular}{@{}c@{}}
            $\Lambda$\mbox{ is any } \\
            \mbox{\ref{c0AndlInftyModc0}}
        \end{tabular} & 
        \begin{tabular}{@{}c@{}}
            $\operatorname{Card}(\Lambda)<\aleph_0$ \\
            \mbox{\ref{c0AndlInftyModc0}}
        \end{tabular} & 
        \begin{tabular}{@{}c@{}}
            $\operatorname{Card}(\Lambda)< \aleph_0$ \\
            \mbox{\ref{c0AndlInftyModc0}}
        \end{tabular} & 
        \begin{tabular}{@{}c@{}}
            $\Lambda$\mbox{ is any } \\
            \mbox{\ref{c0AndlInftyModc0}}
        \end{tabular} \\ 
    \hline
        $\mathbb{C}_\lambda$ &
        \begin{tabular}{@{}c@{}}
            $\lambda$\mbox{ is any } \\
            \mbox{\ref{OneDimlInftyc0ModMetTopProjIngFlat}}
        \end{tabular} & 
        \begin{tabular}{@{}c@{}}
            $\lambda$\mbox{ is any } \\
            \mbox{\ref{OneDimlInftyc0ModMetTopProjIngFlat}}
        \end{tabular} & 
        \begin{tabular}{@{}c@{}}
            $\lambda$\mbox{ is any } \\
            \mbox{\ref{OneDimlInftyc0ModMetTopProjIngFlat}}
        \end{tabular} &
        \begin{tabular}{@{}c@{}}
            $\lambda$\mbox{ is any } \\
            \mbox{\ref{OneDimlInftyc0ModMetTopProjIngFlat}}
        \end{tabular} & 
        \begin{tabular}{@{}c@{}}
            $\lambda$\mbox{ is any } \\
            \mbox{\ref{OneDimlInftyc0ModMetTopProjIngFlat}}
        \end{tabular} & 
        \begin{tabular}{@{}c@{}}
            $\lambda$\mbox{ is any } \\
            \mbox{\ref{OneDimlInftyc0ModMetTopProjIngFlat}}
        \end{tabular} \\
    \hline 
    \multicolumn{7}{c}{
        \mbox{
            Homologically trivial $c_0(\Lambda)$- 
            and $\ell_\infty(\Lambda)$-modules in topological theory
        }
    } \\
    \hline & 
    \multicolumn{3}{c|}{
        $c_0(\Lambda)$-modules
    } & 
    \multicolumn{3}{c|}{
        $\ell_\infty(\Lambda)$-modules
    } \\
    \hline & 
        \mbox{Projectivity} & 
        \mbox{Injectivity} & 
        \mbox{Flatness} & 
        \mbox{Projectivity} & 
        \mbox{Injectivity} & 
        \mbox{Flatness} \\ 
    \hline 
        $\ell_1(\Lambda)$ & 
        \begin{tabular}{@{}c@{}}
            $\Lambda$\mbox{ is any }  \\
            \mbox{\ref{c0AndlInftyModl1}}
        \end{tabular} & 
        \begin{tabular}{@{}c@{}}
            $\Lambda$\mbox{ is any } \\
            \mbox{\ref{c0AndlInftyModl1}}
        \end{tabular} & 
        \begin{tabular}{@{}c@{}}
            $\Lambda$\mbox{ is any } \\
            \mbox{\ref{c0AndlInftyModl1}}
        \end{tabular} & 
        \begin{tabular}{@{}c@{}}
            $\Lambda$\mbox{ is any }  \\
            \mbox{\ref{c0AndlInftyModl1}}
        \end{tabular} & 
        \begin{tabular}{@{}c@{}}
            $\Lambda$\mbox{ is any } \\
            \mbox{\ref{c0AndlInftyModl1}}
        \end{tabular} & 
        \begin{tabular}{@{}c@{}}
            $\Lambda$\mbox{ is any } \\
            \mbox{\ref{c0AndlInftyModl1}}
        \end{tabular} \\
    \hline
        $\ell_p(\Lambda)$ & 
        \begin{tabular}{@{}c@{}}
            $\operatorname{Card}(\Lambda)<\aleph_0$ \\
            \mbox{\ref{c0AndlInftyModlp}}
        \end{tabular} & 
        \begin{tabular}{@{}c@{}}
            $\operatorname{Card}(\Lambda)<\aleph_0$ \\
            \mbox{\ref{c0AndlInftyModlp}}
        \end{tabular} & 
        \begin{tabular}{@{}c@{}}
            $\operatorname{Card}(\Lambda)<\aleph_0$ \\
            \mbox{\ref{c0AndlInftyModlp}}
        \end{tabular} & 
        \begin{tabular}{@{}c@{}}
            $\operatorname{Card}(\Lambda)<\aleph_0$ \\
            \mbox{\ref{c0AndlInftyModlp}}
        \end{tabular} & 
        \begin{tabular}{@{}c@{}}
            $\operatorname{Card}(\Lambda)<\aleph_0$ \\
            \mbox{\ref{c0AndlInftyModlp}}
        \end{tabular} & 
        \begin{tabular}{@{}c@{}}
            $\operatorname{Card}(\Lambda)<\aleph_0$ \\
            \mbox{\ref{c0AndlInftyModlp}}
        \end{tabular} \\
    \hline
        $\ell_\infty(\Lambda)$ & 
        \begin{tabular}{@{}c@{}}
            $\operatorname{Card}(\Lambda)<\aleph_0$ \\
            \mbox{\ref{c0AndlInftyModlIfty}}
        \end{tabular} & 
        \begin{tabular}{@{}c@{}}
            $\Lambda$\mbox{ is any } \\
            \mbox{\ref{c0AndlInftyModlIfty}}
        \end{tabular} & 
        \begin{tabular}{@{}c@{}}
            $\operatorname{Card}(\Lambda)<\aleph_0$ \\
            \mbox{\ref{c0AndlInftyModlIfty}}
        \end{tabular} & 
        \begin{tabular}{@{}c@{}}
            $\Lambda$\mbox{ is any } \\
            \mbox{\ref{c0AndlInftyModlIfty}}
        \end{tabular} & 
        \begin{tabular}{@{}c@{}}
            $\Lambda$\mbox{ is any } \\
            \mbox{\ref{c0AndlInftyModlIfty}}
        \end{tabular} & 
        \begin{tabular}{@{}c@{}}
            $\Lambda$\mbox{ is any } \\
            \mbox{\ref{c0AndlInftyModlIfty}}
        \end{tabular} \\ 
    \hline
        $c_0(\Lambda)$ &
        \begin{tabular}{@{}c@{}}
            $\operatorname{Card}(\Lambda)<\aleph_0$ \\
            \mbox{\ref{c0AndlInftyModc0}}
        \end{tabular} & 
        \begin{tabular}{@{}c@{}}
            $\operatorname{Card}(\Lambda)<\aleph_0$ \\
            \mbox{\ref{c0AndlInftyModc0}}
        \end{tabular} & 
        \begin{tabular}{@{}c@{}}
            $\Lambda$\mbox{ is any } \\
            \mbox{\ref{c0AndlInftyModc0}}
        \end{tabular} & 
        \begin{tabular}{@{}c@{}}
            $\operatorname{Card}(\Lambda)<\aleph_0$ \\
            \mbox{\ref{c0AndlInftyModc0}}
        \end{tabular} & 
        \begin{tabular}{@{}c@{}}
            $\operatorname{Card}(\Lambda)<\aleph_0$ \\
            \mbox{\ref{c0AndlInftyModc0}}
        \end{tabular} & 
        \begin{tabular}{@{}c@{}}
            $\Lambda$\mbox{ is any } \\
            \mbox{\ref{c0AndlInftyModc0}}
        \end{tabular} \\ 
    \hline
        $\mathbb{C}_\lambda$ & 
        \begin{tabular}{@{}c@{}}
            $\lambda$\mbox{ is any } \\
            \mbox{\ref{OneDimlInftyc0ModMetTopProjIngFlat}}
        \end{tabular} & 
        \begin{tabular}{@{}c@{}}
            $\lambda$\mbox{ is any } \\
            \mbox{\ref{OneDimlInftyc0ModMetTopProjIngFlat}}
        \end{tabular} & 
        \begin{tabular}{@{}c@{}}
            $\lambda$\mbox{ is any } \\
            \mbox{\ref{OneDimlInftyc0ModMetTopProjIngFlat}}
        \end{tabular} & 
        \begin{tabular}{@{}c@{}}
            $\lambda$\mbox{ is any } \\
            \mbox{\ref{OneDimlInftyc0ModMetTopProjIngFlat}}
        \end{tabular} & 
        \begin{tabular}{@{}c@{}}
            $\lambda$\mbox{ is any } \\
            \mbox{\ref{OneDimlInftyc0ModMetTopProjIngFlat}}
        \end{tabular} & 
        \begin{tabular}{@{}c@{}}
            $\lambda$\mbox{ is any } \\
            \mbox{\ref{OneDimlInftyc0ModMetTopProjIngFlat}}
        \end{tabular} \\
    \hline
    \multicolumn{7}{c}{
        \mbox{
            Homologically trivial $c_0(\Lambda)$- 
            and $\ell_\infty(\Lambda)$-modules in relative theory
        }
    } \\
    \hline & 
    \multicolumn{3}{c|}{
        $c_0(\Lambda)$-modules
    } & 
    \multicolumn{3}{c|}{
        $\ell_\infty(\Lambda)$-modules
    } \\
    \hline & 
        \mbox{Projectivity} & 
        \mbox{Injectivity} & 
        \mbox{Flatness} & 
        \mbox{Projectivity} & 
        \mbox{Injectivity} & 
        \mbox{Flatness} \\ 
    \hline
        $\ell_1(\Lambda)$ & 
        \begin{tabular}{@{}c@{}}
            $\Lambda$\mbox{ is any } \\
            \mbox{\ref{c0AndlInftyModsRelTh}, (i)}
        \end{tabular} & 
        \begin{tabular}{@{}c@{}}
            $\Lambda$\mbox{ is any } \\
            \mbox{\ref{c0AndlInftyModsRelTh}, (i)}
        \end{tabular} & 
        \begin{tabular}{@{}c@{}}
            $\Lambda$\mbox{ is any } \\
            \mbox{\ref{c0AndlInftyModsRelTh}, (i)}
        \end{tabular} & 
        \begin{tabular}{@{}c@{}}
            $\Lambda$\mbox{ is any } \\
            \mbox{\ref{c0AndlInftyModsRelTh}, (ii)}
        \end{tabular} & 
        \begin{tabular}{@{}c@{}}
            $\Lambda$\mbox{ is any } \\
            \mbox{\ref{c0AndlInftyModsRelTh}, (ii)}
        \end{tabular} & 
        \begin{tabular}{@{}c@{}}
            $\Lambda$\mbox{ is any } \\
            \mbox{\ref{c0AndlInftyModlIfty}, (ii)}
        \end{tabular} \\
    \hline 
        $\ell_p(\Lambda)$ & 
        \begin{tabular}{@{}c@{}}
            $\Lambda$\mbox{ is any } \\
            \mbox{\ref{c0AndlInftyModsRelTh}, (i)}
        \end{tabular} & 
        \begin{tabular}{@{}c@{}}
            $\Lambda$\mbox{ is any }  \\
            \mbox{\ref{c0AndlInftyModsRelTh}, (i)}
        \end{tabular} & 
        \begin{tabular}{@{}c@{}}
            $\Lambda$\mbox{ is any } \\
            \mbox{\ref{c0AndlInftyModsRelTh}, (i)}
        \end{tabular} & 
        \begin{tabular}{@{}c@{}}
            $\Lambda$\mbox{ is any }  \\
            \mbox{\ref{c0AndlInftyModsRelTh}, (ii)}
        \end{tabular} & 
        \begin{tabular}{@{}c@{}}
            $\Lambda$\mbox{ is any } \\
            \mbox{\ref{c0AndlInftyModsRelTh}, (ii)}
        \end{tabular} & 
        \begin{tabular}{@{}c@{}}
            $\Lambda$\mbox{ is any } \\
            \mbox{\ref{c0AndlInftyModlIfty}, (ii)}
        \end{tabular} \\
    \hline
        $\ell_\infty(\Lambda)$ & 
        \begin{tabular}{@{}c@{}}
            $\operatorname{Card}(\Lambda)<\aleph_0$ \\
            \mbox{\ref{c0AndlInftyModsRelTh}, (i)}
        \end{tabular} & 
        \begin{tabular}{@{}c@{}}
            $\Lambda$\mbox{ is any } \\
            \mbox{\ref{c0AndlInftyModsRelTh}, (i)}
        \end{tabular} & 
        \begin{tabular}{@{}c@{}}
            $\Lambda$\mbox{ is any } \\
            \mbox{\ref{c0AndlInftyModsRelTh}, (i)}
        \end{tabular} & 
        \begin{tabular}{@{}c@{}}
            $\Lambda$\mbox{ is any }  \\
            \mbox{\ref{c0AndlInftyModsRelTh}, (ii)}
        \end{tabular} & 
        \begin{tabular}{@{}c@{}}
            $\Lambda$\mbox{ is any } \\
            \mbox{\ref{c0AndlInftyModsRelTh}, (ii)}
        \end{tabular} & 
        \begin{tabular}{@{}c@{}}
            $\Lambda$\mbox{ is any } \\
            \mbox{\ref{c0AndlInftyModlIfty}, (ii)}
        \end{tabular} \\
    \hline
        $c_0(\Lambda)$ &
        \begin{tabular}{@{}c@{}}
            $\Lambda$\mbox{ is any } \\
            \mbox{\ref{c0AndlInftyModsRelTh}, (i)}
        \end{tabular} & 
        \begin{tabular}{@{}c@{}}
            $\operatorname{Card}(\Lambda)<\aleph_0$ \\
            \mbox{\ref{c0AndlInftyModsRelTh}, (i) }
        \end{tabular} & 
        \begin{tabular}{@{}c@{}}
            $\Lambda$\mbox{ is any } \\
            \mbox{\ref{c0AndlInftyModsRelTh}, (i)}
        \end{tabular} & 
        \begin{tabular}{@{}c@{}}
            $\Lambda$\mbox{ is any } \\
            \mbox{\ref{c0AndlInftyModsRelTh}, (ii)}
        \end{tabular} & 
        \begin{tabular}{@{}c@{}} 
            $\Lambda$\mbox{ is any } \\
            \mbox{\ref{c0AndlInftyModsRelTh}, (ii)} 
        \end{tabular} & 
        \begin{tabular}{@{}c@{}}
            $\Lambda$\mbox{ is any }  \\
            \mbox{\ref{c0AndlInftyModlIfty}, (ii)}
        \end{tabular} \\
    \hline 
        $\mathbb{C}_\lambda$ & 
        \begin{tabular}{@{}c@{}}
            $\lambda$\mbox{ is any } \\
            \mbox{\ref{c0AndlInftyModsRelTh}, (i)}
        \end{tabular} & 
        \begin{tabular}{@{}c@{}}
            $\lambda$\mbox{ is any }  \\
            \mbox{\ref{c0AndlInftyModsRelTh}, (i)}
        \end{tabular} & 
        \begin{tabular}{@{}c@{}}
            $\lambda$\mbox{ is any } \\
            \mbox{\ref{c0AndlInftyModsRelTh}, (i)}
        \end{tabular} & 
        \begin{tabular}{@{}c@{}}
            $\lambda$\mbox{ is any }  \\
            \mbox{\ref{c0AndlInftyModsRelTh}, (ii)}
        \end{tabular} & 
        \begin{tabular}{@{}c@{}}
            $\lambda$\mbox{ is any } \\
            \mbox{\ref{c0AndlInftyModsRelTh}, (ii)}
        \end{tabular} & 
        \begin{tabular}{@{}c@{}}
            $\lambda$\mbox{ is any }  \\
            \mbox{\ref{c0AndlInftyModlIfty}, (ii)}
        \end{tabular} \\
    \hline
    \end{longtable}
\end{scriptsize}

%-------------------------------------------------------------------------------
%	C_0(S)-modules
%-------------------------------------------------------------------------------

\subsection{
    \texorpdfstring{$C_0(S)$}{C0(S)}-modules
}\label{SubSectionC0SModules}

This section is devoted to study of homological triviality of classical modules
over algebra $C_0(S)$, where $S$ is a locally compact Hausdorff space. By
classical we mean modules $C_0(S)$, $M(S)$ and $L_p(S,\mu)$ for positive measure
$\mu\in M(S)$. The pointwise multiplication plays the role of outer action for
these modules.

Further we give short preliminaries on these modules. Recall that
${C_0(S)}^*\isom{\mathbf{Ban}_1}M(S)$ and
${L_p(S,\mu)}^*\isom{\mathbf{Ban}_1}L_{p^*}(S,\mu)$ for $1\leq p<+\infty$. 
In fact these identifications are isomorphisms of left and 
right $C_0(S)$-modules. For a given positive measure $\mu\in M(S)$ 
by $M_s(S,\mu)$ we shall denote the closed $C_0(S)$-submodule of $M(S)$ 
consisting of measures strictly singular with respect to $\mu$. Then the well 
known Lebesgue decomposition theorem can be stated 
as $M(S)\isom{\mathbf{Ban}_1}L_1(S,\mu)\bigoplus_1 M_s(S,\mu)$. Even
more, this identification is an isomorphism of left and right $C_0(S)$-modules. 

We obliged to emphasize here that we consider only finite Borel regular positive
measures. This shall simplify many considerations. For example, any atom of
regular measure on a locally compact Hausdorff spaces is a point
[\cite{BourbElemMathIntegLivVI}, chapter 5, \S 5, exercise 7]. Since we consider
only finite measures, one can not say that this section simply generalizes
results of the previous one. Strictly speaking these sections are different,
though their methods have much in common.

For a fixed point $s\in S$ by $\mathbb{C}_s$ we denote left or right Banach
$C_0(S)$-module $\mathbb{C}$ with outer action defined by
$$
a\cdot_s z=a(s)z,\qquad z\cdot_s a=a(s)z
$$
\begin{proposition}\label{OneDimC0SModMetTopRelProjIngFlat} Let $S$ be a locally
compact Hausdorff space and let $s\in S$. Then 

\begin{enumerate}[label = (\roman*)]
    \item $\mathbb{C}_s$ is metrically, topologically or relatively projective 
    as $C_0(S)$-module iff $s$ is an isolated point of $S$;

    \item $\mathbb{C}_s$ is metrically, topologically and relatively flat as
    $C_0(S)$-module.

    \item $\mathbb{C}_s$ is metrically, topologically and relatively injective 
    as $C_0(S)$-module;

\end{enumerate}
\end{proposition}
\begin{proof} $(i)$ If $\mathbb{C}_s$ is metrically or topologically or
relatively projective, then by
proposition~\ref{MetProjIsTopProjAndTopProjIsRelProj} it is at least relatively
projective. Now from [\cite{HelBanLocConvAlg}, proposition 7.1.31] we know that
the latter forces $s$ to be an isolated point of $S$. Conversely, assume $s$ is
an isolated point of $S$. One can easily check that the maps
$\pi:{C_0(S)}_+\to\mathbb{C}_s:a\oplus_1 z\mapsto a(s)+z$ and
$\sigma:\mathbb{C}_s\to {C_0(S)}_+:z\mapsto z\delta_s\oplus_1 0$ are contractive
$C_0(S)$-morphisms. Since $\pi\sigma=1_{\mathbb{C}_s}$, then $\mathbb{C}_s$ is a
retract of ${C_0(S)}_+$ in $C_0(S)-\mathbf{mod}_1$. From
propositions~\ref{RetrMetCTopProjIsMetCTopProj} and~\ref{UnitalAlgIsMetTopProj} 
it follows that $\mathbb{C}_s$ is metrically and topologically projective left
$C_0(S)$-module. From~\ref{MetProjIsTopProjAndTopProjIsRelProj} we conclude that
$\mathbb{C}_s$ is also relatively projective $C_0(S)$-module.

$(ii)$ By [\cite{HelBanLocConvAlg}, theorem 7.1.87] the algebra $C_0(S)$ is
relatively amenable. Since this algebra is a $C^*$-algebra it is $1$-relatively 
amenable [\cite{RundeAmenConstFour}, example 3]. Clearly, $\mathbb{C}_s$ is a
$1$-dimensional $L_1$-space and an essential $C_0(S)$-module. Therefore, by
proposition~\ref{MetTopEssL1FlatModAoverAmenBanAlg} this module is metrically
flat. Now the result follows from
proposition~\ref{MetFlatIsTopFlatAndTopFlatIsRelFlat}.

$(iii)$ From paragraph $(ii)$ and proposition~\ref{MetCTopFlatCharac} it follows
that $\mathbb{C}_s^*$ is metrically injective. By
proposition~\ref{MetFlatIsTopFlatAndTopFlatIsRelFlat} it is topologically and
relatively injective too. It remains to note that
$\mathbb{C}_s\isom{\mathbf{mod}_1-A}\mathbb{C}_s^*$. 
\end{proof}

\begin{proposition}\label{C0SC0SModMetTopRelProjInjFlat} Let $S$ be a locally
compact Hausdorff space and let $s\in S$. Then

\begin{enumerate}[label = (\roman*)]
    \item $C_0(S)$ is $\langle$~metrically / 
    topologically / relatively~$\rangle$
    projective as $C_0(S)$-module iff $S$ is $\langle$~compact / compact /
    paracompact~$\rangle$;

    \item $C_0(S)$ is metrically injective as $C_0(S)$-module iff $S$ is 
    a Stonean space, if $C_0(S)$ is topologically or relatively injective then
    $S$ is compact and $S=\beta(S\setminus \{ s\})$ for any limit point $s$;

    \item $C_0(S)$ is metrically, topologically and relatively flat as
    $C_0(S)$-module.
\end{enumerate}
\end{proposition}
\begin{proof} We regard $C_0(S)$ as a two-sided ideal of $C_0(S)$. Recall that
$\operatorname{Spec}(C_0(S))$ is homeomorphic to $S$ [\cite{HelHomolBanTopAlg},
corollary 3.1.6].

$(i)$ It is enough to note that by
$\langle$~proposition~\ref{IdealofCommCStarAlgMetTopProjCharac} /
proposition~\ref{IdealofCommCStarAlgMetTopProjCharac} /
[\cite{HelHomolBanTopAlg}, chapter IV,\S\S 2--3]~$\rangle$ the spectrum of
$C_0(S)$ is $\langle$~compact / compact / paracompact~$\rangle$. 

$(ii)$ The result on metric injectivity is a weakened version of
proposition~\ref{MetInjCStarAlgCharac}. The result on relatively injectivity 
follows from~[\cite{NemANoteOnRelInjC0ModC0}, theorem 4.4]. It remains to recall
that by proposition~\ref{MetInjIsTopInjAndTopInjIsRelInj} metrically injective
module is relatively injective.

$(iii)$ From proposition~\ref{IdealofCstarAlgisMetTopFlat} it immediately follows
that $C_0(S)$-module $C_0(S)$ is metrically and topologically flat. By
proposition~\ref{MetFlatIsTopFlatAndTopFlatIsRelFlat} it is also relatively
flat.
\end{proof}

\begin{proposition}\label{AtomsOfRelProjLpMod} Let $S$ be a locally compact
Hausdorff space, $\mu$ be a finite Borel regular positive measure on $S$. Assume
$1\leq p\leq+\infty$ and $C_0(S)$-module $L_p(S,\mu)$ is relatively projective.
Then any atom of $\mu$ is an isolated point in $S$.
\end{proposition} 
\begin{proof} Assume $\mu$ has at least one atom, otherwise there is nothing to
prove. From [\cite{BourbElemMathIntegLivVI}, chapter 5, \S 5, exercise 7] we
know that any atom of $\mu$ is a point. Call it $s$. Consider well defined
linear maps $\pi:L_p(\Omega,\mu)\to\mathbb{C}_s:f\mapsto f(s)$ and
$\sigma:\mathbb{C}_s\to L_p(\Omega,\mu):z\mapsto z\delta_s$. One can easily
check that these maps are $C_0(S)$-morphisms and $\pi\sigma=1_{\mathbb{C}_s}$.
Therefore, $\mathbb{C}_s$ is a retract of $L_p(S,\mu)$ in $C_0(S)-\mathbf{mod}$.
By assumption, the latter module is relatively projective, so by
[\cite{HelBanLocConvAlg}, proposition 7.1.6] the $C_0(S)$-module $\mathbb{C}_s$
is relatively projective. By paragraph $(i)$ of
proposition~\ref{C0SC0SModMetTopRelProjInjFlat} we see that $s$ is an isolated
point of $S$.
\end{proof}

\begin{proposition}\label{L1C0SModMetTopRelProjInjFlat} Let $S$ be a locally
compact Hausdorff space and $\mu$ be a finite Borel regular positive measure on
$S$. Then 

\begin{enumerate}[label = (\roman*)]
    \item $L_1(S,\mu)$ is metrically or topologically or relatively projective 
    as $C_0(S)$-module iff $\mu$ is purely atomic and all its atoms are 
    isolated points in $S$;

    \item $L_1(S,\mu)$ is metrically, topologically and relatively injective as
    $C_0(S)$-module;

    \item $L_1(S,\mu)$ is metrically, topologically and relatively flat as
    $C_0(S)$-module.
\end{enumerate}
\end{proposition}
\begin{proof} $(i)$ If $L_1(S,\mu)$ is metrically or topologically or relatively
projective, then by proposition~\ref{MetProjIsTopProjAndTopProjIsRelProj} it is
at least relatively projective. Now from [\cite{NemRelProjModLp}, theorem 2] 
the measure $\mu$ is purely atomic and all atoms are isolated  points.
Conversely, assume that $\mu$ is purely atomic and all atoms are isolated
points. By $S_a^{\mu}$ we denote the set of these atoms. Now one can easily show
that the linear map 
$i:L_1(S,\mu)\to\bigoplus_1 \{\mathbb{C}_s:s\in S_a^{\mu} \}
:f\mapsto \bigoplus_1 \{\mu( \{s \})f(s):s\in S_a^{\mu} \}$ is an isometric
isomorphism of $C_0(S)$-modules. By paragraphs $(i)$ of
propositions~\ref{MetTopProjModCoprod}
and~\ref{OneDimC0SModMetTopRelProjIngFlat} the $C_0(S)$-module $\bigoplus_1
\{\mathbb{C}_s:s\in S_a^{\mu} \}$ is metrically projective. Therefore so does
$L_1(S,\mu)$. By proposition~\ref{MetProjIsTopProjAndTopProjIsRelProj} it is
also topologically and relatively projective.

$(ii)$ By paragraph $(iii)$ of proposition~\ref{C0SC0SModMetTopRelProjInjFlat} 
the $C_0(S)$-module $C_0(S)$ is metrically flat. From
proposition~\ref{DualMetTopProjIsMetrInj} we get that
$M(S)\isom{\mathbf{mod}_1-C_0(S)}{C_0(S)}^*$ is metrically injective. Since
$M(S)\isom{\mathbf{mod}_1-C_0(S)}L_1(S,\mu)\bigoplus_1 M_s(S,\mu)$, then
$L_1(S,\mu)$ is a retract of $M(S)$ in $\mathbf{mod}_1-C_0(S)$. So by
proposition~\ref{RetrMetCTopInjIsMetCTopInj} the $C_0(S)$-module $L_1(S,\mu)$ is
metrically injective. Relative and topological injectivity of $L_1(S,\mu)$
follows from proposition~\ref{MetInjIsTopInjAndTopInjIsRelInj}.

$(iii)$ By [\cite{HelBanLocConvAlg}, theorem 7.1.87] the algebra $C_0(S)$ is
relatively amenable. Since this algebra is a $C^*$-algebra it is $1$-relatively 
amenable [\cite{RundeAmenConstFour}, example 3]. Since $L_1(S,\mu)$ is 
an essential $C_0(S)$-module which tautologically an $L_1$-space, then by
proposition~\ref{MetTopEssL1FlatModAoverAmenBanAlg} this module is metrically
flat. From proposition~\ref{MetFlatIsTopFlatAndTopFlatIsRelFlat} the
$C_0(S)$-module $L_1(S,\mu)$ is also topologically and relatively flat.
\end{proof}

\begin{proposition}\label{LpC0SModMetTopRelProjIngFlat} Let $S$ be a locally
compact Hausdorff space and $\mu$ be a finite Borel regular positive measure on
$S$. Assume $1<p<+\infty$, then 

\begin{enumerate}[label = (\roman*)]
    \item $L_p(S,\mu)$ is relatively injective and flat, but relatively 
    projective iff $\mu$ is purely atomic and all atoms are isolated points;

    \item $L_p(S,\mu)$ is topologically projective or injective or flat 
    iff $\mu$ is purely atomic with finitely many atoms;

    \item if $L_p(S,\mu)$ is metrically projective or injective or flat, 
    then $\mu$ is purely atomic with finitely many atoms.
\end{enumerate}
\end{proposition}
\begin{proof} $(i)$ By [\cite{HelBanLocConvAlg}, theorem 7.1.87] the algebra
$C_0(S)$ is relatively amenable. Now from [\cite{HelBanLocConvAlg}, theorem
7.1.60] it follows that $L_p(S,\mu)$ is relatively flat for all $1<p<+\infty$.
Note that $L_p(S,\mu)\isom{\mathbf{mod}_1-C_0(S)}{L_{p^*}(S,\mu)}^*$. Then from
[\cite{HelBanLocConvAlg}, proposition 7.1.42] we get that $L_p(S,\mu)$ is
relatively injective for any $1<p<+\infty$. Now assume that $L_p(S,\mu)$ is
relatively projective, then by [\cite{NemRelProjModLp}, theorem 2]
the measure $\mu$ is purely atomic and all its atoms are isolated points.
Conversely, let $\mu$ be purely atomic with all atoms isolated. Denote by
$S_a^{\mu}$ the set of these atoms. Since $S_a^{\mu}$ is discrete, then
$C_0(S_a^{\mu})$ is relatively biprojective [\cite{HelHomolBanTopAlg}, theorem
4.5.26]. Then $L_p(S,\mu)$ is relatively projective $C_0(S_a^{\mu})$-module
because it is essential module over relatively biprojective algebra with
two-sided bounded approximate identity. Clearly $C_0(S_a^{\mu})$ is a two-sided
ideal of $C_0(S)$, so from [\cite{RamsHomPropSemgroupAlg}, proposition 2.3.2(i)]
we get that $L_p(S,\mu)$ is relatively projective as $C_0(S)$-module.

$(ii), (iii)$ Assume that $L_p(S,\mu)$ is metrically or topologically projective
or injective or flat $C_0(S)$-module. Since $L_p(S,\mu)$ is reflexive and
$C_0(S)$ is an $\mathscr{L}_\infty^g$-space, then $L_p(S,\mu)$ is finite
dimensional by
corollary~\ref{NoInfDimRefMetTopProjInjFlatModOverMthscrL1OrLInfty}. The latter
is equivalent to measure $\mu$ being purely atomic with finitely many atoms. On
the other hand, if $\mu$ is purely atomic with finitely many atoms, then
$L_p(S,\mu)$ is topologically isomorphic to $L_1(S,\mu)$ as left or right
$C_0(S)$-module. The latter module is topologically projective, injective and
flat for our measure by proposition~\ref{L1C0SModMetTopRelProjInjFlat}. Hence so
does $L_p(S,\mu)$.
\end{proof}

\begin{proposition}\label{LinftyC0SModMetTopRelProjIngFlat} Let $S$ be a locally
compact Hausdorff space and $\mu$ be a finite Borel regular positive measure on
$S$. Then

\begin{enumerate}[label = (\roman*)]
    \item if $L_\infty(S,\mu)$ is metrically, topologically or relatively 
    projective, then $\mu$ is normal and with is pseudocompact support; 

    \item $L_\infty(S,\mu)$ is metrically, topologically and relatively 
    injective as $C_0(S)$-module;

    \item $L_\infty(S,\mu)$ is relatively flat $C_0(S)$-module.
\end{enumerate}
\end{proposition}
\begin{proof} $(i)$ From [\cite{NemRelProjModLp}, theorem 3] it follows 
that $\operatorname{supp}(\mu)$ is pseudocompact and $\mu$ is inner open 
regular. Since $\mu$ is regular and finite it is 
normal [\cite{NemRelProjModLp}, proposition 9].
    
$(ii)$ Since
$L_\infty(S,\mu)\isom{\mathbf{mod}_1-C_0(S)}{L_1(S,\mu)}^*$, then the result
immediately follows from proposition~\ref{DualMetTopProjIsMetrInj} and paragraph
$(iii)$ of proposition~\ref{L1C0SModMetTopRelProjInjFlat}.

$(iii)$ By [\cite{HelBanLocConvAlg}, theorem 7.1.87] the algebra $C_0(S)$ is
relatively amenable. Any left Banach module over relatively amenable Banach
algebra is relatively flat [\cite{HelBanLocConvAlg}, theorem 7.1.60]. In
particular $L_\infty(S,\mu)$ is relatively flat $C_0(S)$-module.
\end{proof}

\begin{proposition}\label{MSC0SModMetTopRelProjIngFlat} Let $S$ be a locally
compact Hausdorff space and $\mu$ be a finite Borel regular positive measure on
$S$. Then

\begin{enumerate}[label = (\roman*)]
    \item $M(S)$ is metrically or topologically or relatively projective as
    $C_0(S)$-module iff $S$ is discrete; 

    \item $M(S)$ is metrically, topologically and relatively injective as
    $C_0(S)$-module; 

    \item $M(S)$ is metrically, topologically and relatively flat as
    $C_0(S)$-module.
\end{enumerate}
\end{proposition}
\begin{proof} $(i)$ If $M(S)$ is metrically or topologically or relatively
projective, then by proposition~\ref{MetProjIsTopProjAndTopProjIsRelProj} it is
at least relatively projective. For arbitrary $s\in S$ consider measure
$\mu=\delta_s$ and recall the decomposition
$M(S)\isom{C_0(S)-\mathbf{mod}_1}L_1(S,\mu)\bigoplus_1 M_s(S,\mu)$. Then
$L_1(S,\mu)$ is a retract of $M(S)$ in $C_0(S)-\mathbf{mod}_1$. So from
[\cite{HelBanLocConvAlg}, proposition 7.1.6] we get that $L_1(S,\mu)$ is
relatively projective $C_0(S)$-module. Since $s$ is the only atom of $\mu$, then
from proposition~\ref{L1C0SModMetTopRelProjInjFlat} it follows that $s$ is an
isolated point in $S$. Since $s\in S$ is arbitrary, then $S$ is discrete.
Conversely, assume $S$ is discrete. Then $C_0(S)=c_0(S)$, and
$M(S)
\isom{C_0(S)-\mathbf{mod}_1}{C_0(S)}^*
\isom{C_0(S)-\mathbf{mod}_1}\ell_1(S)
\isom{C_0(S)-\mathbf{mod}_1}
\bigoplus_1 \{\mathbb{C}_s:s\in S \}$. The latter $C_0(S)$-module is metrically
projective by paragraphs $(i)$ of propositions~\ref{MetTopProjModCoprod}
and~\ref{OneDimC0SModMetTopRelProjIngFlat}. Therefore $M(S)$ is metrically
projective $C_0(S)$-module too. By
proposition~\ref{MetProjIsTopProjAndTopProjIsRelProj} it is also topologically
and relatively projective.

$(ii)$ Since $M(S)\isom{\mathbf{mod}_1-C_0(S)}{C_0(S)}^*$, then the result
immediately follows from proposition~\ref{DualMetTopProjIsMetrInj} and paragraph
$(iii)$ of proposition~\ref{C0SC0SModMetTopRelProjInjFlat}.

$(iii)$ By [\cite{HelBanLocConvAlg}, theorem 7.1.87] the algebra $C_0(S)$ is
relatively amenable. Since this algebra is a $C^*$-algebra it is $1$-relatively 
amenable [\cite{RundeAmenConstFour}, example 2]. Note that $M(S)$ is 
an essential $C_0(S)$-module which as Banach space is an $L_1$-space
[\cite{DalLauSecondDualOfMeasAlg}, discussion after proposition 2.14]. Then by
proposition~\ref{MetTopEssL1FlatModAoverAmenBanAlg} this module is metrically
flat. From proposition~\ref{MetFlatIsTopFlatAndTopFlatIsRelFlat} it is also
topologically and relatively flat.
\end{proof}

Results of this section are summarized in the following three tables. Each cell
contains a condition under which the respective module has the respective
property and propositions where it is proved. We use ``?'' symbol to indicate 
open problems. Open problems of this section are divided into three parts:
injectivity of $C_0(S)$, projectivity of $L_\infty(S,\mu)$ and flatness of
$L_\infty(S,\mu)$. Complete description of relatively and topologically
injective $C_0(S)$-modules $C_0(S)$ seems quite a challenge for one simple
reason --- still there is no standard category of functional analysis where even
topologically injective objects were fully understood. The question of relative
projectivity of $C_0(S)$-module $L_\infty(S,\mu)$ is rather old. It seems that
even relative projectivity of $L_\infty(S,\mu)$ is a rare property. Our
conjecture that $\mu$ must be purely atomic with finitely many atoms. Finally we
presume that a necessary condition for metric and topological flatness of
$C_0(S)$-module $L_\infty(S,\mu)$ is compactness of $S$.

This section cotaines many other unsolved problems. Usually we have a strong
necessary condition. We use ${}^{*}$ to indicate that. As for partial results, 
we don't have a criterion of homological triviality of $C_0(S)$-modules 
$L_p(S,\mu)$ in metric theory for $1<p<+\infty$. Using advanced
Banach geometric techniques on factorization constants through finite
dimensional Hilbert spaces one may show that atoms count for metrically
projective modules $L_p(S,\mu)$ doesn't exceed some universal constant. It seems
that $L_p(S,\mu)$ is homologically trivial $C_0(S)$-module in metric theory only
for purely atomic measures with unique atom. 

\begin{scriptsize}
    \begin{longtable}{|c|c|c|c|} 
    \multicolumn{4}{c}{
        \mbox{
            Homologically trivial $C_0(S)$-modules in metric theory
        }
    }                                                                                                                                                                                                                                                                                                                                                                                                                               \\
    \hline & 
    \mbox{Projectivity} & 
    \mbox{Injectivity} & 
    \mbox{Flatness} \\
    \hline
        $L_1(S,\mu)$ & 
        \begin{tabular}{@{}c@{}}
            $\mu$\mbox{ is purely atomic, all } \\ 
            \mbox{ atoms are isolated points } \\
            \mbox{\ref{L1C0SModMetTopRelProjInjFlat}} (i)
        \end{tabular} & 
        \begin{tabular}{@{}c@{}}
            $\mu$\mbox{ is any }  \\
            \mbox{\ref{L1C0SModMetTopRelProjInjFlat}} (ii)
        \end{tabular} & 
        \begin{tabular}{@{}c@{}}
            $\mu$\mbox{ is any }  \\
            \mbox{\ref{L1C0SModMetTopRelProjInjFlat}} (iii)
        \end{tabular} \\
    \hline
        $L_p(S,\mu)$ & 
        \begin{tabular}{@{}c@{}}
            $\mu$\mbox{ is purely atomic } \\ 
            \mbox{ with finitely many atoms } \\ 
            \mbox{\ref{LpC0SModMetTopRelProjIngFlat}} (iii)${}^{*}$
        \end{tabular} & 
        \begin{tabular}{@{}c@{}}
            $\mu$\mbox{ is purely atomic } \\ 
            \mbox{ with finitely many atoms } \\ 
            \mbox{\ref{LpC0SModMetTopRelProjIngFlat}} (iii)${}^{*}$
        \end{tabular} & 
        \begin{tabular}{@{}c@{}}
            $\mu$\mbox{ is purely atomic } \\ 
            \mbox{ with finitely many atoms } \\ 
            \mbox{\ref{LpC0SModMetTopRelProjIngFlat}} (iii)${}^{*}$
        \end{tabular} \\
    \hline
        $L_\infty(S,\mu)$ & 
        \begin{tabular}{@{}c@{}} 
            $\mu$ is normal, with \\
            pseudocompact support \\
            \mbox{\ref{LinftyC0SModMetTopRelProjIngFlat}} (i)${}^{*}$
        \end{tabular} & 
        \begin{tabular}{@{}c@{}}
            $\mu$\mbox{ is any } \\
            \mbox{\ref{LinftyC0SModMetTopRelProjIngFlat}} (ii)
        \end{tabular} & 
        \begin{tabular}{@{}c@{}} 
            {?}
        \end{tabular} \\
    \hline
        $M(S)$ & 
        \begin{tabular}{@{}c@{}}
            $S$\mbox{ is discrete } \\
            \mbox{\ref{MSC0SModMetTopRelProjIngFlat}}
        \end{tabular} & 
        \begin{tabular}{@{}c@{}}
            $S$\mbox{ is any } \\
            \mbox{\ref{MSC0SModMetTopRelProjIngFlat}}
        \end{tabular} & 
        \begin{tabular}{@{}c@{}}
            $S$\mbox{ is any } \\
            \mbox{\ref{MSC0SModMetTopRelProjIngFlat}}
        \end{tabular} \\
    \hline
        $C_0(S)$ & 
        \begin{tabular}{@{}c@{}}
            $S$\mbox{ is compact } \\
            \mbox{\ref{C0SC0SModMetTopRelProjInjFlat}} (i)
        \end{tabular} & 
        \begin{tabular}{@{}c@{}}
            $S$\mbox{ is Stonean } \\
            \mbox{\ref{C0SC0SModMetTopRelProjInjFlat}} (ii) 
        \end{tabular} & 
        \begin{tabular}{@{}c@{}}
            $S$\mbox{ is any } \\
            \mbox{\ref{C0SC0SModMetTopRelProjInjFlat}} (iii)
        \end{tabular} \\
    \hline
        $\mathbb{C}_s$ & 
        \begin{tabular}{@{}c@{}}
            $s$\mbox{ is an isolated point } \\
            \mbox{\ref{OneDimC0SModMetTopRelProjIngFlat}}
        \end{tabular} & 
        \begin{tabular}{@{}c@{}}
            $s$\mbox{ is any } \\
            \mbox{\ref{OneDimC0SModMetTopRelProjIngFlat}}
        \end{tabular} &
        \begin{tabular}{@{}c@{}}
            $s$\mbox{ is any } \\
            \mbox{\ref{OneDimC0SModMetTopRelProjIngFlat}}
        \end{tabular} \\
    \hline
    \multicolumn{4}{c}{
        \mbox{
            Homologically trivial $C_0(S)$-modules in topological theory
        }
    } \\
    \hline & 
        \mbox{Projectivity} & 
        \mbox{Injectivity} & 
        \mbox{Flatness} \\
    \hline
        $L_1(S,\mu)$ & 
        \begin{tabular}{@{}c@{}}
            $\mu$\mbox{ is purely atomic, all } \\ 
            \mbox{ atoms are isolated points } \\
            \mbox{\ref{L1C0SModMetTopRelProjInjFlat}}
        \end{tabular} & 
        \begin{tabular}{@{}c@{}}
            $\mu$\mbox{ is any }  \\
            \mbox{\ref{L1C0SModMetTopRelProjInjFlat}}
        \end{tabular} & 
        \begin{tabular}{@{}c@{}}
            $\mu$\mbox{ is any } \\
            \mbox{\ref{L1C0SModMetTopRelProjInjFlat}}
        \end{tabular} \\
    \hline
        $L_p(S,\mu)$ & 
        \begin{tabular}{@{}c@{}}
            $\mu$\mbox{ is purely atomic } \\ 
            \mbox{ with finitely many atoms } \\
            \mbox{\ref{LpC0SModMetTopRelProjIngFlat}} (ii)
        \end{tabular} & 
        \begin{tabular}{@{}c@{}}
            $\mu$\mbox{ is purely atomic } \\ 
            \mbox{ with finitely many atoms } \\
            \mbox{\ref{LpC0SModMetTopRelProjIngFlat}} (ii)
        \end{tabular} & 
        \begin{tabular}{@{}c@{}}
            $\mu$\mbox{ is purely atomic } \\ 
            \mbox{ with finitely many atoms } \\
            \mbox{\ref{LpC0SModMetTopRelProjIngFlat}} (ii)
        \end{tabular} \\
    \hline
        $L_\infty(S,\mu)$ & 
        \begin{tabular}{@{}c@{}} 
            $\mu$ is normal, with \\
            pseudocompact support \\
            \mbox{\ref{LinftyC0SModMetTopRelProjIngFlat}} (i)${}^{*}$
        \end{tabular} & 
        \begin{tabular}{@{}c@{}}
            $\mu$\mbox{ is any } \\
            \mbox{\ref{LinftyC0SModMetTopRelProjIngFlat}} (ii)
        \end{tabular} & 
        \begin{tabular}{@{}c@{}}
            {?}
        \end{tabular} \\
    \hline
        $M(S)$ & 
        \begin{tabular}{@{}c@{}}
            $S$\mbox{ is discrete } \\
            \mbox{\ref{MSC0SModMetTopRelProjIngFlat}} (i)
        \end{tabular} & 
        \begin{tabular}{@{}c@{}}
            $S$\mbox{ is any } \\
            \mbox{\ref{MSC0SModMetTopRelProjIngFlat}} (ii)
        \end{tabular} & 
        \begin{tabular}{@{}c@{}}
            $S$\mbox{ is any } \\
            \mbox{\ref{MSC0SModMetTopRelProjIngFlat}} (iii)
        \end{tabular} \\
    \hline
        $C_0(S)$ & 
        \begin{tabular}{@{}c@{}}
            $S$\mbox{ is compact } \\
            \mbox{\ref{C0SC0SModMetTopRelProjInjFlat}} (i)
        \end{tabular} & 
        \begin{tabular}{@{}c@{}} 
            $S=\beta(S\setminus \{s \})$ \\
            for any limit point $s$ \\
            \mbox{\ref{C0SC0SModMetTopRelProjInjFlat}} (ii)${}^{*}$
        \end{tabular} & 
        \begin{tabular}{@{}c@{}}
            $S$\mbox{ is any } \\
            \mbox{\ref{C0SC0SModMetTopRelProjInjFlat}} (iii)
        \end{tabular} \\
    \hline
        $\mathbb{C}_s$ & 
        \begin{tabular}{@{}c@{}}
            $s$\mbox{ is an isolated point } \\
            \mbox{\ref{OneDimC0SModMetTopRelProjIngFlat}}
        \end{tabular} & 
        \begin{tabular}{@{}c@{}}
            $s$\mbox{ is any } \\
            \mbox{\ref{OneDimC0SModMetTopRelProjIngFlat}}
        \end{tabular} & 
        \begin{tabular}{@{}c@{}}
            $s$\mbox{ is any } \\
            \mbox{\ref{OneDimC0SModMetTopRelProjIngFlat}}
        \end{tabular} \\
    \hline
    \multicolumn{4}{c}{
        \mbox{
            Homologically trivial $C_0(S)$-modules in relative theory
        }
    } \\
    \hline & 
        \mbox{Projectivity} & 
        \mbox{Injectivity} & 
        \mbox{Flatness} \\
    \hline
        $L_1(S,\mu)$ & 
        \begin{tabular}{@{}c@{}}
            $\mu$\mbox{ is purely atomic, all } \\ 
            \mbox{ atoms are isolated points } \\
            \mbox{\ref{L1C0SModMetTopRelProjInjFlat}}
        \end{tabular} & 
        \begin{tabular}{@{}c@{}}
            $\mu$\mbox{ is any }  \\
            \mbox{\ref{L1C0SModMetTopRelProjInjFlat}}
        \end{tabular} & 
        \begin{tabular}{@{}c@{}}
            $\mu$\mbox{ is any } \\
            \mbox{\ref{L1C0SModMetTopRelProjInjFlat}}
        \end{tabular} \\
    \hline
        $L_p(S,\mu)$ & 
        \begin{tabular}{@{}c@{}}
            $\mu$\mbox{ is purely atomic, all } \\ 
            \mbox{ atoms are isolated points } \\
            \mbox{\ref{LpC0SModMetTopRelProjIngFlat}} (i)
        \end{tabular} & 
        \begin{tabular}{@{}c@{}}
            $\mu$\mbox{ is any } \\
            \mbox{\ref{LpC0SModMetTopRelProjIngFlat}} (i)
        \end{tabular} & 
        \begin{tabular}{@{}c@{}}
            $\mu$\mbox{ is any } \\
            \mbox{\ref{LpC0SModMetTopRelProjIngFlat}} (i)
        \end{tabular} \\
    \hline
        $L_\infty(S,\mu)$ & 
        \begin{tabular}{@{}c@{}} 
            $\mu$ is normal, with \\
            pseudocompact support \\
            \mbox{\ref{LinftyC0SModMetTopRelProjIngFlat}} (i)${}^{*}$
        \end{tabular} & 
        \begin{tabular}{@{}c@{}}
            $\mu$\mbox{ is any } \\
            \mbox{\ref{LinftyC0SModMetTopRelProjIngFlat}} (ii)
        \end{tabular} & 
        \begin{tabular}{@{}c@{}}
            $\mu$\mbox{ is any } \\
            \mbox{\ref{LinftyC0SModMetTopRelProjIngFlat}} (iii)
        \end{tabular} \\
    \hline
        $M(S)$ & 
        \begin{tabular}{@{}c@{}}
            $S$\mbox{ is discrete } \\
            \mbox{\ref{MSC0SModMetTopRelProjIngFlat}} (i)
        \end{tabular} & 
        \begin{tabular}{@{}c@{}}
            $S$\mbox{ is any } \\
            \mbox{\ref{MSC0SModMetTopRelProjIngFlat}} (ii)
        \end{tabular} & 
        \begin{tabular}{@{}c@{}}
            $S$\mbox{ is any } \\
            \mbox{\ref{MSC0SModMetTopRelProjIngFlat}} (iii)
        \end{tabular} \\
    \hline
        $C_0(S)$ & 
        \begin{tabular}{@{}c@{}}
            $S$\mbox{ is paracompact } \\
            \mbox{\ref{C0SC0SModMetTopRelProjInjFlat}} (i)
        \end{tabular} & 
        \begin{tabular}{@{}c@{}} 
            $S=\beta(S\setminus \{s \})$ \\
            for any limit point $s$ \\
            \mbox{\ref{C0SC0SModMetTopRelProjInjFlat}} (ii)${}^{*}$
        \end{tabular} & 
        \begin{tabular}{@{}c@{}}
            $S$\mbox{ is any } \\
            \mbox{\ref{C0SC0SModMetTopRelProjInjFlat}} (iii)
        \end{tabular} \\
    \hline  
        $\mathbb{C}_s$ & 
        \begin{tabular}{@{}c@{}}
            $s$\mbox{ is an isolated point } \\
            \mbox{\ref{OneDimC0SModMetTopRelProjIngFlat}}
        \end{tabular} & 
        \begin{tabular}{@{}c@{}}
            $s$\mbox{ is any } \\
            \mbox{\ref{OneDimC0SModMetTopRelProjIngFlat}}
        \end{tabular} & 
        \begin{tabular}{@{}c@{}}
            $s$\mbox{ is any } \\
            \mbox{\ref{OneDimC0SModMetTopRelProjIngFlat}}
        \end{tabular} \\
    \hline
    \end{longtable}
\end{scriptsize}

%-------------------------------------------------------------------------------
%	Applications to harmonic analysis
%-------------------------------------------------------------------------------

\section{
    Applications to modules of harmonic analysis
}\label{SectionApplicationsToModulesOfHarmonicAnalysis}

%-------------------------------------------------------------------------------
%	Preliminaries on harmonic analysis
%-------------------------------------------------------------------------------

\subsection{
    Preliminaries on harmonic analysis
}\label{SectionPreliminariesOnHarmonicAnalysis} 

Let $G$ be a locally compact group. Its identity we shall denote by $e_G$. By
well known Haar's theorem [\cite{HewRossAbstrHarmAnalVol1},section 15.8] there
exists a unique up to positive constant Borel regular measure $m_G$ which is
finite on all compact sets, positive on all open sets and left translation
invariant, that is $m_G(sE)=m_G(E)$ for all $s\in G$ and $E\in Bor(G)$. It is
called the left Haar measure of group $G$. If $G$ is compact we assume
$m_G(G)=1$. If $G$ is infinite and discrete we choose $m_G$ as counting measure.
For each $s\in G$ the map $m:Bor(G)\to[0,+\infty]:E\mapsto m_G(Es)$ is also a
left Haar measure, so from uniqueness we infer that $m(E)=\Delta_G(s)m_G(E)$ for
some $\Delta_G(s)>0$. The function $\Delta_G:G\to(0,+\infty)$ is called the
modular function of the group $G$. It is clear that
$\Delta_G(st)=\Delta_G(s)\Delta_G(t)$ for all $s,t\in G$. Groups with modular
function equal to one are called unimodular. Examples of groups with unimodular
function include compact groups, commutative groups and discrete groups. In what
follows we use the notation $L_p(G)$ instead of $L_p(G,m_G)$ 
for $1\leq p\leq+\infty$. For a fixed $s\in G$ we define the left shift operator
$L_s:L_1(G)\to L_1(G):f\mapsto(t\mapsto f(s^{-1}t))$ and the right shift
operator $R_s:L_1(G)\to L_1(G):f\mapsto (t\mapsto f(ts))$. 

Group structure of $G$ allows us to introduce the Banach algebra structure on
$L_1(G)$. For a given $f,g\in L_1(G)$ we define their convolution as
$$
(f\convol g)(s)=\int_G f(t)g(t^{-1}s)dm_G(t)=\int_G f(st)g(t^{-1})dm_G(t)
$$
$$=\int_G f(st^{-1})g(t)\Delta_G(t^{-1})dm_G(t)
$$
for almost all $s\in G$. In this case $L_1(G)$ endowed with convolution 
product becomes a Banach algebra. The Banach algebra $L_1(G)$ has a 
contractive two-sided approximate identity consisting of positive compactly 
supported continuous functions. The algebra $L_1(G)$ is unital iff $G$ is 
discrete, and in this case $\delta_{e_G}$ is the identity of $L_1(G)$. The 
group structure of $G$ allows us to turn the Banach space of complex finite 
Borel regular measures $M(G)$ into the Banach algebra too. We define 
convolution of two measures $\mu,\nu\in M(G)$ as
$$
(\mu\convol \nu)(E)=\int_G\nu(s^{-1}E)d\mu(s)=\int_G\mu(Es^{-1})d\nu(s)
$$
for all $E\in Bor(G)$. The Banach space $M(G)$ along with this convolution is 
a unital Banach algebra. The role of identity is played by Dirac delta 
measure $\delta_{e_G}$ supported on $e_G$. In fact $M(G)$ is a coproduct 
in $L_1(G)-\mathbf{mod}_1$ (but not in $M(G)-\mathbf{mod}_1$) of two-sided 
ideal $M_a(G)$ of measures absolutely continuous with respect to $m_G$ and 
subalgebra $M_s(G)$ of measures singular with respect to $m_G$. Note 
that $M_a(G)\isom{M(G)-\mathbf{mod}_1}L_1(G)$ and $M_s(G)$ is an 
annihilator $L_1(G)$-module. Finally, $M(G)=M_a(G)$ iff $G$ is discrete. 

Now we proceed to the discussion of standard left and right modules 
over $L_1(G)$ and $M(G)$. Since $L_1(G)$ can be regarded as two-sided ideal 
of $M(G)$ because of isometric left and 
right $M(G)$-morphism $i:L_1(G)\to M(G):f\mapsto f m_G$ it is enough to define 
module structure over $M(G)$. For $1\leq p<+\infty$ and 
any $f\in L_p(G)$, $\mu\in M(G)$ we define
$$
(\mu\convol_p f)(s)=\int_G f(t^{-1}s)d\mu(t), \qquad\qquad (f\convol_p
\mu)(s)=\int_G f(st^{-1}){\Delta_G(t^{-1})}^{1/p}d\mu(t)
$$
These module actions turn any Banach space $L_p(G)$ for $1\leq p<+\infty$ into 
the left and right $M(G)$-module. Note that for $p=1$ and $\mu\in M_a(G)$ we 
get the usual definition of convolution. For $1<p\leq +\infty$ and 
any $f\in L_p(G)$, $\mu\in M(G)$ we define module actions
$$
(\mu\cdot_p f)(s)=\int_G {\Delta_G(t)}^{1/p}f(st)d\mu(t), \qquad\qquad (f\cdot_p
\mu)(s)=\int_G f(ts)d\mu(t)
$$
These module actions turn any Banach space $L_p(G)$ for $1<p\leq+\infty$ into 
the left and right $M(G)$-module too. This special choice of module structure 
nicely interacts with duality. Indeed we have 
and ${(L_p(G),\convol_p)}^*\isom{\mathbf{mod}_1-M(G)}(L_{p^*}(G),\cdot_{p^*})$ 
for all $1\leq p<+\infty$. Finally, the Banach space $C_0(G)$ also becomes left 
and right $M(G)$-module when endowed with $\cdot_\infty$ in the role of module 
action. Even more, $C_0(G)$ is a closed left and right $M(G)$-submodule 
of $L_\infty(G)$ 
and ${(C_0(G),\cdot_\infty)}^*\isom{M(G)-\mathbf{mod}_1}(M(G),\convol)$.

A character on a locally compact group $G$ is by definition a continuous 
homomorphism from $G$ to $\mathbb{T}$. The set of characters on $G$ forms a 
group denoted by $\widehat{G}$. It becomes a locally compact group when 
considered with compact open topology. Any character $\gamma\in\widehat{G}$ 
gives rise to the continuous character 
$\varkappa_\gamma^L
:L_1(G)\to\mathbb{C}
:f\mapsto \int_G f(s)\overline{\gamma(s)}dm_G(s)$ on $L_1(G)$. In fact all 
characters of $L_1(G)$ arise this way. This result is due to 
Gelfand [\cite{KaniBanAlg}, theorems 2.7.2, 2.7.5]. Similarly, for 
each $\gamma\in\widehat{G}$ we have a character on $M(G)$ defined by 
$\varkappa_\gamma^M
:M(G)\to\mathbb{C}
:\mu\mapsto\int_{G} \overline{\gamma(s)}d\mu(s)$. By $\mathbb{C}_\gamma$ we 
denote the respective augmentation left and right $L_1(G)$- or $M(G)$-module. 
Their module actions are defined by
$$
f\cdot_{\gamma}z=z\cdot_{\gamma}f=\varkappa_\gamma^L(f)z \qquad\qquad
\mu\cdot_{\gamma}z=z\cdot_{\gamma}\mu=\varkappa_\gamma^M(\mu)z
$$
for all $f\in L_1(G)$, $\mu\in M(G)$ and $z\in\mathbb{C}$. 

One of the numerous definitions of amenable group says, that a locally compact 
group $G$ is amenable if there exists an $L_1(G)$-morphism of right 
modules $M:L_\infty(G)\to\mathbb{C}_{e_{\widehat{G}}}$ such 
that $M(\chi_G)=1$ [\cite{HelBanLocConvAlg}, section VII.2.5]. We can even 
assume that $M$ is contractive [\cite{HelBanLocConvAlg}, remark 7.1.54].

Most of results of this section that not supported with references are 
presented in a full detail in [\cite{DalBanAlgAutCont}, section 3.3].

%-------------------------------------------------------------------------------
%	L_1(G)-modules
%-------------------------------------------------------------------------------

\subsection{
    \texorpdfstring{$L_1(G)$}{L1(G)}-modules
}\label{SubSectionL1GModules}

Metric homological properties of most of the standard $L_1(G)$-modules of 
harmonic analysis are studied in~\cite{GravInjProjBanMod}. We borrow these 
ideas to unify approaches to metrical and topological homological properties 
of modules over group algebras.

\begin{proposition}\label{LInfIsL1MetrInj} Let $G$ be a locally compact group. 
Then $L_1(G)$ is metrically and topologically flat $L_1(G)$-module, 
i.e. $L_1(G)$-module $L_\infty(G)$ is metrically and topologically injective.
\end{proposition} 
\begin{proof} Since $L_1(G)$ has contractive approximate identity, 
then $L_1(G)$ is metrically and topologically flat $L_1(G)$-module 
by proposition~\ref{MetTopFlatIdealsInUnitalAlg}. 
Since $L_\infty(G)\isom{\mathbf{mod}_1-L_1(G)}{L_1(G)}^*$, then by 
proposition~\ref{MetCTopFlatCharac} it is metrically and topologically injective.
\end{proof}

\begin{proposition}\label{OneDimL1ModMetTopProjCharac} Let $G$ be a locally 
compact group, and $\gamma\in\widehat{G}$. Then the following are equivalent:

\begin{enumerate}[label = (\roman*)]
    \item $G$ is compact;

    \item $\mathbb{C}_\gamma$ is metrically projective $L_1(G)$-module;

    \item $\mathbb{C}_\gamma$ is topologically projective $L_1(G)$-module.
\end{enumerate}
\end{proposition}
\begin{proof} $(i)\implies (ii)$ Consider $L_1(G)$-morphisms 
$\sigma^+:\mathbb{C}_\gamma\to {L_1(G)}_+:z\mapsto z\gamma \oplus_1 0$ 
and $\pi^+:{L_1(G)}_+\to\mathbb{C}_\gamma: f\oplus_1 w\to f\cdot_{\gamma}1+w$. 
One can easily check 
that $\Vert\pi^+\Vert=\Vert\sigma^+\Vert=1$ 
and $\pi^+\sigma^+=1_{\mathbb{C}_\gamma}$. Therefore $\mathbb{C}_\gamma$ is a 
retract of ${L_1(G)}_+$ in $L_1(G)-\mathbf{mod}_1$. From 
propositions~\ref{UnitalAlgIsMetTopProj} 
and~\ref{RetrMetCTopProjIsMetCTopProj} it follows 
that $\mathbb{C}_\gamma$ is metrically projective.

$(ii)\implies (iii)$ See
proposition~\ref{MetProjIsTopProjAndTopProjIsRelProj}.

$(iii)\implies (i)$ Consider $L_1(G)$-morphism
$\pi:L_1(G)\to\mathbb{C}_\gamma:f\mapsto f\cdot_{\gamma} 1$. It is easy to see
that $\pi$ is strictly coisometric. Since $\mathbb{C}_\gamma$ is topologically
projective, then there exists an $L_1(G)$-morphism $\sigma:\mathbb{C}_\gamma\to
L_1(G)$ such that $\pi\sigma=1_{\mathbb{C}_\gamma}$. Let $f=\sigma(1)\in L_1(G)$
and ${(e_\nu)}_{\nu\in N}$ be a standard approximate identity of $L_1(G)$. Since
$\sigma$ is an $L_1(G)$-morphism, then for all $s,t\in G$ we have 
$$
f(s^{-1}t)
=L_s(f)(t)
=\lim_\nu L_s(e_\nu\convol \sigma(1))(t)
=\lim_\nu((\delta_s\convol e_\nu)\convol \sigma(1))(t)
=\lim_\nu\sigma((\delta_s\convol e_\nu)\cdot_{\gamma} 1)(t)
$$
$$
=\lim_\nu\sigma(\varkappa_\gamma^L(\delta_s\convol e_\nu))(t)
=\lim_\nu\varkappa_\gamma^L(\delta_s\convol e_\nu)\sigma(1)(t)
=\lim_\nu(e_\nu\convol\gamma)(s^{-1})f(t)
=\gamma(s^{-1})f(t).
$$
Therefore, for the function $g(t):=\gamma(t^{-1})f(t)$ in $L_1(G)$ we have
$g(st)=g(t)$ for all $s,t\in G$. Thus $g$ is a constant function in $L_1(G)$,
which is possible only for compact group $G$.
\end{proof}

\begin{proposition}\label{OneDimL1ModMetTopInjFlatCharac} Let $G$ be a locally
compact group, and $\gamma\in\widehat{G}$. Then the following are equivalent:

\begin{enumerate}[label = (\roman*)]
    \item $G$ is amenable;

    \item $\mathbb{C}_\gamma$ is metrically injective $L_1(G)$-module;

    \item $\mathbb{C}_\gamma$ is topologically injective $L_1(G)$-module.

    \item $\mathbb{C}_\gamma$ is metrically flat $L_1(G)$-module;

    \item $\mathbb{C}_\gamma$ is topologically flat $L_1(G)$-module.
\end{enumerate}
\end{proposition}
\begin{proof} $(i)\implies (ii)$ Since $G$ is amenable, then we have
contractive $L_1(G)$-morphism $M:L_\infty(G)\to\mathbb{C}_{e_{\widehat{G}}}$
with $M(\chi_G)=1$. Consider linear 
operators $\rho:\mathbb{C}_\gamma\to L_\infty(G):z\mapsto z\overline{\gamma}$ 
and $\tau:L_\infty(G)\to\mathbb{C}_\gamma:f\mapsto M(f\gamma)$. These are
$L_1(G)$-morphisms of right $L_1(G)$-modules. We shall check this for operator
$\tau$: for all $f\in L_\infty(G)$ and $g\in L_1(G)$ we have
$$
\tau(f\cdot_\infty g)
=M((f\cdot_\infty g)\gamma)
=M(f\gamma\cdot_\infty g\overline{\gamma})
=M(f\gamma)\cdot_{e_{\widehat{G}}} g\overline{\gamma}
=M(f\gamma)\varkappa_\gamma^L(g)
=\tau(f)\cdot_{\gamma} g.
$$  
It is easy to check that $\rho$ and $\tau$ are contractive and
$\tau\rho=1_{\mathbb{C}_\gamma}$. Therefore $\mathbb{C}_\gamma$ is a retract of
$L_\infty(G)$ in $\mathbf{mod}_1-L_1(G)$. From
propositions~\ref{LInfIsL1MetrInj} and~\ref{RetrMetCTopInjIsMetCTopInj} it 
follows that $\mathbb{C}_\gamma$ is metrically injective as $L_1(G)$-module.

$(ii)\implies (iii)$ See proposition~\ref{MetInjIsTopInjAndTopInjIsRelInj}.

$(iii) \implies (i)$ Since $\rho$ is an isometric $L_1(G)$-morphism of right
$L_1(G)$-modules and $\mathbb{C}_\gamma$ is topologically injective as
$L_1(G)$-module, then $\rho$ is a coretraction in $\mathbf{mod}-L_1(G)$. Denote
its left inverse morphism by $\pi$, then
$\pi(\overline{\gamma})=\pi(\rho(1))=1$. Consider bounded linear functional
$M:L_\infty(G)\to\mathbb{C}_\gamma:f\mapsto \pi(f\overline{\gamma})$. For all
$f\in L_\infty(G)$ and $g\in L_1(G)$ we have
$$
M(f\cdot_\infty g)
=\pi((f\cdot_\infty g)\overline{\gamma})
=\pi(f\overline{\gamma}\cdot_\infty g\gamma)
=\pi(f\overline{\gamma})\cdot_{\gamma} g\gamma
=M(f)\varkappa_\gamma^L(g\gamma)
=M(f)\cdot_{e_{\widehat{G}}}g.
$$
Therefore $M$ is an $L_1(G)$-morphism, but we also have
$M(\chi_G)=\pi(\overline{\gamma})=1$. Therefore $G$ is amenable.

$(ii) \Longleftrightarrow (iv)$, $(iii) \Longleftrightarrow (v)$ Note that
$\mathbb{C}_\gamma^*\isom{\mathbf{mod}_1-L_1(G)}\mathbb{C}_\gamma$, so all
equivalences  follow from three previous paragraphs and
proposition~\ref{MetCTopFlatCharac}.
\end{proof}

In the next proposition we shall study specific ideals of Banach algebra
$L_1(G)$. They are of the form $L_1(G)\convol\mu$ for some idempotent measure
$\mu$. In fact, this class of ideals in case of commutative compact groups $G$
coincides with those left ideals of $L_1(G)$ that admit a right bounded
approximate identity.

\begin{proposition}\label{CommIdealByIdemMeasL1MetTopProjCharac} Let $G$ be a
locally compact group and  $\mu\in M(G)$ be an idempotent measure, that is
$\mu\convol\mu=\mu$. If the left ideal $I=L_1(G)\convol\mu$ of Banach algebra
$L_1(G)$ is topologically projective $L_1(G)$-module, then $\mu=p m_G$, for some
$p\in I$.
\end{proposition}
\begin{proof} Let $\phi:I\to L_1(G)$ be arbitrary morphism of left
$L_1(G)$-modules. Consider $L_1(G)$-morphism $\phi':L_1(G)\to
L_1(G):x\mapsto\phi(x\convol\mu)$. By Wendel's theorem [\cite{WendLeftCentrzrs},
theorem 1], there exists a measure $\nu\in M(G)$ such that
$\phi'(x)=x\convol\nu$ for all $x\in L_1(G)$. In particular,
$\phi(x)=\phi(x\convol\mu)=\phi'(x)=x\convol\nu$ for all $x\in I$. It is now
clear that $\psi:I\to I:x\mapsto\nu\convol x$ is a morphism of right $I$-modules
satisfying $\phi(x)y=x\psi(y)$ for all $x,y\in I$. By paragraph $(ii)$ of
lemma~\ref{GoodIdealMetTopProjIsUnital} the ideal $I$ has a right identity, say
$e\in I$. Then $x\convol\mu=x\convol\mu\convol e$ for all $x\in L_1(G)$. Two
measures are equal if their convolutions with all functions of $L_1(G)$ coincide
[\cite{DalBanAlgAutCont}, corollary 3.3.24], so $\mu=\mu\convol e m_G$. Since
$e\in I\subset L_1(G)$, then $\mu=\mu\convol e m_G\in M_a(G)$. 
Set $p=\mu\convol e\in I$, then $\mu=p m_G$.
\end{proof}

We conjecture that the left ideal $L_1(G)\convol \mu$ for idempotent measure
$\mu$ is metrically projective $L_1(G)$-module iff $\mu=p m_G$ where $p\in I$
and $\Vert p\Vert=1$.

\begin{theorem}\label{L1ModL1MetTopProjCharac} Let $G$ be a locally compact
group. Then the following are equivalent:

\begin{enumerate}[label = (\roman*)]
    \item $G$ is discrete;

    \item $L_1(G)$ is metrically projective $L_1(G)$-module;

    \item $L_1(G)$ is topologically projective $L_1(G)$-module.
\end{enumerate}
\end{theorem}
\begin{proof} $(i)\implies (ii)$ If $G$ is discrete, then $L_1(G)$ is unital
with unit of norm $1$. By  proposition~\ref{UnIdeallIsMetTopProj} we see that
$L_1(G)$ is metrically projective as $L_1(G)$-module.

$(ii)\implies (iii)$ See
proposition~\ref{MetProjIsTopProjAndTopProjIsRelProj}.

$(iii) \implies (i)$ Clearly, $\delta_{e_G}$ is an idempotent measure. Since
$L_1(G)=L_1(G)\convol \delta_{e_G}$ is topologically projective, then by
proposition~\ref{CommIdealByIdemMeasL1MetTopProjCharac} 
we have $\delta_{e_G}=f m_G$ for some $f\in L_1(G)$. This is possible only 
if $G$ is discrete.
\end{proof}

Note that $L_1(G)$-module $L_1(G)$ is relatively projective for any locally
compact group $G$ [\cite{HelBanLocConvAlg}, exercise 7.1.17].

\begin{proposition}\label{L1MetTopProjAndMetrFlatOfMeasAlg} Let $G$ be 
a locally compact group. Then the following are equivalent:

\begin{enumerate}[label = (\roman*)]
    \item $G$ is discrete;

    \item $M(G)$ is metrically projective $L_1(G)$-module;

    \item $M(G)$ is topologically projective $L_1(G)$-module;

    \item $M(G)$ is metrically flat $L_1(G)$-module.
\end{enumerate}
\end{proposition}
\begin{proof} 
$(i)\implies (ii)$ We have $M(G)\isom{L_1(G)-\mathbf{mod}_1}L_1(G)$ for
discrete $G$, so the result follows from theorem~\ref{L1ModL1MetTopProjCharac}. 

$(ii)\implies (iii)$ See
proposition~\ref{MetProjIsTopProjAndTopProjIsRelProj}.

$(ii)\implies (iv)$ See proposition~\ref{MetTopProjIsMetTopFlat}.

$(iii)\implies (i)$ Note that $M(G)\isom{L_1(G)-\mathbf{mod}_1}
L_1(G)\bigoplus_1 M_s(G)$, so $M_s(G)$ is topologically projective by
proposition~\ref{MetTopProjModCoprod}. Note that $M_s(G)$ is an annihilator
$L_1(G)$-module, then by proposition~\ref{MetTopProjOfAnnihModCharac} the
algebra $L_1(G)$ has a right identity. Recall that $L_1(G)$ also has a two-sided
bounded approximate identity, so $L_1(G)$ is unital. The last is equivalent to
$G$ being discrete.

$(iv)\implies (i)$ Note that $M(G)\isom{L_1(G)-\mathbf{mod}_1}
L_1(G)\bigoplus_1 M_s(G)$, so $M_s(G)$ is metrically flat by
proposition~\ref{MetTopFlatModCoProd}. Note that $M_s(G)$ is an annihilator
$L_1(G)$-module, then by proposition~\ref{MetTopFlatAnnihModCharac} it is equal
to zero. The last is equivalent to $G$ being discrete.
\end{proof}

\begin{proposition}\label{MeasAlgIsL1TopFlat} Let $G$ be a locally compact
group. Then $M(G)$ is topologically flat $L_1(G)$-module.
\end{proposition}
\begin{proof} Since $M(G)$ is an $L_1$-space it is a fortiori an
$\mathscr{L}_1^g$-space. Since $M_s(G)$ is complemented in $M(G)$, then $M_s(G)$
is an $\mathscr{L}_1^g$-space too [\cite{DefFloTensNorOpId}, corollary
23.2.1(2)]. Since $M_s(G)$ is an annihilator $L_1(G)$-module, then from
proposition~\ref{MetTopFlatAnnihModCharac} we have that $M_s(G)$ is
topologically flat $L_1(G)$-module. The $L_1(G)$-module $L_1(G)$ is also
topologically flat by proposition~\ref{LInfIsL1MetrInj}. Since
$M(G)\isom{L_1(G)-\mathbf{mod}_1}L_1(G)\bigoplus_1 M_s(G)$, then $M(G)$ is
topologically flat $L_1(G)$-module by proposition~\ref{MetTopFlatModCoProd}.
\end{proof}

%-------------------------------------------------------------------------------
%	M(G)-modules
%-------------------------------------------------------------------------------

\subsection{
    \texorpdfstring{$M(G)$}{M (G)}-modules
}\label{SubSectionMGModules}

We turn to the study of standard $M(G)$-modules of harmonic analysis. As we
shall see most of results can be derived from results on $L_1(G)$-modules.

\begin{proposition}\label{MGMetTopProjInjFlatRedToL1} Let $G$ be a locally
compact group, and $X$ be $\langle$~essential / faithful / essential~$\rangle$
$L_1(G)$-module. Then

\begin{enumerate}[label = (\roman*)]
    \item $X$ is metrically $\langle$~projective / injective / flat~$\rangle$
    $M(G)$-module iff it is metrically $\langle$~projective / injective /
    flat~$\rangle$ $L_1(G)$-module;

    \item $X$ is topologically $\langle$~projective / injective / flat~$\rangle$
    $M(G)$-module iff it is topologically $\langle$~projective / injective /
    flat~$\rangle$ $L_1(G)$-module.
\end{enumerate}
\end{proposition}
\begin{proof} Recall that $L_1(G)\isom{L_1(G)-\mathbf{mod}_1}M_a(G)$ is a
two-sided complemented in $\mathbf{Ban}_1$ ideal of $M(G)$. Now $(i)$ and $(ii)$
follow from proposition $\langle$~\ref{MetTopProjUnderChangeOfAlg}
/~\ref{MetTopInjUnderChangeOfAlg}  /~\ref{MetTopFlatUnderChangeOfAlg}~$\rangle$.
\end{proof} 

It is worth to mention here that the $L_1(G)$-modules $C_0(G)$, $L_p(G)$ for
$1\leq p<\infty$ and $\mathbb{C}_\gamma$ for $\gamma\in\widehat{G}$ are
essential and $L_1(G)$-modules $C_0(G)$, $M(G)$, $L_p(G)$ for $1\leq p\leq
\infty$ and $\mathbb{C}_\gamma$ for $\gamma\in\widehat{G}$ are faithful. 

\begin{proposition}\label{MGModMGMetTopProjFlatCharac} Let $G$ be a locally
compact group. Then $M(G)$ is metrically and topologically projective
$M(G)$-module. As the consequence it is metrically and topologically flat
$M(G)$-module.
\end{proposition} 
\begin{proof} Since $M(G)$ is a unital algebra, the metric and topological
projectivity of $M(G)$ follow from proposition~\ref{UnitalAlgIsMetTopProj}. It
remains to apply proposition~\ref{MetTopProjIsMetTopFlat}.
\end{proof}

%-------------------------------------------------------------------------------
%	Banach geometric restrictions
%-------------------------------------------------------------------------------

\subsection{
    Banach geometric restriction
}\label{SubSectionBanachGeometricRestriction}

In this section we shall show that many modules of harmonic analysis are fail to
be metrically or topologically projective, injective or flat for purely Banach
geometric reasons. 

\begin{proposition}\label{StdModAreNotRetrOfL1LInf} Let $G$ be an infinite
locally compact group. Then

\begin{enumerate}[label = (\roman*)]
    \item $L_1(G)$, $C_0(G)$, $M(G)$, ${L_\infty(G)}^*$ are not 
    topologically injective Banach spaces;

    \item $C_0(G)$, $L_\infty(G)$ are not complemented in any $L_1$-space.
\end{enumerate}
\end{proposition}
\begin{proof}
Since $G$ is infinite all modules in question are infinite dimensional.

$(i)$ If an infinite dimensional Banach space is topologically injective, then it
contains a copy of $\ell_\infty(\mathbb{N})$ [\cite{RosOnRelDisjFamOfMeas},
corollary 1.1.4], and consequently a copy of $c_0(\mathbb{N})$. The Banach space
$L_1(G)$ is weakly sequentially complete [\cite{WojBanSpForAnalysts}, corollary
III.C.14], so by corollary 5.2.11 in~\cite{KalAlbTopicsBanSpTh} it can't contain
a copy of $c_0(\mathbb{N})$. Therefore, $L_1(G)$ is not topologically injective
Banach space.  If $M(G)$ is topologically injective Banach space, then so does
its complemented subspace $M_a(G)\isom{\mathbf{Ban}_1}L_1(G)$. By previous
argument this is impossible. So $M(G)$ is not topologically injective as Banach
space. By corollary 3 of~\cite{LauMingComplSubspInLInfOfG} the space $C_0(G)$ is
not complemented in $L_\infty(G)$. Then $C_0(G)$ can't be topologically
injective either. The Banach space $L_1(G)$ is complemented in
${L_\infty(G)}^*\isom{\mathbf{Ban}_1}{L_1(G)}^{**}$ [\cite{DefFloTensNorOpId},
proposition  B10]. Therefore if ${L_\infty(G)}^*$ is topologically injective as
Banach space, then so does its retract $L_1(G)$. By previous arguments this is
impossible, so ${L_\infty(G)}^*$ is not topologically injective Banach space.

$(ii)$ If $C_0(G)$ is a retract of $L_1$-space, then
$M(G)\isom{\mathbf{Ban}_1}{C_0(G)}^*$ is a retract of $L_\infty$-space, 
so it must be a topologically injective Banach space. This contradicts 
paragraph $(i)$, so $C_0(G)$ is not a retract of $L_1$-space. 
Note that $\ell_\infty(\mathbb{N})$ embeds in $L_\infty(G)$, then so 
does $c_0(\mathbb{N})$. So if $L_\infty(G)$ is
a retract of $L_1$-space, then there would exist an $L_1$-space containing a
copy of $c_0(\mathbb{N})$. This is impossible as already showed in paragraph
$(i)$.
\end{proof}

From now on by $A$ we denote either $L_1(G)$ or $M(G)$. Recall that $L_1(G)$ and
$M(G)$ are both $L_1$-spaces.

\begin{proposition}\label{StdModAreNotL1MGMetTopProjInjFlat} Let $G$ be an
infinite locally compact group. Then

\begin{enumerate}[label = (\roman*)]
    \item $C_0(G)$, $L_\infty(G)$ are neither topologically nor metrically 
    projective $A$-modules;

    \item $L_1(G)$, $C_0(G)$, $M(G)$, ${L_\infty(G)}^*$ are neither 
    topologically nor metrically injective $A$-modules;

    \item $L_\infty(G)$, $C_0(G)$ are neither topologically nor metrically flat
    $A$-modules.
\end{enumerate}
\end{proposition}

$(iv)$ $L_p(G)$ for $1<p<\infty$ are neither topologically nor metrically
projective, injective or flat $A$-flat.

\begin{proof} $(i)$ The result follows from
propositions~\ref{TopProjInjFlatModOverL1Charac} paragraph $(i)$ and
~\ref{StdModAreNotRetrOfL1LInf} paragraph $(ii)$.

$(ii)$ The result follows from propositions~\ref{TopProjInjFlatModOverL1Charac}
paragraph $(ii)$ and~\ref{StdModAreNotRetrOfL1LInf}.

$(iii)$ Note that ${C_0(G)}^*\isom{\mathbf{mod}_1-A}M(G)$. Now the result 
follows from paragraph $(i)$ and proposition~\ref{MetCTopFlatCharac}.

$(iv)$ Since $L_p(G)$ is reflexive for $1<p<\infty$ the result follows
from~\ref{NoInfDimRefMetTopProjInjFlatModOverMthscrL1OrLInfty}.
\end{proof}

It remains to consider metric and topological homological properties of
$A$-modules when $G$ is finite.

\begin{proposition}\label{LpFinGrL1MGMetrInjProjCharac} Let $G$ be a non trivial
finite group and $1\leq p\leq \infty$. Then the $A$-module $L_p(G)$ is
metrically $\langle$~projective / injective~$\rangle$ iff $\langle$~$p=1$ /
$p=\infty$~$\rangle$
\end{proposition}
\begin{proof} 
Assume $L_p(G)$ is metrically $\langle$~projective / injective~$\rangle$ as
$A$-module. Since $L_p(G)$ is finite dimensional, then by paragraphs $(i)$ and
$(ii)$ of proposition~\ref{TopProjInjFlatModOverL1Charac} we have 
identifications $\langle$~$L_p(G)\isom{\mathbf{Ban}_1}\ell_1(\mathbb{N}_n)$ /
$L_p(G)
\isom{\mathbf{Ban}_1}
C(\mathbb{N}_n)
\isom{\mathbf{Ban}_1}
\ell_\infty(\mathbb{N}_n)$~$\rangle$, 
where $n=\operatorname{Card}(G)>1$. Now we use the result of theorem
1 from~\cite{LyubIsomEmdbFinDimLp} for Banach spaces over field $\mathbb{C}$: if
for $2\leq m\leq k$ and $1\leq r,s\leq \infty$, there exists an isometric
embedding from $\ell_r(\mathbb{N}_m)$ into $\ell_s(\mathbb{N}_k)$, then either
$r=2$, $s\in 2\mathbb{N}$ or $r=s$. Therefore $\langle$~$p=1$ /
$p=\infty$~$\rangle$. The converse easily follows from
$\langle$~theorem~\ref{L1ModL1MetTopProjCharac} /
proposition~\ref{LInfIsL1MetrInj}~$\rangle$
\end{proof}

\begin{proposition}\label{StdModFinGrL1MGMetrInjProjFlatCharac} Let $G$ be a
finite group. Then

\begin{enumerate}[label = (\roman*)]
    \item $C_0(G)$, $L_\infty(G)$ are metrically injective $A$-modules;

    \item $C_0(G)$, $L_p(G)$ for $1<p\leq\infty$ are metrically projective
    $A$-modules iff $G$ is trivial;

    \item $M(G)$, $L_p(G)$ for $1\leq p<\infty$ are metrically injective
    $A$-modules iff $G$ is trivial;

    \item $C_0(G)$, $L_p(G)$ for $1<p\leq\infty$ are metrically 
    flat $A$-modules iff $G$ is trivial.
\end{enumerate}
\end{proposition}
\begin{proof}
$(i)$ Since $G$ is finite then $C_0(G)=L_\infty(G)$. The result follows from
proposition~\ref{LInfIsL1MetrInj}.

$(ii)$ If $G$ is trivial, that is $G= \{e_G \}$, then $L_p(G)=C_0(G)=L_1(G)$ and
the result follows from paragraph $(i)$. If $G$ is non trivial, then we recall
that $C_0(G)=L_\infty(G)$ and use
proposition~\ref{LpFinGrL1MGMetrInjProjCharac}.

$(iii)$ If $G= \{e_G \}$, then $M(G)=L_p(G)=L_\infty(G)$ and the result follows
from paragraph $(i)$. If $G$ is non trivial, then we note that $M(G)=L_1(G)$ and
use proposition~\ref{LpFinGrL1MGMetrInjProjCharac}.

$(iv)$ From paragraph $(iii)$ it follows that $L_p(G)$ for $1\leq p<\infty$ is
metrically injective $A$-module iff $G$ is trivial. Now the result follows from
proposition~\ref{MetCTopFlatCharac} and the facts that
${C_0(G)}^*\isom{\mathbf{mod}_1-L_1(G)}M(G)\isom{\mathbf{mod}_1-L_1(G)}L_1(G)$,
${L_p(G)}^*\isom{\mathbf{mod}_1-L_1(G)}L_{p^*}(G)$ for $1\leq p^*<\infty$.
\end{proof}

It is worth to mention here that if we would consider all Banach spaces over the
field of real numbers, then $L_\infty(G)$ and $L_1(G)$ would be metrically
projective and injective respectively,  additionally for $G$ consisting of two
elements, because
$$
L_\infty(\mathbb{Z}_2)
\isom{L_1(\mathbb{Z}_2)-\mathbf{mod}_1}
\mathbb{R}_{\gamma_0}\bigoplus\nolimits_1\mathbb{R}_{\gamma_1},
\qquad
L_1(\mathbb{Z}_2)
\isom{L_1(\mathbb{Z}_2)-\mathbf{mod}_1}
\mathbb{R}_{\gamma_0}\bigoplus\nolimits_\infty\mathbb{R}_{\gamma_1}
$$
for $\gamma_0,\gamma_1\in\widehat{\mathbb{Z}_2}$ defined by
$\gamma_0(0)=\gamma_0(1)=\gamma_1(0)=-\gamma_1(1)=1$. Here $\mathbb{Z}_2$
denotes the unique group of two elements.

\begin{proposition}\label{StdModFinGrL1MGTopInjProjFlatCharac} Let $G$ be a
finite group. Then the $A$-modules $C_0(G)$, $M(G)$, $L_p(G)$ 
for $1\leq p\leq \infty$ are both topologically projective, injective and flat.
\end{proposition} 
\begin{proof}
For finite group $G$ we have $M(G)=L_1(G)$ and $C_0(G)=L_\infty(G)$, so these
modules do not require special considerations. Since $M(G)=L_1(G)$, we can
restrict our considerations to the case $A=L_1(G)$. The identity map
$i:L_1(G)\to L_p(G):f\mapsto f$ is a topological isomorphism of Banach spaces,
because $L_1(G)$ and $L_p(G)$ for $1\leq p<+\infty$ are of equal finite
dimension. Since $G$ is finite, it is unimodular. Therefore, the module actions
in $(L_1(G),\convol)$ and $(L_p(G),\convol_p)$ coincide for $1\leq p<+\infty$
and $i$ is an isomorphism in $L_1(G)-\mathbf{mod}$ and $\mathbf{mod}-L_1(G)$.
Similarly one can show that $(L_\infty(G),\cdot_\infty)$ and $(L_p(G),\cdot_p)$
for $1<p\leq+\infty$ are isomorphic in $L_1(G)-\mathbf{mod}$ and
$\mathbf{mod}-L_1(G)$. Finally, one can easily check that $(L_1(G),\convol)$ and
$(L_\infty(G),\cdot_\infty)$ are isomorphic in $L_1(G)-\mathbf{mod}$ and
$\mathbf{mod}-L_1(G)$ via the 
map $j:L_1(G)\to L_\infty(G):f\mapsto(s\mapsto f(s^{-1}))$. Therefore all 
the discussed modules are isomorphic. It remains to
recall that $L_1(G)$ is topologically projective and flat by
theorem~\ref{L1ModL1MetTopProjCharac} and proposition~\ref{LInfIsL1MetrInj},
while $L_\infty(G)$ is topologically injective by
proposition~\ref{LInfIsL1MetrInj}.
\end{proof}

Now we can summarize results on homological properties of modules of harmonic
analysis into three tables. Each cell of the table contains a condition under
which the respective module has respective property and propositions where this
is proved. We shall mention that results for modules $L_p(G)$ are valid for both
module actions $\convol_p$ and $\cdot_p$. Characterization and proofs for
homologically trivial modules $\mathbb{C}_\gamma$ in case of relative theory is
the same as in
propositions~\ref{OneDimL1ModMetTopProjCharac},
~\ref{OneDimL1ModMetTopInjFlatCharac}
and~\ref{OneDimL1ModMetTopInjFlatCharac}. As usually, we use ${}^{*}$
indicates that only a necessary conditions is known. As we showed above even
topological theory is too restrictive for $L_1(G)$ to be projective as
$L_1(G)$-module. Similarly a Banach space is topologically projective iff it is
an $L_1$-space, and the underlying measure space is atomic. This analogy
confirms important role of Banach geometry in metric and topological Banach
homology.

\begin{scriptsize}
    \begin{longtable}{|c|c|c|c|c|c|c|} 
    \multicolumn{7}{c}{
        \mbox{
            Homologically trivial $L_1(G)$- and $M(G)$-modules in metric theory
        }
    }                                                                                                                                                                                                                                                                                                                                                                                                                                                                                                                                                                                                                                                                                                                                                                                                                                                                                                                                             \\
    \hline & 
        \multicolumn{3}{c|}{
            $L_1(G)$-modules
        } & 
        \multicolumn{3}{c|}{
            $M(G)$-modules
        } \\
    \hline & 
        \mbox{Projectivity} & 
        \mbox{Injectivity} & 
        \mbox{Flatness} & 
        \mbox{Projectivity} &
        \mbox{Injectivity} & 
        \mbox{Flatness} \\ 
    \hline
        $L_1(G)$ & 
        \begin{tabular}{@{}c@{}}
            $G$\mbox{ is discrete } \\
            \mbox{\ref{L1ModL1MetTopProjCharac}}
        \end{tabular} & 
        \begin{tabular}{@{}c@{}}
            $G= \{e_G \}$ \\
            \mbox{\ref{StdModAreNotL1MGMetTopProjInjFlat}},
            \mbox{\ref{StdModFinGrL1MGMetrInjProjFlatCharac}}
        \end{tabular} & 
        \begin{tabular}{@{}c@{}}
            $G$\mbox{ is any } \\
            \mbox{\ref{LInfIsL1MetrInj}}
        \end{tabular} & 
        \begin{tabular}{@{}c@{}}
            $G$\mbox{ is discrete } \\
            \mbox{\ref{L1ModL1MetTopProjCharac}},
            \mbox{\ref{MGMetTopProjInjFlatRedToL1}}
        \end{tabular} & 
        \begin{tabular}{@{}c@{}}
            $G= \{e_G \}$ \\
            \mbox{\ref{StdModAreNotL1MGMetTopProjInjFlat}},
            \mbox{\ref{StdModFinGrL1MGMetrInjProjFlatCharac}}
        \end{tabular} & 
        \begin{tabular}{@{}c@{}}
            $G$\mbox{ is any } \\
            \mbox{\ref{LInfIsL1MetrInj}},
            \mbox{\ref{MGMetTopProjInjFlatRedToL1}}
        \end{tabular} \\
    \hline 
        $L_p(G)$ & 
        \begin{tabular}{@{}c@{}}
            $G= \{e_G \}$ \\
            \mbox{\ref{StdModAreNotL1MGMetTopProjInjFlat}},
            \mbox{\ref{LpFinGrL1MGMetrInjProjCharac}}
        \end{tabular} &
        \begin{tabular}{@{}c@{}}
            $G= \{e_G \}$ \\
            \mbox{\ref{StdModAreNotL1MGMetTopProjInjFlat}},
            \mbox{\ref{LpFinGrL1MGMetrInjProjCharac}}
        \end{tabular} & 
        \begin{tabular}{@{}c@{}}
            $G= \{e_G \}$ \\
            \mbox{\ref{StdModAreNotL1MGMetTopProjInjFlat}},
            \mbox{\ref{StdModFinGrL1MGMetrInjProjFlatCharac}}
        \end{tabular} & 
        \begin{tabular}{@{}c@{}}
            $G= \{e_G \}$ \\
            \mbox{\ref{StdModAreNotL1MGMetTopProjInjFlat}},
            \mbox{\ref{LpFinGrL1MGMetrInjProjCharac}}
        \end{tabular} & 
        \begin{tabular}{@{}c@{}}
            $G= \{e_G \}$ \\
            \mbox{\ref{StdModAreNotL1MGMetTopProjInjFlat}},
            \mbox{\ref{LpFinGrL1MGMetrInjProjCharac}}
        \end{tabular} & 
        \begin{tabular}{@{}c@{}}
            $G= \{e_G \}$ \\
            \mbox{\ref{StdModAreNotL1MGMetTopProjInjFlat}},
            \mbox{\ref{StdModFinGrL1MGMetrInjProjFlatCharac}}
        \end{tabular} \\
    \hline
        $L_\infty(G)$ & 
        \begin{tabular}{@{}c@{}}
            $G= \{e_G \}$ \\
            \mbox{\ref{StdModAreNotL1MGMetTopProjInjFlat}},
            \mbox{\ref{LpFinGrL1MGMetrInjProjCharac}}
        \end{tabular} & 
        \begin{tabular}{@{}c@{}}
            $G$\mbox{ is any } \\
            \mbox{\ref{LInfIsL1MetrInj}}
        \end{tabular} &
        \begin{tabular}{@{}c@{}}
            $G= \{e_G \}$ \\
            \mbox{\ref{StdModAreNotL1MGMetTopProjInjFlat}},
            \mbox{\ref{StdModFinGrL1MGMetrInjProjFlatCharac}}
        \end{tabular} & 
        \begin{tabular}{@{}c@{}}
            $G= \{e_G \}$ \\
            \mbox{\ref{StdModAreNotL1MGMetTopProjInjFlat}},
            \mbox{\ref{LpFinGrL1MGMetrInjProjCharac}}
        \end{tabular} & 
        \begin{tabular}{@{}c@{}}
            $G$\mbox{ is any } \\
            \mbox{\ref{LInfIsL1MetrInj}},
            \mbox{\ref{MGMetTopProjInjFlatRedToL1}}
        \end{tabular} & 
        \begin{tabular}{@{}c@{}}
            $G= \{e_G \}$ \\
            \mbox{\ref{StdModAreNotL1MGMetTopProjInjFlat}},
            \mbox{\ref{StdModFinGrL1MGMetrInjProjFlatCharac}}
        \end{tabular} \\ 
    \hline
        $M(G)$ & 
        \begin{tabular}{@{}c@{}}
            $G$\mbox{ is discrete } \\
            \mbox{\ref{L1MetTopProjAndMetrFlatOfMeasAlg}}
        \end{tabular} & 
        \begin{tabular}{@{}c@{}}
            $G= \{e_G \}$ \\
            \mbox{\ref{StdModAreNotL1MGMetTopProjInjFlat}},
            \mbox{\ref{StdModFinGrL1MGMetrInjProjFlatCharac}}
        \end{tabular} & 
        \begin{tabular}{@{}c@{}}
            $G$\mbox{ is discrete } \\
            \mbox{\ref{MeasAlgIsL1TopFlat}}
        \end{tabular} & 
        \begin{tabular}{@{}c@{}}
            $G$\mbox{ is any } \\
            \mbox{\ref{MGModMGMetTopProjFlatCharac}}
        \end{tabular} & 
        \begin{tabular}{@{}c@{}}
            $G= \{e_G \}$ \\
            \mbox{\ref{StdModAreNotL1MGMetTopProjInjFlat}},
            \mbox{\ref{StdModFinGrL1MGMetrInjProjFlatCharac}}
        \end{tabular} & 
        \begin{tabular}{@{}c@{}}
            $G$\mbox{ is any } \\
            \mbox{\ref{MGModMGMetTopProjFlatCharac}}
        \end{tabular} \\ 
    \hline
        $C_0(G)$ & 
        \begin{tabular}{@{}c@{}}
            $G= \{e_G \}$ \\
            \mbox{\ref{StdModAreNotL1MGMetTopProjInjFlat}},
            \mbox{\ref{StdModFinGrL1MGMetrInjProjFlatCharac}}
        \end{tabular} & 
        \begin{tabular}{@{}c@{}}
            $G$\mbox{ is finite } \\
            \mbox{\ref{StdModAreNotL1MGMetTopProjInjFlat}},
            \mbox{\ref{StdModFinGrL1MGMetrInjProjFlatCharac}}
        \end{tabular} & 
        \begin{tabular}{@{}c@{}}
            $G= \{e_G \}$ \\
            \mbox{\ref{StdModAreNotL1MGMetTopProjInjFlat}},
            \mbox{\ref{StdModFinGrL1MGMetrInjProjFlatCharac}}
        \end{tabular} & 
        \begin{tabular}{@{}c@{}}
            $G= \{e_G \}$ \\
            \mbox{\ref{StdModAreNotL1MGMetTopProjInjFlat}},
            \mbox{\ref{StdModFinGrL1MGMetrInjProjFlatCharac}}
        \end{tabular} & 
        \begin{tabular}{@{}c@{}}
            $G$\mbox{ is finite } \\
            \mbox{\ref{StdModAreNotL1MGMetTopProjInjFlat}},
            \mbox{\ref{StdModFinGrL1MGMetrInjProjFlatCharac}}
        \end{tabular} & 
        \begin{tabular}{@{}c@{}}
            $G= \{e_G \}$ \\
            \mbox{\ref{StdModAreNotL1MGMetTopProjInjFlat}},
            \mbox{\ref{StdModFinGrL1MGMetrInjProjFlatCharac}}
        \end{tabular} \\ 
    \hline 
        $\mathbb{C}_\gamma$ & 
        \begin{tabular}{@{}c@{}}
            $G$\mbox{ is compact } \\
            \mbox{\ref{OneDimL1ModMetTopProjCharac}}
        \end{tabular} & 
        \begin{tabular}{@{}c@{}}
            $G$\mbox{ is amenable } \\
            \mbox{\ref{OneDimL1ModMetTopInjFlatCharac}}
        \end{tabular} & 
        \begin{tabular}{@{}c@{}}
            $G$\mbox{ is amenable } \\
            \mbox{\ref{OneDimL1ModMetTopInjFlatCharac}}
        \end{tabular} & 
        \begin{tabular}{@{}c@{}}
            $G$\mbox{ is compact } \\
            \mbox{\ref{OneDimL1ModMetTopProjCharac}},
            \mbox{\ref{MGMetTopProjInjFlatRedToL1}}
        \end{tabular} & 
        \begin{tabular}{@{}c@{}}
            $G$\mbox{ is amenable } \\
            \mbox{\ref{OneDimL1ModMetTopInjFlatCharac}},
            \mbox{\ref{MGMetTopProjInjFlatRedToL1}}
        \end{tabular} & 
        \begin{tabular}{@{}c@{}}
            $G$\mbox{ is amenable } \\
            \mbox{\ref{OneDimL1ModMetTopInjFlatCharac}},
            \mbox{\ref{MGMetTopProjInjFlatRedToL1}}
        \end{tabular} \\ 
    \hline
        \multicolumn{7}{c}{
            \mbox{
                Homologically trivial $L_1(G)$- and $M(G)$-modules 
                in topological theory
            }
        } \\
    \hline & 
        \multicolumn{3}{c|}{
            $L_1(G)$-modules
        } & 
        \multicolumn{3}{c|}{
            $M(G)$-modules
        } \\
    \hline & 
        \mbox{Projectivity} & 
        \mbox{Injectivity} & 
        \mbox{Flatness} & 
        \mbox{Projectivity} & 
        \mbox{Injectivity} & 
        \mbox{Flatness} \\ 
    \hline
        $L_1(G)$ & 
        \begin{tabular}{@{}c@{}}
            $G$\mbox{ is discrete } \\
            \mbox{\ref{L1ModL1MetTopProjCharac}}
        \end{tabular} & 
        \begin{tabular}{@{}c@{}}
            $G$\mbox{ is finite } \\
            \mbox{\ref{StdModAreNotL1MGMetTopProjInjFlat}},
            \mbox{\ref{StdModFinGrL1MGTopInjProjFlatCharac}}
        \end{tabular} & 
        \begin{tabular}{@{}c@{}}
            $G$\mbox{ is any } \\
            \mbox{\ref{LInfIsL1MetrInj}}
        \end{tabular} & 
        \begin{tabular}{@{}c@{}}
            $G$\mbox{ is discrete } \\
            \mbox{\ref{L1ModL1MetTopProjCharac}},
            \mbox{\ref{MGMetTopProjInjFlatRedToL1}}
        \end{tabular} & 
        \begin{tabular}{@{}c@{}}
            $G$\mbox{ is finite } \\
            \mbox{\ref{StdModAreNotL1MGMetTopProjInjFlat}},
            \mbox{\ref{StdModFinGrL1MGTopInjProjFlatCharac}}
        \end{tabular} & 
        \begin{tabular}{@{}c@{}}
            $G$\mbox{ is any } \\
            \mbox{\ref{LInfIsL1MetrInj}},
            \mbox{\ref{MGMetTopProjInjFlatRedToL1}}
        \end{tabular} \\ 
    \hline
        $L_p(G)$ & 
        \begin{tabular}{@{}c@{}}
            $G$\mbox{ is finite } \\
            \mbox{\ref{StdModAreNotL1MGMetTopProjInjFlat}},
            \mbox{\ref{StdModFinGrL1MGTopInjProjFlatCharac}}
        \end{tabular} & 
        \begin{tabular}{@{}c@{}}
            $G$\mbox{ is finite } \\
            \mbox{\ref{StdModAreNotL1MGMetTopProjInjFlat}},
            \mbox{\ref{StdModFinGrL1MGTopInjProjFlatCharac}}
        \end{tabular} & 
        \begin{tabular}{@{}c@{}}
            $G$\mbox{ is finite } \\
            \mbox{\ref{StdModAreNotL1MGMetTopProjInjFlat}},
            \mbox{\ref{StdModFinGrL1MGTopInjProjFlatCharac}}
        \end{tabular} & 
        \begin{tabular}{@{}c@{}}
            $G$\mbox{ is finite } \\
            \mbox{\ref{StdModAreNotL1MGMetTopProjInjFlat}},
            \mbox{\ref{StdModFinGrL1MGTopInjProjFlatCharac}}
        \end{tabular} & 
        \begin{tabular}{@{}c@{}}
            $G$\mbox{ is finite } \\
            \mbox{\ref{StdModAreNotL1MGMetTopProjInjFlat}},
            \mbox{\ref{StdModFinGrL1MGTopInjProjFlatCharac}}
        \end{tabular} & 
        \begin{tabular}{@{}c@{}}
            $G$\mbox{ is finite } \\
            \mbox{\ref{StdModAreNotL1MGMetTopProjInjFlat}},
            \mbox{\ref{StdModFinGrL1MGTopInjProjFlatCharac}}
        \end{tabular} \\ 
    \hline
        $L_\infty(G)$ &
        \begin{tabular}{@{}c@{}}
            $G$\mbox{ is finite } \\
            \mbox{\ref{StdModAreNotL1MGMetTopProjInjFlat}},
            \mbox{\ref{StdModFinGrL1MGTopInjProjFlatCharac}}
        \end{tabular} & 
        \begin{tabular}{@{}c@{}}
            $G$\mbox{ is any } \\
            \mbox{\ref{LInfIsL1MetrInj}}
        \end{tabular} & 
        \begin{tabular}{@{}c@{}}
            $G$\mbox{ is finite } \\
            \mbox{\ref{StdModAreNotL1MGMetTopProjInjFlat}},
            \mbox{\ref{StdModFinGrL1MGTopInjProjFlatCharac}}
        \end{tabular} & 
        \begin{tabular}{@{}c@{}}
            $G$\mbox{ is finite } \\
            \mbox{\ref{StdModAreNotL1MGMetTopProjInjFlat}},
            \mbox{\ref{StdModFinGrL1MGTopInjProjFlatCharac}}
        \end{tabular} & 
        \begin{tabular}{@{}c@{}}
            $G$\mbox{ is any } \\
            \mbox{\ref{LInfIsL1MetrInj}},
            \mbox{\ref{MGMetTopProjInjFlatRedToL1}}
        \end{tabular} & 
        \begin{tabular}{@{}c@{}}
            $G$\mbox{ is finite } \\
            \mbox{\ref{StdModAreNotL1MGMetTopProjInjFlat}},
            \mbox{\ref{StdModFinGrL1MGTopInjProjFlatCharac}}
        \end{tabular} \\ 
    \hline
        $M(G)$ & 
        \begin{tabular}{@{}c@{}}
            $G$\mbox{ is discrete } \\
            \mbox{\ref{L1MetTopProjAndMetrFlatOfMeasAlg}}
        \end{tabular} & 
        \begin{tabular}{@{}c@{}}
            $G$\mbox{ is finite } \\
            \mbox{\ref{StdModAreNotL1MGMetTopProjInjFlat}},
            \mbox{\ref{StdModFinGrL1MGTopInjProjFlatCharac}}
        \end{tabular} & 
        \begin{tabular}{@{}c@{}}
            $G$\mbox{ is any } \\
            \mbox{\ref{MeasAlgIsL1TopFlat}}
        \end{tabular} & 
        \begin{tabular}{@{}c@{}}
            $G$\mbox{ is any } \\
            \mbox{\ref{MGModMGMetTopProjFlatCharac}}
        \end{tabular} & 
        \begin{tabular}{@{}c@{}}
            $G$\mbox{ is finite } \\
            \mbox{\ref{StdModAreNotL1MGMetTopProjInjFlat}},
            \mbox{\ref{StdModFinGrL1MGTopInjProjFlatCharac}}
        \end{tabular} & 
        \begin{tabular}{@{}c@{}}
            $G$\mbox{ is any } \\
            \mbox{\ref{MGModMGMetTopProjFlatCharac}}
        \end{tabular} \\ 
    \hline
        $C_0(G)$ & 
        \begin{tabular}{@{}c@{}}
            $G$\mbox{ is finite } \\
            \mbox{\ref{StdModAreNotL1MGMetTopProjInjFlat}},
            \mbox{\ref{StdModFinGrL1MGTopInjProjFlatCharac}}
        \end{tabular} & 
        \begin{tabular}{@{}c@{}}
            $G$\mbox{ is finite } \\
            \mbox{\ref{StdModAreNotL1MGMetTopProjInjFlat}},
            \mbox{\ref{StdModFinGrL1MGTopInjProjFlatCharac}}
        \end{tabular} & 
        \begin{tabular}{@{}c@{}}
            $G$\mbox{ is finite } \\
            \mbox{\ref{StdModAreNotL1MGMetTopProjInjFlat}},
            \mbox{\ref{StdModFinGrL1MGTopInjProjFlatCharac}}
        \end{tabular} & 
        \begin{tabular}{@{}c@{}}
            $G$\mbox{ is finite } \\
            \mbox{\ref{StdModAreNotL1MGMetTopProjInjFlat}},
            \mbox{\ref{StdModFinGrL1MGTopInjProjFlatCharac}}
        \end{tabular} & 
        \begin{tabular}{@{}c@{}}
            $G$\mbox{ is finite } \\
            \mbox{\ref{StdModAreNotL1MGMetTopProjInjFlat}},
            \mbox{\ref{StdModFinGrL1MGTopInjProjFlatCharac}}
        \end{tabular} & 
        \begin{tabular}{@{}c@{}}
            $G$\mbox{ is finite } \\
            \mbox{\ref{StdModAreNotL1MGMetTopProjInjFlat}},
            \mbox{\ref{StdModFinGrL1MGTopInjProjFlatCharac}}
        \end{tabular} \\ 
    \hline
        $\mathbb{C}_\gamma$ & 
        \begin{tabular}{@{}c@{}}
            $G$\mbox{ is compact } \\
            \mbox{\ref{OneDimL1ModMetTopProjCharac}}
        \end{tabular} & 
        \begin{tabular}{@{}c@{}}
            $G$\mbox{ is amenable } \\
            \mbox{\ref{OneDimL1ModMetTopInjFlatCharac}}
        \end{tabular} & 
        \begin{tabular}{@{}c@{}}$G$\mbox{ is amenable } \\
            \mbox{\ref{OneDimL1ModMetTopInjFlatCharac}}
        \end{tabular} & 
        \begin{tabular}{@{}c@{}}
            $G$\mbox{ is compact } \\
            \mbox{\ref{OneDimL1ModMetTopProjCharac}},
            \mbox{\ref{MGMetTopProjInjFlatRedToL1}}
        \end{tabular} & 
        \begin{tabular}{@{}c@{}}
            $G$\mbox{ is amenable } \\
            \mbox{\ref{OneDimL1ModMetTopInjFlatCharac}},
            \mbox{\ref{MGMetTopProjInjFlatRedToL1}}
        \end{tabular} & 
        \begin{tabular}{@{}c@{}}
            $G$\mbox{ is amenable } \\
            \mbox{\ref{OneDimL1ModMetTopInjFlatCharac}},
            \mbox{\ref{MGMetTopProjInjFlatRedToL1}}
        \end{tabular} \\ 
    \hline
        \multicolumn{7}{c}{
            \mbox{
                Homologically trivial $L_1(G)$- and $M(G)$-modules 
                in relative theory
            }
        } \\
    \hline & 
    \multicolumn{3}{c|}{
        $L_1(G)$-modules
    } & 
    \multicolumn{3}{c|}{
        $M(G)$-modules
    } \\
    \hline & 
        \mbox{Projectivity} & 
        \mbox{Injectivity} & 
        \mbox{Flatness} & 
        \mbox{Projectivity} & 
        \mbox{Injectivity} & 
        \mbox{Flatness} \\ 
    \hline 
        $L_1(G)$ & 
        \begin{tabular}{@{}c@{}}
            $G$\mbox{ is any } \\
            \mbox{\cite{DalPolHomolPropGrAlg}, \S 6}
        \end{tabular} & 
        \begin{tabular}{@{}c@{}}
            $G$\mbox{ is amenable } \\
            \mbox{ and discrete } \\
            \mbox{\cite{DalPolHomolPropGrAlg}, \S 6}
        \end{tabular} & 
        \begin{tabular}{@{}c@{}}
            $G$\mbox{ is any } \\
            \mbox{\cite{DalPolHomolPropGrAlg}, \S 6}
        \end{tabular} &
        \begin{tabular}{@{}c@{}}
            $G$\mbox{ is any } \\
            \mbox{\cite{RamsHomPropSemgroupAlg}, \S 3.5}
        \end{tabular} &
        \begin{tabular}{@{}c@{}}
            $G$\mbox{ is amenable } \\
            \mbox{ and discrete } \\
            \mbox{\cite{RamsHomPropSemgroupAlg}, \S 3.5}
        \end{tabular} &
        \begin{tabular}{@{}c@{}}
            $G$\mbox{ is any } \\
            \mbox{\cite{RamsHomPropSemgroupAlg}, \S 3.5}
        \end{tabular} \\ 
    \hline
        $L_p(G)$ & 
        \begin{tabular}{@{}c@{}}
            $G$\mbox{ is compact } \\
            \mbox{\cite{DalPolHomolPropGrAlg}, \S 6}
        \end{tabular} & 
        \begin{tabular}{@{}c@{}}
            $G$\mbox{ is amenable } \\
            \mbox{\cite{RachInjModAndAmenGr}}
        \end{tabular} & 
        \begin{tabular}{@{}c@{}}
            $G$\mbox{ is amenable } \\
            \mbox{\cite{RachInjModAndAmenGr}}
        \end{tabular} & 
        \begin{tabular}{@{}c@{}}
            $G$\mbox{ is compact } \\
            \mbox{\cite{RamsHomPropSemgroupAlg}, \S 3.5}
        \end{tabular} & 
        \begin{tabular}{@{}c@{}}
            $G$\mbox{ is amenable } \\
            \mbox{\cite{RamsHomPropSemgroupAlg}, \S 3.5},
            \mbox{\cite{RachInjModAndAmenGr}}
        \end{tabular} &
        \begin{tabular}{@{}c@{}}
            $G$\mbox{ is amenable } \\
            \mbox{\cite{RamsHomPropSemgroupAlg}, \S 3.5}
        \end{tabular} \\
    \hline
        $L_\infty(G)$ & 
        \begin{tabular}{@{}c@{}}
            $G$\mbox{ is finite } \\
            \mbox{\cite{DalPolHomolPropGrAlg}, \S 6}
        \end{tabular} & 
        \begin{tabular}{@{}c@{}}
            $G$\mbox{ is any } \\
            \mbox{\cite{DalPolHomolPropGrAlg}, \S 6}
        \end{tabular} & 
        \begin{tabular}{@{}c@{}}
            $G$\mbox{ is amenable } \\
            \mbox{\cite{DalPolHomolPropGrAlg}, \S 6}
        \end{tabular} & 
        \begin{tabular}{@{}c@{}}
            $G$\mbox{ is finite } \\
            \mbox{\cite{RamsHomPropSemgroupAlg}, \S 3.5}
        \end{tabular} & 
        \begin{tabular}{@{}c@{}}
            $G$\mbox{ is any } \\
            \mbox{\cite{RamsHomPropSemgroupAlg}, \S 3.5}
        \end{tabular} & 
        \begin{tabular}{@{}c@{}}
            $G$\mbox{ is amenable } \\ 
            \mbox{\cite{RamsHomPropSemgroupAlg}, \S 3.5}${}^{*}$
        \end{tabular} \\ 
    \hline
        $M(G)$ & 
        \begin{tabular}{@{}c@{}}
            $G$\mbox{ is discrete } \\
            \mbox{\cite{DalPolHomolPropGrAlg}, \S 6}
        \end{tabular} & 
        \begin{tabular}{@{}c@{}}
            $G$\mbox{ is amenable }\\
            \mbox{\cite{DalPolHomolPropGrAlg}, \S 6}
        \end{tabular} & 
        \begin{tabular}{@{}c@{}}
            $G$\mbox{ is any } \\
            \mbox{\cite{RamsHomPropSemgroupAlg}, \S 3.5}
        \end{tabular} & 
        \begin{tabular}{@{}c@{}}
            $G$\mbox{ is any } \\ 
            \mbox{\cite{RamsHomPropSemgroupAlg}, \S 3.5}
        \end{tabular} & 
        \begin{tabular}{@{}c@{}}
            $G$\mbox{ is amenable } \\
            \mbox{\cite{RamsHomPropSemgroupAlg}, \S 3.5}
        \end{tabular} & 
        \begin{tabular}{@{}c@{}}
            $G$\mbox{ is any } \\
            \mbox{\cite{RamsHomPropSemgroupAlg}, \S 3.5}
        \end{tabular} \\ 
    \hline 
        $C_0(G)$ & 
        \begin{tabular}{@{}c@{}}
            $G$\mbox{ is compact } \\ 
            \mbox{\cite{DalPolHomolPropGrAlg}, \S 6}
        \end{tabular} & 
        \begin{tabular}{@{}c@{}}
            $G$\mbox{ is finite } \\ 
            \mbox{\cite{DalPolHomolPropGrAlg}, \S 6}
        \end{tabular} & 
        \begin{tabular}{@{}c@{}}
            $G$\mbox{ is amenable } \\ 
            \mbox{\cite{DalPolHomolPropGrAlg}, \S 6}
        \end{tabular} & 
        \begin{tabular}{@{}c@{}}
            $G$\mbox{ is compact } \\ 
            \mbox{\cite{RamsHomPropSemgroupAlg}, \S 3.5}
        \end{tabular} & 
        \begin{tabular}{@{}c@{}}
            $G$\mbox{ is finite } \\
            \mbox{\cite{RamsHomPropSemgroupAlg}, \S 3.5}
        \end{tabular} & 
        \begin{tabular}{@{}c@{}}
            $G$\mbox{ is amenable } \\
            \mbox{\cite{RamsHomPropSemgroupAlg}, \S 3.5}
        \end{tabular} \\ 
    \hline
        $\mathbb{C}_\gamma$ & 
        \begin{tabular}{@{}c@{}}
            $G$\mbox{ is compact } \\
            \mbox{\ref{OneDimL1ModMetTopProjCharac}}
        \end{tabular} & 
        \begin{tabular}{@{}c@{}}
            $G$\mbox{ is amenable } \\
            \mbox{\ref{OneDimL1ModMetTopInjFlatCharac}}
        \end{tabular} & 
        \begin{tabular}{@{}c@{}}
            $G$\mbox{ is amenable } \\
            \mbox{\ref{OneDimL1ModMetTopInjFlatCharac}}
        \end{tabular} & 
        \begin{tabular}{@{}c@{}}
            $G$\mbox{ is compact } \\
            \mbox{\ref{OneDimL1ModMetTopProjCharac}},
            \mbox{\ref{MGMetTopProjInjFlatRedToL1}}
        \end{tabular} & 
        \begin{tabular}{@{}c@{}}
            $G$\mbox{ is amenable } \\
            \mbox{\ref{OneDimL1ModMetTopInjFlatCharac}},
            \mbox{\ref{MGMetTopProjInjFlatRedToL1}}
        \end{tabular} & 
        \begin{tabular}{@{}c@{}}
            $G$\mbox{ is amenable } \\
            \mbox{\ref{OneDimL1ModMetTopInjFlatCharac}},
            \mbox{\ref{MGMetTopProjInjFlatRedToL1}}
        \end{tabular} \\
    \hline
    \end{longtable}
\end{scriptsize}

%% file: main.bbl
\begin{thebibliography}{68}
\providecommand{\natexlab}[1]{#1}
\providecommand{\url}[1]{\texttt{#1}}
\expandafter\ifx\csname urlstyle\endcsname\relax
  \providecommand{\doi}[1]{doi: #1}\else
  \providecommand{\doi}{doi: \begingroup \urlstyle{rm}\Url}\fi

\bibitem[Helemskii(2006)]{HelLectAndExOnFuncAn}
A.Ya. Helemskii.
\newblock \emph{Lectures and exercises on functional analysis}, volume 233.
\newblock American Mathematical Society Providence, RI, 2006.

\bibitem[Kashiwara and Schapira(2006)]{KashivShapCatsAndSheavs}
M.~Kashiwara and P.~Schapira.
\newblock \emph{Categories and sheaves}, volume 332.
\newblock Springer, 2006.

\bibitem[Engelking(1989)]{EngelGenTop}
R.~Engelking.
\newblock \emph{General topology}.
\newblock Heldermann, 1989.

\bibitem[Bourbaki(1963)]{BourbElemMathGenTopLivIII}
N.~Bourbaki.
\newblock \emph{{\'E}l{\'e}ments de math{\'e}matique. Nouvelle {\'e}dition. Les
  Structures fondamentales de l'analyse.[Livre III.] Topologie
  g{\'e}n{\'e}rale.}
\newblock Hermann, 1963.

\bibitem[Fremlin(2003)]{FremMeasTh}
D.H. Fremlin.
\newblock \emph{Measure Theory, Vol. 1-5}.
\newblock Torres Fremlin, 2003.

\bibitem[Conway(1985)]{ConwACoursInFuncAn}
J.B. Conway.
\newblock \emph{A course in functional analysis}.
\newblock Springer-Verlag, 1985.

\bibitem[Carothers(2005)]{CarothShortCourseBanSp}
N.L. Carothers.
\newblock \emph{A short course on Banach space theory}, volume~64.
\newblock Cambridge University Press, 2005.

\bibitem[Albiac and Kalton(2006)]{KalAlbTopicsBanSpTh}
F.~Albiac and N.J. Kalton.
\newblock \emph{Topics in Banach space theory}, volume 233.
\newblock Springer, 2006.

\bibitem[Fabian and Habala(2011)]{FabHabBanSpTh}
M.~Fabian and P.~Habala.
\newblock \emph{Banach space theory}.
\newblock Springer, 2011.

\bibitem[Diestel et~al.(2008)Diestel, Fourie, and J.]{DiestMetTheoryOfTensProd}
J.~Diestel, J.H. Fourie, and Swart J.
\newblock \emph{The metric theory of tensor products: Grothendieck's
  r{\'e}sum{\'e} revisited}.
\newblock American Mathematical Society, 2008.

\bibitem[Dales et~al.(2012)Dales, Lau, and Strauss]{DalLauSecondDualOfMeasAlg}
H.G. Dales, A.T.-M. Lau, and D.~Strauss.
\newblock Second duals of measure algebras.
\newblock \emph{Dissertationes Math. (Rozprawy Mat.)}, 481:\penalty0 1--121,
  2012.

\bibitem[Defant and Floret(1992)]{DefFloTensNorOpId}
A.~Defant and K.~Floret.
\newblock \emph{Tensor norms and operator ideals}, volume 176.
\newblock Elsevier, 1992.

\bibitem[J. and A.(1968)]{LinPelAbsSumOpInLpSpAndApp}
Lindenstrauss J. and Pelczynski A.
\newblock Absolutely summing operators in $\mathscr{L}_p$-spaces and their
  applications.
\newblock \emph{Stud. Math.}, 29\penalty0 (3):\penalty0 275--326, 1968.

\bibitem[Wojtaszczyk(1996)]{WojBanSpForAnalysts}
P.~Wojtaszczyk.
\newblock \emph{Banach spaces for analysts}, volume~25.
\newblock Cambridge University Press, 1996.

\bibitem[Grothendieck(1953)]{GrothApllFaiblCompSpCK}
A.~Grothendieck.
\newblock Sur les applications lin{\'e}aires faiblement compactes d'espaces du
  type $\it{C(K)}$.
\newblock \emph{Canad. J. Math.}, 5\penalty0 (1953):\penalty0 129--173, 1953.

\bibitem[Megginson(1998)]{MeggIntroBanSpTh}
R.E. Megginson.
\newblock \emph{An introduction to Banach space theory}, volume 183.
\newblock Springer, 1998.

\bibitem[Lacey(1974)]{LaceyIsomThOfClassicBanSp}
H.E. Lacey.
\newblock \emph{The isometric theory of classical Banach spaces}.
\newblock Springer, 1974.

\bibitem[Sherman(1951)]{SherOrderInOpAlg}
S.~Sherman.
\newblock Order in operator algebras.
\newblock \emph{Amer. J. Math.}, pages 227--232, 1951.

\bibitem[Kadison(1951)]{KadOrderPropOfBoundSAOps}
R.V. Kadison.
\newblock Order properties of bounded self-adjoint operators.
\newblock \emph{Proc. Amer. Math. Soc.}, 2\penalty0 (3):\penalty0 505--510,
  1951.

\bibitem[Diestel et~al.(1995)Diestel, Jarchow, and Tonge]{DiestAbsSumOps}
J.~Diestel, H.~Jarchow, and A.~Tonge.
\newblock \emph{Absolutely summing operators}, volume~43.
\newblock Cambridge University Press, 1995.

\bibitem[Grothendieck(1955{\natexlab{a}})]{GrothProdTenTopNucl}
A.~Grothendieck.
\newblock Produits tensoriels topologiques et espaces nucl{\'e}aires.
\newblock \emph{S{\'e}minaire Bourbaki}, 2:\penalty0 193--200,
  1955{\natexlab{a}}.

\bibitem[Ryan(2013)]{RyanIntroTensNormsBanSp}
R.A. Ryan.
\newblock \emph{Introduction to tensor products of Banach spaces}.
\newblock Springer Science \& Business Media, 2013.

\bibitem[Helemskii(1993)]{HelBanLocConvAlg}
A.Ya. Helemskii.
\newblock \emph{Banach and locally convex algebras}.
\newblock Oxford University Press, 1993.

\bibitem[Helemskii(1989)]{HelHomolBanTopAlg}
A.Ya. Helemskii.
\newblock \emph{The homology of Banach and topological algebras}, volume~41.
\newblock Springer, 1989.

\bibitem[Dales(2000)]{DalBanAlgAutCont}
H.G. Dales.
\newblock \emph{Banach algebras and automatic continuity}.
\newblock Clarendon Press, 2000.

\bibitem[Kaniuth(2009)]{KaniBanAlg}
E.~Kaniuth.
\newblock \emph{A course in commutative Banach algebras}, volume 246.
\newblock Springer, 2009.

\bibitem[Blackadar(2006)]{BlackadarOpAlg}
B.~Blackadar.
\newblock Operator algebras.
\newblock \emph{Encyclopaedia Math. Sci.}, 122, 2006.

\bibitem[Murphy(2014)]{MurphyCStarAlgsAndOpTh}
G.J. Murphy.
\newblock \emph{$\it{C^*}$-algebras and operator theory}.
\newblock Academic Press, 2014.

\bibitem[White(1996)]{WhiteInjmoduAlg}
M.C. White.
\newblock Injective modules for uniform algebras.
\newblock \emph{Proc. Lond. Math. Soc.}, 3\penalty0 (1):\penalty0 155--184,
  1996.

\bibitem[Helemskii(2013)]{HelMetrFrQMod}
A.Ya. Helemskii.
\newblock Metric freeness and projectivity for classical and quantum normed
  modules.
\newblock \emph{Sb. Mat.}, 204\penalty0 (7):\penalty0 1056--1083, 2013.

\bibitem[K{\"o}the(1966)]{KotheTopProjBanSp}
G.~K{\"o}the.
\newblock Hebbare lokalkonvexe r{\"a}ume.
\newblock \emph{Math. Ann.}, 165\penalty0 (3):\penalty0 181--195, 1966.

\bibitem[Grothendieck(1955{\natexlab{b}})]{GrothMetrProjFlatBanSp}
A.~Grothendieck.
\newblock Une caract{\'e}risation vectorielle-m{\'e}trique des espaces
  $\it{L}_1$.
\newblock \emph{Canad. J. Math.}, 7:\penalty0 552--561, 1955{\natexlab{b}}.

\bibitem[Ramsden(2009)]{RamsHomPropSemgroupAlg}
P.~Ramsden.
\newblock \emph{Homological properties of semigroup algebras}.
\newblock PhD thesis, The University of Leeds, 2009.

\bibitem[Dales(2003)]{DalesIntroBanAlgOpHarmAnal}
H.G. Dales.
\newblock \emph{Introduction to Banach algebras, operators, and harmonic
  analysis}, volume~57.
\newblock Cambridge University Press, 2003.

\bibitem[Johnson and J.(2001)]{JohnLinHandbookGeomBanSp}
W.B. Johnson and Lindenstrauss J.
\newblock \emph{Handbook of the geometry of Banach spaces}, volume~2.
\newblock Elsevier, 2001.

\bibitem[Helemskii(2011)]{HelMetrFlatNorMod}
A.Ya. Helemskii.
\newblock Metric version of flatness and hahn-banach type theorems for normed
  modules over sequence algebras.
\newblock \emph{Stud. Math.}, 206\penalty0 (2):\penalty0 135--160, 2011.

\bibitem[Blecher and Ozawa(2015)]{PosAndApproxIdinBanAlg}
D.P. Blecher and N.~Ozawa.
\newblock Real positivity and approximate identities in banach algebras.
\newblock \emph{Pacific J. Math.}, 277\penalty0 (1):\penalty0 1--59, 2015.

\bibitem[Doran and Wichmann(1979)]{AppIdAndFactorInBanAlg}
R.S. Doran and J.~Wichmann.
\newblock \emph{Approximate identities and factorization in Banach modules}.
\newblock Springer, 1979.

\bibitem[Blecher and Read(2011)]{BleContrAppIdInOpAlg}
D.P. Blecher and C.J. Read.
\newblock Operator algebras with contractive approximate identities.
\newblock \emph{J. Funct. Anal.}, 261\penalty0 (1):\penalty0 188--217, 2011.

\bibitem[Graven(1979)]{GravInjProjBanMod}
A.W.M. Graven.
\newblock Injective and projective banach modules.
\newblock In \emph{Indag. Math. (Proceedings)}, volume~82, pages 253--272.
  Elsevier, 1979.

\bibitem[Bourgain(1981)]{BourgOnTheDPP}
J.~Bourgain.
\newblock On the dunford-pettis property.
\newblock \emph{Proc. Amer. Math. Soc.}, 81\penalty0 (2):\penalty0 265--272,
  1981.

\bibitem[Dales and Polyakov(2004)]{DalPolHomolPropGrAlg}
H.G. Dales and M.E. Polyakov.
\newblock Homological properties of modules over group algebras.
\newblock \emph{Proc. Lond. Math. Soc.}, 89\penalty0 (2):\penalty0 390--426,
  2004.

\bibitem[Racher(2013)]{RachInjModAndAmenGr}
G.~Racher.
\newblock Injective modules and amenable groups.
\newblock \emph{Comment. Math. Helv.}, 88\penalty0 (4):\penalty0 1023--1031,
  2013.

\bibitem[Pier(1984)]{PierAmenLCA}
J.-P. Pier.
\newblock \emph{Amenable locally compact groups}.
\newblock Wiley-Interscience, 1984.

\bibitem[Rosenthal(1966)]{RosProjTransInvSbspLpG}
H.P. Rosenthal.
\newblock \emph{Projections onto translation-invariant subspaces of $L_p(G)$}.
\newblock American Mathematical Society, 1966.

\bibitem[Blecher and Kania(2014)]{BleKanFinGenCStarAlgHilbMod}
D.P. Blecher and T.~Kania.
\newblock Finite generation in $\it{C^*}$-algebras and hilbert
  $\it{C^*}$-modules.
\newblock \emph{Stud. Math.}, 224\penalty0 (2):\penalty0 143--151, 2014.

\bibitem[Cushing and Lykova(2013)]{LykProjOfBanAndCStarAlgsOfContFld}
D.~Cushing and Z.A. Lykova.
\newblock Projectivity of banach and $\it{C^*}$-algebras of continuous fields.
\newblock \emph{Q. J. Math.}, 64\penalty0 (2):\penalty0 341--371, 2013.

\bibitem[Kaplansky(1951)]{KaplProjInBanAlg}
I.~Kaplansky.
\newblock Projections in banach algebras.
\newblock \emph{Ann. of Math.}, pages 235--249, 1951.

\bibitem[Berberian(1972)]{BerbBaerStarRings}
S.K. Berberian.
\newblock \emph{Baer $*$-Rings: Reprint of the 1972 Edition with Errata List
  and Later Developments Indicated}, volume 195.
\newblock Springer Science \& Business Media, 1972.

\bibitem[Hamana(1978)]{HamInjEnvBanMod}
M.~Hamana.
\newblock Injective envelopes of banach modules.
\newblock \emph{T{\^o}hoku Math. J. (2)}, 30\penalty0 (3):\penalty0 439--453,
  1978.

\bibitem[M.(1960)]{TakHanBanThAndJordDecomOfModMap}
Takesaki M.
\newblock On the hahn-banach type theorem and the jordan decomposition of
  module linear mapping over some operator algebras.
\newblock In \emph{Kodai Mathematical Seminar Reports}, volume~12, pages 1--10.
  Tokyo Institute of Technology, Department of Mathematics, 1960.

\bibitem[Lau et~al.(1996)Lau, Loy, and
  Willis]{LauLoyWillisAmnblOfBanAndCStarAlgsOfLCG}
A.T.-M. Lau, R.J. Loy, and G.A. Willis.
\newblock Amenability of banach and $\it{C^*}$-algebras on locally compact
  groups.
\newblock \emph{Studia Mathematica}, 119\penalty0 (2):\penalty0 161--178, 1996.

\bibitem[Gordon and Lewis(1974)]{GorLewAbsSmOpAndLocUncondStrct}
Y.~Gordon and D.R. Lewis.
\newblock Absolutely summing operators and local unconditional structures.
\newblock \emph{Acta Math.}, 133\penalty0 (1):\penalty0 27--48, 1974.

\bibitem[Brown and Ozawa(2008)]{BroOzaCStarAlgFinDimApprox}
N.P. Brown and N.~Ozawa.
\newblock \emph{$\it{C^*}$-algebras and finite-dimensional approximations},
  volume~88.
\newblock American Mathematical Society, 2008.

\bibitem[Haagerup(1983)]{HaaNucCStarAlgAmen}
U.~Haagerup.
\newblock All nuclear $\it{C^*}$-algebras are amenable.
\newblock \emph{Invent. Math.}, 74\penalty0 (2):\penalty0 305--319, 1983.

\bibitem[Runde(2006)]{RundeAmenConstFour}
V.~Runde.
\newblock The amenability constant of the fourier algebra.
\newblock \emph{Proc. Amer. Math. Soc.}, 134\penalty0 (5):\penalty0 1473--1481,
  2006.

\bibitem[Smith and Williams(1986)]{SmithDecompPropCStarAlg}
R.R. Smith and D.P. Williams.
\newblock The decomposition property for $\it{C^*}$-algebras.
\newblock \emph{J. Operator Theory}, 16:\penalty0 51--74, 1986.

\bibitem[Sakai(1964)]{SakWeakCompOpOnOpAlg}
S.~Sakai.
\newblock Weakly compact operators on operator algebras.
\newblock \emph{Pacific J. Math.}, 14\penalty0 (2):\penalty0 659--664, 1964.

\bibitem[Davidson(1996)]{DavCSatrAlgByExmpl}
K.R. Davidson.
\newblock \emph{$\it{C^*}$-algebras by example}, volume~6.
\newblock American Mathematical Society, 1996.

\bibitem[Lyubich and Shatalova(2005)]{LyubIsomEmdbFinDimLp}
Yu.I. Lyubich and O.A. Shatalova.
\newblock Isometric embeddings of finite-dimensional $\ell_p$-spaces over the
  quaternions.
\newblock \emph{St. Petersburg Math. J.}, 16\penalty0 (1):\penalty0 9--24,
  2005.

\bibitem[Takesaki(1979)]{TakThOpAlgVol1}
M.~Takesaki.
\newblock \emph{Theory of operator algebras}, volume~I.
\newblock Springer-Verlag, 1979.

\bibitem[Nemesh(2022{\natexlab{a}})]{NemANoteOnRelInjC0ModC0}
N.T. Nemesh.
\newblock A note on relatively injective $\it{C_0(S)}$-modules $\it{C_0(S)}$.
\newblock \emph{Funct. Anal. Appl.}, 55:\penalty0 298--303, 2022{\natexlab{a}}.

\bibitem[Bourbaki(2004)]{BourbElemMathIntegLivVI}
N.~Bourbaki.
\newblock Integration. i. chapters 1--6. elements of mathematics, 2004.

\bibitem[Nemesh(2022{\natexlab{b}})]{NemRelProjModLp}
N.T. Nemesh.
\newblock Relative projectivity of modules $\it{L_p}$.
\newblock \emph{Math. Notes}, 111:\penalty0 103--114, 2022{\natexlab{b}}.

\bibitem[Hewit and Ross(1979)]{HewRossAbstrHarmAnalVol1}
E.~Hewit and K.A. Ross.
\newblock \emph{Abstract harmonic analysis}, volume~1.
\newblock Springer-Verlag, 1979.

\bibitem[Wendel(1952)]{WendLeftCentrzrs}
J.G. Wendel.
\newblock Left centralizers and isomorphisms of group algebras.
\newblock \emph{Pacific J. Math.}, 2\penalty0 (3):\penalty0 251--261, 1952.

\bibitem[Rosenthal(1970)]{RosOnRelDisjFamOfMeas}
H.~Rosenthal.
\newblock On relatively disjoint families of measures, with some applications
  to banach space theory.
\newblock \emph{Stud. Math.}, 37\penalty0 (1):\penalty0 13--36, 1970.

\bibitem[Lau and Losert(1990)]{LauMingComplSubspInLInfOfG}
A.T.-M. Lau and V.~Losert.
\newblock Complementation of certain subspaces of $\it{L_\infty(G)}$ of a
  locally compact group.
\newblock \emph{Pacific J. Math.}, 141\penalty0 (2):\penalty0 295--310, 1990.

\end{thebibliography}
